\documentclass[DIV=12,numbers=enddot,leqno,bibliography=totoc]{scrartcl}
\usepackage{./q-HodgeStyle_minimal}
\title{\texorpdfstring{$q$}{q}-de Rham cohomology and topological Hochschild homology over \texorpdfstring{$\ku$}{ku}}
\author{Ferdinand Wagner}

\newcommand{\BMSev}{\mathrm{BMS}\hspace{-0.1ex}\mhyph\hspace{-0.1ex}\mathrm{ev}}
\newcommand{\HRWev}{\mathrm{HRW}\hspace{-0.2ex}\mhyph\hspace{-0.1ex}\mathrm{ev}}
\newcommand{\Pev}{\mathrm{P}\hspace{-0.1ex}\mhyph\hspace{-0.1ex}\mathrm{ev}}

\newcommand{\ISPhat}{\IS_{\smash{\widehat{P}_p}\vphantom{_p}}}
\newcommand{\ISPhatbullet}{\IS_{\smash{\widehat{P}_p^\bullet}\vphantom{_p}}}

\newcommand{\Cyclonic}{\cat{CycnSp}}

\newcommand{\TCn}[1]{\TC^{-(#1)}}

\begin{document}
	\maketitle
	
	\begin{abstract}
		\textbf{Abstract. --- }
		Hodge-filtered derived de Rham cohomology of a ring $R$ can be described (up to completion and shift) as the graded pieces of the even filtration on $\HC^-(R)$. In this paper we show a deformation of this result: If $R$ admits a spherical $\IE_2$-lift, then the graded pieces of the even filtration on $\TC^-(\ku\otimes\IS_R/\ku)$ form a certain filtration on the $q$-de Rham cohomology of $R$, which $q$-deforms the Hodge filtration.
		
		We also explain how the associated \emph{Habiro--Hodge complex} from \cite{qWittHabiro} can be described in terms of the genuine equivariant structure on $\THH(\KU\otimes\IS_R/\KU)$. As a special case, we'll obtain homotopy-theoretic construction of the \emph{Habiro ring of a number field} from \cite{HabiroRingOfNumberField}.
	\end{abstract}

	\tableofcontents
	\renewcommand{\SectionPrefix}{\textrm{\S}}
	\renewcommand{\SubsectionPrefix}{\textrm{\S}}
	
	\newpage

	\section{Introduction}\label{sec:Intro}
	
	Let $\ku$ denote the connective complex $K$-theory spectrum. The ring $\pi_0(\ku^{\h S^1})\cong \ku^0(\B S^1)$ contains a canonical element $q$, which corresponds to the standard representation of $S^1$ acting via rotations on $\IC$. In this paper we explain that this is the \enquote{same $q$} as in $q$-de Rham cohomology.
	
	\subsection{\texorpdfstring{$q$}{q}-Hodge filtrations from \texorpdfstring{$\THH$}{THH} over \texorpdfstring{$\ku$}{ku}}
	
	Many interesting cohomology theories in arithmetic geometry can be obtained as graded pieces of motivic filtrations on localising invariants. In the case of de Rham cohomology, this is particularly well understood: The corresponding localising invariant is given by \emph{Hochschild homology} (and its cousins, negative cyclic and periodic homology), the motivic filtration is given by the \emph{even filtration} of Hahn--Raksit--Wilson \cite{EvenFiltration}.
	
	More precisely, combining \cite[Theorem~\chref{1.1}]{AntieauHC-} and \cite[Theorem~\chref{5.0.2}]{EvenFiltration}, one has the following result:
	
	\begin{thm}[Antieau--Hahn--Raksit--Wilson]
		Let $R$ be a quasi-syntomic ring. Then completion of the Hodge-filtrered derived de Rham complex of $R$ agrees \embrace{up to shift} with the the graded pieces of the $S^1$-equivariant even filtration on $\HC^-(R/\IZ)$:
		\begin{equation*}
			\fil_{\Hodge}^\star\hatdeRham_{R/\IZ}\simeq\Sigma^{-2*}\gr_{\ev,\h S^1}^*\HC^-(R/\IZ)\,.
		\end{equation*}
	\end{thm}
	The goal of this paper is to show a deformation of this theorem, in which $\IZ$ gets deformed to $\ku$ and de Rham cohomology gets deformed to $q$-de Rham cohomology.
	
	\begin{thm}[see \cref{thm:qdeRhamkuGlobal}]\label{thm:qdeRhamkuIntro}
		Let $R$ be a quasi-syntomic ring such that $2\in R^\times$. Assume that $R$ admits a lift to a connective $\IE_2$-ring spectrum $\IS_R$ such that $\IS_R\otimes\IZ\simeq R$. Then the derived $q$-de Rham complex $\qdeRham_{R/\IZ}$ can be equipped with a $q$-deformation of the Hodge filtration $\fil_{\qHodge}^\star\qdeRham_{R/\IZ}$, and the completion of this filtration agrees \embrace{up to shift} with the graded pieces of the $S^1$-equivariant even filtration on $\TC^-(\ku\otimes\IS_R/\ku)$:
		\begin{equation*}
			\fil_{\qHodge}^\star\qhatdeRham_{R/\IZ}\simeq\Sigma^{-2*}\gr_{\ev,\h S^1}^*\TC^-(\ku\otimes\IS_R/\ku)\,.
		\end{equation*}
	\end{thm}
	
	Before we discuss \cref{thm:qdeRhamkuIntro}, let us comment on the origins of the result and highlight some of the preceding work of Arpon Raksit and Sanath Devalapurkar that this theorem crucially relies on.
	
	\begin{numpar}[Relation to work of Raksit.]
		In the case where $\IS_R$ is the flat spherical polynomial ring $\IS[x]$, \cref{thm:Raksit} was first shown in unpublished work of Raksit, who also gave an explicit description of the $q$-deformed Hodge filtration in this case:
		\begin{thm}[Raksit, unpublished; see \cref{thm:Raksit}]\label{thm:RaksitIntro}
			Let $\qOmega_{\IZ[x]/\IZ, \square}^*$ be the coordinate-dependent $q$-de Rham complex, where the choice of coordinates is the identity $\square\colon \IZ[x]\rightarrow \IZ[x]$. Equip $\qOmega_{\IZ[x]/\IZ, \square}^*$ with the filtration by subcomplexes $\fil_{\qHodge,\square}^0\qOmega_{\IZ[x]/\IZ, \square}^*\coloneqq \qOmega_{\IZ[x]/\IZ, \square}^*$ and
			\begin{equation*}
				\fil_{\qHodge,\square}^i\qOmega_{\IZ[x]/\IZ, \square}^*\coloneqq \Bigl((q-1)^i\IZ[x]\qpower\overset{\q\nabla}{\longrightarrow}(q-1)^{i-1}\IZ[x]\qpower\d x\Bigr)
			\end{equation*}
			for all $i\geqslant 1$. Then
			\begin{equation*}
				\fil_{\qHodge,\square}^i\qOmega_{\IZ[x]/\IZ, \square}^*\simeq \Sigma^{-2i}\gr_{\ev,\h S^1}^i\TC^-\bigl(\ku[x]/\ku\bigr)\,.
			\end{equation*}
		\end{thm}
		To the author's knowledge, Raksit's result marks the discovery of the relation between $q$-de Rham cohomology and $\TC^-(-/\ku)$.
		
		Moreover, Raksit shows a generalisation of \cref{thm:RaksitIntro}, in which $\ku$ can be replaced by the connecitve cover $\tau_{\geqslant 0}E$ of any $2$-periodic $\IE_\infty$-ring spectrum $E$. The result is a version of the $q$-de Rham complex, in which the differentials send $x^m\mapsto \langle m\rangle_E x^{m-1}\d x$, where $\langle m\rangle_E$ denotes the reduced $m$-series of the formal group of $E$. It would be very interesting to see whether such a variant of $q$-de Rham cohomology can be made coordinate-independent. Some speculation about the (im-)possibility of this can be found in \cite[\S{\chpageref[4.3]{32}}]{DevalapurkarMisterka}.
	\end{numpar}

	\begin{numpar}[Relation to work of Devalapurkar.]
		The crucial input in our proof of \cref{thm:qdeRhamkuIntro} is the following theorem that was conjectured by Lurie and Nikolaus (for all~$p$) and finally proved (for $p>2$) in Devalapurkar's thesis:
		\begin{thm}[Devalapurkar {\cite[Theorem~\chref{6.4.1}]{DevalapurkarSpherochromatism}}]\label{thm:kutCpIntro}
			For primes $p > 2$, there exists an $S^1\times\IZ_p^\times$-equivariant equivalence of $\IE_\infty$-ring spectra
			\begin{equation*}
				\THH\bigl(\IZ_p[\zeta_p]/\IS_p\qpower\bigr)_p^\complete\overset{\simeq}{\longrightarrow} \tau_{\geqslant 0}\left(\ku^{\t C_p}\right)\,.
			\end{equation*}
		\end{thm}
		This theorem allows us to construct a comparison between the $p$-completions of $q$-de Rham cohomology and $\TC^-(-/\ku)$, as was observed both by Devalapurkar and the author (see \cite[Theorem~\chref{6.4.2}]{DevalapurkarSpherochromatism} and the discussion afterwards as well as \cite{qHodge}, arXiv versions $\leqslant3$). The idea is to construct an $S^1$-equivariant map $\THH(R[\zeta_p]/\IS\qpower)\rightarrow \THH(\ku\otimes\IS_R/\ku)^{\t C_p}$ as follows:
		\begin{equation*}
			\begin{tikzcd}
				\THH\bigl(R[\zeta_p]/\IS\qpower\bigr)\ar[rr,dashed] &[-0.5em] &[-0.5em] \THH(\ku\otimes\IS_R/\ku)^{\t C_p}\\
				\uar["\simeq"]\THH(\IS_R)\otimes \THH\bigl(\IZ_p[\zeta_p]/\IS\qpower\bigr)\rar & \THH(\IS_R)^{\t C_p}\otimes\ku^{\t C_p}\rar & \bigl(\THH(\IS_R)\otimes\ku\bigr)^{\t C_p}\uar["\simeq"']
			\end{tikzcd}
		\end{equation*}
		Here the bottom left arrow is given by the cyclotomic Frobenius on $\THH(\IS_R)$ and the map from \cref{thm:kutCpIntro}. Now the $S^1$-equivariant even filtration on $\THH(R[\zeta_p]/\IS\qpower)$ yields $p$-completed $q$-de Rham cohomology by a variant of the arguments in \cite[\S{\chref[section]{11}}]{BMS2}, as we'll explain in \cref{appendix:BMS2}. On the other hand, $(\THH(\ku\otimes\IS_R/\ku)^{\t C_p})^{\h S^1}\simeq \TP(\ku\otimes\IS_R/\ku)_p^\complete$ is already close to the $p$-completion of $\TC^-(\ku\otimes\IS_R/\ku)$. After some massaging, this will yield \cref{thm:qdeRhamkuIntro}.
		%
	\end{numpar}
	
	Let us now give several remarks on \cref{thm:qdeRhamkuIntro}. We begin with the notions of \emph{even filtrations} that we use.
	
	\begin{numpar}[Even filtrations.]
		Since $\IS_R$ is only assumed to be $\IE_2$, we cannot use the even filtration from \cite{EvenFiltration} on $\TC^-(\ku\otimes\IS_R/\ku)$. Instead we'll work with Pstr\k{a}gowski's \emph{perfect even filtration} \cite{PerfectEvenFiltration}, which is already defined for $\IE_1$-ring spectra.
		
		We'll also need an $S^1$-equivariant version of Pstr\k{a}gowski's construction, which first appears in \cite[Definition~\chref{2.55}]{CyclotomicSynthetic} but is originally due to Raksit: We let $\IS_{\ev}\coloneqq \fil_{\ev}^\star \IS$ and $\IT_{\ev}\coloneqq \fil_{\ev}^\star\IS[S^1]$ denote the even filtrations of $\IS$ and $\IS[S^1]$, respectively, and then define a filtered version of $S^1$-fixed points via $(-)^{\h\IT_{\ev}}\coloneqq \Hhom_{\IT_{\ev}}^\star(\IS_{\ev},-)$, where $\Hhom_{\IT_{\ev}}^\star$ denotes the internal $\Hom$ in filtered $\IT_{\ev}$-modules. Finally, we let
		\begin{equation*}
			\fil_{\ev,\h S^1}^\star\TC^-(\ku\otimes\IS_R/\ku)\coloneqq \bigl(\fil_{\ev}^\star \THH(\ku\otimes\IS_R/\ku)\bigr)^{\h\IT_{\ev}}\,,
		\end{equation*}
		where $\fil_{\ev}^\star \THH(\ku\otimes\IS_R/\ku)$ denotes the non-equivariant even filtration of $\THH(\ku\otimes\IS_R/\ku)$, regarded as a left module over itself.
	\end{numpar}
	\begin{numpar}[The solid even filtration.]
		We'll first show a $p$-complete version of \cref{thm:qdeRhamkuIntro} (see \cref{thm:qdeRhamkupComplete}) and then deduce the global theorem via an adelic gluing argument. For the $p$-complete version, it will be very convenient for us to work in the setting of \emph{solid condensed spectra} in the sense of Clausen--Scholze (see \cref{par:CondensedRecollections}--\cref{par:SolidpComplete} for a brief recap).
		
		To this end, we'll develop a variant of Pstr\k{a}gowski's perfect even filtration in the solid setting in \cref{sec:SolidEvenFiltration}. The idea is simple: For an $\IE_1$-algebra $R$ in solid condensed spectra, the $\infty$-category of left $R$-modules has a compact generator given by $\Null_R\coloneqq \cofib(R[\{\infty\}]\rightarrow R[\IN\cup\{\infty\}])$, the free left $R$-module on a nullsequence. We then call a left $R$-module \emph{solid perfect even} if it is contained in the sub-$\infty$-category generated by $\Null_R$ under even shifts, extensions, and retracts. With this definition, we'll adapt the constructions from \cite{PerfectEvenFiltration} in a straightforward way.
		
		As we'll see in \cref{sec:SolidEvenFiltration}, not all of the nice properties of the perfect even filtration carry over to the solid case, since Pstr\k{a}wski frequently exploits the fact that perfect even modules are dualisable, which fails in the solid setting. However, under certain additional \emph{nuclearity} assumptions, everything works as expected. These assumptions are satisfied in the cases we're interested in. In particular, they are satisfied in the discrete case (see \cref{cor:SolidvsDiscreteEvenFiltration}), and so our construction will recover Pstr\k{a}gowski's.
	\end{numpar}
	
	Let us now discuss to what extent the assumptions in \cref{thm:qdeRhamkuIntro} are optimal.
	
	\begin{numpar}[The theorem for $\IE_1$-lifts.]
		For the perfect even filtration of $\THH(\ku\otimes\IS_R/\ku)$ to be defined, it needs to be an $\IE_1$-algebra, which requires $\IS_R$ to be $\IE_2$-algebra. However, in the case where $R$ only admits an $\IE_1$-lift $\IS_R$, it can still happen that $\THH(\ku\otimes\IS_R/\ku)_p^\complete$ is concentrated in even degrees for some primes~$p$. More generally, it can happen that $\IS_R$ admits a cosimplicial resolution by $\IE_1$-algebras $\IS_R\rightarrow \IS_{R^\bullet}$ for which $\THH(\ku\otimes\IS_{R^\bullet}/\ku)_p^\complete$ is even. In this case, we can define an ad-hoc even filtration by
		\begin{equation*}
			\fil_{\ev}^\star\THH(\ku\otimes\IS_R/\ku)_p^\complete\coloneqq \limit_{\IDelta}\tau_{\geqslant 2\star}\THH(\ku\otimes\IS_{R^\bullet}/\ku)_p^\complete\,,
		\end{equation*}
		and then one can put $\fil_{\ev,\h S^1}^\star\TC^-(\ku\otimes\IS_R/\ku)_p^\complete\coloneqq (\fil_{\ev}^\star\THH(\ku\otimes\IS_R/\ku)_p^\complete)^{\h\IT_{\ev}}$ again. This agrees with $\limit_{\IDelta}\tau_{\geqslant 2\star}\TC^-(\ku\otimes\IS_{R^\bullet}/\ku)_p^\complete$.
		
		It turns out that the $p$-complete version of \cref{thm:qdeRhamkuIntro} is still true for these ad-hoc even filtrations. Moreover, we'll see in \cref{thm:qdeRhamkuRegularQuotient} that the filtration $\fil_{\qHodge}^\star(\qdeRham_{R/\IZ})_p^\complete$ admits a very simple explicit description in this case. We'll use this in \cref{subsec:Raksit} to prove \cref{thm:RaksitIntro}, and in joint work with Meyer \cite{qHodge} to compute $\pi_*\TC^-((\ku\otimes\IS/p^\alpha)/\ku)$.
	\end{numpar}
	\begin{numpar}[The theorem at $p=2$.]
		We expect that the assumption $2\in R^\times$ in \cref{thm:qdeRhamkuIntro} can be removed once \cref{thm:kutCpIntro} is proved for $p=2$ as well. In any case, the $\IE_1$-version of the theorem discussed above can be proved unconditionally for $p=2$.
	\end{numpar}
	\begin{numpar}[Are lifts to $\ku$ enough?]\label{par:Liftstoku}
		It's natural to ask if the $\IE_2$-lift $\IS_R$ in \cref{thm:qdeRhamkuIntro} can be replaced by the weaker datum of an $\IE_2$-$\ku$-algebra $\ku_R$ satisfying $\ku_R\otimes_{\ku}\IZ\simeq R$. Although we don't know any counterexample, we consider this unlikely. At the very least, the $\IE_1$-version of the theorem is provably wrong if only an $\IE_1$-lift $\ku_R$ is assumed. Here's a counterexample: Let $\IZ_p\{x\}_\infty$ denote the free $p$-complete perfect $\delta$-ring on a generator $x$ and let $R\coloneqq \IZ_p\{x\}_\infty/x$. Since $\IZ_p\{x\}_\infty$ is a perfect $\delta$-ring, it lifts uniquely to $\ku$ (even to the sphere spectrum), and so $\ku_R\coloneqq \ku_{\IZ_p\{x\}_\infty}/x$ can be equipped with an $\IE_1$-$\ku$-algebra structure via \cite[Corollary~\chref{3.2}]{AngeltveitEvenQuotients}. However, it can be shown that the filtration $\fil_{\qHodge}^\star(\qdeRham_{R/\IZ})_p^\complete$ from \cref{thm:qdeRhamkuRegularQuotient} is \emph{not} a $q$-deformation of the Hodge filtration in this case (see \cite[Example~\chref{4.24}]{qWittHabiro}), and so the $p$-complete version of \cref{thm:qdeRhamkuIntro} cannot hold in this case.
		
		This explains why we don't expect a lift to $\ku$ to be enough. We do expect, however, that it's enough to have a lift $j_R$ to the \emph{image of $J$-spectrum} $j\coloneqq \tau_{\geqslant 0}(\IS_{K(1)})$. Indeed, if the diagram
		\begin{equation*}
			\begin{tikzcd}
				& \THH(j)_p^\complete\dlar\dar\\
				j\rar & \THH(\IZ_p)_p^\complete
			\end{tikzcd}
		\end{equation*}
		were $S^1$-equivariantly commutative, we could use the construction discussed below \cref{thm:kutCpIntro} to get an $S^1$-equivariant map $\THH(R[\zeta_p]/\IS\qpower)\rightarrow \THH(\ku\otimes_j j_R/\ku)^{\t C_p}$, which could then be used to show the $p$-complete version of \cref{thm:qdeRhamkuIntro} with $\IS_R$ replaced by $j_R$.
		
		Unfortunately, the diagram above is \emph{not} $S^1$-equivariantly commutative, similar to what happens for $\THH(\IZ_p)\rightarrow \IZ_p\rightarrow \THH(\IF_p)$. But the issue doesn't seem to be too serious. For example, a chromatic height-$2$ analogue of \cite[Theorem~\chref{0.3.6}]{DevalapurkarRaksitTHH} would likely fix this, in the same way as the cited result fixes the problems for $\THH(\IZ_p)\rightarrow \IZ_p\rightarrow \THH(\IF_p)$.
	\end{numpar}
	
	\subsection{Habiro descent, homotopically}
	
	Let us now discuss the $q$-deformed Hodge filtration $\fil_{\qHodge}^\star\qdeRham_{R/\IZ}$ from \cref{thm:qdeRhamkuIntro}. It's natural to ask whether the $q$-de Rham complex in general can be equipped with a $q$-deformation of the Hodge filtration. This question is studied in the companion paper \cite{qWittHabiro}. It turns out that $\fil_{\qHodge}^\star\qdeRham_{R/\IZ}$ is a \emph{$q$-Hodge filtration} in the sense of \cite[Definition~\chref{3.2}]{qWittHabiro}; this roughly means that the rationalisation and the rationalised $p$-completions of $\fil_{\qHodge}^\star\qdeRham_{R/\IZ}$ behave as expected. It is \emph{not} always possible to find a $q$-Hodge filtration (the ring from \cref{par:Liftstoku} is a counterexample), so \cref{thm:qdeRhamkuIntro} provides a large source of examples where it works.
	
	To any $q$-Hodge filtration on $\qdeRham_{R/\IZ}$ we associate a \emph{$q$-Hodge complex} in \cite[\chref{3.5}]{qWittHabiro}. It is defined as
	\begin{equation*}
		\qHodge_{R/\IZ}\coloneqq \colimit\Bigl( \fil_{\qHodge}^0\qdeRham_{R/\IZ}\xrightarrow{(q-1)} \fil_{\qHodge}^1\qdeRham_{R/\IZ}\xrightarrow{(q-1)}\dotsb\Bigr)_{(q-1)}^\complete
	\end{equation*}
	(we've abusingly suppressed the choice of $q$-Hodge filtration in the notation $\qHodge_{R/\IZ}$, but it will always be clear from the context). As a straightforward corollary of \cref{thm:qdeRhamkuIntro}, one finds that
	\begin{equation*}
		\qHodge_{R/\IZ}[\beta^{\pm 1}]\simeq \Sigma^{-2*}\gr_{\ev,\h S^1}^*\TC^-(\KU\otimes\IS_R/\KU)\,.
	\end{equation*}
	
	In \cite[Theorem~\chref{3.11}]{qWittHabiro} we show that the $q$-Hodge complex $\qHodge_{R/\IZ}$ is, in a non-trivial way, the $(q-1)$-completion of an object $\qHhodge_{R/\IZ}$, which lives over the \emph{Habiro ring}
	\begin{equation*}
		\Hh\coloneqq \limit_{m\in\IN}\IZ[q]_{(q^m-1)}^\complete\,.
	\end{equation*}
	We call $\qHhodge_{R/\IZ}$ the \emph{Habiro--Hodge complex}. Our second main goal in this paper is to show that in the situation where the $q$-Hodge filtration $\fil_{\qHodge}^\star\qdeRham_{R/\IZ}$ comes from homotopy theory by means of \cref{thm:qdeRhamkuIntro}, the Habiro--Hodge complex $\qHhodge_{R/\IZ}$ can also be described homotopically.
	
	\begin{numpar}[Cyclonic spectra.]
		This description will involve some \emph{genuine} equivariant homotopy theory. Since this is not standard in arithmetic geometry, we'll give a review of the necessary parts of the theory in \cref{subsec:GenuineRecollections}. We'll then work with in the $\infty$-category of \emph{cyclonic spectra} $\Cyclonic$, introduced by Barwick--Glasman \cite{BarwickGlasmanCyclonic}; see \cref{par:Cyclonic}. Roughly, a cyclonic spectrum is a spectrum $X$ with an $S^1$-action, equipped with compatible genuine enhancements of the action of the finite cyclic subgroups $C_m\subseteq S^1$ for all $m\in\IN$.
		
		The spectrum $\THH(\IS_R)$ admits a natural cyclotomic structure, which induces a cyclonic structure. Similarly, $\ku$ admits a natural genuine $S^1$-equivariant structure, which again induces a cyclonic structure. Hence $\THH(\IS_R)\otimes\ku\simeq \THH(\ku\otimes\IS_R/\ku)$ can be equipped with a cyclonic structure too. We define the \emph{$m$\textsuperscript{th} topological cyclonic homology} \cref{def:ThisShouldBeTR} as
		\begin{equation*}
			\TCn{m}(\ku\otimes\IS_R/\ku)\coloneqq \bigl(\THH(\ku\otimes\IS_R/\ku)^{C_m}\bigr)^{\h (S^1/C_m)}\,.
		\end{equation*}
		Here $(-)^{C_m}$ denotes the genuine $C_m$-fixed points and $(-)^{\h (S^1/C_m)}$ the homotopy fixed points of the residual action. We note that $\TCn{m}(-/\ku)$ is similar to the construction of topological restriction homology $\TR(-)$; however, the restriction maps along which $\TR(-)$ is the limit do not exist for $\TCn{m}(-/\ku)$, and $\TR(-)$ doesn't take the residual actions into account.
		
		The same constructions work for $\KU$, and so we can define $\TCn{m}(\KU\otimes\IS_R/\KU)$ analogously. To obtain the Habiro--Hodge complex, we'll construct an appropriate even filtration on $\TCn{m}(\KU\otimes\IS_R/\KU)$.
	\end{numpar}
	\begin{numpar}[Genuine equivariant even filtrations.]
		For a cyclonic $\IE_1$-algebra~$X$ whose geometric fixed points $X^{\Phi C_m}$ are bounded below for all~$m\in\IN$, the genuine fixed points can be recovered from the geometric fixed points via the formula
		\begin{equation*}
			X^{C_m}\overset{\simeq}{\longrightarrow}\eq\Biggl(\prod_{d\mid m}(X^{\Phi C_d})^{\h C_{m/d}}\overset{\operatorname{can}}{\underset{\phi}{\doublemorphism}}\prod_{p\vphantom{|}}\prod_{pd\mid m}\bigl((X^{\Phi C_d})^{\t C_p}\bigr)^{\h C_{m/pd}}\Biggr)
		\end{equation*}
		(see \cref{lem:GenuineFromGeometric}). In the case where $X$ is an $\IE_1$-structure in $\Cyclonic$, this allows us to define an even filtration $\fil_{\ev,C_m}^\star X^{C_m}$ as follows: Equip each $X^{\Phi C_d}$ with the non-equivariant perfect even filtration, apply the filtered versions of $(-)^{\h C_{m/d}}$ and $(-)^{\t C_p}$ from \cite[\S{\chref[subsection]{2.3}}]{CyclotomicSynthetic}, then finally take the equaliser in filtered objects.
		
		In this way we can construct $\fil_{\ev,C_m}^\star\THH(\ku\otimes\IS_R/\ku)^{C_m}$. The same construction cannot be used for $\fil_{\ev,C_m}^\star\THH(\KU\otimes\IS_R/\KU)^{C_m}$, as $\THH(\KU\otimes\IS_R/\KU)$ is not bounded below, but instead we can simply use the filtered localisation of $\fil_{\ev,C_m}^\star\THH(\ku\otimes\IS_R/\ku)^{C_m}$ at the Bott element $\beta$. Finally, we put
		\begin{equation*}
			\fil_{\ev,S^1}^\star\TCn{m}(\KU\otimes\IS_R/\KU)\coloneqq \bigl(\fil_{\ev,C_m}^\star\THH(\KU\otimes\IS_R/\KU)^{C_m}\bigr)^{\h(\IT/C_m)_{\ev}}
		\end{equation*}
		(see \cref{par:CyclonicEvenFiltrationTHHKU} for the details).
		
		While these even filtrations are defined in a rather ad-hoc way, we expect that there exists a more canonical construction. An intrinsically defined genuine equivariant even filtration is currently under development (the author has been informed of independent work in progress by Jeremy Hahn and Lucas Piessevaux) and we hope that it will agree with our constructions in the cases at hand.
	\end{numpar}
	With these even filtrations, we can finally formulate the second main result of this paper.
	\begin{thm}[see \cref{thm:qdeRhamKUGenuine}]\label{thm:HabiroIntro}
		The Habiro--Hodge complex $\qHhodge_{R/\IZ}$ associated to the $q$-Hodge filtration $\fil_{\qHodge}^\star\qdeRham_{R/\IZ}$ from \cref{thm:qdeRhamkuIntro} satisfies
		\begin{equation*}
			\qHhodge_{R/\IZ}[\beta^{\pm 1}]\simeq \Sigma^{-2*}\gr^*\Bigl(\limit_{m\in\IN}\fil_{\ev,C_m}^\star\TCn{m}(\KU\otimes\IS_R/\KU)\Bigr)\,.
		\end{equation*}
	\end{thm}
	As a consequence, we obtain a homotopical construction of the \emph{Habiro ring of a number field} $\Hh_{\Oo_F[1/\Delta]}$ from \cite{HabiroRingOfNumberField}:
	\begin{cor}[see \cref{cor:HabiroRingOfNumberFieldKU}]
		Let $F$ be a number field and let $\Delta$ be divisible by~$6$ and by the discriminant of $F$. Let $\IS_{\Oo_F[1/\Delta]}$ denote the unique lift of $\Oo_F[1/\Delta]$ to an étale extension of $\IS$. Then
		\begin{equation*}
			\Hh_{\Oo_F[1/\Delta]}\cong \pi_0\Bigl(\limit_{m\in\IN}\bigl(\THH(\KU\otimes\IS_{\Oo_F[1/\Delta]}/\KU)^{C_m}\bigr)^{\h (S^1/C_m)}\Bigr)\,.
		\end{equation*}
	\end{cor}
	
	\subsection{Organisation of this paper}
	
	In \cref{sec:SolidEvenFiltration}, we introduce a version of Pstr\k{a}gowski's perfect even filtration in the solid setting. In \cref{sec:SolidEvenFiltrationku}, we'll study the solid even filtration on $\THH$. In particular, we'll show that it can often be computed by even cosimplicial resolutions, and satisfies the expected base change properties.
	
	In \cref{sec:qdeRhamku}, we'll show \cref{thm:qdeRhamkuIntro}. Subsections \crefrange{subsec:qdeRhamkupComplete}{subsec:QuasiRegular} are devoted to the $p$-completed version of the theorem; the global version will be deduced in \cref{subsec:qdeRhamkuGlobal} via an adelic gluing argument. In \cref{sec:Genuine}, we show \cref{thm:HabiroIntro}. In subsections \crefrange{subsec:GenuineRecollections}{subsec:Genuineku}, we set up the formalism of cyclonic spectra and explain how $\ku$ fits into this. In \cref{subsec:GenuineHabiroDescent} we'll show the Habiro descent theorem.
	
	In~\cref{sec:Examples} we'll discuss several examples. It is a priori not clear how to find rings $R$ to which \cref{thm:qdeRhamkuIntro} can be applied. In \cref{subsec:CyclotomicBases} we'll construct a large supply of such rings~$R$. In \cref{subsec:Raksit} we prove \cref{thm:RaksitIntro} and in \cref{subsec:HabiroRingOfNumberField} we show how the Habiro ring of a number field can be recovered from our formalism.
	
	Finally, there will be three appendices. In \cref{appendix:BMS2}, we discuss a \cite{BMS2}-style construction of the $q$-de Rham complex. In \cref{appendix:E2}, we give a proof of the folklore theorem that flat polynomial rings over $\IS$ in any number of variables admit an even $\IE_2$-cell structure. In \cref{appendix:EquivariantSnaith}, we explain how the equivariant Snaith theorem from \cite{EquivariantSnaith} can be made $\IE_\infty$, which will be needed in our discussion of the genuine $S^1$-equivariant structure on $\ku$.
	
	\begin{numpar}[Notation and conventions.]\label{par:Notation}
		Throughout the article, we freely use the language of $\infty$-categories and we'll adopt the following conventions:
		\begin{alphanumerate}
			\item \textbf{Stable $\boldsymbol\infty$-categories.} We let $\Sp$ denote the $\infty$-category of spectra. For an ordinary ring $R$, we let $\Dd(R)$ denote the derived $\infty$-category of $R$. We often implicitly regard objects of $\Dd(R)$ as spectra via the Eilenberg--MacLane functor $\H$, but we'll always suppress this functor in our notation. For a stable $\infty$-category $\Cc$, we let $\Hom_\Cc(-,-)$ denote the mapping spectra in $\Cc$. The shift functor and its inverse will always be denoted by $\Sigma$ and $\Sigma^{-1}$ (even for $\Dd(R)$), to avoid confusion with shifts in graded or filtered objects.
			
			\item \textbf{Symmetric monoidal $\infty$-categories.} If no confusion can occur, we denote the tensor unit by $\IUnit$ and the tensor product by $\otimes$. Whenever we consider a symmetric monoidal $\infty$-category $\Cc$ which is stable or presentable, we always implicitly assume that the tensor product commutes with finite colimits or arbitrary colimits, respectively. In the presentable case, we let $\Hhom_\Cc(-,-)$ denote the \emph{internal Hom} in $\Cc$ and $X^\vee\coloneqq \Hhom_\Cc(X,\IUnit)$ the \emph{dual} of an object $X \in \Cc$.
			
			\item \textbf{Graded and filtered objects.} For a stable $\infty$-category $\Cc$, we let $\Gr(\Cc)$ and $\Fil(\Sp)$ denote the $\infty$-categories of \emph{graded} and \emph{\embrace{descendingly} filtered objects in $\Cc$}. The shift in graded or filtered objects will be denoted $(-)(1)$. An object with a descending filtration is typically denoted
			\begin{equation*}
				\fil^\star X=\Bigl(\dotsb\leftarrow \fil^nX\leftarrow \fil^{n+1}X\leftarrow\dotsb\Bigl)
			\end{equation*}
			and we let $\gr^*X$ denote the \emph{associated graded}, given by $\gr^nX\coloneqq \cofib(\fil^{n+1}X\rightarrow \fil^nX)$. We mostly work with filtrations that are constant in degrees $\leqslant 0$ (such as the Hodge filtration). In this case we'll abusingly write $\fil^\star X=(\fil^0X\leftarrow \fil^1 X\leftarrow\dotsb)$; this should be interpreted as the constant $\fil^0X$-valued filtration in degrees $\leqslant 0$.
			
			If $\Cc$ is symmetric monoidal and the tensor product $-\otimes-$ commutes with colimits in both variables, we equip $\Gr(\Cc)$ and $\Fil(\Cc)$ with their canonical symmetric monoidal structures given by Day convolution. We'll use the fact that $\Fil(\Cc)\simeq \Mod_{\IUnit_{\Gr}[t]}\Gr(\Cc)$, where $\IUnit_{\Gr}$ denotes the tensor unit in $\Gr(\Cc)$ and $t$ sits in graded degree~$-1$; see e.g.\ \cite[Proposition~\chref{3.2.9}]{RaksitFilteredCircle}. Under this equivalence, passing to the associated graded corresponds to \enquote{modding out~$t$}, i.e.\ the base change $\IUnit_{\Gr}\otimes_{\IUnit_{\Gr}[t]}-$.
			
			We say that $\fil^\star X$ is an \emph{exhaustive filtration on $X$} if $X\simeq \colimit_{n\rightarrow -\infty}\fil^nX$. We say that a filtered object $\fil^\star X$ is \emph{complete} if $0\simeq \limit_{n\rightarrow \infty}\fil^nX$. We define the \emph{completion} $\fil^\star \widehat{X}\coloneqq \limit_{n\rightarrow\infty}\cofib(\fil^{\star+n} X\rightarrow\fil^{\star}X)$. By construction, there's a pullback square
			\begin{equation*}
				\begin{tikzcd}
					\fil^\star X\rar\dar\drar[pullback] & \fil^\star \widehat{X}\dar\\
					X\rar & \widehat{X}
				\end{tikzcd}
			\end{equation*}
			We'll often refer to this by saying that \emph{every filtration is the pullback of its completion}.
			
			\item \textbf{Condensed mathematics.} Whenever we use condensed mathematics, we work in the light condensed setting. We'll distinguish between the words \emph{static} (\enquote{un-animated}) for a spectrum concentrated in degree~$0$, and \emph{discrete} (\enquote{un-condensed}) for a condensed spectrum with the discrete topology.
			
			\item \textbf{Homotopy classes of $\boldsymbol{\ku^{\h S^1}}$.} We denote by $\beta\in\pi_2(\ku)$ the Bott element and by $q\in\pi_0(\ku^{\h S^1})$ the class corresponding to the standard representation of $S^1$ on $\IC$. There's a unique complex orientation $t\in\pi_{-2}(\ku^{\h S^1})$ satisfying $q-1=\beta t$; then $\pi_*(\ku^{\h S^1})\cong \IZ[\beta]\llbracket t\rrbracket$. Here we purposely use the same symbol as for $\IUnit_{\Gr}[t]$ above: We'll often regard graded $\pi_*(\ku^{\h S^1})\cong \IZ[\beta]\llbracket t\rrbracket$-modules as filtered objects using the apparent graded $\IZ[t]$-module structure.
			
			\item \textbf{Arithmetic fracture squares and gluing.} We'll often use that for any spectrum $X$ and any positive integer~$N$, there are canonical pullback squares
			\begin{equation*}
				\begin{tikzcd}
					X\rar\dar\drar[pullback] & \prod_{p\vphantom{\mid N}}\widehat{X}_p\dar\\
					X\otimes\IQ\rar & \prod_{p\vphantom{\mid N}}\widehat{X}_p\otimes\IQ
				\end{tikzcd}\quad\text{and}\quad \begin{tikzcd}
				X\rar\dar\drar[pullback] & \prod_{p\mid N}\widehat{X}_p\dar\\
				X\bigl[\localise{N}\bigr]\rar & \prod_{p\mid N}\widehat{X}_p\bigl[\localise{p}\bigr]
				\end{tikzcd}
			\end{equation*}
			and we'll call these \emph{arithmetic fracture squares}.
			
			\item \textbf{Completed ($\boldsymbol q$-)de Rham complexes.} To avoid excessive use of completions, we adopt the convention that all ($q$-)de Rham or cotangent complexes relative to a $p$-complete ring will be implicitly $p$-completed. So for example, while $\deRham_{R/\IZ}$ would denote the usual derived de Rham complex of $R$, $\qdeRham_{R/\IZ_p}$ would denote the $p$-completed derived $q$-de Rham complex of $R$.
		\end{alphanumerate}
		
		
	\end{numpar}

	\begin{numpar}[Acknowledgements.]
		First and foremost I'd like to thank Sanath Devalapurkar and Arpon Raksit for sharing and discussing their (by then) unpublished results with me. It will undoubtedly become apparent to the reader how much of an intellectual debt this paper owes to their work. Moreover, I would like to thank Peter Scholze, for his suggestion to compute $\TCref(\ku\otimes\IQ/\ku)$, which has led to the discovery of the results in \cref{sec:qdeRhamku} of this paper, and for many fruitful discussions. I would also like to thank Johannes Anschütz, who originally raised the question whether there's a $\TR$-style description of $q$-de Rham--Witt complexes at my master's thesis defense. This question can now partially be answered by the results of \cref{sec:Genuine}. Finally, I'm grateful to Gabriel Angelini-Knoll, Ben Antieau, Omer Bojann, Robert Burklund, Bastiaan Cnossen, Jeremy Hahn, Kaif Hilman, Dominik Kirstein, Christian Kremer, Akhil Mathew, and Lucas Piessevaux for helpful conversations.
		
		This work was carried out while I was a Ph.D.\ student at the MPIM/University Bonn, and I would like to thank these institutions for their hospitality. I'm grateful for the financial support provided by the DFG through Peter Scholze's Leibniz prize.
	\end{numpar}
	
	\newpage

	\section{The solid even filtration}\label{sec:SolidEvenFiltration}
	
	In this section we'll sketch how to adapt Pstr\k{a}gowski's perfect even filtration \cite{PerfectEvenFiltration} to $\IE_1$-algebras in solid condensed spectra. This facilitates many $p$-completion arguments later on. However, as we'll see, not all of the nice properties of the perfect even filtration carry over to the solid condensed case. But in the cases we need---and probably most cases of interest in general---it works as expected. It would be desirable to develop a more complete (and perhaps less naive) theory of the perfect even filtration in the condensed setting.
	
	Before we begin, let us briefly recall the solid condensed setting. There are no properly published sources yet, so we have to refer the reader to the recordings of \cite{AnalyticStacks} and the unfinished notes \cite{SolidNotes}.

	\begin{numpar}[Solid condensed recollections.]\label{par:CondensedRecollections}
		Let $\Cond(\Sp)$ denote the $\infty$-category of \emph{\embrace{light} condensed spectra}, that is, hypersheaves of spectra on the site of light profinite sets as defined by Clausen and Scholze \cite{AnalyticStacks}. The evaluation at the point $(-)(*)\colon \Cond(\Sp)\rightarrow \Sp$ admits a fully faithful symmetric monoidal left adjoint $(\underline{-})\colon \Sp\rightarrow\Cond(\Sp)$, sending a spectrum $X$ to the \emph{discrete} condensed spectrum $\underline{X}$.
		
		One can develop a theory of \emph{solid condensed spectra} along the lines of \cite[Lectures~\href{https://www.youtube.com/watch?v=bdQ-_CZ5tl8&list=PLx5f8IelFRgGmu6gmL-Kf_Rl_6Mm7juZO}{5}--\href{https://www.youtube.com/watch?v=KKzt6C9ggWA&list=PLx5f8IelFRgGmu6gmL-Kf_Rl_6Mm7juZO}{6}]{AnalyticStacks}. Let $\Null\coloneqq \cofib(\IS[\{\infty\}]\rightarrow\IS[\IN\cup\{\infty\}])$ be the free condensed spectrum on a null sequence. Let $\sigma\colon \Null\rightarrow \Null$ be the endomorphism induced by the shift map $(-)+1\colon \IN\cup\{\infty\}\rightarrow \IN\cup\{\infty\}$. Recall that a condensed spectrum $M$ is called \emph{solid} if
		\begin{equation*}
			1-\sigma^*\colon \Hhom_\IS(\Null,M)\overset{\simeq}{\longrightarrow}\Hhom_\IS(\Null,M)
		\end{equation*}
		is an equivalence, where $\Hhom_\IS$ denotes the internal Hom in $\Cond(\Sp)$. We let $\Sp_\solid\subseteq \Cond(\Sp)$ denote the full sub-$\infty$-category of solid condensed spectra. Then $\Sp_\solid$ is closed under all limits and colimits. This implies that the inclusion $\Sp_\solid\subseteq \Cond(\Sp)$ admits a left adjoint $(-)^\solid\colon \Cond(\Sp)\rightarrow \Sp_\solid$. It satisfies $(M\otimes N)^\solid\simeq (M^\solid\otimes N)^\solid$, which allows us to endow $\Sp_\solid$ with a symmetric monoidal structure, called the \emph{solid tensor product}, via $M\soltimes N\coloneqq (M\otimes N)^\solid$.
	\end{numpar}
	\begin{numpar}[Solid condensed spectra and $p$-completions.]\label{par:SolidpComplete}
		If $X$ is a $p$-complete spectrum, then $\underline{X}$ is usually \emph{not} $p$-complete in $\Cond(\Sp)$ because $(\underline{-})$ doesn't commute with limits. After passing to $p$-completions, we still get an adjunction on $p$-complete objects $(\underline{-})_p^\complete\colon \Sp_p^\complete\shortdoublelrmorphism\Cond(\Sp)_p^\complete\noloc(-)(*)$ and the left adjoint is still fully faithful because the unit is still an equivalence.
		
		It's straightforward to check that any discrete condensed spectrum is solid. By closure under limits it follows that $(\underline{-})_p^\complete\colon \Sp_p^\complete\rightarrow\Cond(\Sp)_p^\complete$ takes values in $\Sp_\solid$.  The solid tensor product has the magical property that if $M$ and $N$ are $p$-complete and bounded below solid condensed spectra, then $M\soltimes N$ is again $p$-complete; see \cite[Lecture~\href{https://youtu.be/KKzt6C9ggWA?list=PLx5f8IelFRgGmu6gmL-Kf_Rl_6Mm7juZO&t=3288}{6}]{AnalyticStacks} or \cite[Proposition~\chref{A.3}]{BoscoRationalHodge}. In particular, the fully faithful embedding $(\underline{-})_p^\complete\colon \Sp_p^\complete\rightarrow \Sp_\solid$ is symmetric monoidal when restricted to bounded below objects.
	\end{numpar}
	
	\subsection{Definitions and basic properties}
	
	In the following we let $R$ be an $\IE_1$-algebra in the symmetric monoidal $\infty$-category of solid condensed spectra $\Sp_\solid$ and we let
	\begin{equation*}
		-\soltimes_R-\colon \RMod_R(\Sp_\solid)\times \LMod_R(\Sp_\solid)\longrightarrow \Sp_\solid
	\end{equation*}
	denote the relative tensor product over $R$. We start setting up the theory in a completely analogous way to \cite[\S\S{\chref[section]{2}}--{\chref[section]{3}}]{PerfectEvenFiltration}.
	
	\begin{numpar}[Solid perfect even modules.]\label{par:SolidPerfectEven}
		We let $\Null_R\coloneqq R\soltimes \Null^\solid$, where we define $\Null\coloneqq \cofib(\IS[\{\infty\}]\rightarrow \IS[\IN\cup\{\infty\}])$ to be the free condensed spectrum on a null sequence as in \cref{par:CondensedRecollections}. It can be shown that the solidification $\Null^\solid$ agrees with $\prod_\IN\IS$ and defines a compact generator of $\Sp_\solid$, so that $\Null_R$ is a compact generator of $\LMod_R(\Sp_\solid)$.
		
		We say that an $R$-module $Q$ is \emph{solid perfect even} if it is contained in the smallest sub-$\infty$-category
		\begin{equation*}
			\Perf_{\ev}(R_\solid)\subseteq \LMod_R(\Sp_\solid)
		\end{equation*}
		which contains $\Sigma^{2n}\Null_R$ for all $n\in\IZ$ and is closed under extensions and retracts.
		
		Note that $R[S]^\solid$ is solid perfect even for all light condensed sets $S$. Also note that in contrast to the uncondensed situation, it is no longer true that $\Perf_{\ev}(R_\solid)$ is closed under duals. Already for $R=\IS$ we have $\Hom_\IS(\Null_\IS,\IS)\simeq \bigoplus_\IN\IS$, which is not solid perfect even. This is ultimately the reason why the solid theory is not quite as nice.
	\end{numpar}
	
	\begin{numpar}[The solid even filtration.]\label{par:SolidEvenFiltration}
		Equip $\Perf_{\ev}(R_\solid)$ with a Grothendieck topology in which covers are maps $P\rightarrow Q$ whose fibre is again solid perfect even. Every left $R$-module $M$ defines a $\Sp_\solid$-valued sheaf on the additive site $\Perf_{\ev}(R_\solid)$ via
		\begin{equation*}
			\Hhom_R(-,M)\colon \Perf_{\ev}(R_\solid)^\op\longrightarrow \Sp_\solid\,.
		\end{equation*}
		We can form its truncations $\tau_{\geqslant 2n}\Hhom_R(-,M)$ in the sheaf $\infty$-category $\Sh(\Perf_{\ev}(R_\solid),\Sp_\solid)$ and then define the \emph{solid even filtration of $M$} as the sections
		\begin{equation*}
			\fil_{\ev/R}^\star M\coloneqq \Gamma_{\Perf_{\ev}(R_\solid)}\bigl(R,\tau_{\geqslant 2\star}\Hhom_R(-,M)\bigr)\,.
		\end{equation*}
		If $R$ is clear from the context, we'll often just write $\fil_{\ev}^\star M$. In particular, if we write $\fil_{\ev}^\star R$, it is understood that we take the solid even filtration of $R$ over itself.
		
		For any half-integer weight $w$, we also define the \emph{even sheaf of weight $w$}, denoted $\Ff_M(w)$, as the sheafification of the presheaf of solid abelian groups $\pi_{2w}\Hhom_R(-,M)\colon \Perf_{\ev}(R_\solid)^\op\rightarrow \cat{Ab}_\solid$. For $w=0$ we just write $\Ff_M\coloneqq \Ff_M(0)$. We call $M$ \emph{solid homologically even} if $\Ff_M(w)\cong 0$ for all proper half-integers $w\in\frac12+\IZ$.
		
		The results from  \cite[\S{\chref[section]{2}}]{PerfectEvenFiltration} can be carried over verbatim to the solid setting. In particular, it's still true that an $R$-module $E$, whose condensed homotopy groups $\pi_*(E)$ are concentrated in even degrees, will be homogically even and its solid even filtration will be the double-speed Whitehead filtration $\fil_{\ev/R}^\star E\simeq \tau_{\geqslant 2\star}(E)$.
	\end{numpar}
	\begin{numpar}[Monoidality of the solid even filtration.]
		The arguments from \cite[\S{\chref[section]{3}}]{PerfectEvenFiltration} can mostly be adapted to the solid situation, but we need some enriched $\infty$-category to do so.
		
		Let us first set up the enriched setting. We use the formalism from \cite{HeineEnriched}. The $\infty$-category $\LMod_R(\Sp_\solid)$ is naturally a module over $\Sp_{\solid}$ in $\Pr^\L$ and so it will be enriched in the sense of \cite{HeineEnriched}. Explicitly, for left $R$-modules $M$ and $N$, the mapping spectrum $\Hom_R(M,N)$ comes with a natural condensed structure $\Hhom_R(M,N)$ which will be solid if $N$ is (we've already used this in \cref{par:SolidEvenFiltration}). Restricting the module structure, we see that $\LMod_R(\Sp_\solid)$ is also a module over the connective part $\Sp_{\solid,\geqslant 0}$ in $\Pr^\L$, which yields an enrichment given by $\tau_{\geqslant 0}\Hhom_R(M,N)$. The full sub-$\infty$-category $\Perf_{\ev}(R_\solid)\subseteq \LMod_R(\Sp_\solid)$ inherits an enrichment over $\Sp_{\solid,\geqslant 0}$. There is an established notion of an \emph{enriched presheaf $\infty$-category} $\PSh^{\Sp_{\solid,\geqslant 0}}(\Perf_{\ev}(R_\solid),\Sp_{\solid},\geqslant 0)$ with an enriched Yoneda embedding; see \cite{HinichYoneda,HeineTwoModels}. By considering enriched presheaves which are additive and local with respect to all covering sieves, we can also define an enriched version of additive sheaves. To avoid cumbersome notation, we'll drop the superscript and just write $\Sh_\Sigma(\Perf_{\ev}(R_\solid),\Sp_{\solid,\geqslant 0})$ and $\Sh_\Sigma(\Perf_{\ev}(R_\solid),\Sp_\solid)$ in the following, implicitly assuming that all sheaves are enriched over $\Sp_{\solid,\geqslant 0}$.
		
		Let us now explain how to adapt \cite[\S{\chref[section]{3}}]{PerfectEvenFiltration} to turn the solid even filtration into a lax symmetric monoidal functor
		\begin{equation*}
			\fil_{\ev/-}^\star(-)\colon \LMod(\Sp_\solid)\longrightarrow \LMod(\Fil\Sp_\solid)\,.
		\end{equation*}
		Let $\Uu^{\geqslant 0}$ and $\Uu$ denote the cocartesian unstraightenings of the functors lax symmetric monoidal functors $R\mapsto \Sh_\Sigma(\Perf_{\ev}(R_\solid),\Sp_{\solid,\geqslant 0})$ and $\Sh_\Sigma(\Perf_{\ev}(R_\solid),\Sp_\solid)$. The $\infty$-category of enriched (pre)sheaves satisfies a similar universal property as usual; see \cite[Theorem~{\chpageref[5.1]{61}}]{HeineEnriched}. As in \cite[Construction~\chref{3.8}]{PerfectEvenFiltration}, we obtain a symmetric monoidal natural transformation between the lax symmetric monoidal functors $R\mapsto \Sh_\Sigma(\Perf_{\ev}(R_\solid),\Sp_{\solid,\geqslant 0})$ and $R\mapsto \LMod_R(\Sp_\solid)$. Applying unstraightening, we obtain a diagram
		\begin{equation*}
			\begin{tikzcd}
				\Uu^{\geqslant 0}\ar[rr,"F"]\drar & & \LMod(\Sp_\solid)\dlar\\
				& \Alg_{\IE_1}(\Sp_\solid) &
			\end{tikzcd}
		\end{equation*}
		where the vertical arrows are cocartesian fibrations and the top horizontal arrow $F$ is symmetric monoidal.
		
		The functor $F$ admits a fibre-wise right adjoint: In the fibre over $R$, the right adjoint is given by the restricted enriched Yoneda embedding $\LMod_R(\Sp_\solid)\rightarrow\Sh_\Sigma(\Perf_{\ev}(R_\solid),\Sp_{\solid,\geqslant 0})$ sending $M\mapsto \tau_{\geqslant 0}\Hhom_R(-,M)$. Since our sheaves take values in $\Sp_{\solid,\geqslant 0}$, the truncation can be performed section-wise and no sheafification is necessary. By \cite[Corollary~\chref{7.3.2.7}]{HA}, the fibre-wise right adjoints assemble into a lax symmetric monoidal right adjoint $G\colon \LMod(\Sp_{\solid})\rightarrow\Uu^{\geqslant 0}$. We'll now study the composition
		\begin{equation*}
			\LMod(\Sp_\solid)\overset{G}{\longrightarrow}\Uu^{\geqslant 0}\longrightarrow \Uu\,.
		\end{equation*}
		In the fibre over $R$, this composition is given by sending $M\mapsto \nu_R(M)\coloneqq \tau_{\geqslant 0}\Hhom_R(-,M)$, where now the truncation is performed in  $\Sh_\Sigma(\Perf_{\ev}(R_\solid),\Sp_{\solid})$.
		
		Another application of the universal property \cite[Theorem~{\chpageref[5.1]{61}}]{HeineEnriched} allows us to extend the lax symmetric monoidal functor $\tau_{\geqslant -2\star}\Hhom_\IS(-,\IS)\colon \IZ\rightarrow \Sh_\Sigma(\Perf_{\ev}(\IS_\solid),\Sp_{\solid})$ to a lax symmetric monoidal functor
		\begin{equation*}
			\Fil\Sp_\solid\longrightarrow \Sh_\Sigma\bigl(\Perf_{\ev}(\IS_\solid),\Sp_{\solid}\bigr)
		\end{equation*}
		As in \cite[Construction~\chref{3.20}]{PerfectEvenFiltration}, for any $R\in\Alg_{\IE_1}(\Sp_{\solid})$, $\Sh_\Sigma(\Perf_{\ev}(R_\solid),\Sp_{\solid})$ is a module over $\Sh_\Sigma(\Perf_{\ev}(\IS_\solid),\Sp_{\solid})$ and thus over $\Fil\Sp_\solid$. Therefore, if $X$ and $Y$ are $\Sp_\solid$-valued sheaves on $\Perf_{\ev}(R_\solid)$, we can define a filtered solid condensed mapping spectrum $\Hhom^\star(X,Y)$. Using the enriched Yoneda lemma of \cite{HinichYoneda}, we can argue as in \cite[Lemma~\chref{3.23}]{PerfectEvenFiltration} to show
		\begin{equation*}
			\Hhom^\star\bigl(\nu_R(R),\nu_R(M)\bigr)\simeq \fil_{\ev/R}^\star M\,.
		\end{equation*}
		Now consider the functor $R\mapsto \Sh_\Sigma(\Perf_{\ev}(R_\solid),\Sp_{\solid})$. As in \cite[Construction~\chref{3.27}]{PerfectEvenFiltration} we can refine it to a lax symmetric monoidal functor $\Alg_{\IE_1}(\Sp_\solid)\rightarrow \Alg_{\IE_0}(\Mod_{\Fil\Sp_\solid}(\Pr^\L))$.
		
		We don't know if this functor factors through the image of the fully faithful embedding $\Alg_{\IE_1}(\Fil\Sp_\solid)\hookrightarrow \Alg_{\IE_0}(\Mod_{\Fil\Sp_\solid}(\Pr^\L))$, as it does in the uncondensed setting.%
		\footnote{In particular, we don't know if the analogue of \cite[Proposition~\chref{3.26}]{PerfectEvenFiltration} is true, i.e.\ whether
			\begin{equation*}
				\Hhom^\star\bigl(\nu_R(R),-\bigr)\colon \Sh_\Sigma\bigl(\Perf_{\ev}(R_\solid),\Sp_{\solid}\bigr)\longrightarrow \LMod_{\underline{\operatorname{End}}^\star(\nu_R(R))}(\Fil\Sp_\solid)
			\end{equation*}
			is an equivalence. The problem is that the even filtration $\fil_{\ev/R}^\star(M)$ only knows about the values of the sheaf $\tau_{\geqslant 2\star}\Hom_R(-,M)$ on $R$ (plus even shifts, extensions, and retracts thereof), but not about the value on $\Null_R$.}
		But this fully faithful embedding has a right adjoint by \cite[Theorem~\chref{4.8.5.11}]{HA}, which sends a $\Fil\Sp_\solid$-module $\Mm$ with a distinguished object $X\in\Mm$ to $\underline{\operatorname{End}}^\star(X)\in\Alg_{\IE_1}(\Fil\Sp_{\solid})$. Composing with this right adjoint allows us to turn $R\mapsto \fil_{\ev/R}^\star R$ into a lax symmetric monoidal functor
		\begin{equation*}
			\fil_{\ev/-}^\star(-)\colon \Alg_{\IE_1}(\Sp_{\solid})\longrightarrow\Alg_{\IE_1}(\Fil\Sp_\solid)
		\end{equation*}
		and provides a symmetric monoidal natural transformation from $R\mapsto \Sh_\Sigma(\Perf_{\ev}(R_\solid),\Sp_{\solid})$ to $R\mapsto \Mod_{\fil_{\ev/R}^\star R}(\Fil\Sp_\solid)$. The unstraightening of the latter functor is the the pullback of $\LMod(\Fil\Sp_{\solid})\rightarrow \Alg_{\IE_1}(\Fil\Sp_{\solid})$ along $\fil_{\ev/-}(-)^\star$. We obtain a diagram
		\begin{equation*}
			\begin{tikzcd}[column sep=large]
				\Uu\rar\dar & \LMod(\Fil\Sp_{\solid})\dar\\
				\Alg_{\IE_1}(\Sp_{\solid})\rar["\fil_{\ev/-}^\star(-)"] & \Alg_{\IE_1}(\Fil\Sp_{\solid})
			\end{tikzcd}
		\end{equation*}
		with lax symmetric monoidal horizontal arrows. We can now finally define a functorial lax symmetric monoidal solid even filtration as the composite
		\begin{equation*}
			\fil_{\ev/-}(-)\colon \LMod(\Sp_\solid)\overset{G}{\longrightarrow}\Uu^{\geqslant 0}\longrightarrow \Uu\longrightarrow \LMod(\Fil\Sp_\solid)\,.
		\end{equation*}
	\end{numpar}
	\begin{numpar}[Calculus of solid evenness.]\label{par:CalculusOfSolidEvenness}
		Deviating from \cite[Definition~\chref{4.4}]{PerfectEvenFiltration}, let us call a left $R$-module $M$ \emph{solid ind-perfect even} if it can be written as a filtered colimit of solid perfect evens, and \emph{solid even flat} if $\pi_*(E\soltimes_RM)$ is concentrated in even degrees for any right $R$-module $E$ such that $\pi_*(E)$ is concentrated in even degrees. In the uncondensed setting these notions are equivalent by the \enquote{even Lazard theorem} \cite[Theorem~\chref{4.14}]{PerfectEvenFiltration}. In the solid setting it is still true that solid ind-perfect even modules are solid even flat (as we'll see). However, we don't know if the converse is true. Similarly, we don't know if \cite[Theorem~\chref{4.16}]{PerfectEvenFiltration} still works. In \cref{subsec:Nuclear}, we'll discuss what the problem is, and in \cref{subsec:SolidEvenFlat} we'll see how to fix this, at least under certain additional assumptions.
		
		Despite these problems, the formalism of \emph{$\pi_*$-even envelopes} can entirely be carried over to the solid setting: Any left $R$-module $M$ admits a map $M\rightarrow E$ such that:
		\begin{alphanumerate}
			\item $\cofib(M\rightarrow E)$ is ind-solid perfect even.
			\item $\pi_*(E)$ is concentrated in even degrees.
			\item for any other map $M\rightarrow F$ into a left $R$-module $F$ such that $\pi_*(F)$ is even, a dashed arrow can be found to make the following diagram commutative:
			\begin{equation*}
				\begin{tikzcd}
					& M\dlar\drar & \\
					E\ar[rr,dashed] & & F
				\end{tikzcd}
			\end{equation*}
		\end{alphanumerate}
		The proof is the same as in the uncondensed setting, except that we have to consider maps $\Sigma^n\Null_R\rightarrow M$ from odd suspensions of $\Null_R$.
	\end{numpar}
	\begin{numpar}[Comparison with the uncondensed theory.]\label{par:ComparisonWithUsualEven}
		Let $R$ be a discrete solid condensed ring and let $M$ be a discrete left $R$-module. Let $\fil_{\Pev}^\star M$ be Pstr\k{a}gowski's perfect even filtation, regarded as a filtered discrete solid spectrum. Since Pstr\k{a}gowski's category $\Perf_{\ev}(R)$ is a full sub-$\infty$-category of $\Perf_{\ev}(R_\solid)$, we get a canonical comparison map
		\begin{equation*}
			\fil_{\Pev}^\star M\longrightarrow \fil_{\ev}^\star M\,.
		\end{equation*}
		As a consequence of the fact that $\pi_*$-even envelopes still work, we'll be able to show in \cref{cor:SolidvsDiscreteEvenFiltration} that this comparison map is an equivalence whenever $M$ is homologically even!
	\end{numpar}
	
	\subsection{Recollections on trace-class morphisms and nuclear objects}\label{subsec:Nuclear}
	
	In contrast to the mostly smooth sailing of \cref{par:SolidPerfectEven}--\labelcref{cor:SolidvsDiscreteEvenFiltration}, it's not so clear how to transport Pstr\k{a}gowski's discussion of even flatness---in particular, the powerful results \cite[Theorems~\chref{4.14} and~\chref{4.16}]{PerfectEvenFiltration}---to the solid setting. The main problem is the following: In the proofs, Pstr\k{a}gowski repeatedly uses the trick that a map $P\rightarrow Q$ of perfect even $R$-modules can be equivalently described by a map $\IS\rightarrow P^\vee\otimes_RQ$. This doesn't work anymore in the solid setting, since most solid perfect even $R$-modules are \emph{not} dualisable, the quintessential example being $\Null_R$.
	
	This is not the first time that such a problem occurs in solid condensed mathematics. The usual way to deal with these issues (which will also work in our case) is to replace dualisable objects by the weaker notions of \emph{trace-class morphisms} and \emph{nuclear objects} that we'll review in this subsection.
	
	\begin{numpar}[Trace-class morphisms.]\label{par:TraceClass}
		Let $\Cc$ be a presentable symmetric monoidal%
		\footnote{By convention, this includes the assumption that $-\otimes-$ commutes with colimits in both variables, so the adjoint functor theorem is applicable.}
		$\infty$-category.
		Let $R$ be an $\IE_1$-algebra in $\Cc$. By Lurie's adjoint functor theorem, for all left $R$-modules $M$ and $N$ there exists an object $\Hhom_R(M,N)\in \Cc$ characterised by
		\begin{equation*}
			\Hom_\Cc\bigl(-,\Hhom_R(M,N)\bigr)\simeq \Hom_R(M\otimes -,N)\,.
		\end{equation*}
		We remark that $\Hhom_R(M,R)$ is naturally a right $R$-module. A morphism $\varphi\colon M\rightarrow N$ of left $R$-modules is called \emph{trace-class} if there exists a morphism $\eta\colon \IUnit_\Cc\rightarrow \Hhom_R(M,R)\otimes_RN$, such that $\varphi$ is the composition
		\begin{equation*}
			M\simeq M\otimes\IUnit_\Cc\overset{\eta}{\longrightarrow} M\otimes\Hhom_R(M,R)\otimes_RN\xrightarrow{\ev_M}R\otimes_RN\simeq N\,.
		\end{equation*}
		We often call $\eta$ the \emph{classifier of $\varphi$}.
	\end{numpar}
	
	Trace-class morphism have a number of nice properties. We'll often use the properties from  \cite[Lemma~\chref{8.2}]{Complex} as well as the following lemma.
	\begin{lem}\label{lem:TraceClassAbstractNonsense}
		Let $F\colon \Cc\rightarrow \Dd$ be a symmetric monoidal functor between presentable symmetric monoidal $\infty$-categories. Let $R\in\Alg_{\IE_1}(\Cc)$. By abuse of notation, we'll denote both $\Hhom_R(-,R)$ and $\Hhom_{F(R)}(-,F(R))$ by $(-)^\vee$.
		\begin{alphanumerate}
			\item There exists a natural transformation $F((-)^\vee)\Rightarrow F(-)^\vee$.\label{enum:DualTransformation}
			\item If $M\rightarrow N$ is a trace-class morphism in $\LMod_R(\Cc)$, then $N^\vee\rightarrow M^\vee$ is trace-class in $\RMod_R(\Cc)$ and $F(M)\rightarrow F(N)$ is trace-class in $\LMod_{F(R)}(\Dd)$.\label{enum:DualTraceClass}
			\item The commutative square in $\RMod_{F(R)}(\Dd)$ formed by the morphisms from \cref{enum:DualTransformation} and \cref{enum:DualTraceClass}\label{enum:DualTraceClassLift}
			\begin{equation*}
				\begin{tikzcd}
					F(N^\vee)\rar\dar & F(M^\vee)\dar\\
					F(N)^\vee\rar\urar[dashed] & F(M)^\vee
				\end{tikzcd}
			\end{equation*}
			admits a canonical diagonal map $F(N)^\vee\rightarrow F(M^\vee)$ that makes both triangles commute.
		\end{alphanumerate}
	\end{lem}
	\begin{proof}
		The natural transformation from \cref{enum:DualTransformation} is adjoint to $F((-)^\vee)\otimes_{F(R)} F(-)\Rightarrow F(R)$, which is in turn given by applying $F$ to the evaluation $(-)^\vee\otimes_R(-)\Rightarrow R$.
		
		Now let $M\rightarrow N$ be trace-class in $\LMod_R(\Cc)$ with classifier $\IUnit_\Cc\rightarrow M^\vee\otimes_R N$. If we apply $F$ to the classifier and compose with the morphism $F(M^\vee)\rightarrow F(M)^\vee$ from \cref{enum:DualTransformation}, we obtain a morphism $\IUnit_\Dd\rightarrow F(M^\vee)\otimes_{F(R)} F(N)\rightarrow F(M)^\vee\otimes_{F(R)} F(N)$, which serves as a classifier for $F(M)\rightarrow F(N)$. If we compose instead with $N\rightarrow N^{\vee\vee}$, we obtain $\IUnit_\Cc\rightarrow M^\vee\otimes_R N\rightarrow M^\vee\otimes_R N^{\vee\vee}$, which serves as a classifier for $N^\vee\rightarrow M^\vee$ being trace-class. This shows \cref{enum:DualTraceClass}. To show \cref{enum:DualTraceClassLift}, we construct the diagonal map $F(N)^\vee\rightarrow F(M^\vee)$ as follows:
		\begin{equation*}
			F(N)^\vee\longrightarrow F(M^\vee\otimes_RN)\otimes_\Dd F(N)^\vee\simeq F(M^\vee)\otimes_{F(R)}F(N)
			\otimes_\Dd F(N)^\vee\longrightarrow F(M^\vee)\,.
		\end{equation*}
		Here we use the classifier $\IUnit_\Cc\rightarrow M^\vee\otimes_RN$ and the evaluation map for $F(N)$.
	\end{proof}
	
	\begin{numpar}[Nuclear objects]\label{par:Nuclear}
		In addition to the assumptions from \cref{par:TraceClass}, let us now assume that $\Cc$ is stable, compactly generated, and $\IUnit_\Cc$ is compact.
		\begin{alphanumerate}
			\item A left $R$-module $M$ is called \emph{nuclear} if every morphism $P\rightarrow M$ from a compact left $R$-module $P$ is trace-class.
			\item We call a left $R$-module $M$ \emph{basic nuclear} if $M$ can be written as a sequential colimit $M\simeq \colimit (M_0\rightarrow M_1\rightarrow \dotsb)$ such that each transition map $M_n\rightarrow M_{n+1}$ is trace-class.
		\end{alphanumerate}
		We let $\Nuc(\LMod_R(\Cc))\subseteq \LMod_R(\Cc)$ denote the full sub-$\infty$-category spanned by the nuclear left $R$-modules.
	\end{numpar}
	\begin{thm}\label{thm:Nuclear}
		Let $\Cc$ be a presentable stable symmetric monoidal $\infty$-category such that $\Cc$ is compactly generated and the tensor unit $\IUnit_\Cc\in\Cc$ is compact. Let $R\in \Alg_{\IE_1}(\Cc)$
		\begin{alphanumerate}
			\item $\Nuc(\LMod_R(\Cc))\subseteq \LMod_R(\Cc)$ closed under shifts and colimits. Moreover, if $M$ is a nuclear left $R$-module and $X\in \Nuc(\Cc)$, then $M\otimes X\in \Nuc(\LMod_R(\Cc))$.\label{enum:Nuc}
			\item $\Nuc(\LMod_R(\Cc))$ is $\omega_1$-compactly generated and the $\omega_1$-compact objects are precisely the basic nuclears.\label{enum:Nucw1Generated}
			\item If $R\rightarrow S$ is a map of $\IE_1$-algebras in $\Cc$, then $S\otimes_R-\colon \LMod_R(\Cc)\rightarrow \LMod_S(\Cc)$ preserves the full sub-$\infty$-categories of nuclear objects.\label{enum:NucBaseChange}
			\item Suppose that for all compact left $R$-modules $P$ and all compact $C\in \Cc$ the tensor product $P\otimes C$ is still compact as a left $R$-module. If $P$ is compact and $M$ is nuclear, the natural map\label{enum:NucForgetfulFunctor}
			\begin{equation*}
				\Hhom_R(P,R)\otimes_R M\overset{\simeq}{\longrightarrow}\Hhom_R(P,M)
			\end{equation*}
			is an equivalence. Furthermore, if $R\rightarrow S$ is a map of $\IE_1$-algebras in $\Cc$ such that $S$ is nuclear as a left $R$-module, then the forgetful functor $\LMod_S(\Cc)\rightarrow \LMod_R(\Cc)$ preserves the full sub-$\infty$-categories of nuclear objects.
		\end{alphanumerate}
	\end{thm}
	\begin{proof}[Proof sketch]
		For parts~\cref{enum:Nuc} and~\cref{enum:Nucw1Generated}, the case $R\simeq \IUnit_\Cc$ is covered in \cite[Theorem~\chref{8.6}]{Complex}; the arguments given therein apply verbatim for general $R$ as well. For~\cref{enum:NucBaseChange}, it's straightforward to check that $S\otimes_R-$ preserves trace-class maps, hence basic nuclear objects and thus all nuclear objects by \cref{enum:Nucw1Generated}.
		
		For \cref{enum:NucForgetfulFunctor}, the assumption implies that every compact left $R$-module is also \emph{internally compact} in the sense that $\Hhom_R(P,-)$ preserves filtered colimits. We may thus reduce to the case where $M$ is basic nuclear. Write $M$ as a sequential colimit $M\simeq \colimit(M_0\rightarrow M_1\rightarrow \dotsb)$ with trace-class transition maps. If $\eta\colon \IUnit_\Cc\rightarrow \Hhom_R(M_n,R)\otimes_R M_{n+1}$ is a classifier for $M_n\rightarrow M_{n+1}$ and $c\colon {\Hhom_R(P,M_n)}\otimes{\Hhom_R(M_n,R)}\rightarrow \Hhom_R(P,R)$ is the canonical composition map, we get a commutative diagram
		\begin{equation*}
			\begin{tikzcd}
				\Hhom_R(P,M_n)\rar\dar["\eta"'] & \Hhom_R(P,M_{n+1})\\
				\Hhom_R(P,M_n)\otimes\Hhom_R(M_n,R)\otimes_R M_{n+1}\rar["c"] & \Hhom_R(P,R)\otimes_R M_{n+1}\uar
			\end{tikzcd}
		\end{equation*}
		Using these diagrams for all $n$ we see that $\colimit \Hhom_R(P,R)\otimes_R M_n\rightarrow \colimit\Hhom_R(P,M_n)$ has an inverse. It follows that $\Hhom_R(P,R)\otimes_RM\simeq \Hhom_R(P,M)$, as desired.
		
		Now let $N$ be a nuclear left $S$-module and let $P\rightarrow N$ be a map from a compact left $R$-module. Then $S\otimes_RP\rightarrow N$ is trace-class, because it factors through $S\otimes_RP\rightarrow S\otimes_RN$ and $S\otimes_R-$ preserves trace-class morphisms. If $\eta\colon \IUnit_\Cc\rightarrow \Hhom_S(S\otimes_RP,S)\otimes_SN$ is a classifier, we note $\Hhom_S(S\otimes_RP,S)\simeq \Hhom_R(P,S)\simeq \Hhom_R(P,R)\otimes_RS$ by our assumption that $S$ is nuclear. Thus $\Hhom_S(S\otimes_RP,S)\otimes_SN\simeq \Hhom_R(P,R)\otimes_RN$ and so $\eta$ is also a classifier witnessing $P\rightarrow N$ being trace-class. This shows that the forgetful functor $\LMod_S(\Cc)\rightarrow \LMod_R(\Cc)$ preserves the full sub-$\infty$-categories of nuclear objects.
	\end{proof}
	\begin{rem}\label{rem:NucInd}
		If $\Cc_0$ is a small stable symmetric monoidal $\infty$-category, then \cref{thm:Nuclear} can be applied to $\Ind(\Cc_0)$. Since every trace-class map in $\Ind(\Cc_0)$ factors through a compact object by \cite[Lemma~\chref{8.4}]{Complex}, we see that the basic nuclear objects in $\Ind(\Cc_0)$ are of the form $\indcolim(X_1\rightarrow X_2\rightarrow \dotsb)$, where each $X_n\rightarrow X_{n+1}$ is trace-class in $\Cc_0$.
		
		If $\Cc$ is a presentable stable symmetric monoidal $\infty$-category (hence $\Cc$ is large unless $\Cc\simeq 0$), one can still make sense of $\Nuc\Ind(\Cc)$ without running into set-theoretic problems. Indeed, if $\kappa$ is a sufficiently large regular cardinal such that $\Cc$ is $\kappa$-compactly generated and $\IUnit$ is $\kappa$-compact, the same argument as in \cite[Lemma~\chref{8.4}]{Complex} shows that every trace-class morphism in $\Cc$ factors through a $\kappa$-compact object. Then every basic nuclear object is equivalent to one in which each $X_n$ is $\kappa$-compact and so the basic nuclear objects in form an essentially small $\infty$-category. We may then define $\Nuc\Ind(\Cc)$ as $\Ind_{\omega_1}(-)$ of the $\infty$-category of basic nuclear objects.
	\end{rem}
	
	\subsection{Solid even flatness in the nuclear case}\label{subsec:SolidEvenFlat}
	In this subsection we explain that the analogues of \cite[Theorems~\chref{4.14} and~\chref{4.16}]{PerfectEvenFiltration} are still true under certain additional nuclearity assumptions.

	\begin{numpar}[Assumptions on~$R$.]\label{par:SolidAssumptions}
		From now on let us assume that our solid $\IE_1$-algebra $R$ satisfies the following condition:
		\begin{alphanumerate}\itshape
			\item[R] $\Hhom_R(\Null_R,R)$ is nuclear and solid ind-perfect even both as a left $R$-module and as a right $R$-module.\label{enum:SolidAssumptions}
		\end{alphanumerate}	
		Here we use that $\Null_R\simeq \prod_\IN\IS\soltimes R$ is naturally a bimodule over $R$. Also note that Assumption~\cref{enum:SolidAssumptions} implies that that $\Hhom_R(P,R)$ is nuclear and solid ind-perfect even for any solid perfect even left  or right $R$-module $P$.
	\end{numpar}
	\begin{lem}\label{lem:NuclearityForDiscretepComplete}
		Let $R^\circ$ be a discrete condensed $\IE_1$-ring spectrum and let $M^\circ$ be any discrete condensed left $R^\circ$-module.
		\begin{alphanumerate}
			\item Assumption~\cref{par:SolidAssumptions}\cref{enum:SolidAssumptions} is satisfied for $R=R^\circ$. Moreover, $M^\circ$ is nuclear as a left $R^\circ$-module.\label{enum:Rdiscrete}
			\item Assumption~\cref{par:SolidAssumptions}\cref{enum:SolidAssumptions} is satisfied for $R=(R^\circ)_p^\complete$. Moreover, if $R^\circ$ is connective, then $(M^\circ)_p^\complete$ is nuclear over $(R^\circ)_p^\complete$.\label{enum:Rpcomplete}
			\item  Assumption~\cref{par:SolidAssumptions}\cref{enum:SolidAssumptions} is satisfied for $R=(R^\circ)_p^\complete[1/p]$. Moreover, if $R^\circ$ is connective, then $(M^\circ)_p^\complete[1/p]$ is nuclear over $(R^\circ)_p^\complete[1/p]$.\label{enum:Ranalytic}
		\end{alphanumerate}
	\end{lem}
	\begin{proof}
		In the following, we won't specify whether we're working with left  or right $R$-modules, since the arguments will be valid in either case. For arbitrary solid $\IE_1$-algebras~$R$, we have $\Hhom_R(\Null_R,R)\simeq \Hhom_\IS(\prod_\IN\IS,R)$. If $R=R^\circ$ is discrete, then $\Hhom_\IS(\prod_\IN\IS,R)\simeq \bigoplus_\IN R^\circ$, which is solid ind-perfect even. Since $R$ is nuclear over itself and nuclear objects are closed under shifts and colimits, it follows that every discrete $R$-module is nuclear. This shows \cref{enum:Rdiscrete}.
		
		If $R=(R^\circ)_p^\complete$, then the same argument shows $\Hhom_\IS(\prod_\IN\IS,R)\simeq (\bigoplus_\IN R^\circ)_p^\complete$. To show the solid ind-perfect evenness condition, write
		\begin{equation*}
			\biggl(\bigoplus_{\IN} R^\circ\biggr)_p^\complete\simeq \colimit_{\substack{f\colon \IN\rightarrow \IN,\\ f(n)\rightarrow \infty}}\prod_{\IN}p^{f(n)}R\,,
		\end{equation*}
		where the colimit is taken over all functions $f\colon \IN\rightarrow \IN$ such that $f(n)\rightarrow \infty$ as $n\rightarrow\infty$. We claim that whenever $g\leqslant f$ is growing so slowly that $f(n)-g(n)\rightarrow \infty$, the transition map $\prod_\IN p^{f(n)}R\rightarrow \prod_\IN p^{g(n)}R$ is trace-class and factors through $\Null_R$. This will show that every map from a compact left $R$-module to $(\bigoplus_\IN R^\circ)_p^\complete$ is trace-class and factors through $\Null_R$, so that $(\bigoplus_\IN R^\circ)_p^\complete$ is nuclear and solid ind-perfect even by the solid analogue of \cite[Proposition~\chref{4.3}]{PerfectEvenFiltration}.
		
		To show the claim, we may as well assume $g=0$ and show that $(p^{f(n)})_{n\in\IN}\colon \prod_\IN R\rightarrow \prod_\IN R$ is trace-class and factors through $\Null_R$. Let $e_n$ denote the $n$\textsuperscript{th} basis vector in the standard basis of $\bigoplus_{\IN}R^\circ$. Then $\sum p^{f(n)}(e_n\otimes e_n)$ is a well-defined $\pi_0$-class in $(\bigoplus_\IN \tau_{\geqslant 0}(R^\circ))_p^\complete\soltimes_{\tau_{\geqslant 0}(R)}\prod_\IN \tau_{\geqslant 0}(R)$, since the solid tensor product of connective $p$-complete objects will be $p$-complete again. The image of this $\pi_0$-class in $(\bigoplus_\IN R^\circ)_p^\complete \soltimes_R\prod_\IN R$ defines a morphism
		\begin{equation*}
			\IS\longrightarrow \Hhom_R(\Null_R,R)\soltimes_R\prod_\IN R\,,
		\end{equation*}
		which classifies a trace-class map $\Null_R\rightarrow \prod_\IN R$. By inspection, this is a factorisation of $(p^{f(n)})_{n\in\IN}\colon \prod_\IN R\rightarrow \prod_\IN R$, as desired.
		
		This argument shows, in particular, that the $p$-completion of any countable direct sum of copies of $R^\circ$ is nuclear over $R$. We deduce the same for arbitrary direct sums, as $p$-completion commutes with $\omega_1$-filtered colimits. Now suppose $R^\circ$ is connective. First consider the case where $M^\circ$ is bounded below.  Let $M$ be the $p$-completion of $M^\circ$. Define a sequence of left $R$-modules $M_0,M_1,\dotsc$ as follows: $M_0\coloneqq M$; for $n\geqslant 0$, we choose a map $\bigoplus \Sigma^nR^\circ\rightarrow M_n$ that is surjective on $\pi_n$ and then define $M_{n+1}\coloneqq \cofib(\bigoplus \Sigma^n R^\circ\rightarrow M_n)_p^\complete$. Then $M\simeq \colimit \fib(M\rightarrow M_n)$; note that the colimit doesn't need to be $p$-completed, since each term is $p$-complete and in each homotopical degree the colimit stabilises after finitely many steps. Thus, it will be enough to check that each $\fib(M\rightarrow M_n)$ is nuclear, which follows from our observation that $p$-completions of arbitrary direct sums of copies of $R^\circ$ are nuclear. This shows that $(M^\circ)_p^\complete$ is nuclear in the bounded below case. For general $M^\circ$, note that $(M^\circ)_p^\complete$ and $(\tau_{\geqslant -n} M^\circ)_p^\complete$ agree in homotopical degrees $\geqslant -n+1$. It follows that $(M^\circ)_p^\complete\simeq \colimit_{n\geqslant 0} (\tau_{\geqslant -n}M^\circ)_p^\complete$. By the bounded below case, this is a (non-$p$-completed) colimit of nuclear objects and so $(M^\circ)_p^\complete$ must be nuclear too. This finishes the proof of \cref{enum:Rpcomplete}.
		
		If $R=(R^\circ)_p^\complete[1/p]$, then $\Hhom_\IS(\prod_\IN\IS,R)\simeq (\bigoplus_\IN R^\circ)_p^\complete[1/p]$ by compactness of $\prod_\IN\IS$. The desired assertions then follow from \cref{enum:Rpcomplete} using base change for nuclear modules (\cref{thm:Nuclear}\cref{enum:NucBaseChange}). This shows \cref{enum:Ranalytic}.
	\end{proof}	
	
	Under Assumption~\cref{par:SolidAssumptions}\cref{enum:SolidAssumptions}, we can show the following weaker analogue of the \enquote{even Lazard theorem} \cite[Theorem~\chref{4.14}]{PerfectEvenFiltration}.
	\begin{lem}\label{lem:LazardsTheorem}
		Let $R$ be a solid condensed $\IE_1$-ring spectrum and let $M$ be a left $R$-module.
		\begin{alphanumerate}
			\item If $M$ is solid ind-perfect even, then $M$ is solid even flat.\label{enum:IndPerfectEven}
			\item Let $M$ be solid even flat. If $R$  satisfies Assumption~\cref{par:SolidAssumptions}\cref{enum:SolidAssumptions} and $M$ is nuclear, then is solid ind-perfect even.\label{enum:EvenFlat}
		\end{alphanumerate}
	\end{lem}
	\begin{proof}
		For~\cref{enum:IndPerfectEven}, we only need to check that $\Null_R$ is solid even flat. This follows from the fact that $\Null_\IZ\simeq \prod_\IN\IZ$ is flat for the solid tensor product on $\cat{Ab}_\solid$ by \cite[Lecture~\href{https://youtu.be/KKzt6C9ggWA?list=PLx5f8IelFRgGmu6gmL-Kf_Rl_6Mm7juZO&t=3027}{6}]{AnalyticStacks}.
		
		For~\cref{enum:EvenFlat}, let $\varphi\colon P\rightarrow M$ be a map from a compact left $R$-module. By the solid analogue of \cite[Proposition~\chref{4.3}]{PerfectEvenFiltration}, it will be enough to show that $\varphi$ factors through a solid perfect even. Since $M$ is nuclear, $\varphi$ will be trace-class, with classifier $\eta\colon \IS\rightarrow \Hhom_R(P,R)\soltimes_R M$. As in the proof of \cite[Theorem~\chref{4.14}]{PerfectEvenFiltration}, let us choose a map $\Hhom_R(P,R)\rightarrow E$ whose suspension is a $\pi_*$-even envelope in right $R$-modules. Then $\pi_*(E\soltimes_R M)$ is concentrated in odd degrees, hence the composite
		\begin{equation*}
			\IS\longrightarrow \Hhom_R(P,R)\soltimes_R M\longrightarrow E\soltimes_R M
		\end{equation*}
		must vanish.%
		\footnote{This argument still works with condensed homotopy groups since any cover of the one-point set $*$ in the site of light profinite sets is split.}
		It follows that the classifier $\eta$ lifts to a map $\eta'\colon\IS\rightarrow \Sigma^{-1}C\soltimes_R M$, where $C\simeq \cofib(\Hhom_R(P,R)\rightarrow E)$. By definition of $\pi_*$-even envelopes, $\Sigma^{-1}C$ is solid ind-perfect even as a right $R$-module. Writing $\Sigma^{-1}C$ as a filtered colimit of solid perfect evens and using that $\IS$ is compact, we obtain a further factorisation
		\begin{equation*}
			\begin{tikzcd}
				\IS\rar["\eta''"]\drar["\eta"'] & Q\soltimes_R M \dar\\
				& \Hhom_R(P,R)\soltimes_RM
			\end{tikzcd}
		\end{equation*}
		where $Q$ is solid perfect even. Assumption~\cref{par:SolidAssumptions}\cref{enum:SolidAssumptions} guarantees that $\Hhom_R(P,R)$ is nuclear, hence the composition $Q\rightarrow \Sigma^{-1}C\rightarrow \Hhom_R(P,R)$ is trace-class as a map of right $R$-modules. Choose a classifier $\vartheta\colon \IS\rightarrow \Hhom_R(P,R)\soltimes_R\Hhom_R(Q,R)$. We see that the original map $\varphi\colon P\rightarrow M$ is given by tensoring $P$ with $\eta''$ and $\vartheta$ and then applying the evaluation maps $\ev_Q\colon \Hhom_R(Q,R)\soltimes Q\rightarrow R$ and  $\ev_P\colon P\soltimes\Hhom_R(P,R)\rightarrow R$. This can be done in any order, hence $\varphi$ also agrees with the composition
		\begin{equation*}
			P\overset{\vartheta}{\longrightarrow}P\soltimes P^\vee\soltimes_R\Hhom_R(Q,R)\overset{\ev_P}{\longrightarrow}\Hhom_R(Q,R)\overset{\eta''}{\longrightarrow}\Hhom_R(Q,R)\soltimes Q\soltimes_RM\xrightarrow{\ev_M}M\,,
		\end{equation*}
		where we wrote $P^\vee\coloneqq\Hhom_R(P,R)$ for short. We conclude that $\varphi$ factors through $\Hhom_R(Q,R)$. Again by Assumption~\cref{par:SolidAssumptions}\cref{enum:SolidAssumptions}, $\Hhom_R(Q,R)$ is a filtered colimit of solid perfect even left $R$-modules. Since $P$ is compact, we conclude that $\varphi\colon P\rightarrow M$ factors through a solid perfect even left $R$-module, as desired.
	\end{proof}
	We can also show the following weaker analogue of \cite[Theorem~\chref{4.16}]{PerfectEvenFiltration}.
	\begin{lem}\label{lem:HomologicallyEven}
		Let $R$ be a solid condensed $\IE_1$-ring spectrum and let $M$ be a left $R$-module.
		\begin{alphanumerate}
			\item $M$ is solid homologically even if and only if every map $P\rightarrow \Sigma M$, where $P$ is solid perfect even, factors through a map $P\rightarrow \Sigma Q$, where $Q$ is solid perfect even.\label{enum:HomologicallyEvenSuspension}
			\item Suppose $M$ is solid homologically even. If $E$ is a solid even flat right $R$-module such that $\pi_*(E)$ is even, then any map $\IS\rightarrow E\soltimes_R\Sigma M$ vanishes.\label{enum:HomologicallyEven}
			\item Suppose $R$ satisfies Assumption~\cref{par:SolidAssumptions}\cref{enum:SolidAssumptions} and $M$ is nuclear. Suppose furthermore that for any solid ind-perfect even right $R$-module $E$ such that $\pi_*(E)$ is even, any morphism $\IS\rightarrow E\soltimes_R\Sigma M$ vanishes. Then $M$ is solid homologically even. In particular, this applies if $M$ is nuclear and solid even flat.\label{enum:TensorProductWithEvenFlat}
		\end{alphanumerate}
	\end{lem}
	\begin{proof}
		For part~\cref{enum:HomologicallyEvenSuspension}, the proof of \cite[Theorem~\chref{4.16}(2)]{PerfectEvenFiltration} can be copied verbatim. For~\cref{enum:HomologicallyEven}, let $\eta\colon \IS\rightarrow E\soltimes_R \Sigma M$ be any map. Let $M\rightarrow F$ be a $\pi_*$-even envelope and let $C\coloneqq \cofib(M\rightarrow F)$. Since $E$ is solid even flat, $\pi_*(E\soltimes_R\Sigma F)$ is concentrated in odd degrees and so the composite
		\begin{equation*}
			\IS\longrightarrow E\soltimes_R\Sigma M\longrightarrow E\soltimes_R\Sigma F
		\end{equation*}
		must vanish. Choosing a null-homotopy, we see that $\eta$ factors through a map $\eta'\colon \IS\rightarrow E\soltimes_R C$. By assumption, $C$ is solid ind-perfect even. Since $\IS$ is compact, $\eta'$ factors through another map $\eta''\colon \IS\rightarrow E\soltimes_RP$, where $P$ is solid perfect even. Since $M$ is solid homologically even, \cref{enum:HomologicallyEvenSuspension} shows that the composite $P\rightarrow C\rightarrow \Sigma M$ factors through $\Sigma Q$, where $Q$ is solid perfect even. Now $Q$ is solid even flat by \cref{lem:LazardsTheorem}\cref{enum:IndPerfectEven} and so $\pi_*(E\soltimes_R\Sigma Q)$ is concentrated in odd degrees. Thus any map $\IS\rightarrow E\soltimes_R\Sigma Q$ vanishes. Composing with $\Sigma Q\rightarrow \Sigma M$, we find that our original map $\IS\rightarrow E\soltimes_R\Sigma M$ must vanish as well, as desired.
		
		Let us now show \cref{enum:TensorProductWithEvenFlat}. Let $P\rightarrow \Sigma M$ be any map from a solid perfect even. Since $M$ is assumed to be nuclear, any such map is trace-class. Choose a classifier $\eta\colon \IS\rightarrow \Hhom_R(P,R)\soltimes_R \Sigma M$ as well as a $\pi_*$-even envelope $\Hhom_R(P,R)\rightarrow E$ in right $R$-modules. By Assumption~\cref{par:SolidAssumptions}\cref{enum:SolidAssumptions}, $\Hhom_P(P,R)$ is solid ind-perfect even, hence the same is true for any $\pi_*$-even envelope. Our assumption then implies that any map $\IS\rightarrow E\soltimes_R \Sigma M$ vanishes. It follows that $\eta$ factors through a map $\eta'\colon \IS\rightarrow \Sigma^{-1}C\soltimes_R \Sigma M$, where $C\coloneqq \cofib(\Hhom_R(P,R)\rightarrow E)$. By assumption, $C$ is solid ind-perfect even; since $\IS$ is compact, we find a solid perfect even right $R$-module $Q$ and a commutative diagram
		\begin{equation*}
			\begin{tikzcd}
				\IS\rar["\eta''"]\drar["\eta"'] & \Sigma^{-1}Q\soltimes_R \Sigma M \dar\\
				& \Hhom_R(P,R)\soltimes_R\Sigma M
			\end{tikzcd}
		\end{equation*}
		By Assumption~\cref{par:SolidAssumptions}\cref{enum:SolidAssumptions}, $\Hhom_R(P,R)$ is nuclear as a right $R$-module and so the composition $\Sigma^{-1}Q\rightarrow \Sigma^{-1}C\rightarrow \Hhom_R(P,R)$ is trace-class. Arguing as in the proof of \cref{lem:LazardsTheorem}\cref{enum:EvenFlat}, we find that our original map $P\rightarrow \Sigma M$ factors through $\Hhom_R(\Sigma^{-1}Q,R)$. By Assumption~\cref{par:SolidAssumptions}\cref{enum:SolidAssumptions} again, $\Hhom_R(Q,R)$ is solid ind-perfect even. Writing $\Hhom_R(\Sigma^{-1}Q,R)\simeq \Sigma \Hhom_R(Q,R)$ as a filtered colimit of suspensions of solid perfect even left $R$-modules and using that $P$ is compact, we deduce that $P\rightarrow \Sigma M$ factors through the suspension of a solid perfect even left $R$-module, as desired.
		
		For the \enquote{in particular}, just observe that $M$ being solid even flat implies that $\pi_*(E\soltimes_R \Sigma M)$ is concentrated in odd degrees and so indeed any map $\IS\rightarrow E\soltimes_R\Sigma M$ vanishes.
	\end{proof}
	
	We can also show that the solid even filtration recovers Pstr\k{a}gowski's perfect even filtration in the homologically even case.
	
	\begin{cor}\label{cor:SolidvsDiscreteEvenFiltration}
		Let $R^\circ$ be a discrete condensed $\IE_1$-ring spectrum and let $M^\circ$ be a discrete condensed homologically even left $R^\circ$-module.
		\begin{alphanumerate}
			\item $M^\circ$ is solid homologically even as well.\label{enum:DiscreteImpliesSolidHomologicallyEven}
			\item The comparison map from Pstr\k{a}gowski's perfect even filtration to the solid even filtration \embrace{see \cref{par:ComparisonWithUsualEven}} is an equivalence\label{enum:DiscretevsSolidEvenFiltration}
			\begin{equation*}
				\fil_{\Pev}^\star M^\circ\overset{\simeq}{\longrightarrow} \fil_{\ev}^\star M^\circ
			\end{equation*}
		\end{alphanumerate}
		In particular, this applies in the case $M^\circ=R^\circ$.
	\end{cor}
	\begin{proof}
		For \cref{enum:DiscreteImpliesSolidHomologicallyEven} we'll verify the criterion from \cref{lem:HomologicallyEven}\cref{enum:HomologicallyEvenSuspension}. Let $\varphi\colon P\rightarrow \Sigma M^\circ$ be any map, where $P$ is solid perfect even. Since $M^\circ$ is nuclear by \cref{lem:NuclearityForDiscretepComplete}\cref{enum:Rdiscrete}, this map must be trace-class, with witness $\eta\colon \IS\rightarrow \Hhom_{R^\circ}(P,R^\circ)\otimes_R\Sigma M^\circ$. Now $\Hhom_{R^\circ}(P,R^\circ)$ is solid ind-perfect even. In fact, it is discrete and a filtered colimit of discrete left $R^\circ$-modules which are perfect even in Pstr\k{a}gowski's sense (which we'll call \emph{discrete} perfect even in the following). Indeed, this is clearly true for $\Hhom_{R^\circ}(\Null_{R^\circ},R^\circ)\simeq \bigoplus_\IN R^\circ$ and then it follows in general.
		
		Since $\IS$ is compact, $\eta$ must factor through a map $\eta^\circ\colon \IS\rightarrow
		P^\circ\otimes_R\Sigma M^\circ$, where $P^\circ$ is discrete perfect even. Then $P^\circ$ is dualisable and so $\eta^\circ$ corresponds to a map $\varphi^\circ\colon \Hhom_{R^\circ}(P^\circ,R^\circ)\rightarrow \Sigma M^\circ$ through which our original map $\varphi$ factors. Since $\Hhom_{R^\circ}(P^\circ,R^\circ)$ is still discrete perfect even, the assumption that $M^\circ$ is discrete homologically even ensures that $\varphi^\circ$ factors through a map $\Hhom_{R^\circ}(P^\circ,R^\circ)\rightarrow \Sigma Q^\circ$, with $Q^\circ$ discrete perfect even and so we're done.
		
		To show \cref{enum:DiscretevsSolidEvenFiltration}, it's straightforward to check that the construction of a $\pi_*$-even envelopes of $M$ as a discrete left $R$-module in \cite[Proposition~\chref{4.11}]{PerfectEvenFiltration} also yields a $\pi_*$-even envelope as a solid condensed left $R$-module.%
		\footnote{Implicitly, we use that discrete condensed abelian groups have vanishing higher cohomology on any light profinite set; see \cite[Lecture~\href{https://youtu.be/EW39K0J7Hqo?list=PLx5f8IelFRgGmu6gmL-Kf_Rl_6Mm7juZO&t=2142}{4}]{AnalyticStacks}.}
		Assuming homological evenness, both $\gr_{\Pev}^*M$ and $\gr_{\ev}^*M$ can be computed by repeatedly taking $\pi_*$-even envelopes, as explained in \cite[\S{\chref[section]{5}}]{PerfectEvenFiltration}. It follows that $\fil_{\Pev}^\star M\rightarrow\fil_{\ev}^\star M$ is an equivalence on associated gradeds. Since both filtrations are exhaustive, we conclude.
	\end{proof}
	
	\subsection{Solid faithfully flat descent in the nuclear case}\label{subsec:FlatDescent}
	
	In this subsection we'll show a flat descent result for the solid even filtration. We start with the definition of faithful flatness; it is slightly more restrictive than \cite[Definition~\chref{6.15}]{PerfectEvenFiltration}, but we expect that this doesn't cause any problems in practice.
	\begin{defi}\label{def:SolidFaithfullyEvenFlat}
		A map $R\rightarrow S$ of solid condensed $\IE_1$-algebras is called \emph{solid faithfully even flat} if $S$ and $\cofib(R\rightarrow S)$ are solid even flat both as left  and as right $R$-modules.
	\end{defi}
	
	\begin{thm}\label{thm:SolidEvenFlatDescent}
		Let $R\rightarrow S$ be a solid faithfully even flat map of solid condensed $\IE_1$-algebras such that $R$ satisfies Assumption~\cref{par:SolidAssumptions}\cref{enum:SolidAssumptions} and $S$ is nuclear as a left $R$-module. We denote the \v Cech nerve of $R\rightarrow S$ by $R\rightarrow S^\bullet$. Then for every nuclear solid homologically even left $R$-module $M$, the canonical map
		\begin{equation*}
			\fil_{\ev/R}^\star M\longrightarrow \limit_{\IDelta}\fil_{\ev/R}^\star (S^\bullet\soltimes_R M)
		\end{equation*}
		is an equivalence up to completing the filtrations on either side.
	\end{thm}
	\begin{proof}
		Put $C\coloneqq \cofib(R\rightarrow S)$ for short. First observe that $S\otimes_RM$ and $C\otimes_RM$ are again nuclear by \cref{thm:Nuclear}\cref{enum:NucBaseChange} and~\cref{enum:NucForgetfulFunctor}. If $E$ is any $\pi_*$-even and solid even flat right $R$-module, then $E\soltimes_RS$ is $\pi_*$-even and solid even flat since $S$ is solid even flat both as as a left  and as a right $R$-module. Using that $M$ is solid homologically even, we find that any map $\IS\rightarrow E\soltimes_RS\soltimes_R\Sigma M$ vanishes by \cref{lem:HomologicallyEven}\cref{enum:HomologicallyEven}. Since $S\soltimes_RM$ is nuclear, we conclude that it must be solid homologically even by \cref{lem:HomologicallyEven}\cref{enum:TensorProductWithEvenFlat}. The same argument applies to $C\soltimes_RM$.
		
		Therefore we get a short exact sequence $0\rightarrow \Ff_M\rightarrow \Ff_{S\soltimes_RM}\rightarrow \Ff_{C\soltimes_RM}\rightarrow 0$. Arguing as in the proof of \cite[Theorem~\chref{6.26}]{PerfectEvenFiltration}, we conclude that the Moore complex
		\begin{equation*}
			0\longrightarrow \Ff_M\longrightarrow \Ff_{S\soltimes_RM}\longrightarrow \Ff_{S\soltimes_RS\soltimes_RM}\longrightarrow \dotsb
		\end{equation*}
		is exact. Replacing $M$ by an even suspension, we deduce the same for $\Ff_{(-)}(w)$ for every integral weight $w\in\IZ$. For proper half-integral weights $w\in\frac12+\IZ$ this is true as well for trivial reasons, since our argument above shows that all terms in $S^\bullet\soltimes_R M$ are homologically even. We can thus apply the solid analogue of \cite[Proposition~\chref{5.5}]{PerfectEvenFiltration}.
	\end{proof}
	We also need the following variant of faithfully flat descent.
	\begin{thm}\label{thm:SolidEvenFlatDescentVariant}
		Let $R_0$ be a solid condensed $\IE_\infty$-algebra and let $S_0$ be an $\IE_1$-algebra in $R_0$-modules such that $R_0\rightarrow S_0$ is solid faithfully even flat and $S_0$ is nuclear over $R_0$. We denote the \v Cech nerve of $R_0\rightarrow S_0$ by $R_0\rightarrow S_0^\bullet$. Let $R_0\rightarrow R$ be another map of solid condensed $\IE_1$-algebras such that $R$ satisfies Assumption~\cref{par:SolidAssumptions}\cref{enum:SolidAssumptions}. Then for every solid homologically flat $R_0$-module $M_0$, the canonical map
		\begin{equation*}
			\fil_{\ev/R}^\star (R\soltimes_{R_0}M_0)\longrightarrow \limit_{\IDelta}\fil_{\ev/R}^\star (R\soltimes_{R_0}M_0\soltimes_{R_0}S_0^\bullet)
		\end{equation*}
		is an equivalence up to completing the filtrations on both sides.
	\end{thm}
	\begin{proof}
		This doesn't follow from \cref{thm:SolidEvenFlatDescent} since we can't produce an $\IE_1$-structure on $R\soltimes_{R_0}S_0$. But the argument can be adapted in a straightforward way.
		
		Let $C_0\coloneqq \cofib(R_0\rightarrow S_0)$. A combination of \cref{thm:Nuclear}\cref{enum:NucBaseChange} and~\cref{enum:NucForgetfulFunctor} shows again that $R\soltimes_{R_0}M_0\soltimes_{R_0}S_0$ and $R\soltimes_{R_0}M_0\soltimes_{R_0}C_0$ are nuclear over $R$. Moreover, both are solid even flat as left $R$-modules, hence solid homologically even by \cref{lem:HomologicallyEven}\cref{enum:TensorProductWithEvenFlat}. It follows that
		\begin{equation*}
			0\rightarrow \Ff_{R\soltimes_{R_0}M_0}\longrightarrow\Ff_{R\soltimes_{R_0}M_0\soltimes_{R_0}S_0}\longrightarrow \Ff_{R\soltimes_{R_0}M_0\soltimes_{R_0}C_0}\longrightarrow 0
		\end{equation*}
		is a short exact sequence. Since the cosimplicial $R_0$-module $M_0\soltimes_{R_0}S_0\soltimes_{R_0}S_0^\bullet$ is split, we can still use an analogous argument as in the proof of \cite[Theorem~\chref{6.26}]{PerfectEvenFiltration} to conclude that the Moore complex
		\begin{equation*}
			0\longrightarrow \Ff_{R\soltimes_{R_0}M_0}\longrightarrow \Ff_{R\soltimes_{R_0}M_0\soltimes_{R_0}S_0}\longrightarrow \Ff_{R\soltimes_{R_0}M_0\soltimes_{R_0}S_0\soltimes_{R_0}S_0}\longrightarrow \dotsb
		\end{equation*}
		is exact. The same follows for $\Ff_{(-)}(w)$ for every half-integral weight $w$: If $w\in\IZ$, replace $M_0$ by an even suspension, otherwise exactness holds for trivial reasons as the whole complex vanishes by solid homological evenness. We can thus apply the solid analogue of \cite[Proposition~\chref{5.5}]{PerfectEvenFiltration} again to finish the proof.
	\end{proof}
	\begin{rem}
		Note that $M=R$ satisfies the nuclearity and homological evenness assumption in \cref{thm:SolidEvenFlatDescent}. Similarly, $M_0=R$ satisfies the assumptions in \cref{thm:SolidEvenFlatDescentVariant}. So in either case we get a way of computing $\fil_{\ev/R}^\star R$ via descent, provided $R$ satisfies Assumption~\cref{par:SolidAssumptions}\cref{enum:SolidAssumptions}. 
	\end{rem}
	
	\newpage

	\section{The solid even filtration for \texorpdfstring{$\THH$}{THH}}\label{sec:SolidEvenFiltrationku}
	
	The purpose of this section is to construct and study an appropriate even filtration on $\TC^-(\ku_R/\ku_A)$, where $\ku_A$ and $\ku_R$ denote certain lifts to $\ku$ of rings $A$ and $R$ (subject to strong additional assumptions to be specified below). In the subsequent section \cref{sec:qdeRhamku} we'll show that the associated graded of this even filtration is closely related to the $q$-de Rham complex $\qdeRham_{R/A}$.
	
	Throughout \cref{sec:SolidEvenFiltrationku} and \crefrange{subsec:qdeRhamkupComplete}{subsec:QuasiRegular}, we fix a prime~$p$ as well as rings $A$ and $R$ satisfying the following assumptions: 
	
	\begin{numpar}[Assumptions on $A$.]\label{par:AssumptionsOnA}
		We let $A$ be a $p$-complete and \emph{$p$-completely perfectly covered} $\delta$-ring. That is, the Frobenius $\phi\colon A\rightarrow A$ is $p$-completely faithfully flat; equivalently, $A$ admits a $p$-completely faithfully flat $\delta$-ring map $A\rightarrow A_\infty$ into a perfect $\delta$-ring. We assume that $A$ is equipped with the following additional structure:
		\begin{alphanumerate}\itshape
			\item[^{\t C_p\!}] $A$ has a lift to a $p$-complete connective $\IE_\infty$-ring spectrum $\IS_A$ such that $\IS_A\otimes_{\IS_p}\IZ_p\simeq A$ and such that the {Tate-valued Frobenius}\label{enum:CyclotomicLift}
			\begin{equation*}
				\phi_{\t C_p}\colon \IS_A\longrightarrow \IS_A^{\t C_p}
			\end{equation*}
			agrees with the $\delta$-ring Frobenius $\phi\colon A\rightarrow A$ on $\pi_0$. Furthermore, $\phi_{\t C_p}$ must be equipped with an $S^1$-equivariant structure as a map of $\IE_\infty$-ring spectra, where $\IS_A$ receives the trivial $S^1$-action and $\IS_A^{\t C_p}$ the induced $S^1\simeq S^1/C_p$-action.
		\end{alphanumerate}
		The $S^1$-equivariant structure in \cref{enum:CyclotomicLift} ensures that $\IS_A$ is a \emph{$p$-cyclotomic base}: By the universal property of $\THH$, the augmentation $\THH(\IS_A)\rightarrow \IS_A$ becomes a map of $\IE_\infty$-algebras in cyclotomic spectra in a unique way, where the $p$-cyclotomic Frobenius on $\IS_A$ is $\phi_{\t C_p}$ with its chosen $S^1$-equivariant structure. In particular, $\THH(-/\IS_A)\simeq \THH(-)\otimes_{\THH(\IS_A)}\IS_A$ carries a $p$-cyclotomic structure. We also put $\ku_A\coloneqq (\ku\otimes\IS_A)_p^\complete$.
	\end{numpar}
	\begin{numpar}[Assumptions on $R$.]\label{par:AssumptionsOnR}
		We let $R$ be a $p$-complete $A$-algebra of bounded $p^\infty$-torsion. We assume that $R$ is $p$-\emph{quasi-lci} over $A$ in the sense that the cotangent complex $\L_{R/A}$ has $p$-complete $\Tor$-amplitude in homological degrees $[0,1]$ over $R$. In addition, one of the following two conditions must be satisfied:
		\begin{alphanumerate}\itshape
			\item[\IE_2] $R$ has a lift to a $p$-complete connective $\IE_2$-algebra $\IS_R\in\Alg_{\IE_2}(\Mod_{\IS_A}(\Sp))$ such that  $\IS_R\otimes_{\IS_p}\IZ_p\simeq R$.\label{enum:E2Lift}
			\item[\IE_1] $R$ is $p$-torsion free and has a $p$-quasi-syntomic cover $R\rightarrow R_\infty$ such that:\label{enum:E1Lift}
			\begin{alphanumerate}
				\item $R_\infty/p$ is {relatively semiperfect over $A$} in the sense that its relative Frobenius over the $\delta$-ring $A$ is a surjection $R_\infty/p\otimes_{A,\phi}A\twoheadrightarrow R_\infty/p$.
				\item If $R_\infty^\bullet$ denotes the $p$-completed \v Cech nerve of $R\rightarrow R_\infty$, then the augmented cosimplicial diagram $R\rightarrow R_\infty^\bullet$ has a lift to an augmented cosimplicial diagram $\IS_R\rightarrow \IS_{R_\infty^\bullet}$ in $\Alg_{\IE_1}(\Mod_{\IS_A}(\Sp))$, which is $p$-complete and connective in every degree.
			\end{alphanumerate}
		\end{alphanumerate}
		We put $\ku_R\coloneqq (\ku\otimes \IS_R)_p^\complete$ and, in case~\cref{enum:E1Lift}, $\ku_{R_\infty^\bullet}\coloneqq (\ku\otimes\IS_{R_\infty^\bullet})_p^\complete$.
	\end{numpar}
	\begin{rem}
		Even though the assumptions in~\cref{par:AssumptionsOnA} and~\cref{par:AssumptionsOnR} seem quite restrictive, they allow for many interesting examples, as we'll see in \cref{subsec:CyclotomicBases}.
	\end{rem}
	\begin{rem}\label{rem:AdHocEvenFiltration}
		Let us motivate the rather artificial condition \cref{par:AssumptionsOnR}\cref{enum:E1Lift}. If our lifts are only $\IE_1$, there's no even filtration on $\TC^-(\ku_R/\ku_A)_p^\complete$. However, if $\TC^-(\ku_R/\ku_A)_p^\complete$ happens to be an even spectrum, then we can still consider its double-speed Whitehead filtration $\tau_{\geqslant 2\star}\TC^-(\ku_R/\ku_A)_p^\complete$. This case turns out to be quite interesting: As we'll see in \cref{subsec:QuasiRegular}, the $q$-deformation of the Hodge filtration that we get in this case is independent of the choice of the $\IE_1$-lift $\IS_R$! This is the reason why we don't content ourselves with the $\IE_2$-case.
		
		More generally, given a resolution $\IS_R\rightarrow \IS_{R_\infty^\bullet}$ as in \cref{par:AssumptionsOnR}\cref{enum:E1Lift}, then $\TC^-(\ku_{R_\infty^\bullet}/\ku_A)_p^\complete$ is even in every cosimplicial degree, so we can use it to define an ad-hoc replacement of the even filtration. Indeed, evenness can be checked modulo~$\beta$, so we only need to check that $\HC^-(R_\infty^\bullet/A)_p^\complete$ is even. By assumption, $R_\infty/p$ is relatively semiperfect over $A$, hence the same is true for $R_\infty^\bullet/p$ in every cosimplicial degree. Then the desired evenness follows from \cite[Lemma~\chref{4.18}({\chref[Item]{43}[$a$]})]{qWittHabiro} and \cite[Theorem~\chref{1.17}]{BMS2}.
	\end{rem}
	\begin{rem}
		Throughout \cref{sec:SolidEvenFiltrationku}, we won't use that the lifts $\ku_A$ and $\ku_R$ come from spherical lifts $\IS_A$ and $\IS_R$, nor will we use the structure of a $p$-cyclotomic base on $\IS_A$. But for the comparison with $q$-de Rham cohomology in \cref{sec:qdeRhamku}, these assumptions will become relevant.
	\end{rem}
	\subsection{Solid \texorpdfstring{$\THH$}{THH}}
	Throughout \crefrange{sec:SolidEvenFiltrationku}{sec:qdeRhamku}, we'll work in the world of \emph{solid condensed spectra} (see \cref{par:CondensedRecollections}). In many cases, it makes no difference whether we work solidly or $p$-completely; for the most part, the reader not familiar with the solid theory may safely replace each \enquote{$\solid$} by a $p$-completion. But working solidly has the advantage that that $\THH$ will often automatically be $p$-complete (\cref{lem:SolidTHH}). This simplifies the $p$-completed descent for the even filtration (\cref{lem:eff}) and it makes it much easier to deal with rationalisations, as not having to $p$-complete allows us to appeal directly to the fact that $\ku_p^\complete\otimes\IQ\simeq \IQ_p[\beta]$.
	
	\begin{numpar}[Convention.]\label{conv:SolidConventions}
		For readability we'll adopt the following abusive convention: If $X$ is a $p$-complete spectrum, we'll identify $X$ with the solid condensed spectrum $\underline{X}_p^\complete$, otherwise we identify $X$ with the discrete solid condensed spectrum $\underline{X}$. In particular, we'll regard $\ku$ as a discrete condensed spectrum, but $\ku_R$ and $\ku_A$ as a $p$-complete ones.
	\end{numpar}
	For any $\IE_\infty$-algebra $k$ in $\Sp_\solid$, the module $\infty$-category $\Mod_k(\Sp_\solid)$ is symmetric monoidal for the solid tensor product $-\soltimes_{k}-$. We can then consider topological Hochschild homology inside $\Mod_{k}(\Sp_\solid)$. This yields a functor
	\begin{equation*}
		\THH_\solid(-/k)\colon \Alg_{\IE_1}\bigl(\Mod_{k}(\Sp_\solid)\bigr)\longrightarrow \Mod_{k}(\Sp_\solid)^{\B S^1}\,.
	\end{equation*}
	We also let $\TC_\solid^-(-/k)\coloneqq \THH_\solid(-/k)^{\h S^1}$ and $\TP_\solid(-/k)\coloneqq \THH_\solid(-/k)^{\t S^1}$, where the fixed points and Tate construction are taken inside $\Mod_k(\Sp_\solid)^{\B S^1}$.
	
	\begin{lem}\label{lem:SolidTHH}
		Let $k^\circ$ be a discrete connective $\IE_\infty$-ring spectrum and let $T^\circ$ be a discrete connective $\IE_1$-algebra in $k^\circ$-modules. Let $k\coloneqq (k^\circ)_p^\complete$ and $T\coloneqq (T^\circ)_p^\complete$. Then solid condensed spectrum $\THH_\solid(T/k)$ is the $p$-completion of the discrete spectrum $\THH(T^\circ/k^\circ)$.
	\end{lem}
	\begin{proof}
		By the magical property of the solid tensor product, 
		\begin{equation*}
			\THH_\solid(T/k)\simeq T\soltimes_{T^\op\soltimes_{k}T}T
		\end{equation*}
		is again $p$-complete. Hence we get a map $\THH(T^\circ/k^\circ)_p^\complete\rightarrow \THH_\solid(T/k)$. Whether this map is an equivalence can be checked modulo $p^5$. By Burklund's result \cite[Theorem~\chref{1.2}]{BurklundMooreSpectra}, the quotient $k/p^5\simeq k\soltimes\IS/p^5$ admits an $\IE_2$-$k$-algebra structure, and so we may regard $T/p^5\simeq T\soltimes_{k}k/p^5$ as an $\IE_1$-algebra in the $\IE_1$-monoidal $\infty$-category $\RMod_{k/p^5}(\Sp_\solid)$. Since $k/p^5\simeq k^\circ\otimes\IS/p^5$ and $T/p^5\simeq T^\circ\otimes\IS/p^5$ are discrete and the inclusion of discrete objects into all solid condensed spectra preserves tensor products, we obtain
		\begin{equation*}
			\THH(T^\circ/k^\circ)_p^\complete/p^5\simeq (T/p^5)\soltimes_{(T/p^5)^\op\soltimes_{k/p^5}(T/p^5)}(T/p^5)\simeq \THH_\solid(T/k)/p^5\,.\qedhere
		\end{equation*}
	\end{proof}
	
	\subsection{The solid even filtration via even resolutions}\label{subsec:EvenResolution}
	
	Let us now construct the desired even filtrations. We'll use the adaptation of Pstr\k{a}gowski's \emph{perfect even filtration} to the solid setting that we've sketched in \cref{sec:SolidEvenFiltration}.
	
	Throughout this subsection, we'll fix a connective even $\IE_\infty$-ring spectrum~$k$ such that $\pi_{2*}(k)$ is $p$-torsion free. The example of interest is of course $k= \ku$, but we'll later apply the same results in other cases as well (e.g.\ for $\ku\otimes\IQ$ or the geometric fixed points $\ku^{\Phi C_m}$), so the additional generality will be worthwhile. We put $k_A\coloneqq k\soltimes\IS_A$, $k_R\coloneqq k\soltimes\IS_R$, and in case~\cref{par:AssumptionsOnR}\cref{enum:E1Lift} also $k_{R_\infty^\bullet}\coloneqq k\soltimes\IS_{R_\infty^\bullet}$, where we regard $k$, $\IS_A$, and $\IS_R$ as solid condensed spectra per Convention~\cref{conv:SolidConventions}. Note that these are all even by our assumptions on $k$, $A$, and $R$, but they are not necessarily $p$-complete; in the case $k=\ku$ however, $p$-completeness is satisfied. 
	
	\begin{numpar}[Even filtrations.]\label{par:EvenFiltration}
		If we are in situation~\cref{par:AssumptionsOnR}\cref{enum:E2Lift}, then $\THH_\solid(k_R/k_A)$ is an $\IE_1$-algebra and so we can define 
		\begin{equation*}
			\fil_{\ev}^\star \THH_\solid(k_R/k_A)
		\end{equation*}
		to be its solid even filtration as a module over itself. For $k=\ku$, we'll see in \cref{cor:SolidvsPstragowski} below that $\fil_{\ev}^\star\THH_\solid(\ku_R/\ku_A)$ is the $p$-completion of Pstr\k{a}gowski's perfect even filtration on the discrete $\IE_1$-ring spectrum $\THH(\ku_R/\ku_A)$. For $k=\IZ$, we'll see in \cref{cor:HRWvsPstragowski}, that $\fil_{\ev}^\star \HH_\solid(R/A)$ agrees with the Hahn--Raksit--Wilson/HKR filtration on $\HH(R/A)_p^\complete$.
		
		In situation~\cref{par:AssumptionsOnR}\cref{enum:E1Lift}, $\THH_\solid(k_R/k_A)$ doesn't have any multiplicative structure; instead, we use the following ad-hoc definition as discussed in \cref{rem:AdHocEvenFiltration}:
		\begin{equation*}
			\fil_{\ev}^\star\THH_\solid(k_R/k_A)\coloneqq \limit_{\IDelta}\tau_{\geqslant 2\star}\THH_\solid (k_{R_\infty^\bullet}/k_A)\,.
		\end{equation*}
		To define filtrations on $\TC_\solid^-(k_R/k_A)$ and $\TP_\solid(k_R/k_A)$ in either situation, we use a construction due to Pstr\k{a}gowski and Raksit that will appear in forthcoming work \cite{PstragowskiRaksit} and has already been used in \cite{CyclotomicSynthetic}. Let $\IS_{\ev}\coloneqq\fil_{\ev}^\star \IS$ and $\IT_{\ev}\coloneqq \fil_{\ev}^\star\IS[S^1]$ denote the even filtrations of $\IS$ and $\IS[S^1]$, respectively.%
		\footnote{It doesn't matter whether they are defined in à la Hahn--Raksit--Wilson or à la Pstr\k{a}gowski or in the solid setting. Indeed, by \cite[Theorem~\chref{7.5}]{PerfectEvenFiltration}, the Hahn--Raksit--Wilson filtration is the completion of Pstr\k{a}gowski's filtration in either case (to apply this result, we use that $\IS[S^1]\rightarrow \IS$ and $\IS\rightarrow \MU$ are eff by \cite[Corollary~\chref{2.36}]{CyclotomicSynthetic} and \cite[Proposition~\chref{2.2.20}]{EvenFiltration}). But the filtrations are also exhaustive: For  Pstr\k{a}gowski's, this is always the case, for the Hahn--Raksit--Wilson filtration of connective $\IE_\infty$-rings it is an unpublished result of Burklund and Krause. Finally, the comparison with the solid version is \cref{cor:SolidvsDiscreteEvenFiltration}.}
		Following \cite[Definition~\chref{2.11}]{CyclotomicSynthetic}, we define the $\infty$-category of \emph{synthetic solid condensed spectra} to be $\cat{SynSp}_\solid\coloneqq\Mod_{\IS_{\ev}}(\Fil\Sp_\solid)$. Then $\IT_{\ev}$ is a bicommutative bialgebra in $\cat{SynSp}_\solid$ and we can equip $\Mod_{\IT_{\ev}}(\cat{SynSp}_\solid)$ with the symmetric monoidal structure coming from the coalgebra structure on $\IT_{\ev}$. By monoidality of the even filtration, $\fil_{\ev}^\star\THH_\solid(k_R/k_A)$ is an object in $\Mod_{\IT_{\ev}}(\cat{SynSp}_\solid)$ (in case~\cref{par:AssumptionsOnR}\cref{enum:E2Lift} it is even an $\IE_1$-algebra). We can then finally define the desired filtrations as
		\begin{align*}
			\fil_{\ev,\h S^1}^\star\TC_\solid^-(k_R/k_A)&\coloneqq \bigl(\fil_{\ev}^\star\THH_\solid(k_R/k_A)\bigr)^{\h \IT_{\ev}}\,,\\
			\fil_{\ev,\t S^1}^\star\TP_\solid(k_R/k_A)&\coloneqq \bigl(\fil_{\ev}^\star\THH_\solid(k_R/k_A)\bigr)^{\t \IT_{\ev}}\,,
		\end{align*}
		where the fixed points and Tate constructions $(-)^{\h \IT_{\ev}}$ and $(-)^{\t \IT_{\ev}}$ with respect to $\IT_{\ev}$ are defined as in \cite[\S{\chref[subsection]{2.3}}]{CyclotomicSynthetic}.%
		\footnote{To avoid confusion with the \emph{genuine} fixed points that will appear later, we deviate from the notation in \cite{CyclotomicSynthetic} and write $(-)^{\h \IT_{\ev}}$ instead of $(-)^{\IT_{\ev}}$.}
	\end{numpar}

	In situation~\cref{par:AssumptionsOnR}\cref{enum:E1Lift}, the ad-hoc even filtration being given as a cosimplicial limit gives us good control over it. We'll now show a similar description in situation~\cref{par:AssumptionsOnR}\cref{enum:E2Lift}.
	
	\begin{numpar}[Even resolutions.]\label{par:EvenResolution}
		Assume we're in situation~\cref{par:AssumptionsOnR}\cref{enum:E2Lift}. Let $P\coloneqq \IZ[x_i\ |\ i\in I]$ be a polynomial ring with a surjection $P\twoheadrightarrow R$. Since $\IS_P\coloneqq \IS[x_i\ |\ i\in I]$ is the free $\IE_1$-ring on commuting generators $x_i$, we get an $\IE_1$-map $\IS_P\rightarrow \ku_R$. It is a folklore result that $\IS_P$ admits an even cell decomposition as an $\IE_2$-ring; see \cref{lem:E2CellStructure} for a proof. Since $k_R$ is even, the map $\IS_P\rightarrow k_R$ can be upgraded to an $\IE_2$-map.
		
		Now let $\IZ\rightarrow P^\bullet$ denote the \v Cech nerve of $\IZ\rightarrow P$ and define $\IS\rightarrow \IS_{P^\bullet}$ similarly. We also let $\IZ_p\rightarrow \widehat{P}_p^\bullet$ and $\IS_p\rightarrow \ISPhat$ denote the $p$-completed \v Cech nerves. The \v Cech nerve of the augmentation $\THH_\solid(\ISPhat)\rightarrow \ISPhat$ is the cosimplicial diagram $\THH_\solid(\ISPhat/\ISPhatbullet)$. If we base change this diagram along the $\IE_1$-map $\THH_\solid(\ISPhat)\rightarrow \THH_\solid(\ku_R/\ku_A)$, we get an augmented cosimplicial diagram of left $\THH_\solid(k_R/k_A)$-modules
		\begin{equation*}
			\THH_\solid(k_R/k_A)\longrightarrow \THH_\solid\bigl(k_R/k_A\soltimes \ISPhatbullet\bigr)\,.
		\end{equation*}
		In the case $k=\IZ$, this becomes the descent diagram $\HH_\solid(R/A)\rightarrow \HH_\solid(R/A\soltimes_{\IZ_p} \widehat{P}_p^\bullet)$.
	\end{numpar}
	\begin{rem}
		Instead of the resolution from \cref{par:EvenResolution}, we could also use the following: Let $\IS_{P_\infty}\coloneqq \IS[x_i^{1/p^\infty}\ |\ i\in I]$, let $\IS_P\rightarrow \IS_{P_\infty^\bullet}$ be the \v Cech nerve of $\IS_P\rightarrow \IS_{P_\infty}$ and define 
		\begin{equation*}
			k_{R_\infty^\bullet}\coloneqq \left(k_R\otimes_{\IS_P}\IS_{P_\infty^\bullet}\right)_p^\complete\,.
		\end{equation*}
		In this way we get resolutions of the same form in both cases~\cref{par:AssumptionsOnR}\cref{enum:E1Lift} and~\cref{enum:E2Lift}. Most arguments below would work for this resolution as well, but the one from \cref{par:EvenResolution} is more convenient for \cref{cor:SolidvsPstragowski} and for the global case in \cref{subsec:qdeRhamkuGlobal}.
	\end{rem}
	
	\begin{prop}\label{prop:EvenResolution}
		Assume we are in situation~\cref{par:AssumptionsOnR}\cref{enum:E2Lift}. Then the cosimplicial resolution from~\cref{par:EvenResolution} induces a canonical equivalence
		\begin{equation*}
			\fil_{\ev}^\star\THH_\solid(k_R/k_A)\overset{\simeq}{\longrightarrow} \limit_{\IDelta}\tau_{\geqslant 2\star}\THH_\solid\bigl(k_R/k_A\soltimes \ISPhatbullet\bigr)\,.
		\end{equation*}
	\end{prop}
	
	To prove \cref{prop:EvenResolution}, we'll send two technical lemmas in advance.
	
	\begin{lem}\label{lem:eff}
		The augmentation maps $\THH_\solid(\IS_P)\rightarrow \IS_P$ and $\THH_\solid(\ISPhat)\rightarrow \ISPhat $ are solid faithfully even flat in the sense of \cref{def:SolidFaithfullyEvenFlat}. Moreover, $\IS_P$ is nuclear as a $\THH_\solid(\IS_P)$-module and $\ISPhat $ is nuclear as a $\THH_\solid(\ISPhat )$-module.
	\end{lem}
	\begin{proof}
		The nuclearity assumptions follow from \cref{lem:NuclearityForDiscretepComplete}. We only show solid faithful even flatness for $\THH_\solid(\ISPhat )\rightarrow \ISPhat $; the argument for $\THH_\solid(\IS_P)\rightarrow\IS_P$ is similar (but easier). Let $E$ be a $\pi_*$-even module over $\THH_\solid(\ISPhat )$. We have a convergent spectral sequence
		\begin{equation*}
			\mathrm E^2=\H_*\left(\pi_*(E)\lsoltimes_{\pi_*\THH_\solid(\ISPhat )}\pi_*(\ISPhat )\right)\Longrightarrow \pi_*\left(E\soltimes_{\THH_\solid(\ISPhat )}\ISPhat \right)\,.
		\end{equation*}
		To show that the right-hand side is even, so that $\ISPhat $ will be solid even flat as a $\THH_\solid(\ISPhat )$-module, it will be enough to show that the $\mathrm E^2$-page is concentrated in even bidegrees. The calculation in the proof of \cite[Proposition~\chref{4.2.4}]{EvenFiltration} shows that 
		\begin{equation*}
			\pi_*\THH_\solid(\ISPhat )\cong \pi_*(\ISPhat )\soltimes_{\IZ_p}\Lambda_{\IZ_p}^*(\d x_i\ |\ i\in I)_p^\complete
		\end{equation*}
		is a graded $p$-completed exterior algebra over $\pi_*(\ISPhat )$ on generators $\d x_i$ in bidegree~$(1,0)$. Since $\pi_*(E)$ is concentrated in even degrees, each $\d x_i$ must act by $0$, and so
		\begin{equation*}
			\pi_*(E)\lsoltimes_{\pi_*\THH_\solid(\ISPhat )}\pi_*(\ISPhat )\simeq \pi_*(E)\lsoltimes_{\IZ_p}\Gamma_{\IZ_p}^*(\sigma^2 x_i\ |\ i\in I)_p^\complete\,,
		\end{equation*}
		where $\Gamma_{\IZ_p}^*(\sigma^2 x_i\ |\ i\in I)_p^\complete$ denotes a $p$-completed divided power algebra on generators in bidegree~$(2,0)$. Thus, to show that the $\mathrm E^2$-page is concentrated in even bidegrees, we only need to check that any $p$-completed direct sum $(\bigoplus_J \IZ_p)_p^\complete$ is solid even flat over $\IZ_p$. For finite direct sums this is obvious, for countable direct sums we can use the argument from the proof of \cref{lem:NuclearityForDiscretepComplete}, and for uncountable direct sums we can reduce to the countable case since $p$-completion commutes with $\omega_1$-filtered colimits. This finishes the proof of evenness of the $\mathrm E^2$-page, so that $\ISPhat $ is indeed solid even flat over $\THH_\solid(\ISPhat )$.
		
		Since the unit component $\IZ_p\rightarrow \Gamma_{\IZ_p}^*(\sigma^2 x_i\ |\ i\in I)_p^\complete$ is a direct summand, we see that the condensed homotopy groups
		\begin{equation*}
			\pi_*\left(E\soltimes_{\THH_\solid(\ISPhat)}\cofib\bigl(\THH_\solid(\ISPhat \bigr)\rightarrow\ISPhat )\right)
		\end{equation*}
		are also computed by a spectral sequence with $\mathrm E^2$-page concentrated in even bidegrees. This shows that $\cofib(\THH_\solid(\ISPhat )\rightarrow\ISPhat )$ is also solid even flat over $\THH_\solid(\ISPhat )$ and we're done.
	\end{proof}
	
	\begin{lem}\label{lem:THHSpectralSequence}
		There exists a natural convergent spectral sequence
		\begin{equation*}
			\mathrm{E}_{r,s}^2=\H_r\bigl(\HH_\solid(R/A)\lsoltimes_\IZ\pi_{2s}(k)\bigr)\Longrightarrow \pi_{r+s}\THH_\solid(k_R/k_A)\,.
		\end{equation*} 
	\end{lem}
	\begin{proof}
		The argument is the same as in \cite[Proposition~\chref{4.2.4}]{EvenFiltration} except for different grading conventions. Consider the filtered spectrum $\THH_\solid(\tau_{\geqslant \star}(k_R)/\tau_{\geqslant \star}(k_A))$. This is an exhaustive and complete (due to increasing connectivity) filtration on $\THH_\solid(k_R/k_A)$ and so it determines a convergent spectral sequence. 
		
		It remains to check that the $\mathrm{E}^2$-page has the desired form. The associated graded of the filtered spectrum above is $\THH_\solid(\Sigma^*\pi_*(k_R)/\Sigma^*\pi_*(k_A))$. Since $\pi_*(k_A)$ and $\pi_*(k_R)$ are concentrated in even graded degrees and $\IZ$-linear, the shearing functor $\Sigma^*$ is symmetric monoidal and commutes with $\THH$. The associated graded can thus be rewritten as $\Sigma^*\HH_\solid(\pi_*(k_R)/\pi_*(k_A))\simeq \Sigma^*\HH_\solid(R/A)\lsoltimes_\IZ\pi_*(k)$. This yields the desired $\mathrm{E}^2$-page.
	\end{proof}
	
	\begin{proof}[Proof of \cref{prop:EvenResolution}]
		Using the spectral sequence from \cref{lem:THHSpectralSequence} (applied to $\IS_A\soltimes \ISPhatbullet$ instead of $\IS_A$) and our asssumption that $A\soltimes_\IZ \widehat{P}_p\twoheadrightarrow R$ is $p$-quasi-lci and surjective, we see that $\THH_\solid(k_R/k_A\soltimes \ISPhatbullet)$ is even. It follows by the solid analogue of \cite[Lemma~\chref{2.36}]{PerfectEvenFiltration} that the solid even filtration (taken in left modules over $\THH_\solid(k_R/k_A)$) is the double speed Whitehead filtration
		\begin{equation*}
			\fil_{\ev}^\star\THH_\solid\bigl(k_R/k_A\soltimes \ISPhatbullet\bigr)\simeq \tau_{\geqslant 2\star}\THH_\solid\bigl(k_R/k_A\soltimes \ISPhatbullet\bigr)\,.
		\end{equation*}
		Using the flat descent result from \cref{thm:SolidEvenFlatDescentVariant}, which applies thanks to \cref{lem:eff,lem:NuclearityForDiscretepComplete}\cref{enum:Rpcomplete}, we find that
		\begin{equation*}		
			\fil_{\ev}^\star\THH_\solid(k_R/k_A)\longrightarrow \limit_{\IDelta}\tau_{\geqslant 2\star}\THH_\solid\bigl(k_R/k_A\soltimes \ISPhatbullet\bigr)
		\end{equation*}
		becomes an equivalence upon completion of the filtrations. Since the left-hand side is exhaustive whereas the right-hand side is complete, to finish the proof of the $\THH$ case, it will be enough to check that the right-hand side is also exhaustive.
		
		In other words, we must show $\THH_\solid(k_R/k_A)\simeq \limit_{\IDelta}\THH_\solid(k_R/k_A\soltimes \ISPhatbullet)$. By the same argument as in \cite[Corollary~\chref{3.4}(2)]{BMS2}, it's enough to show instead
		\begin{equation*}
			\THH_\solid(k_R/k_A)\soltimes_k\tau_{\leqslant 2s}k\overset{\simeq}{\longrightarrow}\limit_{\IDelta}\Bigl(\THH_\solid\bigl(k_R/k_A\soltimes \ISPhatbullet\bigr)\soltimes_k\tau_{\leqslant 2s}k\Bigr)
		\end{equation*}
		for all $s\geqslant 0$. This can be checked on associated gradeds in~$s$. So we must show that $\HH_\solid(R/A)\soltimes_{\IZ}\pi_{2s}(k)\simeq \limit_{\IDelta}(\HH_\solid(R/A\soltimes_\IZ\widehat{P}_p^\bullet)\soltimes_\IZ\pi_{2s}(k))$ for all $s\geqslant 0$. By our assumptions on $R$ and $A$, the HKR filtrations $\fil_\mathrm{HKR}^\star\HH_\solid(R/A)$ and $\fil_\mathrm{HKR}^\star\HH_\solid(R/A\soltimes_\IZ\widehat{P}_p^\bullet)$ increase in connectivity as $\star\rightarrow\infty$. They are therefore still complete after $-\soltimes_\IZ\pi_{2s}(k)$. So we may also pass to the associated graded of the HKR filtration. It remains to show that
		\begin{equation*}
			\bigwedge^n\L_{R/A}\lsoltimes_\IZ\pi_{2s}(k)\longrightarrow \limit_{\IDelta} \left(\bigwedge^n\L_{R/A\otimes_{\IZ} P^\bullet}\lsoltimes_\IZ\pi_{2s}(k)\right)
		\end{equation*}
		is an equivalence for all $n,s\geqslant 0$ (here the cotangent complexes are implicitly $p$-completed). By descent for the cotangent complex, this would be true without $-\lsoltimes_\IZ\pi_{2s}(k)$ on either side, so we must check that $-\lsoltimes_\IZ\pi_{2s}(k)$ commutes with the cosimplicial limit. Since $R$ is $p$-quasi-lci over $A$ and $P\twoheadrightarrow R$ is surjective, each $\bigwedge^n\L_{R/A\otimes_\IZ P^\bullet}$ is concentrated in homological degree~$n$. Writing $\bigwedge^n \L_{R/A\otimes_\IZ P^i}\simeq \Sigma^n L_i$, it follows that the cosimplicial limit $\limit_{\IDelta}\bigwedge^n \L_{R/A\otimes_\IZ P^\bullet}$ is given by the unnormalised Moore complex $L_*\simeq (\dotsb\leftarrow L_1\leftarrow L_0)$, sitting in homological degrees $(-\infty,n]$. Now since $\pi_{2s}(k)$ is $p$-torsion free and discrete by our assumptions on $k$, we see that $L_i\lsoltimes_\IZ\pi_{2s}(k)\simeq L_i\soltimes_\IZ\pi_{2s}(k)$ is static. It follows that 
		\begin{equation*}
			L_*\lsoltimes_\IZ\pi_{2s}(k)\simeq \Bigl(\dotsb \leftarrow\bigl(L_1\soltimes_\IZ\pi_{2s}(k)\bigr)\leftarrow \bigl(L_0\soltimes_\IZ\pi_{2s}(k)\bigr)\Bigr)\,.
		\end{equation*}
		So in this case it is indeed true that $-\lsoltimes_\IZ\pi_{2s}(k)$ commutes with the cosimplicial limit. This finishes the proof.
	\end{proof}
	
	\begin{cor}\label{cor:FilEvExhaustiveComplete}
		In both situations~\cref{par:AssumptionsOnR}\cref{enum:E1Lift} and~\cref{par:AssumptionsOnR}\cref{enum:E2Lift}, $\fil_{\ev}^\star\THH_\solid(k_R/k_A)$ is an exhaustive complete filtration on $\THH_\solid(k_R/k_A)$.
	\end{cor}
	\begin{proof}
		In case~\cref{par:AssumptionsOnR}\cref{enum:E1Lift} completeness is clear and exhaustiveness follows from the same argument as in the proof of \cref{prop:EvenResolution} above. In case~\cref{par:AssumptionsOnR}\cref{enum:E2Lift} exhaustiveness is automatic and completeness follows from \cref{prop:EvenResolution}.
	\end{proof}
	
	\begin{cor}\label{cor:Bifiltration}
		Put $(\tau_{\leqslant 2s}k)_A\coloneqq (\IS_A\otimes \tau_{\leqslant 2s}k)_p^\complete$ and $(\tau_{\leqslant 2s}k)_R\coloneqq (\IS_R\otimes \tau_{\leqslant 2s}k)_p^\complete$ for all $s\geqslant0$. In both situations~\cref{par:AssumptionsOnR}\cref{enum:E1Lift} and~\cref{par:AssumptionsOnR}\cref{enum:E2Lift}, consider the bifiltered object given by
		\begin{equation*}
			\fil^{s}\fil_{\ev}^\star \THH_\solid(k_R/k_A)\coloneqq \fil_{\ev}^\star\THH_\solid\bigl((\tau_{\leqslant 2*}k)_R/(\tau_{\leqslant 2s}k)_A\bigr)\,.
		\end{equation*}
		\begin{alphanumerate}
			\item We have $\fil_{\ev}^\star \THH_\solid(k_R/k_A)\simeq \limit_{s\geqslant 0}\fil^s\fil_{\ev}^\star \THH_\solid(k_R/k_A)$.\label{enum:BifiltrationComplete}
			\item If $\fil_{\mathrm{HKR}}^\star$ denotes the usual HKR filtration, then for all $s\geqslant 0$,\label{enum:BifiltrationGraded}
			\begin{equation*}
				\gr^s\fil_{\ev}^\star\THH_\solid(k_R/k_A)\simeq \bigl(\fil_{\mathrm{HKR}}^{\star-s}\HH_\solid(R/A)\bigr)\lsoltimes_\IZ\Sigma^{2s+1}\pi_{2s}(k)\,.
			\end{equation*}
		\end{alphanumerate}
	\end{cor}
	\begin{proof}
		We explain the argument in the context of~\cref{par:AssumptionsOnR}\cref{enum:E1Lift}. The other case is analogous, using the cosimplicial resolution from \cref{prop:EvenResolution} instead. Put $(\tau_{\leqslant 2s}k)_{R_\infty^\bullet}\coloneqq (\IS_{R_\infty^\bullet}\otimes\tau_{\leqslant 2s}k)_p^\complete$ and consider the cosimplicial bifiltered object
		\begin{equation*}
			\fil^s\tau_{\geqslant 2\star}\THH_\solid(k_{R_\infty^\bullet}/k_A)\coloneqq \tau_{\geqslant 2\star}\THH_\solid\bigl((\tau_{\leqslant 2s}k)_{R_\infty^\bullet}/(\tau_{\leqslant 2s}k)_{A}\bigr)\,.
		\end{equation*}
		Then clearly $\tau_{\geqslant 2\star}\THH_\solid(k_{R_\infty^\bullet}/k_A)\simeq \limit_{s\geqslant 0}\fil^s\tau_{\geqslant 2\star}\THH_\solid(k_{R_\infty^\bullet}/k_A)$. Applying $\limit_{\IDelta}$ on both sides already shows~\cref{enum:BifiltrationComplete}. To prove~\cref{enum:BifiltrationGraded}, observe that the functor $\tau_{\geqslant 2\star}(-)$ is non-exact in general, but nevertheless it preserves the cofibre sequence
		\begin{equation*}
			\HH_\solid(R_\infty^\bullet/A)\soltimes_\IZ\Sigma^{2s}\pi_{2s}(k)\longrightarrow \THH_\solid(k_{R_\infty^\bullet}/k_A)\soltimes_k\tau_{\leqslant 2s}k\longrightarrow \THH_\solid(k_{R_\infty^\bullet}/k_A)\soltimes_k\tau_{\leqslant 2(s-1)}k\,.
		\end{equation*}
		Indeed, consider the spectral sequence%
		\footnote{In the construction of the spectral sequence in \cref{lem:THHSpectralSequence} we used the Postnikov filtration $\tau_{\geqslant \star}k$, while here we're working with the double speed Whitehead filtration $\tau_{\leqslant 2\star}k$. We could have used the Postnikov filtration as well to construct a similar spectral sequence as in \cref{lem:THHSpectralSequence}. But we still use the one from \cref{lem:THHSpectralSequence}.}
		from \cref{lem:THHSpectralSequence} with $k$ replaced by $\tau_{\leqslant 2s}(k)$ or $\tau_{\leqslant 2(s-1)}(k)$. Our assumptions on $R_\infty^\bullet$ guarantee that both $\mathrm E^2$-pages are concentrated in even bidegrees and so the spectral sequences collapse. A closer examination of the induced map on $\mathrm E^2$-pages then shows that $\tau_{\geqslant 2\star}(-)$ indeed preserves the cofibre sequence above.
		
		Using this observation, we conclude that the graded pieces of $\fil^s\tau_{\geqslant 2\star}\THH_\solid(k_{R_\infty^\bullet}/k_A)$ are given by%
		\footnote{Note that $\gr^s$ is defined as a cofibre, not a fibre. Hence the extra $\Sigma$.}
		\begin{equation*}
			\gr^s\tau_{\geqslant 2\star}\THH_\solid(k_{R_\infty^\bullet}/k_A)\simeq \Sigma\tau_{\geqslant 2\star}\Bigl(\HH_\solid(R_\infty^\bullet/A)\lsoltimes_\IZ\Sigma^{2s}\pi_{2s}(k)\Bigr)\,.
		\end{equation*}
		The right-hand side agrees with $\tau_{\geqslant 2(\star-s)}\HH_\solid(R_\infty^\bullet/A)\lsoltimes_\IZ\Sigma^{2s+1}\pi_{2s}(k)$ since $\pi_{2s}(k)$ was assumed to be discrete and $p$-torsion free. Now the HKR filtration can be computed as the cosimplicial limit $\fil_{\mathrm{HKR}}^\star\HH_\solid(R/A)\simeq \limit_{\IDelta}\tau_{\geqslant 2\star}\HH_\solid(R_\infty^\bullet/A)$. Thus, to prove \cref{enum:BifiltrationGraded}, it remains to check that $-\lsoltimes_\IZ\pi_{2s}(k)$ commutes with the cosimplicial limit. Since the HKR filtration stays complete after $-\lsoltimes_\IZ\pi_{2s}(k)$ (due to increasing connectivity), we may pass to the associated graded. This reduces us to an assertion that was checked in the proof of \cref{prop:EvenResolution} above.
	\end{proof}

	\begin{cor}\label{cor:TC-TPEvenResolution}
		In situation~\cref{par:AssumptionsOnR}\cref{enum:E1Lift}, the given cosimplicial resolution induces equivalences
		\begin{align*}
			\fil_{\ev,\h S^1}^\star\TC_\solid^-(k_R/k_A)&\overset{\simeq}{\longrightarrow} \limit_{\IDelta}\tau_{\geqslant 2\star}\TC_\solid^-(k_{R_\infty^\bullet}/k_A)\,,\\
			\fil_{\ev,\t S^1}^\star\TP_\solid(k_R/k_A)&\overset{\simeq}{\longrightarrow}\limit_{\IDelta}\tau_{\geqslant 2\star}\TP_\solid(k_{R_\infty^\bullet}/k_A)\,.
		\end{align*}
		If we are in situation~\cref{par:AssumptionsOnR}\cref{enum:E2Lift}, the cosimplicial resolution from~\cref{par:EvenResolution} induces equivalences
		\begin{align*}
			\fil_{\ev,\h S^1}^\star\TC_\solid^-(k_R/k_A)&\overset{\simeq}{\longrightarrow} \limit_{\IDelta}\tau_{\geqslant 2\star}\TC_\solid^-\bigl(k_R/k_A\soltimes \ISPhatbullet\bigr)\,,\\
			\fil_{\ev,\t S^1}^\star\TP_\solid(k_R/k_A)&\overset{\simeq}{\longrightarrow} \limit_{\IDelta}\tau_{\geqslant 2\star}\TP_\solid\bigl(k_R/k_A\soltimes \ISPhatbullet\bigr)\,.
		\end{align*}
	\end{cor}
	\begin{proof}
		To see the assertion for $\TC^-$ in both cases, just observe that $(-)^{\h \IT_{\ev}}$ commutes with the cosimplicial limit and that $(\tau_{\geqslant 2\star}\THH_\solid(-))^{\h \IT_{\ev}}\simeq \tau_{\geqslant 2\star}\TC_\solid^-(-)$ holds in this case by \cite[Lemma~\chref{2.75}(vi)]{CyclotomicSynthetic}. To show the same for $\TP$, we need to commute $(-)_{\h \IT_{\ev}}\simeq \IS_{\ev}\soltimes_{\IT_{\ev}}-$ past the cosimplicial limit.
		
		Let us explain how to do this in case~\cref{par:AssumptionsOnR}\cref{enum:E1Lift}; the other case is analogous. We use the bifiltration from \cref{cor:Bifiltration}. By \cref{cor:Bifiltration}\cref{enum:BifiltrationGraded}, $\cofib(\fil^s\fil_{\ev}^\star\rightarrow\fil_{\ev}^\star)$ is $\star+s$-connective. Using \cref{cor:Bifiltration}\cref{enum:BifiltrationComplete} follows that $(\fil_{\ev}^\star)_{\h \IT_{\ev}}\simeq (\limit_{s\geqslant 0}\fil^s\fil_{\ev}^\star)_{\h \IT_{\ev}}\simeq \limit_{s\geqslant 0}(\fil^s\fil_{\ev}^\star)_{\h \IT_{\ev}}$. So we may pass to the associated graded in $s$-direction and thus, using \cref{cor:Bifiltration}\cref{enum:BifiltrationGraded} again, it will be enough to check
		\begin{equation*}
			\Bigl(\fil_\mathrm{HKR}^\star\HH_\solid(R/A)\lsoltimes_\IZ\pi_{2s}(k)\Bigr)_{\h \IT_{\ev}}\overset{\simeq}{\longrightarrow}\limit_{\IDelta}\Bigl(\bigl(\tau_{\geqslant 2\star}\HH_\solid(R_\infty^\bullet/A)\lsoltimes_\IZ\pi_{2s}(k)\bigr)_{\h \IT_{\ev}}\Bigr)\,.
		\end{equation*}
		Now both sides are $\IZ$-linear. By \cite[Proposition~\chref{2.54}]{CyclotomicSynthetic}, the construction $(-)_{\h\IT_{\ev}}$ agrees with the orbits with respect to Raksit's filtered circle \cite[Notation~\chref{6.3.2}]{RaksitFilteredCircle}. Combining this observation with \cite[Corollary~\chref{3.4}(1)]{BMS2} (plus an easy argument  as in the proof of \cref{prop:EvenResolution} to deal with the extra $-\lsoltimes_\IZ\pi_{2s}(k)$), we conclude that both sides are exhaustive filtrations on $(\HH_\solid(R/A)\lsoltimes_\IZ\pi_{2s}(k))_{\h S^1}$.
		
		The equivalence can now be checked on associated gradeds. By \cite[Proposition~\chref{6.3.3}]{RaksitFilteredCircle}, the $n$\textsuperscript{th} graded piece of $(\fil_{\mathrm{HKR}}^\star\HH_\solid(R/A)\lsoltimes_\IZ\pi_{2s}(k))_{\h \IT_{\ev}}$ will be an iterated extension of $\gr_\mathrm{HKR}^i\HH_\solid(R/A)\lsoltimes_\IZ\pi_{2s}(k)$ for $i=0,1,\dotsc,n$. A similar argument applies on the right-hand side. So we can finally deduce the desired equivalence from \cref{prop:EvenResolution}.
	\end{proof}
	
	\subsection{Base change}\label{subsec:BaseChange}
	
	We continue to fix a $k$ as specified at the beginning of \cref{subsec:EvenResolution}. As a consequence of \cref{prop:EvenResolution}, we show that the even filtrations constructed in \cref{par:EvenFiltration} satisfy all expected base change properties.
	
	\begin{cor}\label{cor:EvenFiltrationBaseChange}
		Let $k\rightarrow l$ be any map of $\IE_\infty$-ring spectra where $l$ is also connective, even, and $p$-torsion free in every homotopical degree. Let $l_A\coloneqq l\soltimes \IS_A$ and $l_R\coloneqq l\soltimes \IS_R$. Let furthermore $k_{\ev}\coloneqq \tau_{\geqslant 2\star}k$ and $l_{\ev}\coloneqq\tau_{\geqslant 2\star}l$. Then the canonical base change morphism is an equivalence
		\begin{equation*}
			\fil_{\ev}^\star \THH_\solid(k_R/k_A)\soltimes_{k_{\ev}}l_{\ev}\overset{\simeq}{\longrightarrow} \fil_{\ev}^\star \THH_\solid(l_R/l_A)\,.
		\end{equation*}
	\end{cor}

	\begin{proof}
		Using \cref{cor:FilEvExhaustiveComplete}, we see that both sides are exhaustive filtrations on $\THH_\solid(l_R/l_A)$. It is thus enough to check the equivalence on associated gradeds. Let us now assume we're in case~\cref{par:AssumptionsOnR}\cref{enum:E2Lift}; the ~\cref{par:AssumptionsOnR}\cref{enum:E2Lift} is analogous using the resolution from \cref{prop:EvenResolution}. Using the spectral sequence from \cref{lem:THHSpectralSequence}, we see that the cosimplicial graded object $\pi_{2*}\THH_\solid(k_{R_\infty^\bullet}/k_A)$ has a finite filtration%
		\footnote{This is \emph{not} the filtration from \cref{cor:Bifiltration}.}
		in every graded degree-wise finite filtration whose associated graded satisfies
		\begin{equation*}
			\gr^*\pi_{2(\star+*)}\THH_\solid(k_{R_\infty^\bullet}/k_A)\simeq \pi_{2\star}\HH_\solid(R_\infty^\bullet/A)\lsoltimes_\IZ\pi_{2*}(k)
		\end{equation*}
		as cosimplicial bigraded objects. Applying $\limit_{\IDelta}$ (which commutes with $-\lsoltimes_\IZ\pi_{2*}(k)$ by the argument in the proof of \cref{prop:EvenResolution}), we find that $\gr_{\ev}^*\THH_\solid(k_R/k_A)$ has a finite filtration in every graded degree in such a way that the associated graded satisfies
		\begin{equation*}
			\gr^*\gr_{\ev}^{\star+*}\THH_\solid(k_R/k_A)\simeq \gr_\mathrm{HKR}^\star\HH_\solid(R/A)\soltimes_\IZ\Sigma^{2*}\pi_{2*}(k)
		\end{equation*}
		as bigraded objects. This equivalence is compatible with $\gr^*k_{\ev}\simeq \Sigma^{2*}\pi_{2*}(k)$, since the latter can be obtained from the spectral sequence for $\THH_\solid(k/k)$. Using the same for $l$, the desired equivalence now follows from the trivial observation
		\begin{equation*}
			\left(\gr_\mathrm{HKR}^\star\HH_\solid(R/A)\soltimes_\IZ\Sigma^{2*}\pi_{2*}(k)\right)\soltimes_{\Sigma^{2*}\pi_{2*}(k)}\Sigma^{2*}\pi_{2*}(l)\simeq \gr_\mathrm{HKR}^\star\HH_\solid(R/A)\soltimes_\IZ\Sigma^{2*}\pi_{2*}(l)\,,
		\end{equation*}
		so we're done.
	\end{proof}
	\begin{cor}\label{cor:EvenFiltrationTC-BaseChange}
		Let $k\rightarrow l$ be as in \cref{cor:EvenFiltrationBaseChange} and put $k^{\h S^1}_{\ev}\coloneqq \tau_{\geqslant 2\star}(k^{\h S^1})$ as well as $l^{\h S^1}_{\ev}\coloneqq \tau_{\geqslant 2\star}(l^{\h S^1})$. Let also $t\in\pi_{-2}(k^{\h S^1})$ be a complex orientation of~$k$. We regard~$t$ as sitting in homotopical degree~$-2$ and filtration degree~$-1$ of $k_{\ev}^{\h S^1}$. Then the canonical base change morphism is an equivalence
		\begin{equation*}
			\left(\fil_{\ev,\h S^1}^\star \TC_\solid^-(k_R/k_A)\soltimes_{k_{\ev}^{\h S^1}}l_{\ev}^{\h S^1}\right)_t^\complete\overset{\simeq}{\longrightarrow} \fil_{\ev,\h S^1}^\star \TC_\solid^-(l_R/l_A)\,.
		\end{equation*}
	\end{cor}
	\begin{proof}
		Using \cref{cor:TC-TPEvenResolution}, we see that both sides are $t$-complete. Upon reduction modulo~$t$, we get the equivalence from \cref{cor:EvenFiltrationBaseChange}.
	\end{proof}
	
	A similar base change equivalence exists for $\fil_{\ev,\t S^1}\TP_\solid(k_R/k_A)$, but one has to be a little careful about completions. One way to formulate the result would be via \cref{cor:EvenFiltrationTC-BaseChange} combined with the following:
	\begin{cor}\label{cor:TC-toTPbaseChange}
		Let $k^{\t S^1}_{\ev}\coloneqq \tau_{\geqslant 2\star}(k^{\t S^1})$. We have a canonical equivalence
		\begin{equation*}
			\fil_{\ev,\h S^1}^\star\TC_\solid^-(k_R/k_A)\soltimes_{k_{\ev}^{\h S^1}}k_{\ev}^{\t S^1}\overset{\simeq}{\longrightarrow}\fil_{\ev,\t S^1}^\star \TP_\solid(k_R/k_A)\,.
		\end{equation*}
	\end{cor}
	\begin{proof}
		Using \cref{cor:TC-TPEvenResolution}, we see that both sides are exhaustive filtrations on $\TP_\solid(k_R/k_A)$. It is thus enough to check the equivalence on associated gradeds. Using \cref{cor:TC-TPEvenResolution}, we find that
		\begin{equation*}
			\gr_{\ev,\h S^1}^*\TC_\solid^-(k_R/k_A)\rightarrow \gr_{\ev,\t S^1}^*\TP_\solid(k_R/k_A)
		\end{equation*}
		is an equivalence in negative graded degrees and that the right-hand side is periodic. Since $-\soltimes_{\gr^*k_{\ev}^{\smash{\h S^1}}}\gr^*k_{\ev}^{\t S^1}$ will also make the left-hand side periodic, we're done.
	\end{proof}
	
	\subsection{Comparison of even filtrations}\label{subsec:ComparisonOfEvenFiltrations}
	
	As another consequence of \cref{prop:EvenResolution}, we can show that the even filtrations from \cref{par:EvenFiltration} agree with the those defined by \cite{BMS2,EvenFiltration,PerfectEvenFiltration}.
	
	\begin{numpar}[Even filtrations on ordinary Hochschild homology.]\label{par:EvenFiltrationII}
		In the case $k=\IZ$, the constructions in \cref{par:EvenFiltration} yield filtrations
		\begin{equation*}
			\fil_{\ev}^\star\HH_\solid(R/A)\,,\quad\fil_{\ev,\h S^1}^\star\HC_\solid^-(R/A)\,,\quad\text{and}\quad\fil_{\ev,\t S^1}^\star\HP_\solid(R/A)\,.
		\end{equation*}
		But $\HH_\solid(R/A)\simeq \HH(R/A)_p^\complete$ is a $p$-complete $\IE_\infty$-ring spectrum and so we can also consider the Hahn--Raksit--Wilson even filtrations
		\begin{equation*}
			\fil_{\HRWev}^\star\HH(R/A)_p^\complete\,,\quad\fil_{\HRWev,\h S^1}^\star\HC^-(R/A)_p^\complete\,,\quad\text{and}\quad\fil_{\HRWev,\t S^1}^\star\HP(R/A)_p^\complete\,.
		\end{equation*}
		These can be regarded as filtrations on $\HH_\solid(R/A)$, $\HC_\solid^-(R/A)$, and $\HP_\solid(R/A)$ in a natural way. For $\HH$, we simply regard $p$-complete spectra as solid condensed spectra per Convention~\cref{conv:SolidConventions} and use \cref{lem:SolidTHH}. For $\HC^-$ and $\HP$, we must be a little more careful: If $\HH(R/A)\rightarrow E$ is an $S^1$-equivariant $\IE_\infty$-map into an even $p$-complete ring spectrum with bounded $p^\infty$-torsion, we regard $E^{\h S^1}$ as a solid condensed spectrum by performing both the $p$-completion and the homotopy fixed points $(-)^{\h S^1}$ in $\Sp_\solid$. We then regard
		\begin{equation*}
			\fil_{\HRWev,\h S^1}^\star\HC^-(R/A)_p^\complete\simeq \limit_{\HH(R/A)\rightarrow E}\tau_{\geqslant 2\star}\bigl(E^{\h S^1}\bigr)\,;
		\end{equation*}
		as a solid condensed spectrum by also performing the limit in $\Sp_\solid$. In the same way we can regard $\fil_{\HRWev,\t S^1}^\star\HP(R/A)_p^\complete$ as a filtered solid condensed spectrum.
		
		If $E$ is even, then the perfect even filtration of $E$ is the double-speed Whitehead filtration $\tau_{\geqslant 2\star}(E)$ by \cite[Lemma~\chref{2.36}]{PerfectEvenFiltration} and its solid analogue. Moreover, $(\tau_{\geqslant 2\star}(E))^{\h \IT_{\ev}}\simeq \tau_{\geqslant 2\star}(E^{\h S^1})$ by \cite[Lemma~\chref{2.75}(vi)]{CyclotomicSynthetic} and similarly $(\tau_{\geqslant 2\star}(E))^{\t \IT_{\ev}}\simeq \tau_{\geqslant 2\star}(E^{\t S^1})$. It follows that there's a canonical map $\fil_{\ev}^\star\rightarrow \fil_{\HRWev}^\star$ in each case.
	\end{numpar}
	
	\begin{cor}\label{cor:HRWvsPstragowski}
		Via the comparison maps constructed in~\cref{par:EvenFiltrationII} above, the filtrations
		\begin{equation*}
			\fil_{\ev}^\star\HH_\solid(R/A)\,,\quad \fil_{\ev,\h S^1}^\star\HC_\solid^-(R/A)\,,\quad\text{and}\quad \fil_{\ev,\t S^1}^\star\HP_\solid(R/A)\,,
		\end{equation*}
		agree with the Hahn--Raksit--Wilson/HKR even filtrations
		\begin{equation*}
			\fil_{\HRWev}\HH(R/A)_p^\complete\,,\quad \fil_{\HRWev,\h S^1}\HC^-(R/A)_p^\complete\,,\quad\text{and}\quad \fil_{\HRWev,\t S^1}\HP(R/A)_p^\complete\,.
		\end{equation*}
	\end{cor}
	\begin{proof}
		The solid even filtration $\fil_{\ev}^\star \HH_\solid(R/A)$ can be computed by a certain cosimplicial resolution (in case~\cref{par:AssumptionsOnR}\cref{enum:E1Lift} by definition, in case~\cref{par:AssumptionsOnR}\cref{enum:E2Lift} by \cref{prop:EvenResolution}). The same resolutions also compute the even filtration of Hahn--Raksit--Wilson. The same argument also works for $\HC_\solid^-$ and $\HP_\solid$ thanks to \cref{cor:TC-TPEvenResolution}.
	\end{proof}
	\begin{rem}\label{rem:EvenFiltrationkuRationalisation}
		For later use, let us point out the following consequence: Using \cref{cor:EvenFiltrationTC-BaseChange} for $\ku\rightarrow \ku\otimes\IQ\simeq \IQ[\beta]$ and $\IZ\rightarrow \IQ[\beta]$, we deduce that
		\begin{equation*}
			\left(\fil_{\ev,\h S^1}^\star\TC_\solid^-(\ku_R/\ku_A)\soltimes\IQ\right)_t^\complete\simeq \left(\fil_{\ev,\h S^1}^\star\HC_\solid^-(R/A)\soltimes_{\IZ_{\ev}^{\h S^1}}\IQ[\beta]_{\ev}^{\h S^1}\right)_t^\complete\,.
		\end{equation*}
		Moreover, the filtration on the right-hand side is the usual Hahn--Raksit--Wilson/HKR even filtration. This will give us good control over the constructions in \cref{sec:qdeRhamku} after rationalisation.
	\end{rem}

	The filtration on $\TC^-(S/\IS_A\qpower)[1/u]_{(p,q-1)}^\complete$ from \cref{prop:qdeRhamTC-}, whose associated graded computes prismatic/$q$-de Rham cohomology, is also recovered by the solid even filtration.
	
	\begin{cor}\label{cor:BMSvsPstragowski}
		If $S$ is any $p$-complete $p$-quasi-lci $A[\zeta_p]$-algebra of bounded $p^\infty$-torsion, then there's a canonical filtered $\IE_\infty$-equivalence
		\begin{equation*}
			\left(\fil_{\ev}^\star\THH_\solid\bigl(S/\IS_A\qpower\bigr)\bigl[\localise{u}\bigr]_p^\complete\right)^{\h \IT_{\ev}}\overset{\simeq}{\longrightarrow}\fil_{\HRWev,\h S^1}^\star\left(\TC^-\bigl(S/\IS_A\qpower\bigr)\bigl[\localise{u}\bigr]_{(p,q-1)}^\complete\right)
		\end{equation*}
		\embrace{where the right-hand side is regarded as a filtered solid condensed spectrum in the way described in~\cref{par:EvenFiltrationII} above}.
	\end{cor}
	
	\begin{proof}
		Let us first construct the canonical map in question. For every $S^1$-equivariant $\IE_\infty$-map $\THH(S/\IS_A\qpower)[1/u]\rightarrow E$ into a $p$-complete even ring spectrum, we get a canonical filtered $\IE_\infty$-map
		\begin{equation*}
			\left(\fil_{\ev}^\star\THH_\solid\bigl(S/\IS_A\qpower\bigr)\bigl[\localise{u}\bigr]_p^\complete\right)^{\h \IT_{\ev}}\longrightarrow (\tau_{\geqslant 2\star}E)^{\h \IT_{\ev}}\simeq \tau_{\geqslant 2\star}\bigl(E^{\h S^1}\bigr)
		\end{equation*}
		using \cite[Lemma~\chref{2.75}(vi)]{CyclotomicSynthetic}. This induces the desired comparison map. To prove that we get an equivalence, we can use the same arguments as before: Choose a polynomial ring $P=\IZ[x_i\ |\ i\in I]$ with a surjection $P\twoheadrightarrow S$ and then show that both sides are computed by the cosimplicial resolution $\tau_{\geqslant 2\star}\TC_\solid^-(S/(\IS_A\soltimes\ISPhatbullet)\qpower)[1/u]_{(p,q-1)}^\complete$.
	\end{proof}
	
	Finally, we show that in the case $k=\ku$ our solid even filtration on $\THH_\solid(\ku_R/\ku_A)$ agrees with the $p$-completion of Pstr\k{a}gowski's perfect even filtration $\fil_{\Pev}^\star \THH(\ku_R/\ku_A)$. This won't be needed in the rest of the text, but it is perhaps a nice sanity check.
	
	\begin{cor}\label{cor:SolidvsPstragowski}
		The canonical map induced by~\cref{par:ComparisonWithUsualEven} is an equivalence
		\begin{equation*}
			\bigl(\fil_{\Pev}^\star\THH(\ku_R/\ku_A)\bigr)_p^\complete\overset{\simeq}{\longrightarrow}\fil_{\ev}^\star\THH_\solid(\ku_R/\ku_A)\,.
		\end{equation*}
	\end{cor}
	\begin{proof}
		Let $T\coloneqq \THH(\ku_R/\ku_A)$ for short. Since $\THH(\IS_P)\rightarrow \IS_P$ is eff, we can compute $\fil_{\Pev}^\star T$ using descent; more precisely, using the uncondensed version of \cref{thm:SolidEvenFlatDescentVariant}. We find that
		\begin{equation*}
			\fil_{\Pev/T}^\star T\longrightarrow \limit_{\IDelta} \fil_{\Pev/T}^\star\bigl(T\otimes_{\THH(\IS_P)}\THH(\IS_P/\IS_{P^\bullet})\bigr)
		\end{equation*}
		is an equivalence up to completing the filtrations on both sides. Let us now study the right-hand side. Fix some cosimplicial degree~$i$ and put $M\coloneqq \THH(\ku_R/\ku_A\otimes\IS_{P^i})$ for short. We claim that there is a canonical equivalence
		\begin{equation*}
			(\fil_{\Pev}^\star M)_p^\complete\overset{\simeq}{\longrightarrow} \fil_{\Pev}^\star\widehat{M}_p\simeq \tau_{\geqslant 2\star}(\widehat{M}_p)\,.
		\end{equation*}
		If we can show this, we're done. Indeed, by comparison with the resolution from \cref{prop:EvenResolution}, we find that $\bigl(\fil_{\Pev}^\star\THH(\ku_R/\ku_A)\bigr)_p^\complete\rightarrow\fil_{\ev}^\star\THH_\solid(\ku_R/\ku_A)$ is an equivalence up to completion. But the filtrations on both sides are exhaustive and the right-hand side is complete by \cref{prop:EvenResolution} again, and so the map must be an equivalence.
		
		To show the claim, first observe that the homotopy groups of $\widehat{M}_p/\beta\simeq \HH(R/A\otimes_\IZ P^i)_p^\complete$ are concentrated in even degrees and $p$-completely flat over $R$, where the $R$-module structure on $\pi_*(\widehat{M}_p/\beta)$ comes from the left $T$-module structure on $M$. We would like to show that the same conclusion is true for $\pi_*(\Hom_T(Q,\widehat{M}_p)/\beta)$ for any perfect even $T$-module $Q$; however, the seemingly obvious argument doesn't quite work, since $T$ is only $\IE_1$ and so there's no left $T$-module structure on $\Hom_T(-,-)$.
		
		To fix this, observe that $T\otimes_{\THH(\IS_P)}\IS_P$ has a right $\IS_P$-module structure commuting with the left $T$-module structure. Restricting to $\pi_0(\IS_P)\cong P$, we get a right homotopy action of $P$ on $T\otimes_{\THH(\IS_P)}\IS_P$. Since $\pi_0\THH(\IS_P)\cong P$ as well, this action agrees with the right action of $P$ on $T$ via $P\twoheadrightarrow R\cong \pi_0(T)$. In particular, the right homotopy action by $P$ factors through $R$. An analogous right homotopy action of $R$ can be constructed on $M\simeq T\otimes_{\THH(\IS_P)}\IS_P^{\otimes_{\smash{\THH(\IS_P)}}(i+1)}$, by picking our favourite tensor factor.
		
		This explains how $\pi_*\Hom_T(-,\widehat{M}_p)$ can be equipped with an $R$-module structure. With this $R$-module structure, it is still true that the homotopy groups $\pi_*(\widehat{M}_p/\beta)$ are concentrated in even degrees and are $p$-completely flat $R$-modules, because $\HH(R/A\otimes_\IZ P^i)$ is commutative. This allows us to deduce that the homotopy groups $\pi_*(\Hom_T(Q,\widehat{M}_p)/\beta)$ are also concetrated in even degrees and $p$-completely flat over $R$ for any perfect even left $T$-module~$Q$. Since $M$ is bounded below, we deduce that also $\Hom_T(Q,\widehat{M}_p)$ is even and its homotopy groups are $p$-completely flat $R$-modules. In particular, this is true for $\widehat{M}_p$ itself. By \cite[Lemma~\chref{4.7}]{BMS2}, the $p^\infty$-torsion in $\pi_{2*}\Hom_T(Q,\widehat{M}_p)$ is therefore bounded. In fact, there's a uniform bound $N$ that works for all $Q$, since we can use the same bound as for~$R$. 
		
		Let us use this to analyse the canonical map
		\begin{equation*}
			\gr_{\Pev}^* \widehat{M}_p\longrightarrow \limit_{\alpha\geqslant 0} \gr_{\Pev}^*(\widehat{M}_p/p^\alpha)\,.
		\end{equation*}
		By definition, $(\gr_{\Pev}^* \widehat{M}_p)/p^\alpha$ is given by the sections over $T$ of the sheafification of the spectra-valued presheaf $\Sigma^{2*}(\pi_{2*}\Hom_T(-,\widehat{M}_p))/p^\alpha$ on the perfect even site $\Perf_{\ev}(T)$. In homotopical degree~$2*$, this presheaf agrees with $\Sigma^{2*}\pi_{2*}\Hom_T(-,\widehat{M}_p/p^\alpha)$, but in homotopical degree~${2*}+1$ it has an extra torsion component. However, if we go from $\alpha+N$ to $\alpha$, then the transition map will vanish on the torsion component, because $N$ is a uniform bound for the $p^\infty$-torsion. Thus, in the limit we get an equivalence $\limit_{\alpha\geqslant 0}(\gr_{\Pev}^*\widehat{M}_p)/p^\alpha\simeq \limit_{\alpha\geqslant 0}\gr_{\Pev}^*(\widehat{M}_p/p^\alpha)$. The left-hand side agrees with $\pi_{2*}(\widehat{M}_p)$ since $\widehat{M}_p$ is already even and $p$-complete. We conclude that
		\begin{equation*}
			\tau_{\geqslant 2\star}(\widehat{M}_p)\simeq \fil_{\Pev}^\star\widehat{M}_p\longrightarrow \limit_{\alpha\geqslant 0}\fil_{\Pev}^\star(\widehat{M}_p/p^\alpha)
		\end{equation*}
		is an equivalence up to completion of the filtration on the right-hand side.
		
		Since $\THH(\IS_P)\rightarrow \IS_P$ is eff, $M$ will be even flat, hence homologically even over $T$. Thus \cite[Remark~\chref{2.35}]{PerfectEvenFiltration} shows $\fil_{\Pev}^\star M\simeq \fil_{\Pev}^{\star-1/2}M$. By definition, $(\fil_{\Pev}^{\star-1/2}M)/p^\alpha$ is given by the sections over $T$ of the sheafification of the spectra-valued presheaf
		\begin{equation*}
			\cofib\bigl(p^\alpha\colon \tau_{\geqslant 2\star-1}\Hom_T(-,M)\longrightarrow \tau_{\geqslant 2\star-1}\Hom_T(-,M)\bigr)
		\end{equation*} 
		on $\Perf_{\ev}(T)$. In homotopical degrees $\geqslant 2\star$, this presheaf agrees with $\tau_{\geqslant 2\star}\Hom_T(-,M/p^\alpha)$, but in homotopical degree~${2\star}-1$ there might be an additional component that injects into $\Sigma^{2\star-1}\pi_{2\star-1}\Hom_T(-,M/p^\alpha)$. However, the transition maps from $\alpha+N$ to $\alpha$ will vanish on this additional component by our uniform $p^\infty$-torsion bound, so in the limit we get an equivalence
		\begin{equation*}
			\left(\fil_{\Pev}^\star M\right)_p^\complete\simeq \limit_{\alpha\geqslant 0}\left(\fil_{\Pev}^\star M\right)/p^\alpha\overset{\simeq}{\longrightarrow}\limit_{\alpha\geqslant 0}\fil_{\Pev}^\star(M/p^\alpha)\,.
		\end{equation*} 
		At this point we've shown that $\left(\fil_{\Pev}^\star M\right)_p^\complete\rightarrow \tau_{\geqslant 2\star}(\widehat{M}_p)$ is an equivalence up to completion. But both sides are already complete: The right-hand side by inspection, the left-hand side by \cite[Theorem~\chref{8.3}(2)]{PerfectEvenFiltration}. So we're done.
	\end{proof}
	\begin{rem}
		The argument can be adapted to any even ring spectrum~$k$ such that $\pi_*(k)$ is a graded polynomial ring over $\IZ$ with finitely many generators in each given degree. In particular, it works for $k=\MU$. We don't know to what extent \cref{cor:SolidvsPstragowski} is true in complete generality. At the very least, one would need some finiteness assumption on $k$; otherwise $k_A$ and $k_R$ won't be $p$-complete in general. 
	\end{rem}

	\newpage
	
	\section{\texorpdfstring{$q$}{q}-de Rham cohomology and \texorpdfstring{$\TC^-$}{TC-} over \texorpdfstring{$\ku$}{ku}}\label{sec:qdeRhamku}
	
	In this section we'll finally formulate and prove the precise relationship between the even filtration on $\TC^-(\ku_R/\ku_A)$ and the $q$-de Rham complex $\qdeRham_{R/A}$.
	
	Before we begin, we remind the reader of our convention from \cref{par:Notation} to regard all ($q$-)de Rham complexes or cotangent complexes relative to a $p$-complete ring (such as $\qdeRham_{R/A}$) as implicitly $p$-completed.
	
	\subsection{The \texorpdfstring{$p$}{p}-complete comparison (case \texorpdfstring{$p>2$}{p>2})}\label{subsec:qdeRhamkupComplete}
	
	We fix a prime~$p>2$. We'll also continue to fix rings $A$ and $R$ satisfying the assumptions from~\cref{par:AssumptionsOnA} and~\cref{par:AssumptionsOnR}.
	
	Our main tool will be a striking result of Devalapurkar. To formulate this result, let us regard $\IZ_p[\zeta_p]$ as a $\IS_p\qpower$-algebra via $q\mapsto \zeta_p$. We let $S^1$ act on $\THH(\IZ_p[\zeta_p]/\IS_p\qpower)_p^\complete$ in the usual way and let $\IZ_p^\times$ act via~\cref{par:AdamsAction}. We let $S^1$ act on $\ku^{\t C_p}$ via the residual $S^1\simeq S^1/C_p$-action and let $\IZ_p^\times$ act via the Adams operations on $\ku_p^\complete$.
	
	\begin{thm}[Devalapurkar {\cite[Theorem~\chref{6.4.1}]{DevalapurkarSpherochromatism}}]\label{thm:kutCp}
		For primes $p > 2$, there exists an $S^1\times\IZ_p^\times$-equivariant equivalence of $\IE_\infty$-ring spectra
		\begin{equation*}
			\THH\bigl(\IZ_p[\zeta_p]/\IS_p\qpower\bigr)_p^\complete\overset{\simeq}{\longrightarrow} \tau_{\geqslant 0}\left(\ku^{\t C_p}\right)\,.
		\end{equation*}
		Moreover, this equivalence fits into a commutative diagram of $S^1$-equivariant $\IE_\infty$-algebras
		\begin{equation*}
			\begin{tikzcd}
				\THH\bigl(\IZ_p[\zeta_p]/\IS_p\qpower\bigr)_p^\complete\rar["\simeq"]\dar &\tau_{\geqslant 0}(\ku^{\t C_p})\dar\\
				\THH(\IF_p)\rar["\simeq"] & \tau_{\geqslant 0}(\IZ_p^{\t C_p})
			\end{tikzcd}
		\end{equation*}
		where the bottom row is the equivalence from \cite[Corollary~\textup{\chref{4.4.13}[IV.4.13]}]{NikolausScholze}.
	\end{thm}
	\begin{rem}
		\cref{thm:kutCp} was conjectured for all~$p$ by Lurie and Nikolaus. By an unpublished result of Nikolaus, \cref{thm:kutCp} is true as an $S^1$-equivariant $\IE_1$-equivalence for all~$p$ (see \cref{thm:kutCpE1} below). As far as the author is aware, constructing an $S^1$-equivariant $\IE_\infty$-equivalence case~$p=2$ is still open.
	\end{rem}
	\begin{rem}\label{rem:kutCpNotation}
		If we also let $q\in\pi_0(\ku^{\h S^1})\cong \ku^0(\B S^1)$ denote the class corresponding to the standard representation of $S^1$ on $\IC$, then the map from \cref{thm:kutCp} sends $q\mapsto q$. 
		
		Moreover, there's a unique complex orientation $t\in\pi_{-2}(\ku^{\h S^1})$ satisfying $q-1=\beta t$. In the following, we'll frequently use $\pi_*(\ku^{\h S^1})\cong \IZ[\beta]\llbracket t\rrbracket$, and we'll identify this graded $\IZ[t]$-algebra with the filtered ring $(q-1)^\star \IZ\qpower$, where $(q-1)$ in degree~$1$ corresponds to $\beta$.
	\end{rem}

	
	\begin{numpar}[The comparison map I.]\label{par:kuComparisonI}
		We import the equivalence from \cref{thm:kutCp} into the solid world via \cref{conv:SolidConventions}. Using this equivalence, we can construct an $S^1$-equivariant map of solid condensed spectra as follows:
		\begin{equation*}
			\begin{tikzcd}
				\bigl(\THH_\solid(\IS_R/\IS_A)\soltimes_{\IS_A,\phi_{\t C_p}}\IS_A\bigr)\soltimes\THH_\solid\bigl(\IZ_p[\zeta_p]/\IS_p\qpower\bigr)\dar["\simeq"']\rar& \THH_\solid(\IS_R/\IS_A)^{\t C_p}\soltimes \ku^{\t C_p}\dar\\
				\THH_\solid\left((R\lotimes_{A,\phi}A)_p^\complete[\zeta_p]/\IS_A\qpower\right)\rar[dashed] & \THH_\solid(\ku_R/\ku_A)^{\t C_p}
			\end{tikzcd}
		\end{equation*}
		The map in the top row is given by $\phi_{p/\IS_A}\soltimes(\labelcref{thm:kutCp})$, where $\phi_{p/\IS_A}$ denotes the relative cyclotomic Frobenius on $\THH(-/\IS_A)$. The right vertical arrow comes from lax symmetric monoidality of $(-)^{\t C_p}$. The left vertical arrow is an equivalence since $\THH$ is symmetric monoidal. So the dashed bottom horizontal arrow exists.
		
		Now $\THH_\solid(\IZ_p[\zeta_p]/\IS_p\qpower)\rightarrow \ku^{\t C_p}$ sends the generator $u\in\pi_2$ to a unit. Indeed, this can be checked modulo~$(q-1)=\beta t$, so we reduce to the same question for $\THH(\IF_p)\rightarrow \IZ_p^{\t C_p}$. Under the equivalence $\IZ_p^{\t C_p}\simeq \THH(\IF_p)^{\t C_p}$, this map becomes the cyclotomic Frobenius for $\THH(\IF_p)$, which is well-known to send $u$ to a unit. The diagram above thus induces an $S^1$-equivariant map
		\begin{equation*}
			\psi_R\colon \THH_\solid\bigl(R^{(p)}[\zeta_p]/\IS_A\qpower\bigr)\bigl[\localise{u}\bigr]\longrightarrow \THH_\solid(\ku_R/\ku_A)^{\t C_p}\,,
		\end{equation*}
		where $R^{(p)}\coloneqq (R\lotimes_{A,\phi}A)_p^\complete$ as in \cref{par:qdeRhamViaTC-}. From $\psi_R$, we can now construct a filtered map
		\begin{equation*}
			\psi_R^\star\colon \fil_{\ev}^\star\TC_\solid^-\bigl(R^{(p)}[\zeta_p]/\IS_A\qpower\bigr)\bigl[\localise{u}\bigr]_{(p,q-1)}^\complete\longrightarrow \fil_{\ev}^\star \TP_\solid(\ku_R/\ku_A)\,,
		\end{equation*}
		where the filtration on the left-hand side agrees with the Bhatt--Morrow--Scholze filtration, the Hahn--Raksit--Wilson, and the Pstr\k{a}gowski--Raksit even filtration. To construct $\psi_R^\star$, we have to distinguish the two cases:
		\begin{alphanumerate}
			\item[\IE_1] In situation~\cref{par:AssumptionsOnR}\cref{enum:E1Lift}, we construct $\psi_R^\star$ as the limit
			\begin{equation*}
				\limit_{\IDelta}\tau_{\geqslant 2\star}\TC_\solid^-\bigl((R_\infty^\bullet)^{(p)}[\zeta_p]/\IS_A\qpower\bigr)\bigl[\localise{u}\bigr]_{(p,q-1)}^\complete\xrightarrow{(\labelcref{par:kuComparisonI})} \limit_{\IDelta}\tau_{\geqslant 2\star}\TP_\solid(\ku_{R_\infty^\bullet}/\ku_A)
			\end{equation*}
			The left-hand side is $\fil_{\ev}^\star\TC_\solid^-(R^{(p)}[\zeta_p]/\IS_A\qpower)[1/u]_{(p,q-1)}^\complete$ by quasi-syntomic descent for the Bhatt--Morrow--Scholze even filtration and the right-hand side is $\fil_{\ev}\TP_\solid(\ku_R/\ku_A)$ by definition.\label{enum:psiRE1}
			\item[\IE_2] In situation~\cref{par:AssumptionsOnR}\cref{enum:E2Lift}, we construct $\psi_R^\star$ by applying $(\fil_{\ev}^\star(-))^{\h (\IT/C_p)_{\ev}}$ to the map from \cref{par:kuComparisonI} and composing with a certain canonical map\label{enum:psiRE2}
			\begin{equation*}
				\left(\fil_{\ev}^\star\THH_\solid(\ku_R/\ku_A)^{\t C_p}\right)^{\h (\IT/C_p)_{\ev}}\longrightarrow \fil_{\ev,\t S^1}^\star\TP_\solid(\ku_R/\ku_A)\,,
			\end{equation*}
			that will be constructed in \cref{par:kuComparisonII} below.
		\end{alphanumerate}		
	\end{numpar}
	\begin{numpar}[Even filtrations and the Tate construction.]\label{par:EvenFiltrationTate}
		To construct such a map, let more generally $T$ be a complex orientable solid $\IE_1$-ring spectrum and let $M$ be an $S^1$-equivariant left $T$-module such that $M^{\h C_p}$ is solid homologically even over $T^{\h C_p}$. Let $T_{\ev}^{\h S^1}\coloneqq \fil_{\ev}^\star T^{\h S^1}$ and $T_{\ev}^{\t S^1}\coloneqq \fil_{\ev}^\star T^{\t S^1}$. First observe that we have an equivalence
		\begin{equation*}
			T_{\ev}^{\t S^1}\soltimes_{T_{\ev}^{\h S^1}}\fil_{\ev/\smash{T^{\h C_p}}}^\star M^{\h C_p}\overset{\simeq}{\longrightarrow}\fil_{\ev/\smash{T^{\t C_p}}}^\star M^{\t C_p}
		\end{equation*}
		Indeed, choose a complex orientation $t\in\pi_{-2}(T^{\h S^1})$. It's well-known that $T^{\t S^1}\simeq T^{\h S^1}[t^{-1}]$ and $M^{\t C_p}\simeq M^{\h C_p}[t^{-1}]$. In particular, we see that both sides above are exhaustive filtrations on $M^{\t C_p}$, and so it's enough to check the equivalence on graded pieces. Since $t$ sits in even degree~$-2$, if we take any $\pi_*$-even envelope over $T^{\h S^1}$ or $T^{\h C_p}$ and invert~$t$, we get a $\pi_*$-even envelope over $T^{\t S^1}$ or $T^{\t C_p}$, respectively. Since the associated graded of the even filtration can be computed by successively taking $\pi_*$-even envelopes (see \cite[\S{\chref[section]{5}}]{PerfectEvenFiltration}; the solid analogue is discussed in \cref{par:CalculusOfSolidEvenness}), the claimed equivalence follows.
		
		Now let $(-)^{\h C_{p,\ev}}$ and $(-)^{\t C_{p,\ev}}$ denote the synthetic fixed point and Tate constructions from \cite[Definition~\chref{2.61}]{CyclotomicSynthetic}. We have canonical maps
		\begin{align*}
			\fil_{\ev/\smash{T^{\h C_p}}}^\star M^{\h C_p}&\longrightarrow \bigl(\fil_{\ev/T}^\star M\bigr)^{\h C_{p,\ev}}\,,\\
			T_{\ev}^{\t S^1}\soltimes_{T_{\ev}^{\h S^1}}\bigl(\fil_{\ev/T}^\star M\bigr)^{\h C_{p,\ev}}&\longrightarrow \bigl(\fil_{\ev/T}^\star M\bigr)^{\t C_{p,\ev}}\,.
		\end{align*}
		Composing these with the equivalence above, we get a canonical map
		\begin{equation*}
			\fil_{\ev/\smash{T^{\t C_p}}}^\star M^{\t C_p}\longrightarrow \bigl(\fil_{\ev/T}^\star M\bigr)^{\t C_{p,\ev}}\,.
		\end{equation*}
	\end{numpar}
	\begin{numpar}[The comparison map II.]\label{par:kuComparisonII}
		To construct the map that we need in \cref{par:kuComparisonI}\cref{enum:psiRE2}, we apply $(-)^{\h (\IT/C_p)_{\ev}}$ to the general construction from \cref{par:EvenFiltrationTate}, where $(-)^{\h (\IT/C_p)_{\ev}}$ denotes fixed points in the sense of \cite[\S{\chref[subsection]{2.3}}]{CyclotomicSynthetic} with respect to the even filtration on $\IS[S^1/C_p]$. It then remains to check that the canonical map
		\begin{equation*}
			\fil_{\ev}^\star\TP_\solid(\ku_R/\ku_A)\overset{\simeq}{\longrightarrow}\Bigl(\bigl(\fil_{\ev}^\star\THH_\solid(\ku_R/\ku_A)\bigr)^{\t C_{p,\ev}}\Bigr)^{\h (\IT/C_p)_{\ev}}
		\end{equation*}
		is an equivalence. To see this, we'll use the cosimplicial resolution from \cref{prop:EvenResolution}. A similar argument as in the proof of \cref{cor:TC-TPEvenResolution} can be used to verify that $(-)^{\t C_{p,\ev}}$ commutes with the cosimplicial limit. We can thus reduce to the case where $\THH_\solid(\ku_R/\ku_A)$ is already even. The desired result then follows from \cite[Lemma~\chref{2.75}(vi)]{CyclotomicSynthetic}, its analogue for $(-)_{\h C_{p,\ev}}$, and the classical fact that $(-)^{\t S^1}\simeq ((-)^{\t C_p})^{\h (S^1/C_p)}$ holds on bounded below $p$-complete spectra by \cite[Lemma~{\chref{2.4.2}[II.4.2]}]{NikolausScholze}.
	\end{numpar}
	
	\begin{numpar}[The $q$-Hodge filtration.]\label{par:qHodgeFiltrationTC-ku}
		We can pass to the $0$\textsuperscript{th} graded piece of our filtered comparison map $\psi_R^\star$ and use \cref{prop:qdeRhamTC-} to obtain a map
		\begin{equation*}
			\psi_R^0\colon \qdeRham_{R/A}\longrightarrow \gr_{\ev,\t S^1}^0\TP_\solid(\ku_R/\ku_A)\simeq \gr_{\ev,\t S^1}^0\TC_\solid^-(\ku_R/\ku_A)\,.
		\end{equation*}
		Now $\gr_{\ev,\h S^1}^*\TC_\solid^-(\ku_R/\ku_A)$ is a graded module over $\gr_{\ev,\h S^1}^*(\ku^{\h S^1})\simeq \Sigma^{2*}\pi_{2*}(\ku^{\h S^1})$. Hence the double shearing $\Sigma^{-2*}\gr_{\ev,\h S^1}^*\TC_\solid^-(\ku_R/\ku_A)$ is a graded module over $\IZ_p[\beta]\llbracket t\rrbracket$, with $\abs{\beta}=2$, $\abs{t}=-2$.%
		\footnote{Also note that since everything is $\IZ$-linear, the double shearing functor $\Sigma^{2*}$ is symmetric monoidal.}
		We can regard $t$ as a filtration parameter, so that the graded $\IZ_p[\beta]\llbracket t\rrbracket$-module $\Sigma^{-2*}\gr_{\ev,\h S^1}^*\TC_\solid^-(\ku_R/\ku_A)$ defines a filtration on $\gr_{\ev,\h S^1}^0\TC_\solid^-(\ku_R/\ku_A)$. We define the \emph{$q$-Hodge filtration} as the pullback
		\begin{equation*}
			\begin{tikzcd}
				\fil_{\qHodge}^\star\qdeRham_{R/A}\rar\dar\drar[pullback] & \Sigma^{-2*}\gr_{\ev,\h S^1}^*\TC_\solid^-(\ku_R/\ku_A)\dar\\
				\qdeRham_{R/A}\rar["\psi_R^0"] & \gr_{\ev,\h S^1}^0\TC_\solid^-(\ku_R/\ku_A)
			\end{tikzcd}
		\end{equation*}
		The name \emph{$q$-Hodge filtration} is justified by the fact that $\fil_{\qHodge}^\star\qdeRham_{R/A}$ is indeed a $q$-deformation of the Hodge filtration on $\deRham_{R/A}$. This is part of the main result of this subsection, which we can now formulate and prove. Here we identify the graded $\IZ[t]$-algebra $\IZ_p[\beta]\tpower$ with the $(q-1)$-adic filtration $(q-1)^\star \IZ_p\qpower$ as explained in \cref{rem:kutCpNotation}.
	\end{numpar}

	\begin{thm}\label{thm:qdeRhamkupComplete}
		Let $p>2$ be a prime and let $A$ and $R$ satisfy the assumptions from~\cref{par:AssumptionsOnA} and~\cref{par:AssumptionsOnR}. Then the map $\psi_R^0$ from \cref{par:kuComparisonII} induces an equivalence of graded $\IZ_p[\beta]\llbracket t\rrbracket$-modules
		\begin{equation*}
			\fil_{\qHodge}^\star \qhatdeRham_{R/A}\overset{\simeq}{\longrightarrow}\Sigma^{-2*}\gr_{\ev,\h S^1}^*\TC_\solid^-(\ku_R/\ku_A)\,,
		\end{equation*}
		where the left-hand side denotes the completion of the $q$-Hodge filtration $\fil_{\qHodge}^*\qdeRham_{R/A}$ from \cref{par:qHodgeFiltrationTC-ku}. Moreover, modulo $\beta$ and after rationalisation, we get equivalences
		\begin{align*}
			\fil_{\qHodge}^\star\qdeRham_{R/A}\lotimes_{\IZ_p[\beta]\tpower}\IZ_p\tpower&\overset{\simeq}{\longrightarrow} \fil_{\Hodge}^\star \deRham_{R/A}\,,\\
			\fil_{\qHodge}^\star\qdeRham_{R/A}\bigl[\localise{p}\bigr]_{(q-1)}^\complete&\overset{\simeq}{\longrightarrow}\fil_{(\Hodge,q-1)}^\star\deRham_{R/A}\bigl[\localise{p}\bigr]\qpower
		\end{align*}
		with the usual Hodge filtration and the combined Hodge and $(q-1)$-adic filtration, respectively.
	\end{thm}
	\begin{rem}
		In case~\cref{par:AssumptionsOnR}\cref{enum:E2Lift}, all equivalences in \cref{thm:qdeRhamkupComplete} are canonically $\IE_1$-monoidal. In fact, if $\IS_R$ can be equipped with an $\IE_n$-algebra structure in $\IS_A$-modules for any $2\leqslant n\leqslant \infty$, then all equivalences will be canonically $\IE_{n-1}$-monoidal. To see this, observe that for any $T\in \Alg_{\IE_2}(\Mod_{\IS_A}(\Sp_\solid))$, we can use the same construction as in \cref{par:kuComparisonI} to produce an $S^1$-equivariant map
		\begin{equation*}
			\THH_\solid \bigl((T\soltimes_{\IS_A,\phi_{\t C_p}}\IS_A)\soltimes\IZ_p[\zeta_p]/\IS_A\qpower\bigr)\bigl[\localise{u}\bigr]\longrightarrow\THH_\solid (\ku\soltimes T/\ku_A)^{\t C_p}\,;
		\end{equation*}
		these maps assemble into a symmetric monoidal transformation of symmetric monoidal functors $\Alg_{\IE_2}(\Mod_{\IS_A}(\Sp_\solid))\rightarrow \Alg_{\IE_1}(\Sp_\solid^{\B S^1})$. If $\IS_R$ admits an $\IE_n$-algebra structure in $\IS_A$-modules, then $\IS_R\in \Alg_{\IE_{n-2}}(\Alg_{\IE_2}(\Mod_{\IS_A}(\Sp_\solid)))$ and so $\psi_R$ is $S^1$-equivariantly $\IE_{n-2}$ as a map in $\Alg_{\IE_1}(\Sp_\solid)$, hence $S^1$-equivariantly $\IE_{n-1}$ as a map in $\Sp_\solid$. The other parts of the construction clearly preserve $\IE_{n-1}$-monoidality.
		
		If we are in case~\cref{par:AssumptionsOnR}\cref{enum:E1Lift}, then a priori we only get $\IE_0$-monoidal structures. However,  we can a posteriori upgrade everything from $\IE_0$ to $\IE_\infty$ by applying \cref{thm:qdeRhamkuRegularQuotient} below to the given resolution $R\rightarrow R_\infty^\bullet$.		
	\end{rem}
	
	The main step in the proof of \cref{thm:qdeRhamkupComplete} is to describe $\psi_R^0$ modulo~$(q-1)$.
	
	\begin{lem}\label{lem:psiRmodq-1}
		The reduction modulo $(q-1)=\beta t$ of the map $\psi_R^0$ from \cref{par:kuComparisonII} agrees with the canonical Hodge completion map
		\begin{equation*}
			\deRham_{R/A}\longrightarrow \hatdeRham_{R/A}\simeq \gr_{\ev,\t S^1}^0\HP_\solid(R/A)\,.
		\end{equation*}
	\end{lem}
	\begin{proof}[Proof \embrace{initial reduction}]
		In the following, we'll assume we're in case~\cref{par:AssumptionsOnR}\cref{enum:E2Lift}. In case~\cref{par:AssumptionsOnR}\cref{enum:E1Lift}, we repeat the arguments below instead for each term in the cosimplicial resolution $R_\infty^\bullet$, with the even filtration replaced by $\tau_{\geqslant 2\star}$.
		
		Put $\ov R\coloneqq R\lotimes_{\IZ_p}\IF_p$ and $\ov R^{(p)}\coloneqq \ov R\lotimes_{A,\phi}A$ for short. If we reduce the diagram from \cref{par:kuComparisonI} modulo $(q-1)=\beta t$, we obtain the following commutative diagram:
		\begin{equation*}
			\begin{tikzcd}
				\bigl(\THH_\solid(\IS_R/\IS_A)\soltimes_{\IS_A,\phi_{\t C_p}}\IS_A\bigr)\soltimes\THH(\IF_p)\dar["\simeq"']\rar& \THH_\solid(\IS_R/\IS_A)^{\t C_p}\soltimes \IZ_p^{\t C_p}\dar\\
				\THH_\solid\bigl(\ov R^{(p)}/\IS_A\bigr)\rar[dashed] & \HH_\solid(R/A)^{\t C_p}
			\end{tikzcd}
		\end{equation*}
		The top row is induced by the equivalence $\THH(\IF_p)\simeq \tau_{\geqslant 0}(\IZ_p^{\t C_p})$ from \cite[Corollary~\textup{\chref{4.4.13}[IV.4.13]}]{NikolausScholze} and the relative cyclotomic Frobenius $\phi_{p/\IS_A}$ for $\THH(-/\IS_A)$. After passing to homotopy $S^1$-fixed points, the bottom row of this diagram factors induces a map
		\begin{equation*}
			\ov\psi_R^{\h S^1}\colon \TC_\solid^-\bigl(\ov R^{(p)}/\IS_A\bigr)\bigl[\localise{u}\bigr]_p^\complete\longrightarrow \HP_\solid(R/A)\,.
		\end{equation*}
		The key observation is now that the map $\ov\psi_R^{\h S^1}$ can be constructed without the choice of a spherical lift $\IS_R$. Let us interrupt the proof for the moment and discuss how this works.
	\end{proof}
	\begin{numpar}[Constructing $\ov\psi_R^{\h S^1}$ without a spherical lift.]\label{par:ConstructionWithoutSphericalLift}
		Let us first assume that $A\cong \W(k)$ is the ring of Witt vectors over a perfect field of characteristic~$p$. In this case, Petrov and Vologodsky \cite{PetrovVologodsky} construct an equivalence $\TP_\solid(\ov R/\IS_A)\simeq \HP_\solid (R/A)$ without choosing any spherical lift $\IS_R$. We claim that this equivalence holds, in fact, for arbitrary~$A$, and that the composition with the relative cyclotomic Frobenius
		\begin{equation*}
			\phi_{p/\IS_A}^{\h S^1}\colon \TC_\solid^-\bigl(\ov R^{(p)}/\IS_A\bigr)\bigl[\localise{u}\bigr]_p^\complete\longrightarrow \TP(\ov R/\IS_A)
		\end{equation*}
		agrees with the map $\ov\psi_R^{\h S^1}$. Both of these claims follow from work of Devalapurkar and Raksit \cite{DevalapurkarRaksitTHH}: They give a new proof of the equivalence $\TP_\solid(\ov R/\IS_A)\simeq \HP_\solid (R/A)$, which works for arbitrary $A$, and from their proof it will be apparent that the maps indeed coincide. The new proof is based on the following result:
	\end{numpar}
	\begin{thm}[Devalapurkar--Raksit \cite{DevalapurkarRaksitTHH}]\label{thm:jtCp}
		Let $j\coloneqq \tau_{\geqslant 0}(\IS_{K(1)})$ be the connective cover of the $K(1)$-local sphere.
		\begin{alphanumerate}
			\item There is an equivalence $\THH(\IZ_p)_p^\complete\simeq \tau_{\geqslant 0}(j^{\t C_p})$ as well as a commutative diagram\label{enum:jtCp}
			\begin{equation*}
				\begin{tikzcd}
					j\rar\dar\drar[commutes,pos=0.33] & \THH(\IZ_p)_p^\complete\dar\dlar[dashed]\\
					\IZ_p\rar & \THH(\IF_p)
				\end{tikzcd}
			\end{equation*}
			of $S^1$-equivariant \embrace{in fact, cyclotomic} $\IE_\infty$-rings. Moreover, there exists a dashed diagonal arrow that makes the upper left but not the lower right triangle commute $S^1$-equivariantly.
			\item The horizontal maps $j\rightarrow \THH(\IZ_p)_p^\complete$ and $\IZ_p\rightarrow \THH(\IF_p)$ are $S^1$-nilpotent, that is, for any spectrum $X$ with $S^1$-action the maps $X\otimes j\rightarrow X\otimes\THH(\IZ_p)_p^\complete$ and $X\otimes \IZ_p\rightarrow X\otimes\THH(\IF_p)$ become equivalences upon $(-)^{\t S^1}$.\label{enum:S1nilpotent}
		\end{alphanumerate}
	\end{thm}
	The new proof of the equivalence $\TP_\solid(\ov R/\IS_A)\simeq \HP_\solid (R/A)$ in \cite[\S{\chref[subsection]{5}}]{DevalapurkarRaksitTHH} then proceeds as follows: By \cref{thm:jtCp}\cref{enum:jtCp} we have an $S^1$-equivariant commutative diagram
	\begin{equation*}
		\begin{tikzcd}
			\THH_\solid(R/\IS_A)\soltimes_j\IZ_p\rar["(\simeq)^{\t S^1}"]\dar["(\simeq)^{\t S^1}"'] & \THH_\solid(R/\IS_A)\soltimes_{\THH_\solid(\IZ_p)}\IZ_p\dar[dashed]\\
			\THH_\solid(R/\IS_A)\soltimes_{j}\THH(\IF_p)\rar["(\simeq)^{\t S^1}"] & \THH_\solid(R/\IS_A)\soltimes_{\THH_\solid(\IZ_p)}\THH(\IF_p)
		\end{tikzcd}
	\end{equation*}
	By \cref{thm:jtCp}\cref{enum:S1nilpotent}, the horizontal arrows and the left vertical arrow become equivalences after applying $(-)^{\t S^1}$.%
	\footnote{The functor $(-)^{\t S^1}$ factors through a certain category, denoted $\widehat{\Mod}\vphantom{d}_{W[S^1]}^{\t}$ by \cite{PetrovVologodsky} and $(\Mod_j^{\t S^1})_{(p,v_1)}^\complete$ by \cite{DevalapurkarSpherochromatism}; the $S^1$-nilpotence property from \cref{thm:jtCp}\cref{enum:S1nilpotent} ensures that $j\rightarrow \THH(\IZ_p)_p^\complete$ and $\IZ_p\rightarrow \THH(\IF_p)$ become equivalences in that category.}
	Hence after $(-)^{\t S^1}$ the dashed vertical arrow exists and it induces the desired equivalence $\HP_\solid(R/A)\simeq \TP_\solid(\ov R/\IS_A)$.
	
	Using $\THH_\solid(R/\IS_A)\simeq \THH_\solid(\IS_R/\IS_A)\soltimes \THH_\solid(\IZ_p)$, it is also apparent that the composition of this equivalence with the relative cyclotomic Frobenius $\phi_{p/\IS_A}^{\h S^1}$ agrees with the map $\ov\psi_R^{\h S^1}$, as we've claimed above.

	\begin{proof}[Proof of \cref{lem:psiRmodq-1} \embrace{end of proof}]
		The proof can now be finished as follows: Let $S$ be a $p$-torsion free $p$-complete $p$-quasi-lci $A$-algebra, put $\ov S\coloneqq S/p$ and $\ov S^{(p)}\coloneqq \ov S\lotimes_{A,\phi}A$. Via quasi-syntomic descent as in the proof of \cref{prop:qdeRhamTC-}, we can define a Bhatt--Morrow--Scholze-style even filtration $\fil_{\BMSev,\h S^1}^\star \TC_\solid^-(\ov S^{(p)}/\IS_A)[1/u]_p^\complete$ together with a map
		\begin{equation*}
			\ov\psi_S^\star\colon \fil_{\BMSev,\h S^1}^\star\TC_\solid^-\bigl(\ov S^{(p)}/\IS_A\bigr)\bigl[\localise{u}\bigr]_p^\complete\longrightarrow \fil_{\BMSev,\t S^1}^\star\HP_\solid(S/A)\,;
		\end{equation*}
		to construct this map, we use \cref{par:ConstructionWithoutSphericalLift} above. By passing to animations, we can also cover the case $S=R$.%
		\footnote{Observe that $\ov R^{(p)}$ might only be an animated ring.}
		A comparison with prismatic cohomology as in the proof of \cref{prop:qdeRhamTC-} shows that the $0$\textsuperscript{th} graded piece of $\ov\psi_S^\star$ has the form
		\begin{equation*}
			\ov\psi_S^0\colon \Prism_{\ov S^{(p)}/A}\simeq \deRham_{S/A}\longrightarrow \hatdeRham_{S/A}\,;
		\end{equation*}
		here we also use the crystalline comparison for prismatic cohomology \cite[Theorem~\chref{5.2}]{Prismatic} and the fact that the de Rham cohomology of $S$ agrees with the crystalline cohomology of its reduction $\ov S$. If we can show that $\ov\psi_S^0$ is the canonical Hodge completion map, then we'll be done, because from the comparison results in \cref{cor:HRWvsPstragowski,cor:BMSvsPstragowski} it's clear that in the case $S=R$ the map $\ov\psi_R^\star$ agrees with the reduction of $\psi_R^0$ modulo $(q-1)$.
		
		To show that $\ov\psi_S^0$ has the desired form, we can now use quasi-syntomic descent. In particular, we may reduce to a situation where $S/p$ is relatively semiperfect over $A$ (i.e.\ the relative Frobenius $S/p\otimes_{A,\phi}A\twoheadrightarrow S/p$ is surjective). Then everything is even, hence both sides of $\ov\psi_S^\star$ are double speed Whitehead filtrations on even spectra and $\ov\psi_S^0$ is a map between two static condensed rings. Whether this map is the correct one can be checked on the level of sets and hence after any $p$-completely faithfully flat base change. Let $A_\infty$ denote the $p$-completed colimit perfection of $A$. By our assumption~\cref{par:AssumptionsOnA}, $A\rightarrow A_\infty$ is $p$-completely faithfully flat, and it can be lifted to an $\IE_\infty$-map $\IS_A\rightarrow \IS_{A_\infty}$ (see \cref{lem:S_Ainfty} for example). Via base change along this map, we may reduce to the case where $A$ is perfect. Then $S/p$ is semiperfect on the nose and so $\A_\inf\coloneqq \W(S^\flat)\twoheadrightarrow S$ is surjective.
		
		Now everything becomes rather explicit: Let $J\coloneqq \ker(\A_\inf\rightarrow R)$ and let $\A_\crys\coloneqq D_{\A_\inf}(J)$ denote the $p$-completed PD-envelope of $J$. It's well-known%
		\footnote{Indeed, the first equivalence follows from the fact that $A$ and $\A_\inf$ being are perfect $\delta$-rings. For the second, note that $\deRham_{R/\A_\inf}$ is $p$-torsion free and contains divided powers for all $x\in J$, as can be seen from $\deRham_{\IZ/\IZ[x]}\rightarrow \deRham_{R/\A_\inf}$. Hence there's a map $\A_\crys\rightarrow \deRham_{R/\A_\inf}$, and this map is an equivalence modulo~$p$ by \cite[Proposition~\chref{8.12}]{BMS2}.}
		that
		\begin{equation*}
			\deRham_{S/A}\simeq \deRham_{R/\A_\inf}\simeq \A_\crys\,.
		\end{equation*}
		Since the un-$p$-completed PD-envelope $\A_\crys^\circ$ of $J\subseteq \A_\inf$ is contained in $\A_\inf[1/p]$, the Hodge completion map $\A_\crys\rightarrow\widehat{\A}_\crys$ is uniquely characterised by the following two properties:
		\begin{alphanumerate}\itshape
			\item It is a map of $\A_\inf$-modules.\label{enum:AinfModuleMap}
			\item It is continuous with respect to the natural topologies on either side.\label{enum:Continuous}
		\end{alphanumerate}
		It's clear from the construction that $\ov\psi_S^0$ satisfies~\cref{enum:Continuous} since it is a map of condensed rings. To see \cref{enum:AinfModuleMap}, just observe that in the construction of $\ov\psi_S^0$, instead of working with $\THH_\solid(-/\IS_A)$, we could have worked with $\THH_\solid(-/\IS_{\A_\inf})$, where $\IS_{\A_\inf}$ denotes the unique lift of the perfect $\delta$-ring $\A_\inf$ to a $p$-complete connective $\IE_\infty$-ring spectrum.
	\end{proof}
	Next let us describe $\psi_R^0$ after rationalisation.
	\begin{lem}\label{lem:psiRrational}
		The rationalisation of the map $\psi_R^0$ from \cref{par:kuComparisonII} fits into a commutative diagram
		\begin{equation*}
			\begin{tikzcd}
				\qdeRham_{R/A}\bigl[\localise{p}\bigr]_{(q-1)}^\complete\rar["\psi_{R,\IQ_p}^0"]\dar["\simeq"'] & \gr_{\ev,\h S^1}^0\TC_\solid^-(\ku_R/\ku_A)\bigl[\localise{p}\bigr]_{(q-1)}^\complete\dar["\simeq"]\\
				\deRham_{R/A}\bigl[\localise{p}\bigr]\qpower\rar & \deRham_{R/A}\bigl[\localise{p}\bigr]_{\Hodge}^\complete\qpower
			\end{tikzcd}
		\end{equation*}
		where the left vertical arrow is the usual equivalence for rationalised $q$-de Rham cohomology, the right vertical arrow is obtained via \cref{rem:EvenFiltrationkuRationalisation}, and the bottom arrow is the natural Hodge completion map.
	\end{lem}
	\begin{proof}
		The following argument was suggested by Peter Scholze (any errors are due to the author). Observe that the usual rationalisation equivalence $\qdeRham_{R/A}[1/p]_{(q-1)}^\complete\simeq \deRham_{R/A}[1/p]\qpower$ is $\IZ_p^\times$-equivariant, where the action on the left-hand side is the one discussed in~\cref{par:AdamsAction} and on the right-hand side $u\in\IZ_p^\times$ acts via $q\mapsto q^u$. Since the equivalence from \cref{thm:kutCp} is also $\IZ_p^\times$-equivariant, we obtain a $\IZ_p^\times$-equivariant map
		\begin{equation*}
			\deRham_{R/A}\bigl[\localise{p}\bigr]\qpower\longrightarrow  \deRham_{R/A}\bigl[\localise{p}\bigr]_{\Hodge}^\complete\qpower\,,
		\end{equation*}
		which we must show to agree with the natural Hodge completion map. In general, if $M\in\Dd(\IQ_p)$ is equipped with the trivial action of $\IZ_p^\times$, there's a functorial equivalence
		\begin{equation*}
			M\overset{\simeq }{\longrightarrow}M\qpower^{\h \IZ_p^\times}\lotimes_{\IZ_{\smash{p}}^{\smash{\h \IZ_{p}^\times}\vphantom{h_p}}}\IZ_p\,.
		\end{equation*}
		Indeed, the fixed points $M\qpower^{\h\IZ_p^\times}$ would be $M\oplus \Sigma^{-1}M$; to kill the shifted copy of $M$, we take the tensor product along $\IZ_{\smash{p}}^{\smash{\h \IZ_{p}^\times}\vphantom{h_p}}\rightarrow\IZ_p$.
		
		Applying this in the situation at hand, we get a map $\deRham_{R/A}[1/p]\rightarrow \deRham_{R/A}[1/p]_{\Hodge}^\complete$. By comparison with the reduction modulo $(q-1)$ and using \cref{lem:psiRmodq-1}, we see that this map must be the canonical Hodge completion map. By applying $(-\lotimes_{\IQ_p}\IQ_p\qpower)_{(q-1)}^\complete$ to this map, we deduce that the original map must have been the natural Hodge completion as well.
	\end{proof}
	
	\begin{proof}[Proof of \cref{thm:qdeRhamkupComplete}]
		By definition of the filtration $\fil_{\qHodge}^\star\qdeRham_{R/A}$ (see \cref{par:qHodgeFiltrationTC-ku}), the base change $\fil_{\qHodge}^\star\qdeRham_{R/A}\lotimes_{\IZ_p[\beta]\tpower}\IZ_p\tpower$ is the pullback of the filtered module $\Sigma^{-2\star}\gr_{\ev,\h S^1}^\star\HC_\solid^-(R/A)$ along $\ov\psi_R^0\colon \deRham_{R/A}\rightarrow \gr_{\ev,\h S^1}^0\HC_\solid^-(R/A)$. The rationalisation $\fil_{\qHodge}^\star\qdeRham_{R/A}[1/p]_{(q-1)}^\complete$ can be described analogously. Using \cref{lem:psiRmodq-1,lem:psiRrational} as well as the fact that any filtration is the pullback of its completion (see \cref{par:Notation}), we deduce that
		\begin{align*}
			\fil_{\qHodge}^\star\qdeRham_{R/A} \lotimes_{\IZ_p[\beta]\tpower}\IZ_p\tpower&\overset{\simeq}{\longrightarrow} \fil_{\Hodge}^\star \deRham_{R/A}\,,\\
			\fil_{\qHodge}^\star \qdeRham_{R/A}\bigl[\localise{p}\bigr]_{(q-1)}^\complete&\overset{\simeq}{\longrightarrow}\fil_{(\Hodge,q-1)}^\star \deRham_{R/A}\bigl[\localise{p}\bigr]\qpower
		\end{align*}
		are indeed equivalences. Finally, whether
		\begin{equation*}
			\fil_{\qHodge}^\star\qhatdeRham_{R/A}\overset{\simeq}{\longrightarrow}\Sigma^{-2*}\gr_{\ev}^*\TC_\solid^-(\ku_R/\ku_A)
		\end{equation*}
		is an equivalence can be checked modulo~$\beta$. By the base change result that we've already shown, this follows from $\fil_{\Hodge}^\star \hatdeRham_{R/A} \simeq \Sigma^{-2*}\gr_{\ev}^*\HC_\solid^-(R/A)$.
	\end{proof}
	
	\subsection{The \texorpdfstring{$p$}{p}-complete comparison (case \texorpdfstring{$p=2$}{p=2})}
	
	In this subsection, we'll discuss how much of \cref{subsec:qdeRhamkupComplete} can be salvaged in the case $p=2$. We expect that \cref{thm:qdeRhamkupComplete} is still true for $p=2$, but our proof fails at several places. Here are the two main issues:
	\begin{alphanumerate}\itshape
		\item[\,!\,] The $S^1$-equivariant $\IE_\infty$-equivalence $\THH(\IZ_p[\zeta_p]/\IS_p\qpower)\simeq \tau_{\geqslant 0}(\ku^{\t C_p})$ from \cref{thm:kutCp} is still conjectural for $p=2$.\label{enum:Problemp=2A}
		\item[\,!!\,] \cref{thm:jtCp} is provably false for $p=2$.\label{enum:Problemp=2B}
	\end{alphanumerate}
	The objection in the second issue is essentially the discrepancy between Nygaard and divided power completion at $p=2$; see \cite[Remark~\chref{0.5.3}]{DevalapurkarRaksitTHH} for example. The goal of this subsection is to show that both issues only affect the case~\cref{par:AssumptionsOnR}\cref{enum:E2Lift}. 
	
	\begin{thm}\label{thm:qdeRhamku2Complete}
		If $R$ satisfies the assumptions from~\cref{par:AssumptionsOnR}\cref{enum:E1Lift}, then the conclusions of~\cref{thm:qdeRhamkupComplete} are true in the case $p=2$ as well.
	\end{thm}
	\begin{rem}
		Note that a priori $\fil_{\qHodge}^\star\qdeRham_{R/A}$ will only be a graded $\IE_0$-algebra over $\IZ_p[\beta]\llbracket t\rrbracket$. A posteriori, we get an $\IE_\infty$-structure by applying \cref{thm:qdeRhamkuRegularQuotient} below to the given cosimplicial resolution $R\rightarrow R_\infty^\bullet$.
	\end{rem}
	
	To show \cref{thm:qdeRhamku2Complete}, let us first address the less serious issue \cref{enum:Problemp=2A} above.
	
	\begin{thm}[Nikolaus, unpublished]\label{thm:kutCpE1}
		For all primes~$p$ there exists an $S^1$-equivariant equivalence of $\IE_1$-ring spectra 
		\begin{equation*}
			\THH\bigl(\IZ_p[\zeta_p]/\IS_p\qpower\bigr)_p^\complete\overset{\simeq}{\longrightarrow} \tau_{\geqslant 0}\left(\ku^{\t C_p}\right)\,,
		\end{equation*}
		compatible with $\THH(\IF_p)\simeq\tau_{\geqslant 0}(\IZ_p^{\t C_p})$. For $p>2$, this equivalence agrees with the underlying $S^1$-equivariant $\IE_1$-equivalence of \cref{thm:kutCp}.
	\end{thm}
	\begin{proof}
		We thank Sanath Devalapurkar for explaining the following argument to us; any errors are our own responsibility. Let us first construct an $S^1$-equivariant $\IE_\infty$-map $\IS\qpower\rightarrow \ku^{\t C_p}$, where the left-hand side receives the trivial $S^1$ action and the right-hand side the residual $S^1\simeq S^1/C_p$-action. It's enough to construct an $S^1$-equivariant $\IE_\infty$-map $\IS\qpower\rightarrow \ku^{\h C_p}$, or equivalently, an $\IE_\infty$-map $\IS\qpower\rightarrow (\ku^{\h C_p})^{\h(S^1/C_p)}\simeq \ku^{\h S^1}$. But the element $q\in \pi_0(\ku^{\h S^1})$ is is detected by an $\IE_\infty$-map $\IS[q]\rightarrow \ku^{\h S^1}$; see \cref{cor:qStrictElement}. This factors over the $(q-1)$-completion $\IS[q]\rightarrow \IS\qpower$ and so we obtain the desired map.
		
		Now let us construct an $\IE_2$-$\IS_p\qpower$-algebra map $\IZ_p[\zeta_p]\rightarrow \ku^{\t C_p}$. To this end, observe that $\IZ_p[\zeta_p]$ is the free $(q-1)$-complete $\IE_2$-$\IS_p\qpower$-algebra satisfying $[p]_q=0$. Indeed, since $[p]_q=0$ holds in $\IZ_p[\zeta_p]$, it certainly receives an $\IE_2$-$\IS_p\qpower$-map from the free guy. Whether this map is an equivalence can be checked modulo $(q-1)$, where it reduces to the classical fact that $\IF_p$ is the free $\IE_2$-algebra satisfying $p=0$. Since $[p]_q=0$ holds in $\pi_*(\ku^{\t C_p})\cong \pi_*(\ku^{\t S^1})/[p]_q$ and any nullhomotopy witnessing this must be unique by evenness, we get our desired $\IE_2$-$\IS_p\qpower$-algebra map $\IZ_p[\zeta_p]\rightarrow \ku^{\t C_p}$. It induces $S^1$-equivariant $\IE_1$-$\IS_p\qpower$-algebra maps
		\begin{equation*}
			\THH\bigl(\IZ_p[\zeta_p]/\IS_p\qpower\bigr)_p^\complete\longrightarrow \THH\bigl(\ku^{\t C_p}/\IS\qpower\bigr)_p^\complete\longrightarrow \ku^{\t C_p}\,,
		\end{equation*}
		where the arrow on the right comes from the universal property of $\THH(-/\IS\qpower)$ on $\IE_\infty$-$\IS\qpower$-algebras.%
		\footnote{In particular, this map $\THH(\ku^{\t C_p}/\IS_p\qpower)_p^\complete\rightarrow \ku^{\t C_p}$ is \emph{not} the usual augmentation, as the augmentation would only be $S^1$-equivariant for the trivial $S^1$-action on $\ku^{\t C_p}$.}
		Since the left-hand side is connective, the above composition factors through an $S^1$-equivariant $\IE_1$-$\IS_p\qpower$-algebra map $\THH(\IZ_p[\zeta_p]/\IS_p\qpower)_p^\complete\rightarrow \tau_{\geqslant 0}(\ku^{\t C_p})$.
		
		We wish to show that this map is an equivalence. This can be checked modulo $(q-1)$, so it will be enough to prove that modulo $(q-1)$ we obtain the equivalence $\THH(\IF_p)\simeq \tau_{\geqslant 0}(\IZ_p^{\t C_p})$ from \cite[Corollary~\textup{\chref{4.4.13}[IV.4.13]}]{NikolausScholze}. To this end, observe that by the universal properties of $\IZ_p[\zeta_p]$ and $\IF_p$ as free $\IE_2$-algebras, the $\IE_\infty$-map $\ku^{\t C_p}\rightarrow \IZ_p^{\t C_p}$ fits into a commutative diagram of $\IE_2$-algebras
		\begin{equation*}
			\begin{tikzcd}
				\IZ_p[\zeta_p]\rar\dar & \ku^{\t C_p}\dar\\
				\IF_p\rar & \IZ_p^{\t C_p}
			\end{tikzcd}
		\end{equation*}
		which on the level of underlying spectra exhibits the bottom row as the mod-$(q-1)$-reduction of the top row. Using the same recipe as above, the bottom row induces an $S^1$-equivariant maps of $\IE_1$-algebras
		\begin{equation*}
			\THH(\IF_p)\longrightarrow \THH\bigl(\IZ_p^{\t C_p}\bigr)\longrightarrow \IZ_p^{\t C_p}
		\end{equation*}
		After passing to connective covers, we get an $S^1$-equivariant $\IE_1$-map $\THH(\IF_p)\rightarrow \tau_{\geqslant 0}(\IZ_p^{\t C_p})$. We claim that this map necessarily agrees with the underlying $\IE_1$-map of the $S^1$-equivariant $\IE_\infty$-equivalence $\THH(\IF_p)\simeq \tau_{\geqslant 0}(\IZ^{\t C_p})$ from \cite[Corollary~\textup{\chref{4.4.13}[IV.4.13]}]{NikolausScholze}. Indeed, by the universal property of $\THH$ for $\IE_\infty$-ring spectra, this equivalence must also be given by a composition as above, where the first arrow is given by the non-equivariant $\IE_\infty$-map $\IF_p\rightarrow \IZ_p^{\t C_p}$ induced by the equivalence. But $\IF_p$ is the free $\IE_2$-algebra with $p=0$. Since $\IZ_p^{\t C_p}$ is even, any nullhomotopy witnessing $p=0$ is unique, and so there's a unique $\IE_2$-map $\IF_p\rightarrow \IZ_p^{\t C_p}$. This shows that the $S^1$-equivariant $\IE_1$-map $\THH(\IF_p)\rightarrow \tau_{\geqslant 0}(\IZ_p^{\t C_p})$ agrees with the equivalence $\THH(\IF_p)\simeq \tau_{\geqslant 0}(\IZ_p^{\t C_p})$ and concludes the proof that $\THH(\IZ_p[\zeta_p]/\IS_p\qpower)_p^\complete\rightarrow \tau_{\geqslant 0}(\ku^{\t C_p})$ is an equivalence.
		
		To show that for $p>2$ this equivalence agrees with the underlying $S^1$-equivariant $\IE_1$-equivalence of \cref{thm:kutCp}, we can use the same argument as above, noting that the $\IE_2$-$\IS\qpower$-algebra map $\IZ_p[\zeta_p]\rightarrow \ku^{\t C_p}$ is unique.
	\end{proof}
	
	We can now show \cref{thm:qdeRhamku2Complete}.
	
	\begin{proof}[Proof sketch of \cref{thm:qdeRhamku2Complete}]
		Let us indicate how to modify the arguments in order to avoid those that don't work for $p=2$. To construct the comparison map $\psi_R^0$ as an $\IE_0$-map, we don't need the full strength of \cref{thm:kutCp}, so \cref{thm:kutCpE1} will suffice. In the proof of \cref{lem:psiRmodq-1}, we don't need quasi-syntomic descent (and in particular, we don't need \cref{thm:jtCp}, so we circumvent the more serious issue~\cref{enum:Problemp=2B} above), since the given resolution $R\rightarrow R_\infty^\bullet$ places us already in a relatively semiperfect situation.
		
		It remains to explain how to adapt the proof of \cref{lem:psiRrational}. We don't know if the $\IZ_p^\times$-equivariance argument still works, but fortunately, we can replace it by a simple argument similar to the proof of \cref{lem:psiRmodq-1}. In the given resolution, $R_\infty^\bullet/p$ is already relatively semiperfect over~$A$ and so $\TC_\solid^-(\ku_{R_\infty^\bullet}\soltimes\IQ_p/\ku_A\soltimes\IQ_p)$ is already even. This reduces the question whether $\psi_{R,\IQ_p}^0$ is the correct map to a question that can be checked on underlying sets. In particular, we can base change again to a situation where $A$ is already perfect, so that $R_\infty^\bullet/p$ is semiperfect on the nose. If we put $\A_\inf^\bullet\coloneqq\W((R_\infty^\bullet)^\flat)$, $J^\bullet\coloneqq \ker(\A_\inf^\bullet\rightarrow R_\infty^\bullet)$, and let $\A_\crys^\bullet$ denote the $p$-completed PD-envelope of $J^\bullet$, then
		\begin{equation*}
			\qdeRham_{R_\infty^\bullet/A}\bigl[\localise{p}\bigr]_{(q-1)}^\complete\simeq \deRham_{R_\infty^\bullet/A}\bigl[\localise{p}\bigr]\qpower\simeq\A_\crys^\bullet\bigl[\localise{p}\bigr]\qpower\,.
		\end{equation*}
		So to prove \cref{lem:psiRrational} in this particular case, we must check whether a certain map $\A_\crys^\bullet[1/p]\qpower\rightarrow \A_\crys^\bullet[1/p]_{\Hodge}^\complete\qpower$ agrees with the canonical Hodge completion map. As in the proof of \cref{lem:psiRmodq-1}, the Hodge completion map is uniquely determined by:
		\begin{alphanumerate}\itshape
			\item It is a map of $\A_\inf^\bullet\qpower$-modules.\label{enum:Ainfq-1ModuleMap}
			\item It is continuous with respect to the natural topologies on either side.\label{enum:Continuousq-1}
		\end{alphanumerate}
		Condition~\cref{enum:Continuousq-1} is again clear from our condensed setup, whereas~\cref{enum:Ainfq-1ModuleMap} follows by working over $\IS_{\A_\inf^\bullet}$ rather than $\IS_A$. This finishes the proof.
	\end{proof}
	
	\subsection{The case of quasi-regular quotients}\label{subsec:QuasiRegular}
	
	Let us continue to fix a prime~$p$ (with $p=2$ allowed). Let $A$ be a $\delta$-ring as in \cref{par:AssumptionsOnA} and suppose that $R$ is an $A$-algebra satisfying \cref{par:AssumptionsOnR}\cref{enum:E1Lift} for the identical cover $\id\colon R\rightarrow R$. In other words, $R$ is a $p$-quasi-lci $A$-algebra with a lift to a $p$-complete connective $\IE_1$-algebra $\IS_R\in\Alg_{\IE_1}(\Mod_{\IS_A}(\Sp))$ such that $R/p$ is relatively semiperfect over~$A$. 
	
	These assumptions ensure that $\qdeRham_{R/A}$ and $\deRham_{R/A}$ are static rings and that the Hodge filtration $\fil_{\Hodge}^\star\deRham_{R/A}$ is a descending filtration by ideals (see \cite[Lemma~\chref{4.18}({\chref[Item]{44}[$b$]})]{qWittHabiro}). As it turns out, the $q$-Hodge filtration from \cref{par:qHodgeFiltrationTC-ku} has a very explicit description in this case.
	
	\begin{thm}\label{thm:qdeRhamkuRegularQuotient}
		Under the assumptions above, the $q$-Hodge filtration $\fil_{\qHodge}^\star\qdeRham_{R/A}$ is the descending filtration by ideals given by the \embrace{$1$-categorical} preimage  of the combined Hodge- and $(q-1)$-adic filtration under the rationalisation map $\qdeRham_{R/A}\rightarrow\deRham_{R/A}[1/p]\qpower$. In other words, there's a pullback
		\begin{equation*}
			\begin{tikzcd}
				\fil_{\qHodge}^\star \qdeRham_{R/A}\rar\dar\drar[pullback] & \fil_{(\Hodge,q-1)}^\star \deRham_{R/A}\bigl[\localise{p}\bigr]\qpower\dar\\
				\qdeRham_{R/A}\rar &  \deRham_{R/A}\bigl[\localise{p}\bigr]\qpower
			\end{tikzcd}
		\end{equation*}
		in the $1$-category of filtered $(q-1)^\star A\qpower$-modules. In particular, $\fil_{\qHodge}^\star \qdeRham_{R/A}$ is independent of the choice of the spherical $\IE_1$-lift $\IS_R$, and canonically a filtered $\IE_\infty$-algebra over the filtered ring $(q-1)^\star A\qpower$.
	\end{thm}
	\begin{proof}
		That $\qdeRham_{R/A}$ is static and $\fil_{\qHodge}^\star \qdeRham_{R/A}$ is a descending filtration by subgroups follows from the corresponding assertions for $\deRham_{R/A}$ and $\fil_{\Hodge}^\star\deRham_{R/A}$, using $\qdeRham_{R/A}/(q-1)\simeq \deRham_{R/A}$ and $\fil_{\qHodge}^\star \qdeRham_{R/A}/\beta\simeq \fil_{\Hodge}^\star\deRham_{R/A}$ by \cref{thm:qdeRhamkupComplete,thm:qdeRhamku2Complete}.
		
		To show the description as a preimage, we first note that $\fil_{\qHodge}^\star\qdeRham_{R/A}$ is the preimage of its completion under $\qdeRham_{R/A}\rightarrow\qhatdeRham_{R/A}$ and likewise for $\fil_{(\Hodge,q-1)}^\star \deRham_{R/A}[1/p]\qpower$. Thus, it remains to show that the filtration on $\pi_0\TC_\solid^-(\ku_R/\ku_A)$ induced by the homotopy fixed point spectral sequence is the preimage of the analogous filtration on $\pi_0\TC_\solid(\ku_R\soltimes\IQ_p/\ku_A\soltimes \IQ_p)$ under the rationalisation map
		\begin{equation*}
			\pi_0\TC_\solid^-(\ku_R/\ku_A)\longrightarrow \pi_0\TC_\solid(\ku_R\soltimes\IQ_p/\ku_A\soltimes \IQ_p)\,.
		\end{equation*}
		As both filtrations are complete, it will be enough to show that the map on associated gradeds is injective. That is, we must show $\pi_{2*}\THH_\solid(\ku_R/\ku_A)\rightarrow \pi_{2*}\THH_\solid(\ku_R\soltimes\IQ_p/\ku_A\soltimes \IQ_p)$ is injective. This can be checked modulo~$\beta$, so we've reduced the problem to checking injectivity of $\pi_{2*}\HH_\solid(R/A)\rightarrow \pi_{2*}\HH_\solid(R\soltimes\IQ_p/A\soltimes \IQ_p)$. By the HKR theorem, we must show that
		\begin{equation*}
			\Sigma^{-n}\bigwedge^n \L_{R/A}\longrightarrow \Sigma^{-n}\bigwedge^n \L_{R/A}\soltimes\IQ_p
		\end{equation*}
		is injective for all $n$. Our assumptions guarantee that $\Sigma^{-1}\L_{R/A}$ is a $p$-completely flat module over the $p$-torsion free ring $R$ and so each $\Sigma^{-n}\bigwedge^n\L_{R/A}$ will be a $p$-torsion free $R$-module.
	\end{proof}

	\subsection{The global case}\label{subsec:qdeRhamkuGlobal}
	
	In this subsection we'll sketch a global analogue of the $p$-complete comparison between $\qdeRham_{R/A}$ and $\TC^-(\ku_R/\ku_A)$ from \cref{subsec:qdeRhamkupComplete}. So let us no longer fix a prime~$p$ and update our assumptions on $A$ and $R$ accordingly.

	\begin{numpar}[New assumptions on $A$ and $R$.]\label{par:NewAssumptions}
		From now on, $A$ and $R$ must satisfy the following:
		\begin{alphanumerate}\itshape
			\item[A] We assume that $A$ is a {perfectly covered $\Lambda$-ring}. That is, the Adams operations $\psi^m\colon A\rightarrow A$ are faithfully flat; equivalently, $A$ admits a faithfully flat $\Lambda$-map $A\rightarrow A_\infty$ into a perfect $\Lambda$-ring. Moreover, we assume that for all primes~$p$ the $p$-completion $\widehat{A}_p$ satisfies \cref{par:AssumptionsOnA}\cref{enum:CyclotomicLift}, with $\IS_{\smash{\widehat{A}}_p}$ denoting the $p$-complete spherical lift.\label{enum:AssumptionsOnAglobal}
			\item[R] We assume that $R$ is a {quasi-lci} $A$-algebra in the sense that the cotangent complex $\L_{R/A}$ has $\Tor$-amplitude in homogical degrees $[0,1]$ over~$R$. In addition, for every prime~$p$, the ring~$R$ must have bounded $p^\infty$-torsion and its $p$-completion $\widehat{R}_p$ must satisfy one of the conditions \cref{par:AssumptionsOnR}\cref{enum:E2Lift} or~\cref{enum:E1Lift} \embrace{but not necessarily the same for every~$p$}.
			We let $\IS_{\smash{\widehat{R}}_p}$ denote the $p$-complete spherical lift of $\widehat{R}_p$.\label{enum:AssumptionsOnRglobal}
		\end{alphanumerate}
		We note that the $p$-complete lifts $\IS_{\smash{\widehat{A}}_p}$ and $\IS_{\smash{\widehat{R}}_p}$ for all primes~$p$ can be glued with $A\otimes\IQ$ and $R\otimes\IQ$ to a connective $\IE_\infty$-ring spectrum $\IS_A$ and a connective $\IE_1$-algebra $\IS_R\in\Alg_{\IE_1}(\Mod_{\IS_A}(\Sp))$ satisfying
		\begin{equation*}
			\IS_A\otimes\IZ\simeq A\quad\text{and}\quad \IS_R\otimes\IZ\simeq R\,.
		\end{equation*}
		By construction, $\IS_A$ acquires the structure of a cyclotomic base. If \cref{par:AssumptionsOnR}\cref{enum:E2Lift} was chosen for every~$p$, then $\IS_R$ will be an $\IE_2$-algebra in $\IS_A$-modules. We also let $\ku_{\smash{\widehat{A}}_p}\coloneqq (\ku\otimes\IS_{\smash{\widehat{A}}_p})_p^\complete$ and $\ku_A\coloneqq \ku\otimes\IS_A$ and define $\ku_{\smash{\widehat{R}}_p}$ and $\ku_R$ analogously.
	\end{numpar}
	\begin{rem}
		Despite the restrictive hypotheses, there are many examples of such $A$ and $R$, as we'll see in \cref{subsec:CyclotomicBases}.
	\end{rem}
	To carry out our global constructions, we'll proceed by gluing the $p$-complete constructions from \cref{subsec:qdeRhamkupComplete} with the rational case. For the gluing we'll the following notion:
	
	\begin{numpar}[Profinite completion.]
		A spectrum $X$ is called \emph{profinite complete} if the canonical map 
		\begin{equation*}
			X\rightarrow \limit_{m\in\IN}X/m\simeq \prod_p\widehat{X}_p
		\end{equation*}
		is an equivalence. The spectrum on the right-hand side will be called the \emph{profinite completion of $X$} and denoted $\widehat{X}$.
		
		Analogous notions can be defined for solid condensed spectra. The the solid tensor product of two bounded below profinite complete spectra will be profinite complete again. For a proof, see \cite[Lemma~\chref{B.8}]{qWittHabiro} and replace each $q$-factorial $(q;q)_n$ by an honest factorial $n!$.
	\end{numpar}
	
	\begin{numpar}[Profinite even filtrations.]\label{par:ProfiniteEvenFiltration}
		Let $\widehat{A}$ and $\widehat{R}$ denote the profinite completions of $A$ and $R$. Let $k$ be any connective even $\IE_\infty$-ring spectrum such that $\pi_*(k)$ is $p$-torsion free for all primes~$p$ (the most relevant case is of course $k=\ku$, but we'll also need $k=\ku\otimes\IQ$ and later $k=\ku^{\Phi C_m}$). Let $k_{\smash{\widehat{A}}}\coloneqq k\soltimes\prod_p\IS_{\smash{\widehat{A}}_p}$ and $k_{\smash{\widehat{R}}}\coloneqq k\soltimes\prod_p\IS_{\smash{\widehat{R}}_p}$. We wish to construct an appropriate even filtration
		\begin{equation*}
			\fil_{\ev}^\star\THH_\solid\bigl(k_{\smash{\widehat{R}}}/k_{\smash{\widehat{A}}}\bigr)\,.
		\end{equation*}
		Once we have this, we can also construct versions for $\TC_\solid^-$ and $\TP_\solid$ via
		\begin{align*}
			\fil_{\ev,\h S^1}^\star\TC_\solid^-\bigl(k_{\smash{\widehat{R}}}/k_{\smash{\widehat{A}}}\bigr)&\coloneqq \bigl(\fil_{\ev}^\star\THH_\solid(k_{\smash{\widehat{R}}}/k_{\smash{\widehat{A}}})\bigr)^{\h \IT_{\ev}}\,,\\
			\fil_{\ev,\t S^1}^\star\TP_\solid\bigl(k_{\smash{\widehat{R}}}/k_{\smash{\widehat{A}}}\bigr)&\coloneqq \bigl(\fil_{\ev}^\star\THH_\solid(k_{\smash{\widehat{R}}}/k_{\smash{\widehat{A}}})\bigr)^{\t \IT_{\ev}}
		\end{align*}
		Before we discuss the construction in general, let us start with two special cases:
		\begin{alphanumerate}
			\item[\IE_1] If we chose condition~\cref{par:AssumptionsOnR}\cref{enum:E1Lift} for all primes~$p$, and $\IS_{\smash{\widehat{R}}_p}\rightarrow \IS_{\smash{\widehat{R}}_{p,\infty}^\bullet}$ are the given cosimplicial resolutions, we put $k_{\smash{\widehat{R}}_\infty^\bullet}\coloneqq k\soltimes\prod_p\IS_{\smash{\widehat{R}}_{p,\infty}^\bullet}$ and define our filtration via\label{enum:ProfiniteEvenFiltrationE1}
			\begin{equation*}
				\fil_{\ev}^\star\THH_\solid\bigl(k_{\smash{\widehat{R}}}/k_{\smash{\widehat{A}}}\bigr)\coloneqq \limit_{\IDelta}\tau_{\geqslant 2\star}\THH_\solid\bigl(k_{\smash{\widehat{R}}_\infty^\bullet}/k_{\smash{\widehat{A}}}\bigr)\,.
			\end{equation*}
			\item[\IE_2] If instead \cref{par:AssumptionsOnR}\cref{enum:E2Lift} was chosen for all primes~$p$, so that $k_{\smash{\widehat{R}}}$ is an $\IE_2$-algebra in $k_{\smash{\widehat{A}}}$-modules, we simply define $\fil_{\ev}^\star\THH_\solid(k_{\smash{\widehat{R}}}/k_{\smash{\widehat{A}}})$ to be the solid even filtration of $\THH_\solid(k_{\smash{\widehat{R}}}/k_{\smash{\widehat{A}}})$ as a left module over itself.\label{enum:ProfiniteEvenFiltrationE2}
		\end{alphanumerate}
		In general, let $P_1$ and $P_2$ be the set of primes where we choose \cref{par:AssumptionsOnR}\cref{enum:E1Lift} and \cref{par:AssumptionsOnR}\cref{enum:E2Lift}, respectively. Let $k_{\smash{\widehat{R}},\IE_1}\coloneqq \prod_{p\in P_1}k_{\smash{\widehat{R}}_p}$ and $k_{\smash{\widehat{R}},\IE_2}\coloneqq \prod_{p\in P_2}k_{\smash{\widehat{R}}_p}$. Then
		\begin{equation*}
			\THH_\solid\bigl(k_{\smash{\widehat{R}}}/k_{\smash{\widehat{A}}}\bigr)\simeq \THH_\solid\bigl(k_{\smash{\widehat{R}},\IE_1}/k_{\smash{\widehat{A}}}\bigr)\times \THH_\solid\bigl(k_{\smash{\widehat{R}},\IE_2}/k_{\smash{\widehat{A}}}\bigr)
		\end{equation*}
		and we can apply the constructions from \cref{enum:ProfiniteEvenFiltrationE1} and \cref{enum:ProfiniteEvenFiltrationE2} to the two factors separately.
		
		The results from \crefrange{subsec:EvenResolution}{subsec:ComparisonOfEvenFiltrations} can all be adapted to the profinite case in a straightforward way and the proofs can be copied verbatim. For example, in case~\cref{enum:ProfiniteEvenFiltrationE2}, let $P\coloneqq \IZ[x_i\ |\ i\in I]$ be a polynomial ring with a surjection $P\twoheadrightarrow R$ and let $\widehat{P}$ be its profinite completion. Let $\IS_P\coloneqq \IS[x_i\ |\ i\in I]$ and let $\IS_{\smash{\widehat{P}}}$ be its profinite completion. Finally, let $\IS\rightarrow \IS_{\smash{\widehat{P}}^\bullet}$ denote the profinitely completed \v Cech nerve of $\IS\rightarrow \IS_{\smash{\widehat{P}}}$. Then 
		\begin{equation*}
			\fil_{\ev}^\star\THH_\solid\bigl(k_{\smash{\widehat{R}}}/k_{\smash{\widehat{A}}}\bigr)\overset{\simeq}{\longrightarrow}\limit_{\IDelta}\tau_{\geqslant 2\star}\THH_\solid\bigl(k_{\smash{\widehat{R}}}/k_{\smash{\widehat{A}}}\soltimes\IS_{\smash{\widehat{P}}^\bullet}\bigr)\,.
		\end{equation*}
		To show this, we can simply copy the proof of \cref{prop:EvenResolution}. The key points are that $\THH_\solid(\IS_{\smash{\widehat{P}}})\rightarrow\IS_{\smash{\widehat{P}}}$ is still solid faithfully even flat, which can be shown by the same argument as in \cref{lem:eff}, and that $\HH_\solid(\widehat{R}/\widehat{A}\soltimes_\IZ\widehat{P}^\bullet)$ is still even.
	\end{numpar}
	
	\begin{lem}
		For $k=\ku$, we have canonical equivalences
		\begin{align*}
			\fil_{\ev}^\star\THH_\solid\bigl(\ku_{\smash{\widehat{R}}}/\ku_{\smash{\widehat{A}}}\bigr)&\overset{\simeq}{\longrightarrow} \prod_p\fil_{\ev}^\star \THH_\solid\bigl(\ku_{\smash{\widehat{R}}_p}/\ku_{\smash{\widehat{A}}_p}\bigr)\,,\\
			\fil_{\ev,\h S^1}^\star\TC_\solid^-\bigl(\ku_{\smash{\widehat{R}}}/\ku_{\smash{\widehat{A}}}\bigr)&\overset{\simeq}{\longrightarrow} \prod_p\fil_{\ev,\h S^1}^\star \TC_\solid^-\bigl(\ku_{\smash{\widehat{R}}_p}/\ku_{\smash{\widehat{A}}_p}\bigr)\,,
			\\\fil_{\ev,\t S^1}^\star\TP_\solid\bigl(\ku_{\smash{\widehat{R}}}/\ku_{\smash{\widehat{A}}}\bigr)&\overset{\simeq}{\longrightarrow} \prod_p\fil_{\ev,\t S^1}^\star \TP_\solid\bigl(\ku_{\smash{\widehat{R}}_p}/\ku_{\smash{\widehat{A}}_p}\bigr)\,.
		\end{align*}
	\end{lem}
	\begin{proof}
		Let us first show the assertion for $\fil_{\ev}^\star\THH_\solid$. Note that $\ku_{\smash{\widehat{A}}}\simeq \prod_p\ku_{\smash{\widehat{A}}_p}\simeq (\ku_A)^\complete$ is the profinite completion of $\ku_A$ and likewise for $\ku_{\smash{\widehat{R}}}$. Using \cref{lem:SolidTHH} and its profinite analogue, we see
		\begin{equation*}
			\THH_\solid\bigl(\ku_{\smash{\widehat{R}}}/\ku_{\smash{\widehat{A}}}\bigr)\simeq \THH(\ku_R/\ku_A)^\complete\simeq \prod_p \THH_\solid\bigl(\ku_{\smash{\widehat{R}}_p}/\ku_{\smash{\widehat{A}}_p}\bigr)\,.
		\end{equation*}
		Applying the same observation to the cosimplicial resolutions $\THH_\solid(\ku_{\smash{\widehat{R}}_\infty^\bullet}/\ku_{\smash{\widehat{A}}})$ (in the special case~\cref{par:ProfiniteEvenFiltration}\cref{enum:ProfiniteEvenFiltrationE1}) or $\THH_\solid(\ku_{\smash{\widehat{R}}}/\ku_{\smash{\widehat{A}}}\soltimes\IS_{\smash{\widehat{P}}^\bullet})$ (in the special case~\cref{par:ProfiniteEvenFiltration}\cref{enum:ProfiniteEvenFiltrationE1}) or a mixture thereof (in the general case), we get the desired equivalence for $\fil_{\ev}^\star\THH_\solid$.
		
		The equivalence for $\fil_{\ev,\h S^1}^\star\TC_\solid^-$ immediately follows. For $\fil_{\ev,\t S^1}^\star\TP_\solid$, we must explain why $(-)_{\h \IT_{\ev}}\simeq \IS_{\ev}\soltimes_{\IT_{\ev}}-$ commutes with the infinite product $\prod_p$. By arguing as in the proof of \cref{cor:TC-TPEvenResolution} (or just reduction modulo~$\beta$), we can reduce this to showing that $(-)_{\h S^1}$ commutes with the infinite product in $\prod_p\fil_{\mathrm{HKR}}^\star \HH_\solid(\widehat{R}_p/\widehat{A}_p)$. Since the HKR filtration increases in connectivity, it's enough to show the same for each graded piece $\prod_p\gr_{\mathrm{HKR}}^n \HH_\solid(\widehat{R}_p/\widehat{A}_p)$. Since $R$ was assumed to be quasi-lci over $A$, each graded piece is concentrated in a finite range of degrees. Thus, in any given homotopical degree, only finitely many cells of $\IC\IP^\infty\simeq \B S^1$ will contribute to $(-)_{\h S^1}$, so it commutes with the infinite product.
	\end{proof}
	
	Finally, we can put everything together.
	\begin{numpar}[Global even filtrations.]\label{par:GlobalEvenFiltration}
		Since $\ku_A$ and $\ku_R$ are discrete, $\THH_\solid$ agrees with the usual $\THH$. We can thus equip $\THH(\ku_R\otimes\IQ/\ku_A\otimes\IQ)$ with the solid even filtration, which agrees with Pstr\k{a}gowski's perfect even filtration by \cref{cor:SolidvsDiscreteEvenFiltration}, and with the Hahn--Raksit--Wilson filtration by \cite[Theorem~\chref{7.5}]{PerfectEvenFiltration} and our assumption that $R$ is quasi-lci over $A$. We can now define an even filtration on $\THH(\ku_R/\ku_A)$ via the pullback diagram
		\begin{equation*}
			\begin{tikzcd}
				\fil_{\ev}^\star \THH(\ku_R/\ku_A)\rar\dar\drar[pullback] & \fil_{\ev}^\star \THH_\solid\bigl(\ku_{\smash{\widehat{R}}}/\ku_{\smash{\widehat{A}}}\bigr)\dar\\
				\fil_{\ev}^\star \THH(\ku_R\otimes\IQ/\ku_A\otimes\IQ)\rar & \fil_{\ev}^\star \THH_\solid\bigl(\ku_{\smash{\widehat{R}}}\soltimes\IQ/\ku_{\smash{\widehat{A}}}\soltimes\IQ\bigr)
			\end{tikzcd}
		\end{equation*}
		where the right vertical map is given by~\cref{par:ProfiniteEvenFiltration} applied to $k=\ku$ and $k=\ku\otimes\IQ$.
		
		We must explain where the bottom horizontal map comes from. It's straightforward to check that $\fil_{\ev}^\star\THH(\ku_R\otimes\IQ/\ku_A\otimes\IQ)\simeq \fil_{\ev}^\star\HH(R/A)\otimes\IQ[\beta]_\mathrm{ev}$. Moreover, since the base change result from \cref{cor:EvenFiltrationBaseChange} is still true in the profinite situation (see the discussion in \cref{par:ProfiniteEvenFiltration}), we can use base change for $\IZ\rightarrow \IQ[\beta]\simeq \ku\otimes\IQ$ to get
		\begin{equation*}
			\fil_{\ev}^\star \THH_\solid\bigl(\ku_{\smash{\widehat{R}}}\soltimes\IQ/\ku_{\smash{\widehat{A}}}\soltimes\IQ\bigr)\simeq \fil_{\ev}^\star \HH_\solid(\widehat{R}/\widehat{A})\soltimes\IQ[\beta]_{\ev}\,.
		\end{equation*}
		Moreover, the profinite analogue of \cref{cor:HRWvsPstragowski} shows that $\fil_{\ev}^\star\HH_\solid(\widehat{R}/\widehat{A})$ agrees with $\prod_p\fil_{\HRWev}^\star\HH(R/A)_p^\complete$. We then have a canonical map $\fil_{\ev}^\star\HH(R/A)\rightarrow \fil_{\ev}^\star\HH_\solid(\widehat{R}/\widehat{A})$, which provides  us with the desired bottom horizontal map in the diagram above.
		
		Once we have constructed $\fil_{\ev}^\star\THH(\ku_R/\ku_A)$, we can also construct filtrations on $\TC^-$ and $\TP$ in the usual manner:
		\begin{align*}
			\fil_{\ev,\h S^1}^\star\TC^-(\ku_R/\ku_A)&\coloneqq \bigl(\fil_{\ev}^\star\THH(\ku_R/\ku_A)\bigr)^{\h \IT_{\ev}}\,,\\
			\fil_{\ev,\t S^1}^\star\TP(\ku_R/\ku_A)&\coloneqq \bigl(\fil_{\ev}^\star\THH(\ku_R/\ku_A)\bigr)^{\t\IT_{\ev}}\,.
		\end{align*}
	\end{numpar}
	Here's a sanity check:
	\begin{lem}\label{lem:GlobalEvenFiltration}
		Suppose we chose condition~\cref{par:AssumptionsOnR}\cref{enum:E2Lift} for all primes~$p$, so that $\ku_R$ is an $\IE_2$-algebra in $\ku_A$-modules. Then $\fil_{\ev}^\star\THH(\ku_R/\ku_A)$ agrees with the solid perfect even filtration on the solid $\IE_1$-ring $\THH_\solid(\ku_R/\ku_A)$, and also with Pstr\k{a}gowski's perfect even filtration $\fil_{\Pev}^\star\THH(\ku_R/\ku_A)$.
	\end{lem}
	\begin{proof}[Proof sketch]
		The solid even filtration agrees with Pstr\k{a}gowski's construction by \cref{cor:SolidvsDiscreteEvenFiltration}. To show that both agree with the pullback $\fil_{\ev}^\star\THH(\ku_R/\ku_A)$ from \cref{par:GlobalEvenFiltration}, we verify that all even filtrations in sight can be computed by cosimplicial resolutions as in \cref{prop:EvenResolution}. To show this, the proof of said proposition can be adapted in a straightforward way. The key points are that $\THH_\solid(\IS_P)\rightarrow \IS_P$ is still solid faithfully even flat by \cref{lem:eff} and that $\HH(R/A\otimes_\IZ P^\bullet)$ is still even.
	\end{proof}
	We're now ready to construct the global comparison with $q$-de Rham cohomology. Due to the problems at $p=2$ that we've discussed at the end of \cref{subsec:qdeRhamkupComplete}, we need a small addendum to the assumptions from \cref{par:NewAssumptions}\cref{enum:AssumptionsOnRglobal}.
	
	\begin{unnumpar}[\cref*{par:NewAssumptions}$\boldsymbol{a}$. New assumptions on~$A$ and~$R$.]
		From now on we'll assume that $R$ satisfies not only \cref{par:NewAssumptions}\cref{enum:AssumptionsOnRglobal} but also:
		\begin{alphanumerate}\itshape 
			\item[R_2] The $2$-adic completion $\widehat{R}_2$ satisfies \cref{par:AssumptionsOnR}\cref{enum:E1Lift}.\label{enum:AssumptionOnR2}
		\end{alphanumerate}
		We note that this is true, in particular, if $2$ is invertible in $R$.
	\end{unnumpar}

	\begin{numpar}[The global comparison map.]\label{par:kuComparisonGlobal}
		Let us denote $\qdeRham_{\smash{\widehat{R}}/\smash{\widehat{A}}}\coloneqq \prod_p\qdeRham_{\smash{\widehat{R}}_p/\smash{\widehat{A}}_p}$ and $\deRham_{\smash{\widehat{R}}/\smash{\widehat{A}}}\coloneqq \prod_p\deRham_{\smash{\widehat{R}}_p/\smash{\widehat{A}}_p}$ for short. Then the global $q$-de Rham complex sits inside a pullback
		\begin{equation*}
			\begin{tikzcd}
				\qdeRham_{R/A}\rar\dar\drar[pullback] & \qdeRham_{\smash{\widehat{R}}/\smash{\widehat{A}}}\dar\\
				\bigl(\deRham_{R/A}\lotimes_{\IZ}\IQ\bigr)\qpower\rar & \bigl(\deRham_{\smash{\widehat{R}}/\smash{\widehat{A}}}\lotimes_\IZ\IQ\bigr)\qpower
			\end{tikzcd}
		\end{equation*}
		(see \cite[Construction~\chref{A.14}]{qWittHabiro}). We claim that this diagram maps canonically to the pullback
		\begin{equation*}
			\begin{tikzcd}
				\gr_{\ev,\h S^1}^0\TC^-(\ku_R/\ku_A)\rar\dar\drar[pullback] & \gr_{\ev,\h S^1}^0 \TC_\solid^-\bigl(\ku_{\smash{\widehat{R}}}/\ku_{\smash{\widehat{A}}}\bigr)\dar\\
				\gr_{\ev,\h S^1}^0\TC^-(\ku_R\otimes\IQ/\ku_A\otimes\IQ)\rar & \gr_{\ev,\h S^1}^0 \TC_\solid^-\bigl(\ku_{\smash{\widehat{R}}}\soltimes\IQ/\ku_{\smash{\widehat{A}}}\soltimes\IQ\bigr)
			\end{tikzcd}
		\end{equation*}
		coming from \cref{par:GlobalEvenFiltration}. To construct this map of pullback squares, we need:
		\begin{alphanumerate}
			\item A map  $\qdeRham_{\smash{\widehat{R}}/\smash{\widehat{A}}}\rightarrow \gr_{\ev,\h S^1}^0\TC_\solid^-(\ku_{\smash{\widehat{R}}}/\ku_{\smash{\widehat{A}}})$. This we get by taking the product of the maps $\psi_{\widehat{R}_p}^0$ from \cref{par:kuComparisonII} for all primes~$p$.\label{enum:ComparisonProfinite}
			\item A map $(\deRham_{R/A}\otimes_\IZ\IQ)\qpower\rightarrow \gr_{\ev}^0\TC^-(\ku_R\otimes\IQ/\ku_A\otimes\IQ)$. Since $\ku_A\otimes\IQ\simeq A\otimes\IQ[\beta]$ and $\ku_R\otimes\IQ\simeq R\otimes\IQ[\beta]$, we get\label{enum:ComparisonRational}
			\begin{equation*}
				\TC^-(\ku_R\otimes\IQ/\ku_A\otimes\IQ)\simeq \HC^-\bigl(R\otimes\IQ[\beta]/A\otimes\IQ[\beta]\bigr)\,.
			\end{equation*}
			A standard computation identifies $\gr_{\ev}^0$ with the Hodge completion $(\deRham_{R/A}\otimes\IQ)_{\Hodge}^\complete\qpower$, so we can choose our desired map to be the Hodge completion map.
			\item A map $(\deRham_{\smash{\widehat{R}}/\smash{\widehat{A}}}\otimes_\IZ\IQ)\qpower\rightarrow \gr_{\ev,\h S^1}^0\TC_\solid^-(\ku_{\smash{\widehat{R}}}\soltimes\IQ/\ku_{\smash{\widehat{A}}}\soltimes\IQ)$. This works as in \cref{enum:ComparisonRational} above.\label{enum:ComparisonProfiniteRational}
		\end{alphanumerate}
		Clearly~\cref{enum:ComparisonRational} and~\cref{enum:ComparisonProfiniteRational} are compatible; compatibility of~\cref{enum:ComparisonProfinite} and~\cref{enum:ComparisonProfiniteRational} will be checked in \cref{lem:ProfiniteRationalCompatibilityCheck} below. So we get our map of pullback squares and thus a map
		\begin{equation*}
			\psi_R^0\colon \qdeRham_{R/A}\longrightarrow \gr_{\ev,\h S^1}^0\TC^-(\ku_R/\ku_A)\,.
		\end{equation*}
	\end{numpar}
	\begin{numpar}[The global $q$-Hodge filtration.]\label{par:GlobalqHodgeFiltration}
		As in the $p$-complete case \cref{par:qHodgeFiltrationTC-ku}, we identify $\Sigma^{-2*}\gr_{\ev,\h S^1}^*(\ku^{\h S^1})\simeq\IZ[\beta]\llbracket t\rrbracket$ with the filtered ring $(q-1)^\star \IZ\qpower$, where $t$ is the filtration parameter and $\beta$ corresponds to $(q-1)$ in filtration degree~$1$. We then define the \emph{$q$-Hodge filtration} as the pullback
		\begin{equation*}
			\begin{tikzcd}
				\fil_{\qHodge}^\star\qdeRham_{R/A}\rar\dar\drar[pullback] & \Sigma^{-2*}\gr_{\ev,\h S^1}^*\TC^-(\ku_R/\ku_A)\dar\\
				\qdeRham_{R/A}\rar["\psi_R^0"] & \gr_{\ev,\h S^1}^0\TC^-(\ku_R/\ku_A)
			\end{tikzcd}
		\end{equation*}
		As the name suggests, $\fil_{\qHodge}^\star\qdeRham_{R/A}$ is indeed a $q$-Hodge filtration in the sense of \cite[Definition~\chref{3.2}]{qWittHabiro}.
	\end{numpar}
	\begin{thm}\label{thm:qdeRhamkuGlobal}
		Suppose $A$ and $R$ satisfy the assumptions from~\cref{par:NewAssumptions} along with the addendum~\cref{enum:AssumptionOnR2}. Then the map $\psi_R^0$ from \cref{par:kuComparisonGlobal} induces an equivalence of graded $\IZ[\beta]\llbracket t\rrbracket$-modules
		\begin{equation*}
			\fil_{\qHodge}^\star \qhatdeRham_{R/A}\overset{\simeq}{\longrightarrow}\Sigma^{-2*}\gr_{\ev,\h S^1}^*\TC^-(\ku_R/\ku_A)\,,
		\end{equation*}
		where the left-hand side denotes the completion of the $q$-Hodge filtration $\fil_{\qHodge}^*\qdeRham_{R/A}$ from \cref{par:GlobalqHodgeFiltration}. Moreover, modulo $\beta$ and after rationalisation, we get equivalences
		\begin{align*}
			\fil_{\qHodge}^\star\qdeRham_{R/A}\lotimes_{\IZ[\beta]\tpower}\IZ\tpower&\overset{\simeq}{\longrightarrow} \fil_{\Hodge}^\star \deRham_{R/A}\,,\\
			\fil_{\qHodge}^\star\bigl(\qdeRham_{R/A}\lotimes_\IZ\IQ\bigr)_{(q-1)}^\complete&\overset{\simeq}{\longrightarrow}\fil_{(\Hodge,q-1)}^\star\bigl(\deRham_{R/A}\lotimes_\IZ\IQ\bigr)\qpower
		\end{align*}
		with the usual Hodge filtration and the combined Hodge and $(q-1)$-adic filtration, respectively. Via these equivalences, $(R,\fil_{\qHodge}^\star\qdeRham_{R/A})$ becomes canonically an object in the $\infty$-category $\cat{AniAlg}_A^{\qHodge}$ from \cite[Definition~\chref{3.2}]{qWittHabiro}.
	\end{thm}
	\begin{rem}
		Fix $2\leqslant n\leqslant \infty$. If for every prime~$p$ either \cref{par:AssumptionsOnR}\cref{enum:E1Lift} was chosen or $\IS_{\smash{\widehat{R}}_p}$ admits an $\IE_n$-algebra structure in $\IS_{\smash{\widehat{A}}_p}$-modules, then all equivalences in \cref{thm:qdeRhamkuGlobal} are canonically $\IE_{n-1}$-monoidal. Indeed, for those primes where $\IS_{\smash{\widehat{R}}_p}$ is $\IE_n$, we get $\IE_{n-1}$-monoidality by carefully tracing through all constructions. For the other primes use \cref{thm:qdeRhamkuRegularQuotient}. It follows that the pair $(R,\fil_{\qHodge}^\star\qdeRham_{R/A})$ is canonically an $\IE_{n-1}$-algebra in $\cat{AniAlg}_A^{\qHodge}$ (compare \cite[\chref{3.50}]{qWittHabiro}).
	\end{rem}
	
	\begin{lem}\label{lem:ProfiniteRationalCompatibilityCheck}
		The maps from \cref{par:kuComparisonGlobal}\cref{enum:ComparisonProfinite} and~\cref{enum:ComparisonProfiniteRational} fit into a commutative diagram
		\begin{equation*}
			\begin{tikzcd}
				\bigl(\qdeRham_{\smash{\widehat{R}}/\smash{\widehat{A}}}\lotimes_\IZ\IQ\bigr)_{(q-1)}^\complete\rar["\cref{par:kuComparisonGlobal}\cref{enum:ComparisonProfinite}"]\dar["\simeq"'] & \Bigl(\gr_{\ev,\h S^1}^0\TC_\solid^-(\ku_{\smash{\widehat{R}}}/\ku_{\smash{\widehat{A}}})\soltimes\IQ\Bigr)_{(q-1)}^\complete\dar["\simeq"]\\
				\bigr(\deRham_{\smash{\widehat{R}}/\smash{\widehat{A}}}\lotimes_\IZ\IQ\bigr)\qpower\rar["\cref{par:kuComparisonGlobal}\cref{enum:ComparisonProfiniteRational}"] & \gr_{\ev,\h S^1}^0\HC_\solid^-\Bigl(\widehat{R}\soltimes\IQ[\beta]/\widehat{A}\soltimes\IQ[\beta]\Bigr)
			\end{tikzcd}
		\end{equation*}
		where the left vertical arrow is the usual equivalence for rationalised $q$-de Rham cohomology and the right vertical arrow is obtained as explained in \cref{par:GlobalEvenFiltration}.
	\end{lem}
	\begin{proof}
		In the following, we'll assume that $2$ is invertible in $R$. To treat the general case, we can just split off the factor $p=2$ from $\bigl(\prod_p\qdeRham_{\smash{\widehat{R}}_p/\smash{\widehat{A}}_p}\lotimes_\IZ\IQ\bigr)_{(q-1)}^\complete$ and use \cref{lem:psiRrational}.%
		\footnote{Recall that \cref{lem:psiRrational} still works for $p=2$ as long as \cref{par:AssumptionsOnR}\cref{enum:E1Lift} was chosen; see the argument in the proof of \cref{thm:qdeRhamku2Complete}.}

		We'll use an adaptation of the argument from the proof of \cref{lem:psiRrational}. Observe that all maps in question are equivariant with respect to the Adams action of $\widehat{\IZ}^\times\coloneqq \prod_p\IZ_p^\times$, so the problem boils down to checking that a certain $\widehat{\IZ}^\times$-equivariant map
		\begin{equation*}
			\bigr(\deRham_{\smash{\widehat{R}}/\smash{\widehat{A}}}\lotimes_\IZ\IQ\bigr)\qpower\longrightarrow \bigr(\deRham_{\smash{\widehat{R}}/\smash{\widehat{A}}}\lotimes_\IZ\IQ\bigr)_{\Hodge}^\complete\qpower
		\end{equation*}
		is the canonical Hodge completion map.
		
		To see this, consider the element $\psi\coloneqq (\zeta_{p-1}(1+p))_p\in\prod_p\IZ_p^\times$, where $\zeta_{p-1}\in\IZ_p^\times$ denotes any primitive $(p-1)$\textsuperscript{st} root of unity. We claim that for any $M\in\Dd(\IZ)$, equipped with the trivial action of $\widehat{\IZ}^\times$, one has a functorial equivalence
		\begin{equation*}
			\bigl(\widehat{M}\lotimes_\IZ\IQ\bigr)\qpower^{\psi=1}\simeq \bigl(\widehat{M}\lotimes_\IZ\IQ\bigr)\oplus\Sigma^{-1}\bigl(\widehat{M}\lotimes_\IZ\IQ\bigr)
		\end{equation*}
		To show the claim, it'll be enough to show $\H_{-1}(\widehat{\IZ}\qpower^{\psi=1}/(q-1)^n)\simeq \widehat{\IZ}\oplus(\text{torsion group})$ for every~$n$. This $\H_{-1}$ agrees with $\pi_{-1}$ of the spectrum
		\begin{equation*}
			\prod_p\bigl((\ku_p^\complete)^{\B S^1}\bigr)^{\psi=1}/t^n\simeq \prod_p \bigl((\ku_p^\complete)^{\psi=1}\bigr)^{\IC\IP^n}
		\end{equation*}
		The homotopy groups of $(\ku_p^\complete)^{\psi=1}$ are $\IZ_p$ in degrees $\{-1,0\}$ and torsion groups in degrees $\geqslant 2p-3$. Since $\IC\IP^n$ has a finite even cell decomposition, the torsion groups in positive degrees will only contribute to $\pi_{-1}\bigl(\prod_p((\ku_p^\complete)^{\psi=1})^{\IC\IP^n}\bigr)$ for finitely many primes, and so the result will indeed be of the form $\widehat{\IZ}\oplus (\text{torsion group})$. This proves the claim.
		
		To deduce that our map above must be the canonical Hodge completion, we apply $(-)^{\psi=1}\lotimes_{\IZ^{\psi=1}}\IZ$ to get a map $\deRham_{\smash{\widehat{R}}/\smash{\widehat{A}}}\lotimes_\IZ\IQ\rightarrow(\deRham_{\smash{\widehat{R}}/\smash{\widehat{A}}}\lotimes_\IZ\IQ)_{\Hodge}^\complete$. By comparison with the reduction modulo $(q-1)$ and \cref{lem:psiRmodq-1} (applied for all primes~$p$), we know that this map must be the canonical Hodge completion. By applying $(-\lotimes_{\IQ}\IQ\qpower)_{(q-1)}^\complete$ to this map, we deduce that our original map must be the Hodge completion as well.
	\end{proof}
	\begin{proof}[Proof sketch of \cref{thm:qdeRhamkuGlobal}]
		Using \cref{cor:HRWvsPstragowski}, we see that the base change of our even filtration $\fil_{\ev,\h S^1}^\star\TC^-(\ku_R/\ku_A)$ along $\ku_{\ev}^{\h S^1}\rightarrow \IZ_{\ev}^{\h S^1}$ is the Hahn--Raksit--Wilson even filtration on $\HC^-(R/A)$. Moreover, it's clear from the construction in \cref{par:kuComparisonGlobal} and \cref{lem:psiRmodq-1} that the induced map
		\begin{equation*}
			\ov\psi_{R}^0\colon \deRham_{R/A}\longrightarrow \hatdeRham_{R/A}\simeq \gr_{\HRWev,\h S^1}^0\HC^-(R/A)
		\end{equation*}
		is the canonical Hodge completion map. Similarly, by the construction in \cref{par:kuComparisonGlobal}\cref{enum:ComparisonRational}, the rationalisation
		\begin{equation*}
			\psi_{R,\IQ}^0\colon \bigl(\deRham_{R/A}\lotimes_\IZ\IQ\bigr)\qpower\longrightarrow \gr_{\ev,\h S^1}^0\TC^-(\ku_R\otimes\IQ/\ku_A\otimes\IQ)
		\end{equation*}
		gets identified with the canonical Hodge completion map. With these two observations, the proof of \cref{thm:qdeRhamkupComplete} can be copied verbatim to show everything but the last claim.
		
		It remains to give $(R,\fil_{\qHodge}^\star\qdeRham_{R/A})$ the structure of an object in $\cat{AniAlg}_A^{\qHodge}$. The equivalences from conditions~(\href{https://guests.mpim-bonn.mpg.de/ferdinand/q-Habiro.pdf#Item.24}{$b$}) and~(\href{https://guests.mpim-bonn.mpg.de/ferdinand/q-Habiro.pdf#Item.25}{$c$}) of \cite[Definition~\chref{3.2}]{qWittHabiro} have already been constructed; the compatibility between them follows directly by comparing the even filtrations on $\TC^-(\ku_R\otimes\IQ/\ku_A\otimes\IQ)\simeq \HC^-(R\otimes\IQ[\beta]/A\otimes\IQ[\beta])$ and $\HC^-(R\otimes\IQ/A\otimes\IQ)$. For condition~(\href{https://guests.mpim-bonn.mpg.de/ferdinand/q-Habiro.pdf#section*.4}{$c_p$}) of \cite[Definition~\chref{3.2}]{qWittHabiro}, we use \cref{thm:qdeRhamkupComplete}; the compatibilities come for free via the adelic gluing constructions in~\cref{par:GlobalEvenFiltration} and~\cref{par:kuComparisonGlobal}.
	\end{proof}

		%
		%
		%

	\newpage
	
	\section{Habiro descent via genuine equivariant homotopy theory}\label{sec:Genuine}
	
	We've seen in \cref{thm:qdeRhamkuGlobal} that the even filtration on $\TC^-(\ku_R/\ku_A)$ gives rise to a $q$-Hodge filtration $\fil_{\qHodge}^\star\qdeRham_{R/A}$ in the sense of \cite[Definition~\chref{3.2}]{qWittHabiro}. In particular, this provides many examples to which \cite[Theorem~\chref{3.11}]{qWittHabiro} can be applied.
	
	The goal of this section is to show that, in the situation at hand, the Habiro descent from that theorem can also be obtained homotopically. As a straightforward corollary of \cref{thm:qdeRhamkuGlobal}, one checks that the $q$-Hodge complex associated to $\fil_{\qHodge}^\star\qdeRham_{R/A}$ agrees with
	\begin{equation*}
		\qHodge_{R/A}\simeq \gr_{\ev,\h S^1}^0\TC^-(\KU_R/\KU_A)\,,
	\end{equation*} 
	where we put $\KU_A\coloneqq \KU\otimes \IS_A$ and $\KU_R\coloneqq \KU\otimes\IS_R$. To get the Habiro descent, we'll show that for every $m\in\IN$ the action of the cyclic subgroup $C_m\subseteq S^1$ on $\THH(\KU_R/\KU_A)$ can be made \emph{genuine}. We'll then construct an even filtration on $(\THH(\KU_R/\KU_A)^{C_m})^{\h (S^1/C_m)}$. The Habiro--Hodge complex $\qHhodge_{R/A}$ will finally be recovered as the $0$\textsuperscript{th} graded piece
	\begin{equation*}
		\qHhodge_{R/A}\simeq \limit_{m\in\IN}\gr_{\ev,S^1}^0\bigl(\THH(\KU_R/\KU_A)^{C_m}\bigr)^{\h (S^1/C_m)}
	\end{equation*}
	
	This section is organised as follows: In \crefrange{subsec:GenuineRecollections}{subsec:Genuineku} we review genuine equivariant homotopy theory, its special case of \emph{cyclonic spectra}, and the genuine equivariant structure on $\ku$. In \cref{subsec:GenuineHabiroDescent}, we finally construct the desired even filtrations in the cyclonic setting and prove that they give rise to the Habiro--Hodge complex from \cite[Theorem~\chref{3.11}]{qWittHabiro}.
	
	\subsection{Recollections on genuine equivariant homotopy theory}\label{subsec:GenuineRecollections}
	
	In this subsection, we briefly review theory of genuine equivariant spectra. We'll follow the model-independent treatment of \cite[Appendix~{\chref[appendix]{C}}]{EquivariantTMF} and the lecture notes \cite{HausmannEAST}.
	
	\begin{numpar}[Genuine equivariant anima.]
		Let $G$ be a compact Lie group (of relevance to us will only be the case of $S^1$ and its finite cyclic subgroups $C_m\subseteq S^1$). We let $\cat{Orb}_G$ denote the category whose objects are quotient spaces $G/H$, where $H\subseteq G$ is a closed subgroup, and whose morphisms are $G$-equivariant maps. $\cat{Orb}_G$ is canonically topologically enriched; through this enrichment we view it as an $\infty$-category.
		
		We define the $\infty$-category of \emph{$G$-anima} (or \emph{$G$-spaces}) as well as its pointed variant as
		\begin{equation*}
			\An^G\coloneqq\PSh(\cat{Orb}_G)\quad\text{and}\quad \An_*^G\coloneqq\PSh(\cat{Orb}_G)_*\,,
		\end{equation*}
		where $\PSh(-)\coloneqq \Fun((-)^\op,\An)$ and $\PSh(-)_*\coloneqq \Fun((-)^\op,\An_*)$ denote the presheaf $\infty$-category and its pointed variant. The pointwise product or smash product induces symmetric monoidal structures on $\An^G$ and $\An_*^G$ and thus turns them into objects in $\CAlg(\Pr^\L)$. We denote the evaluation at $G/H$ by $(-)^H\colon \An^G\rightarrow \An$ and likewise for $\An_*^G$. By construction, these functors are symmetric monoidal.
	\end{numpar}
	\begin{numpar}[Genuine equivariant spectra.]\label{par:GenuineSpectra}
		For every finite-dimensional real $G$-representation $V$, we have a topologically enriched functor $\cat{Orb}_G^\op\rightarrow \cat{Top}_*$ sending $G/H\mapsto S^{V^H}$, where $S^{V^H}$ denotes the $1$-point compactification of the vector space $V^H$. This functor defines a pointed $G$-anima $S^V\in \An_*^G$, which we call the \emph{representation sphere of $V$}. We finally define the $\infty$-category of \emph{genuine $G$-equivariant spectra}
		\begin{equation*}
			\Sp^G\coloneqq \An_*^G\Bigl[\bigl\{(S^V)^{\otimes-1}\bigr\}_V\Bigr]
		\end{equation*}
		to be the initial $\An_*^G$-algebra in $\Pr^\L$ in which all representation spheres $S^V$ become $\otimes$-invertible. Explicitly, $\Sp^G$ can be written as a colimit in $\Pr^\L$ of a diagram whose objects are copies of $\An_*^G$ and whose transition maps are of the form $S^V\wedge -\colon \An_*^G\rightarrow \An_*^G$, where $V$ ranges through finite-dimensional $G$-representations; see \cite[\S{\chref[subsection]{C.1}}]{EquivariantTMF}. By construction, $\Sp^G$ comes with a symmetric monoidal functor
		\begin{equation*}
			\Sigma_G^\infty\colon \An_*^G\longrightarrow \Sp^G
		\end{equation*}
		in $\Pr^\L$, which thus admits a lax monoidal right adjoint $\Omega_G^\infty\colon \Sp^G\rightarrow \An_*^G$.
		
		We let $\Sigma^{V}\colon\Sp^G\rightarrow \Sp^G$ denote the functor $\Sigma_G^\infty S^V\otimes -$. By construction, this functor is an equivalence, and we let $\Sigma^{-V}$ denote its inverse. If $(-)_+\colon \An^G\rightarrow \An_*^G$ denotes the left adjoint of the forgetful functor, we also define
		\begin{equation*}
			\IS_G[-]\colon\An^G\xrightarrow{(-)_+}\An_*^G\xrightarrow{\Sigma_G^\infty} \Sp^G
		\end{equation*}
		and we let $\IS_G\coloneqq \IS_G[*]$ be the \emph{genuine $G$-equivariant sphere spectrum}.
		
		The $\infty$-category $\An_*^G$ is compactly generated, with a set of compact generators given by $(G/H)_+$ for all closed subgroups $H\subseteq G$. The transition maps $S^V\wedge-$ preserve compact objects and $\Pr_\omega^\L\rightarrow \Pr^\L$ preserves colimits. It follows that $\Sp^G$ is compactly generated, with a set of compact generators given by $\Sigma^{-V}\IS_G[G/H]$ for all representation spheres and all closed subgroups $H\subseteq G$. In fact, we can do slightly better; see \cref{lem:SpGCompactGenerators} below.
	\end{numpar}
	
	\begin{numpar}[Pullback functors.]
		Given any morphism $\varphi\colon G\rightarrow K$ of compact Lie groups, we can define a functor $\cat{Orb}_G\rightarrow \cat{Orb}_K$ by sending $G/H\mapsto K/\varphi(H)$. By precomposition, we obtain a symmetric monoidal functor $\varphi^*\colon \An_*^K\rightarrow \An_*^G$ in $\Pr^\L$, which sends representation spheres to representation spheres and therefore determines a unique symmetric monoidal colimit-preserving functor
		\begin{equation*}
			\varphi^*\colon \Sp^K\longrightarrow \Sp^G\,.
		\end{equation*}
	\end{numpar}
	\begin{lem}\label{lem:PullbackCommutesWithOmega}
		For every morphism $\varphi\colon G\rightarrow K$ of compact Lie groups, the following diagrams commute:
		\begin{equation*}
			\begin{tikzcd}
				\An_*^K\rar["\varphi^*"]\dar["\Sigma_K^\infty"'] & \An_*^G\dar["\Sigma_G^\infty"]\\
				\Sp^K\rar["\varphi^*"] & \Sp^G
			\end{tikzcd}\quad\text{and}\quad \begin{tikzcd}
				\An_*^K\rar["\varphi^*"] & \An_*^G\\
				\Sp^K\uar["\Omega_K^\infty"]\rar["\varphi^*"] & \Sp^G\uar["\Omega_G^\infty"']
			\end{tikzcd}
		\end{equation*}
	\end{lem}
	\begin{proof}[Proof sketch]
		The diagram on the left commutes by construction. To see that the diagram on the right commutes as well, rewrite the colimits defining $\Sp^G$ and $\Sp^K$ as limits in $\Pr^\R$. It's then enough to check that $\varphi^*\colon \An_*^K\rightarrow \An_*^G$ intertwines the right adjoints of $S^{V}\wedge-$ and $S^{\varphi^*(V)}\wedge -$ for any finite-dimensional $K$-representation $V$. Since $\varphi^*\colon \An_*^K\rightarrow \An_*^G$ has a left adjoint $\varphi_!$, given by left Kan extension, we may pass to left adjoints and show the equivalent assertion $\varphi_!(S^{\varphi^*(V)}\wedge-)\simeq S^V\wedge\varphi_!(-)$. Now in general, for any functor $\varphi\colon \Cc\rightarrow \Dd$ of small $\infty$-categories, the adjunction $\varphi_!\colon \PSh(\Cc)_*\shortdoublelrmorphism \PSh(\Dd)_*\noloc \varphi^*$ satisfies the \enquote{projection formula} $\varphi_!(\varphi^*(Y)\wedge X)\simeq Y\wedge\varphi_!(X)$ by abstract nonsense.
	\end{proof}
	\begin{lem}\label{lem:InducedGSpectra}
		Let $i\colon H\hookrightarrow G$ be the inclusion of a closed subgroup. Then $i^*\colon \Sp^G\rightarrow \Sp^H$ preserves all limits and  and thus admits a left adjoint $i_!\colon \Sp^H\rightarrow \Sp^G$.%
		\footnote{The functor $i_!$ is usually denoted $\Ind_H^G$ and called \emph{induction}.}
		If we also let $i_!\colon \An_*^H\rightarrow \An_*^G$ denote the left Kan extension functor, then the following diagram commutes:
		\begin{equation*}
			\begin{tikzcd}
				\An_*^H\rar["i_!"]\dar["\Sigma_H^\infty"'] & \An_*^G\dar["\Sigma_G^\infty"]\\
				\Sp^H\rar["i_!"] & \Sp^G
			\end{tikzcd}
		\end{equation*}
		In particular, $i_!\IS_H\simeq \IS_G[G/H]$.
	\end{lem}
	\begin{proof}[Proof sketch]
		To form $\Sp^H$, it's enough to invert all representation spheres of the form $S^{i^*(V)}$ in $\An_*^H$, where $V$ is a finite-dimensional $G$-representation. Thus, we can obtain $\Sp^G$ and $\Sp^H$ by colimit diagrams of the same shape in $\Pr^\L$. Treating them as limit diagrams in $\Pr^\R$ and noting that the transition maps still commute with $i^*\colon \An_*^G\rightarrow \An_*^H$ (see the argument in the proof of \cref{lem:PullbackCommutesWithOmega}) shows that $i^*$ indeed preserves limits. Commutativity of the diagram follows from the right diagram in \cref{lem:PullbackCommutesWithOmega} by passing to left adjoints.
	\end{proof}
	\begin{numpar}[Borel-complete spectra.]\label{par:Borel}
		The full sub-$\infty$-category spanned by $G/\{1\}\in\cat{Orb}_G^\op$ defines a functor $\B G\rightarrow \cat{Orb}_G^\op$. Via precomposition we get a symmetric monoidal functor $\An_*^G\rightarrow \An_*^{\B G}$. Since all representation spheres $S^V\in \An_*^G$ become $\otimes$-invertible under $\Sigma^\infty\colon \An_*^{\B G}\rightarrow \Sp^{\B G}$, we can use the universal property of $\Sp^G$ to obtain a commutative diagram
		\begin{equation*}
			\begin{tikzcd}
				\An_*^G\rar\dar["\Sigma_G^\infty"'] & \An_*^{\B G}\dar["\Sigma^\infty"]\\
				\Sp^G\rar["U_G"] & \Sp^{\B G}
			\end{tikzcd}
		\end{equation*}
		of symmetric monoidal functors in $\Pr^\L$. For a genuine $G$-equivariant spectrum $X$, we think of $U_G(X)$ as the underlying spectrum with its non-genuine $G$-action, and we'll often suppress $U_G$ in the notation. Genuine $G$-equivariant spectra in the image of the right adjoint
		\begin{equation*}
			B_G\colon \Sp^{\B G}\longrightarrow \Sp^{G}
		\end{equation*}
		will be called \emph{Borel-complete} and we call the functor $B_G\circ U_G$ \emph{Borel completion.}
	\end{numpar}
	\begin{lem}
		The functor $B_G\colon \Sp^{\B G}\rightarrow \Sp^{G}$ is fully faithful.
	\end{lem}
	\begin{proof}
		As in \cref{lem:InducedGSpectra}, one shows that $U_G$ also preserves limits and hence admits a left adjoint $L$. It will be enough to show that the unit $u\colon \id\Rightarrow U_G\circ L$ is an equivalence. Since both $U_G$ and $L$ preserve all colimits, we only need to check that $u$ is an equivalence on the generator $\IS[G]$ of $\Sp^{\B G}$. 
		
		To see this, note that the forgetful functor $\Sp^{\B G}\rightarrow \Sp$ is conservative. Moreover, it's clear from the construction that $\Sp^G\rightarrow \Sp^{\B G}\rightarrow \Sp$ equals $e^*\colon \Sp^G\rightarrow \Sp$, where $e\colon \{1\}\hookrightarrow G$ is the inclusion of the identity element. Since $\IS[G]$ is the image of $\IS$ under the left adjoint of $\Sp^{\B G}\rightarrow \Sp$, it will thus be enough to check that $\IS\rightarrow e^*e_!\IS$ is an equivalence. Using the commutative diagram of \cref{lem:InducedGSpectra}, this reduces to checking that $S^0\rightarrow e_!e^*S^0$ is an equivalence in $\An_*$, which is clear since Kan extension along a fully faithful functor is fully faithful.
	\end{proof}
	\begin{numpar}[Genuine fixed points.]\label{par:GenuineFixedPoints}
		For every morphism $\varphi\colon G\rightarrow K$ of compact Lie groups, the right adjoint $\varphi_*\colon \Sp^G\rightarrow \Sp^K$ of $\varphi^*$ is lax symmetric monoidal and still preserves colimits. Indeed, since $\varphi_*$ is an exact functor between compactly generated stable $\infty$-categories, it will be enough to check that $\varphi^*$ preserves compact objects, which is clear from the description of compact generators in \cref{par:GenuineSpectra}. In the case where $\varphi$ is the projection $\pi_G\colon G\rightarrow\{1\}$ to the trivial group, we also denote $\pi_{G,*}$ by
		\begin{equation*}
			(-)^G\colon \Sp^G\longrightarrow \Sp
		\end{equation*}
		and call this the \emph{genuine $G$-fixed points}. We have $(-)^G\simeq \Hom_{\Sp}(\IS,(-)^G)\simeq \Hom_{\Sp^G}(\IS_G,-)$ by adjunction, and so $(-)^G$ is represented by $\IS_G$.

		If $X\in \Sp^G$ and $i\colon H\hookrightarrow G$ is the inclusion of a closed subgroup, we'll usually write $X^H$ instead of $(i^*X)^H$ for brevity. It follows formally that $(-)^H\colon \Sp^G\rightarrow \Sp$ is represented by $i_!\IS_H\simeq \IS_G[G/H]$.
	\end{numpar}
	
	\begin{lem}\label{lem:SpGCompactGenerators}
		The $\infty$-category $\Sp^G$ is compactly generated, with a set of compact generators given by $\Sigma^{-n}\IS_G[G/H]$ for all $n\geqslant 0$ and all closed subgroups $H\subseteq G$.
	\end{lem}
	\begin{proof}
		The following argument is taken from \cite[Proposition~\chref{2.7}]{HausmannEAST}. We use induction on $(\dim G,\abs{\pi_0(G)})$, ordered lexicographically. Suppose a genuine $G$-equivariant spectrum $X$ satisfies $\Hom_{\Sp^G}(\Sigma^{-n}\IS_G[G/H],X)\simeq 0$ for all closed subgroups $H$ and all $n\geqslant 0$. If $i\colon H\hookrightarrow G$ is the inclusion of any such $H$, then for any closed subgroup $K\subseteq H$ we have
		\begin{equation*}
			0\simeq \Hom_{\Sp^G}\bigl(\IS_G[G/K],X\bigr)\simeq \Hom_{\Sp^G}\bigl(i_!\IS_H[H/K],X\bigr)\simeq \Hom_{\Sp^H}\bigl(\IS_H[H/K],i^*(X)\bigr)
		\end{equation*}
		and therefore $i^*(X)\simeq 0$ by the inductive hypothesis. As a consequence, we see that $\Hom_{\Sp^G}(\Sigma^{-V}\IS_G[G/H],X)\simeq 0$ for all proper closed subgroups $H\subsetneq G$ and all finite-dimensional $G$-representations $V$.
		
		It remains to show $\Hom_{\Sp^G}(\Sigma^{-V}\IS_G,X)\simeq 0$ for all $V$. Let $j\colon \cat{Orb}_{<G}\hookrightarrow \cat{Orb}_G$ denote the inclusion of the full sub-$\infty$-category spanned by all objects except the terminal object $G/G$. Let $\An^{<G}_*\coloneqq \PSh(\cat{Orb}_{<G})_*$. A straightforward application of the Kan extension formula shows that the left Kan extension functor $j_!\colon \An_*^{<G}\rightarrow \An_*^G$ is fully faithful, with essential image given by those pointed genuine $G$-equivariant anima $Y$ that satisfy $Y^G\simeq *$ (i.e.\ those presheaves that vanish on $G/G\in \cat{Orb}_G$). Since $\cofib(S^{V^G}\rightarrow S^V)$ is of this form, it can be written as a colimit of $(G/H)_+$ for proper closed subgroups $H\subsetneq G$. It follows that $\cofib(\Sigma^{V^G}X\rightarrow \Sigma^VX)\simeq 0$, since it can be written as a colimit of terms of the form
		\begin{equation*}
			\IS_G[G/H]\otimes X\simeq i_!\IS_H\otimes X\simeq i_!\bigl(\IS_H\otimes i^*(X)\bigr)\simeq 0\,.
		\end{equation*}
		By our assumption on $X$, we also have $\Hom_{\Sp^G}(\IS_G,\Sigma^{V^G}X)\simeq \Hom_{\Sp^G}(\Sigma^{-n}\IS_G,X)\simeq 0$, where $n\coloneqq \dim V^G$. We conclude $0\simeq \Hom_{\Sp^G}(\IS_G,\Sigma^VX)\simeq \Hom_{\Sp^G}(\Sigma^{-V}\IS_G,X)$, as desired.
	\end{proof}
	\begin{lem}\label{lem:SGGHselfDual}
		If $G$ is finite, then the compact objects $\IS_G[G/H]\in \Sp^G$ are self-dual for all subgroups $H\subseteq G$. In particular, $\Sp^G$ is a rigid symmetric monoidal $\infty$-category.
	\end{lem}
	\begin{proof}[Proof sketch]
		We need to construct a coevaluation $\eta\colon \IS_G\rightarrow \IS_G[G/H]\otimes \IS_G[G/H]$ and an evaluation $\epsilon\colon \IS_G[G/H]\otimes\IS_G[G/H]\rightarrow \IS_G$ satisfying the triangle identities. To construct $\epsilon$, we simply apply $\Sigma_G^\infty$ to the map $(G/H\times G/H)_+\rightarrow S^0$ that sends the diagonal to the non-basepoint and everything else to the basepoint.
		
		Let us now construct $\eta$. Let $V\coloneqq \IR[G/H]$. Equip $V$ with an inner product in such a way that $\{\sigma\}_{\sigma\in G/H}$ is an orthonormal basis. Consider the \enquote{diagonal map} $V\rightarrow V\times (G/H)$, whose $\sigma$\textsuperscript{th} component is given by $\IR[G/H]\rightarrow \IR\sigma\rightarrow \IR\sigma\times\{\sigma\}$, where the first map is the orthogonal projection. This map is $G$-equivariant and proper, so it induces a $G$-equivariant map of one-point compactifications, which takes the form $S^V\rightarrow S^V\wedge (G/H)_+$. Applying $\Sigma_G^{\infty-V}$, we obtain a map  $\operatorname{tr}_H^G\colon \IS_G\rightarrow \IS_G[G/H]$, called the \emph{transfer}. Let also $\Delta\colon G/H\rightarrow G/H\times G/H$ denote the diagonal. We finally define $\eta$ as the composite
		\begin{equation*}
			\eta\colon \IS_G\xrightarrow{\operatorname{tr}_H^G}\IS_G[G/H]\xrightarrow{\IS_G[\Delta]}\IS_G[G/H]\otimes\IS_G[G/H]\,.
		\end{equation*}
		The triangle identities can already be verified on the level of genuine $G$-equivariant anima.
	\end{proof}
	\begin{rem}
		For arbitrary compact Lie groups $G$, it is still true that $\IS_G[G/H]$ are dualisable, so that $\Sp^G$ is still rigid. See e.g.\ \cite[\S{\chref[subsection]{2.3}}]{HausmannEAST}.
	\end{rem}
	
	\begin{numpar}[Genuine vs.\ homotopy fixed points.]\label{par:GenuineVsHomotopyFixedPoints}
		By abuse of notation, let us denote the composition ofthe functor  $U_G\colon \Sp^G\rightarrow \Sp^{\B G}$ from \cref{par:Borel} with the homotopy fixed point functor $(-)^{\h G}\colon \Sp^{\B G}\rightarrow \Sp$ also by $(-)^{\h G}$. For any $X\in \Sp^G$ we have
		\begin{equation*}
			\bigl(B_GU_G(X)\bigr)^G\simeq \Hom_{\Sp^G}\bigl(\IS_G,B_GU_G(X)\bigr)\simeq \Hom_{\Sp^{\B G}}\bigl(\IS,U_G(X)\bigr)\simeq X^{\h G}\,.
		\end{equation*}
		Thus, the natural transformation $\id\Rightarrow U_G\circ B_G$ (Borel-completion) induces a symmetric monoidal transformation of lax symmetric monoidal functors
		\begin{equation*}
			(-)^G\Longrightarrow (-)^{\h G}
		\end{equation*}
		In general, this is far from being an equivalence; in fact, the goal of this whole section is to explain how the Habiro descent of the $q$-Hodge complex is accounted for by the failure of $\THH(\KU_R/\KU_A)^{C_m}\rightarrow \THH(\KU_R/\KU_A)^{\h C_m}$ to be an equivalence.
	\end{numpar}
	\begin{numpar}[Geometric fixed points.]\label{par:GeometricFixedPoints}
		The functor $\Sigma^\infty\circ (-)^G\colon \An_*^G\rightarrow \Sp$ is symmetric monoidal and inverts all representation spheres. Therefore it induces a symmetric monoidal functor
		\begin{equation*}
			(-)^{\Phi G}\colon \Sp^G\longrightarrow \Sp
		\end{equation*}
		in $\Pr^\L$, called \emph{geometric fixed points}.%
		\footnote{Geometric fixed points are usually denoted $\Phi^G$. We chose $(-)^{\Phi G}$ to be in line with $(-)^{G}$, $(-)^{\h G}$, and $(-)^{\t G}$.}
		%
		There always exists a natural transformation
		\begin{equation*}
			(-)^G\Longrightarrow (-)^{\Phi G}\,.
		\end{equation*}
		One way to construct this would be as the following composite (see \cite[\S{\chref[subsection]{2.2}}]{HausmannEAST}):
		\begin{equation*}
			X^G\simeq \Hom_{\Sp^G}(\IS_G,X)\longrightarrow\Hom_{\Sp}\bigl(\IS_G^{\Phi G},X^{\Phi G}\bigr)\simeq\Hom_{\Sp}\bigl(\IS,X^{\Phi G}\bigr)\simeq X^{\Phi G}\,.
		\end{equation*}
		Just as for genuine fixed points, for every closed subgroup $H\subseteq G$, we also consider the functor $(-)^{\Phi H}\colon \Sp^G\rightarrow \Sp$, suppressing the pullback $\Sp^G\rightarrow \Sp^H$ in the notation.
	\end{numpar}
	\begin{lem}\label{lem:GenuineGeometricJointlyConservative}
		The family of functors $\{(-)^H\}_{H\subseteq G}$ in $\Fun(\Sp^G,\Sp)$ is jointly conservative. The same is true for $\{(-)^{\Phi H}\}_{H\subseteq G}$.
	\end{lem}
	\begin{proof}
		Both assertions are classical; see e.g.\ \cite[Proposition~\chref{3.3.10}]{SchwedeGlobal} for the case of geometric fixed points. We'll give a proof by abstract nonsense, following \cite[\S\S{\chref[subsection]{2.2}}--{\chref[subsection]{2.3}}]{HausmannEAST}.
		
		For genuine fixed points, joint conservativity follows immediately from \cref{lem:SpGCompactGenerators}. For geometric fixed points, assume $X\in \Sp^G$ satisfies $X^{\Phi H}\simeq 0$ for all~$H$. We wish to show $X\simeq 0$. Arguing by induction over $(\dim G,\abs{\pi_0(G)})$, we may assume $i^*(X)\simeq 0$ for all inclusions $i\colon H\hookrightarrow G$ of proper closed subgroups. As in the proof of \cref{lem:SpGCompactGenerators}, this implies $\IS_G[G/H]\otimes X\simeq 0$ for all such $H$.
		
		As in the proof of \cref{lem:SpGCompactGenerators}, let now $j\colon \cat{Orb}_{<G}\hookrightarrow \cat{Orb}_G$ denote the inclusion of the full sub-$\infty$-category spanned by all objects except $G/G$ and put $\An^{<G}\coloneqq\PSh(\cat{Orb}_{<G})$. Let $s\colon \{G/G\}\hookrightarrow \cat{Orb}_G$ denote the complementary inclusion. Let $j_!\colon \An^{<G}\shortdoublelrmorphism \An^G\noloc j^*$ be the adjunction given by left Kan extension/restriction along $j$ and let $s^*\colon \An_*^G\shortdoublelrmorphism \An_*\noloc s_*$ be the adjunction given by restriction/right Kan extension along $s$. We denote $\mathrm{E}\Pp_G\coloneqq j_!j^*(*)$ and $\widetilde{\mathrm{E}}\Pp_G\coloneqq s_*s^*S^0$ (in the classical setup this has intrinsic meaning; for us it's just notation). Then the Kan extension formula shows that%
		\begin{equation*}
			(\mathrm{E}\Pp_G)^H\simeq \begin{cases*}
				\emptyset & if $H=G$\\
				* & if $H\subsetneq G$
			\end{cases*}\quad\text{and}\quad
			(\widetilde{\mathrm{E}}\Pp_G)^H\simeq \begin{cases*}
				S^0 & if $H=G$\\
				* & if $H\subsetneq G$
			\end{cases*}\,.
		\end{equation*}
		Thus the canonical sequence $(\mathrm{E}\Pp_G)_+\rightarrow S^0\rightarrow \widetilde{\mathrm{E}}\Pp_G$ induced by the universal property of Kan extension is a cofibre sequence in $\An_*^G$. It follows that $\IS_G[\mathrm{E}\Pp_G]\rightarrow \IS_G\rightarrow \Sigma_G^\infty(\widetilde{\mathrm{E}}\Pp_G)$ is a cofibre sequence in $\Sp^G$, respectively. We have $\IS_G[\mathrm{E}\Pp_G]\otimes X\simeq 0$ as $\IS_G[\mathrm{E}\Pp_G]$ is contained in the full sub-$\infty$-category generated under colimits by $\IS_G[G/H]$ for proper closed subgroups $H\subsetneq G$. It will thus be enough to show $\Sigma_G^\infty(\widetilde{\mathrm{E}}\Pp_G)\otimes X\simeq 0$. Since $(-)^{\Phi H}$ is symmetric monoidal, we still have $(\Sigma_G^\infty(\widetilde{\mathrm{E}}\Pp_G)\otimes X)^{\Phi H}\simeq 0$ and so the inductive hypothesis shows $(\Sigma_G^\infty(\widetilde{\mathrm{E}}\Pp_G)\otimes X)^{H}\simeq 0$ for all proper closed subgroups $H\subsetneq G$. It remains to show $(\Sigma_G^\infty(\widetilde{\mathrm{E}}\Pp_G)\otimes X)^{G}\simeq 0$, which follows from the assumption $X^{\Phi G}\simeq 0$ using \cref{lem:EPGeometrixFixedPoints} below.
	\end{proof}
	\begin{lem}\label{lem:EPGeometrixFixedPoints}
		With notation as above, for any $X\in \Sp^G$ there is a functorial equivalence
		\begin{equation*}
			\bigl(\Sigma_G^\infty(\widetilde{\mathrm{E}}\Pp_G)\otimes X\bigr)^G\overset{\simeq}{\longrightarrow} X^{\Phi G}\,.
		\end{equation*}
	\end{lem}
	\begin{proof}
		Let us first construct the functorial map. With notation as in the proof of \cref{lem:GenuineGeometricJointlyConservative} above, we have $\IS_G[\mathrm{E}\Pp_G]^{\Phi G}\simeq \IS[(\mathrm{E}\Pp_G)^G]\simeq 0$. Thus, if we apply the natural transformation $(-)^G\Rightarrow (-)^{\Phi G}$ to the cofibre sequence $\IS_G[\mathrm{E}\Pp_G]\otimes X\rightarrow \IS_G\otimes X\rightarrow \Sigma_G^\infty(\widetilde{\mathrm{E}}\Pp_G)\otimes X$, it will induce the desired map.
		
		Let us now verify that this map is an equivalence. Since $(-)^G$ and $(-)^{\Phi G}$ preserve colimits, it's enough to check the case $X\simeq \IS_G[G/H]$. For proper subgroups $H\subsetneq G$ we have $(G/H)^G\simeq *$ and so $\IS_G[G/H]^{\Phi G}\simeq 0$ as well as
		\begin{equation*}
			\Sigma_G^\infty(\widetilde{\mathrm{E}}\Pp_G)\otimes \IS_G[G/H]\simeq \Sigma_G^\infty \bigl(\widetilde{\mathrm{E}}\Pp_G\wedge (G/H)_+\bigr)\simeq \Sigma_G^\infty(*)\simeq 0\,.
		\end{equation*}
		It remains to show that $\Sigma_G^\infty(\widetilde{\mathrm{E}}\Pp_G)^G\rightarrow \IS_G^{\Phi G}\simeq \IS$ is an equivalence. This can be checked on underlying anima. Using the definition of $\Sp^G$ as a colimit, we see
		\begin{equation*}
			\Omega^\infty\bigl(\Sigma_G^\infty(\widetilde{\mathrm{E}}\Pp_G)^G\bigr)\simeq \Omega^\infty\Hom_{\Sp^G}\bigl(\IS_G,\Sigma_G^\infty(\widetilde{\mathrm{E}}\Pp_G)\bigr)\simeq \colimit_{V\subseteq \Uu}\Map_{\An_*^G}\left(S^V,S^V\wedge \widetilde{\mathrm{E}}\Pp_G\right)\,,
		\end{equation*}
		where $\Uu$ is a complete $G$-universe, that is, a direct sum of countably many copies of each irreducible $G$-representation, and $V$ ranges through all finite-dimensional subrepresentations of $\Uu$. Now recall that $\widetilde{\mathrm{E}}\Pp_G\simeq s_*s^*S^0$. Using the Kan extension formula, it's straightforward to check $S^V\wedge s_*s^*S^0\simeq s_*S^{V^G}$ and so the colimit above can be rewritten as desired:
		\begin{equation*}
			\colimit_{V\subseteq \Uu}\Map_{\An_*^G}\bigl(S^V,s_*S^{V^G}\bigr)\simeq \colimit_{V\subseteq \Uu}\Map_{\An_*}\bigl(S^{V^G},S^{V^G}\bigr)\simeq \Omega^\infty \IS\,.\qedhere
		\end{equation*}
	\end{proof}
	Using a similar argument, we can also show the following assertion:
	\begin{lem}\label{lem:GenuineGeometricBoundedBelow}
		Let $G$ be finite. For a genuine $G$-equivariant spectrum $X$, the following are equivalent: 
		\begin{alphanumerate}
			\item For all subgroups $H\subseteq G$, the genuine fixed points $X^H$ are bounded below.
			\item For all subgroups $H\subseteq G$, the geometric fixed points $X^{\Phi H}$ are bounded below.
		\end{alphanumerate}
	\end{lem}
	\begin{proof}
		Via induction on $\abs{G}$, it will be enough to show under the hypothesis that $X^H$ is bounded below for all proper subgroups $H\subsetneq G$, the genuine fixed points $X^G$ are bounded below if and only if the geometric fixed points $X^{\Phi G}$ are bounded below. Using \cref{lem:EPGeometrixFixedPoints} and the proof of \cref{lem:GenuineGeometricJointlyConservative}, we find a cofibre sequence $(\IS_G[\mathrm{E}\Pp_G]\otimes X)^G\rightarrow X^G\rightarrow X^{\Phi G}$. Moreover, $\IS_G[\mathrm{E}\Pp_G]$ can be written as a colimit of $\IS_G[G/H]$ for proper subgroups $H\subsetneq G$. Thus, it will be enough to show that each $(\IS_G[G/H]\otimes X)^G$ is bounded below (here we use finiteness of $G$ to ensure that there are only finitely many $H$). This follows from
		\begin{equation*}
			\bigl(\IS_G[G/H]\otimes X\bigr)^G\simeq \Hom_{\Sp^G}\bigl(\IS_G,\IS_G[G/H]\otimes X\bigr)\simeq \Hom_{\Sp^G}\bigl(\IS_G[G/H],X\bigr)\simeq X^H\,,
		\end{equation*}
		where we use self-duality of $\IS_G[G/H]$ (\cref{lem:SGGHselfDual})
	\end{proof}
	\begin{numpar}[Inflation maps.]\label{par:Inflations}
		Given any morphism $\varphi\colon G\rightarrow K$ of compact Lie groups, one has a symmetric monoidal natural transformation of lax symmetric monoidal functors
		\begin{equation*}
			\inf_\varphi\colon (-)^K\Longrightarrow \bigl(\varphi^*(-)\bigr)^G
		\end{equation*}
		Indeed, from \cref{par:GenuineFixedPoints} we see that $(-)^G\simeq (-)^K\circ \varphi_*$ and then the desired natural transformation arises by postcomposing the unit transformation $\id\Rightarrow \varphi_*\circ \varphi^*$ with $(-)^K$.
		
		If $\varphi$ is injective, the transformation above is called \emph{restriction} and denoted $\res_G^K$. We're instead interested in the case where $\varphi$ is surjective, where it is customary to call these maps \emph{inflations}. In the surjective case, there's also a symmetric monoidal inflation
		\begin{equation*}
			\inf_\varphi\colon (-)^{\Phi K}\Longrightarrow \bigl(\varphi^*(-)\bigr)^{\Phi G}\,.
		\end{equation*}
		Indeed, on the level of genuine equivariant pointed anima, the pullback $\varphi^*\colon \An_*^K\rightarrow \An_*^G$ satisfies $(-)^K\simeq (\varphi^*(-))^G$ (this needs surjectivity, so that evaluation at $K/K\in \cat{Orb}_K^\op$ agrees with evaluation at $K/\varphi(G)$) and then the desired inflation transformation is induced by the universal property of $\Sp^K$ as an $\An_*^K$-algebra in $\Pr^\L$. It's straightforward to check that for all $X\in \Sp^K$ the diagram
		\begin{equation*}
			\begin{tikzcd}
				X^K\rar["\inf_\varphi"]\dar & (\varphi^*X)^G\dar\\
				X^{\Phi K}\rar["\inf_\varphi"] & (\varphi^*X)^{\Phi G}
			\end{tikzcd}
		\end{equation*}
		commutes functorially in $X$, where the horizontal maps are the inflations and the vertical maps are the ones from \cref{par:GeometricFixedPoints}.
	\end{numpar}
	
	\begin{numpar}[Residual actions.]\label{par:ResidualActions}
		Let $i\colon N\hookrightarrow G$ be the inclusion of a normal subgroup, let $\pi\colon G\rightarrow G/N$ denotes the canonical projection and let $e\colon \{1\}\hookrightarrow G/N$ the inclusion of the identity element. Then the diagram
		\begin{equation*}
			\begin{tikzcd}
				\Sp^{G}\rar["\pi_*"]\dar["i^*"'] & \Sp^{G/N}\dar["e^*"]\\
				\Sp^N\rar["(-)^N"] & \Sp^{\vphantom{N}}
			\end{tikzcd}
		\end{equation*}
		commutes. Indeed, commutativity can be checked after passing to left adjoints, and then it follows from $\pi^*e_!\IS\simeq \pi^*\IS_{G/N}[G/N]\simeq \IS_G[G/N]\simeq i_!\IS_N$, using the diagram from \cref{par:GenuineFixedPoints} to compute the values of $i_!$ and $e_!$.
		
		In particular, for any $X\in \Sp^G$, the genuine fixed points $X^N$ can be equipped with a \emph{residual genuine $G/N$-action}. In a similar way, one can equip $X^{\Phi N}$ with a residual genuine $G/N$-action, and it can be checked that $X^N\rightarrow X^{\Phi N}$ is genuine $G/N$-equivariant.
	\end{numpar}
	\begin{lem}\label{lem:FixedPointsCompose}
		With the residual actions from \cref{par:ResidualActions}, for all $X\in \Sp^G$ we have canonical equivalences
		\begin{equation*}
			X^G\simeq (X^N)^{G/N}\quad\text{and}\quad X^{\Phi G}\simeq (X^{\Phi N})^{\Phi(G/N)}\,.
		\end{equation*}
	\end{lem}
	\begin{proof}
		If $\pi_G\colon G\rightarrow \{1\}$ and $\pi_{G/N}\colon G/N\rightarrow \{1\}$ denote the canonical projections, then clearly $\pi_G^*\simeq \pi^*\circ \pi_{G/N}^*$. Since adjoints compose, the equivalence for genuine fixed points follows. To see the equivalence for geometric fixed points, it's enough to check the case $X\simeq \IS_G[Y]$ for $Y$ a genuine $G$-equivariant anima; this case follows from $Y^G\simeq (Y^N)^{G/N}$.
	\end{proof}
	
	\subsection{The \texorpdfstring{$\infty$}{infinity}-category of cyclonic spectra}\label{subsec:Cyclonic}
	
	After reviewing the general framework of genuine equivariant homotopy theory, from now on we'll restrict to the following special case:
	
	\begin{numpar}[Cyclonic spectra.]\label{par:Cyclonic}
		In the following, we'll consider spectra with an $S^1$-action that is genuine with respect to all finite cyclic subgroups $C_m\subseteq S^1$. These were introduced under the name \emph{cyclonic spectra} by Barwick and Glasman \cite{BarwickGlasmanCyclonic}.	
		
		While the original construction uses spectral Mackey functors, we'll follow \cite[Notation~\chref{2.3}(3)]{NaiveGenuine} and construct $\infty$-category of cyclonic spectra as the full stable sub-$\infty$-category $\Cyclonic\subseteq \Sp^{S^1}$ generated under colimits by $\Sigma^{-n}\IS_{S^1}[S^1/C_m]$ for all finite cyclic subgroups $C_m\subseteq S^1$ and all $n\geqslant 0$.
	\end{numpar}
	\begin{lem}\label{lem:CyclonicJointlyConservative}
		The family of functors $\{(-)^{C_m}\}_{m\in \IN}$ in $\Fun(\Cyclonic,\Sp)$ is jointly conservative. The same is true for $\{(-)^{\Phi C_m}\}_{m\in \IN}$.
	\end{lem}
	\begin{proof}
		For genuine fixed points this follows since $\{\Sigma^{-n}\IS_{S^1}[S^1/C_m]\}_{m\in\IN,n\geqslant 0}$ is a system of generators for $\Cyclonic$ by construction. The assertion about geometric fixed points then follows from \cref{lem:GenuineGeometricJointlyConservative}.
	\end{proof}
	
	\begin{lem}\label{lem:CyclonicSpectra}
		The fully faithful inclusion $j_!\colon \Cyclonic\hookrightarrow\Sp^{S^1}$ admits a right adjoint $j^*\colon \Sp^{S^1}\rightarrow \Cyclonic$ with the following properties:
		\begin{alphanumerate}
			\item $j^*$ still preserves all colimits.\label{enum:CyclonicLocalisation}
			\item The counit transformation $c\colon j_!\circ j^*\Rightarrow \id$ is an equivalence after applying $(-)^{C_m}$ or $(-)^{\Phi C_m}$ for any finite cyclic subgroup $C_m\subseteq S^1$.\label{enum:CyclonicGenuineGeometric}
			\item For all $X,Y\in\Sp^{S^1}$ the canonical map
			\begin{equation*}
				j^*(X\otimes j_!j^*Y)\overset{\simeq}{\longrightarrow} j^*(X\otimes Y)
			\end{equation*}
			is an equivalence. Thus, there's a canonical way to equip $\Cyclonic$ and $j^*\colon \Sp^{S^1}\rightarrow \Cyclonic$ with symmetric monoidal structures.\label{enum:CyclonicSymmetricMonoidal} 
		\end{alphanumerate}
	\end{lem}
	\begin{proof}
		The right adjoint $j^*$ exists since $j_!$ preserves all colimits. Since $\Cyclonic$ is compactly generated and $j_!$ preserves compact objects, $j^*$ preserves filtered colimits and thus all colimits by exactness, proving \cref{enum:CyclonicLocalisation}. By construction,
		\begin{equation*}
			\Hom_{\Sp^{S^1}}\bigl(\IS_{S^1}[S^1/C_n],j_!j^*X\bigr)\simeq \Hom_{\Sp^{S^1}}\bigl(\IS_{S^1}[S^1/C_m],X\bigr)
		\end{equation*}
		and so $(j_!j^*X)^{C_m}\rightarrow X^{C_m}$ is indeed an equivalence. Since this is true for all divisors $d\mid m$, \cref{lem:GenuineGeometricJointlyConservative} shows that $(j_!j^*X)^{\Phi C_m}\rightarrow X^{\Phi C_m}$ is an equivalence as well. This shows \cref{enum:CyclonicGenuineGeometric}.
		
		Whether $j^*(X\otimes j_!j^*Y)\rightarrow j^*(X\otimes Y)$ is an equivalence can be checked on geometric fixed points by \cref{lem:CyclonicJointlyConservative}. But after applying $(j_!(-))^{\Phi C_m}$, both sides become $X^{\Phi C_m}\otimes Y^{\Phi C_m}$ by \cref{enum:CyclonicGenuineGeometric} and symmetric monoidality of $(-)^{\Phi C_m}$. This shows the first claim in~\cref{enum:CyclonicSymmetricMonoidal}; the second claim is general abstract nonsense about localisations of symmetric monoidal $\infty$-categories (see \cite[Proposition~\chref{2.2.1.9}]{HA} for example). 
	\end{proof}
	In the following, we'll usually suppress $j_!$ and $j^*$ in the notation.
	\begin{lem}\label{lem:GenuineGeometricCommute}
		For $m,n\in\IN$, let us identify $C_{mn}/C_m\cong C_n$, $C_{mn}/C_n\cong C_m$. For all cyclonic spectra $X$, the residual actions from \cref{par:ResidualActions} satisfy the following functorial identites:
		\begin{alphanumerate}
			\item $(X^{C_m})^{C_n}\simeq X^{C_{mn}}$, $(X^{\Phi C_m})^{\Phi C_n}\simeq X^{\Phi C_{mn}}$, and $(X^{\h C_m})^{\h C_n}\simeq X^{\h C_{mn}}$.\label{enum:FixedPointsCompose}
			\item If $m$ and $n$ are coprime, then $(X^{C_m})^{\Phi C_n}\simeq (X^{\Phi C_n})^{C_m}$.\label{enum:GenuineGeometricCommuteCoprime}
		\end{alphanumerate}
	\end{lem}
	\begin{proof}
		The first two assertions from \cref{enum:FixedPointsCompose} are special cases of \cref{lem:FixedPointsCompose}, the third assertion is classical. For \cref{enum:GenuineGeometricCommuteCoprime}, let us first note that $\cat{Orb}_{C_{mn}}\simeq \cat{Orb}_{C_m}\times \cat{Orb}_{C_n}$, which easily implies $\Sp_{C_{mn}}\simeq \Sp_{C_m}\otimes \Sp_{C_n}$ for the Lurie tensor product. By construction of geometric fixed points it's clear that $(-)^{\Phi C_n}\colon \Sp_{C_m}\otimes \Sp_{C_n}\rightarrow \Sp_{C_m}$ is given by applying $(-)^{\Phi C_n}\colon \Sp_{C_n}\rightarrow \Sp$ in the second tensor factor. If we can show a similar assertion for $(-)^{C_m}$, we'll be done.
		
		To this end, let $\pi\colon C_{mn}\rightarrow C_n$ and $\pi_{C_m}\colon C_m\rightarrow \{1\}$ denote the canonical projections. It is again clear from the construction that $\pi^*\colon \Sp_{C_n}\rightarrow \Sp_{C_m}\otimes\Sp_{C_n}$ is given by applying $\pi_{C_m}^*\colon \Sp\rightarrow \Sp_{C_m}$ in the first tensor factor. Its right adjoint $\pi_*$ must then also be given by applying the right adjoint $\pi_{C_m,*}$ (which is also a functor in $\Pr^\L$) in the first tensor factor, because we can just apply $-\otimes {\Sp_{C_n}}$ to the unit, the counit, and the triangle identities.
	\end{proof}
	
	Nikolaus--Scholze \cite[Theorem~{\chref{2.6.9}[II.6.9]}]{NikolausScholze} showed that on bounded below objects, the structure of a \emph{cyclotomic spectrum} is equivalent to a \enquote{naive} notion, in which one only asks for $S^1$-equivariant maps $X\rightarrow X^{\t C_p}$. 
	We'll now show a similar result in the cyclonic case. This is based on the following well-known fact (see e.g.\ \cite{HesselholtMadsenKTheoryWittVectors} or \cite[Lemma~{\chref{2.4.5}[II.4.5]}]{NikolausScholze}):
	\begin{lem}
		There's a pullback square of symmetric monoidal transformations between lax symmetric monoidal functors in $\Fun(\Sp^{C_p},\Sp)$
		\begin{equation*}
			\begin{tikzcd}
				(-)^{C_p}\doublear["\textup{(\cref{par:GeometricFixedPoints})}"{black}]{r}\doublear["\textup{(\cref{par:GenuineVsHomotopyFixedPoints})}"{black,swap}]{d}\drar[pullback] & (-)^{\Phi C_p}\doublear{d}\\
				(-)^{\h C_p}\doublear{r} & (-)^{\t C_p}
			\end{tikzcd}
		\end{equation*}
	\end{lem}
	\begin{proof}
		If we regard the orbit $C_p/C_1$ as a genuine $C_p$-equivariant anima via the Yoneda embedding, we find $(C_p/C_1)^{C_p}\simeq \emptyset$ and $(C_p/C_1)^{C_1}\simeq C_p$. By direct inspection, it follows that $\mathrm{E}\Pp_{C_p}\simeq (C_p/C_1)_{\h C_p}$, where $C_p$ acts on $C_p/C_1$ in the obvious way.
		
		For every genuine $C_p$-equivariant spectrum $X$, we've seen in \cref{lem:EPGeometrixFixedPoints} that the fibre of $X^{C_p}\rightarrow X^{\Phi C_p}$ is given by $(\IS_{C_p}[\mathrm{E}\Pp_{C_p}]\otimes X)^{C_p}$. Using that $(-)^{C_p}$ preserves all colimits and $\IS_{C_p}[C_p]$ is self-dual by \cref{lem:SGGHselfDual}, we find
		\begin{equation*}
			\bigl(\IS_{C_p}[\mathrm{E}\Pp_{C_p}]\otimes X\bigr)^{C_p}\simeq \bigl(\IS_{C_p}[C_p]_{\h C_p}\otimes X\bigr)^{C_p}\simeq \bigl((\IS_{C_p}[C_p]\otimes X)^{C_p}\bigr)_{\h C_p}\simeq X_{\h C_p}\,.
		\end{equation*}
		In the case where $X$ is Borel-complete, it's straightforward to check that the induced map $X_{\h C_p}\rightarrow X^{C_p}\simeq X^{\h C_p}$ is the norm map and so $X^{\Phi C_p}\simeq X^{\t C_p}$ for Borel-complete $X$. In general, composing the Borel completion transformation $\id \Rightarrow B_{C_p}U_{C_p}$ with the natural trasformation $(-)^{C_p}\Rightarrow (-)^{\Phi C_p}$, we obtain the desired commutative square. It is a pullback square since the row-wise fibres are given by $(-)_{\h C_p}$, as we've just verified. Symmetric monoidality is also clear from the construction.
	\end{proof}
	
	\begin{numpar}[Naive cyclonic spectra]
		Informally, a \emph{naive cyclonic spectrum} should consist of a collection of spectra $(Y_m)_{m\in \IN}$, each $Y_m$ equipped with an $(S^1/C_m)$-action, together with $(S^1/C_{pm})$-equivariant maps $\phi_{p,m}\colon Y_{pm}\rightarrow Y_m^{\t C_p}$ for all $m$ and all primes~$p$. The intuition is that $Y_m\simeq X^{\Phi C_m}$ records the geometric fixed points of some cyclonic spectrum~$X$. To see obtain the maps $\phi_{p,m}\colon X^{\Phi C_{pm}}\rightarrow (X^{\Phi C_m})^{\t C_p}$ in this case, we plug $X^{\Phi C_m}$ into the natural transformation $(-)^{\Phi C_p}\Rightarrow (-)^{\t C_p}$; by naturality, the map $\phi_{p,m}$ that we obtain is (non-genuinely) $(S^1/C_{pm})$-equivariant.
		%
		%
		
		Formally, we define the $\infty$-category of \emph{naive cyclonic spectra} to be the lax equaliser (in the sense of \cite[Definition~{\chref{2.1.4}[II.1.4]}]{NikolausScholze})
		\begin{equation*}
			\Cyclonic^\mathrm{naiv}\coloneqq \operatorname{LEq}\Biggl(\begin{tikzcd}[cramped,column sep=huge]
				\prod_{m\in\IN}\Sp^{\B(S^1/C_m)}\rar[shift left=0.2em,"\operatorname{can}"]\rar[shift right=0.2em,"((-)^{\t C_p})_{p,m}"'] & \prod_{p\vphantom{\IN}}\prod_{m\in\IN}\Sp^{\B(S^1/C_{pm})}
			\end{tikzcd}\biggr)\,,
		\end{equation*}
		where $p$ runs through all primes, the top functor is given by $(Y_m)_m\mapsto (Y_{pm})_{p,m}$, and the bottom functor is given by $(Y_m)_m\mapsto (Y_m^{\t C_p})_{p,m}$. By the universal property of lax equalisers there is a functor
		\begin{equation*}
			(-)^{\Phi C}\colon \Cyclonic \longrightarrow \Cyclonic^\mathrm{naiv}
		\end{equation*}
		which sends $X\mapsto (X^{\Phi C_m})_{m\in\IN}$, equipped with the canonical maps $\phi_{p,m}\colon X^{\Phi C_{pm}}\rightarrow (X^{\Phi C_m})^{\t C_p}$ described above. Using \cref{lem:LEqSymmetricMonoidal} below, we can also equip $\Cyclonic^\mathrm{naiv}$ with a symmetric monoidal structure in such a way that $(-)^{\Phi C}$ is symmetric monoidal.
		
		Let us also call a cyclonic spectrum $X$ \emph{bounded below} if each $X^{C_m}$ is bounded below (not necessarily with a uniform bound for all~$m$); equivalently by \cref{lem:GenuineGeometricBoundedBelow}, all $X^{\Phi C_m}$ are bounded below. Similarly, a naive cyclonic spectrum $Y=((Y_m)_m,(\phi_{p,m})_{p,m})$ will be called \emph{bounded below} if each $Y_m$ is bounded below (not necessarily with a uniform bound). We denote by $\Cyclonic_+$ and $\Cyclonic_+^\mathrm{naiv}$ the respective full sub-$\infty$-categories of bounded below objects.
	\end{numpar}
	
	\begin{prop}\label{prop:NaiveCyclonic}
		When restricted to the respective full sub-$\infty$-categories of bounded below objects, the functor $(-)^{\Phi C}$ becomes a symmetric monoidal equivalence
		\begin{equation*}
			(-)^{\Phi C}\colon \Cyclonic_+\overset{\simeq}{\longrightarrow} \Cyclonic_+^\mathrm{naiv}\,.
		\end{equation*}
	\end{prop}
	To prove \cref{prop:NaiveCyclonic}, let us first construct the desired symmetric monoidal structure.
	\begin{lem}\label{lem:LEqSymmetricMonoidal}
		Let $F\colon \Cc\rightarrow \Dd$ be a symmetric monoidal functor and let $G\colon \Cc\rightarrow \Dd$ be a lax symmetric monoidal functor of symmetric monoidal $\infty$-categories. Let $F^\otimes$ and $G^\otimes$ denote the corresponding functors between the $\infty$-operads $\Cc^\otimes\rightarrow \cat{Fin}_*$ and $\Dd^\otimes\rightarrow \cat{Fin}_*$ and define
		\begin{equation*}
			\operatorname{LEq}(F,G)^\otimes\coloneqq \operatorname{LEq}(F^\otimes,G^\otimes)\times_{\operatorname{LEq}(\id_{\cat{Fin}_*},\id_{\cat{Fin}_*})}\cat{Fin}_*\,.
		\end{equation*}
		\begin{alphanumerate}
			\item $\operatorname{LEq}(F,G)^\otimes\rightarrow \cat{Fin}_*$ is an $\infty$-operad associated to a symmetric monoidal structure on the $\infty$-category $\operatorname{LEq}(F,G)$ and $\operatorname{LEq}(F,G)\rightarrow \Cc$ is symmetric monoidal.\label{enum:LEqCocartesian}
			\item If $\Cc$ and $\Dd$ are presentably symmetric monoidal, $F$ preserves colimits, and $G$ is accessible, then $\operatorname{LEq}(F,G)$ is again presentably monoidal.\label{enum:LEqPresentable}
		\end{alphanumerate}
	\end{lem}
	\begin{proof}[Proof sketch]
		Let $\langle i\rangle \in \cat{Fin}_*$. Using $\operatorname{LEq}(\id_{\{\langle i\rangle\}},\id_{\{\langle i\rangle\}})\simeq *$, the fact that lax equalisers commute with pullbacks, and the fact that the fibres over $F^\otimes$ and $G^\otimes$ over $\langle i\rangle$ are $F^i\colon \Cc^i\rightarrow \Dd^i$ and $G^i\colon \Cc^i\rightarrow \Dd^i$ respectively, we find that the fibre of $\operatorname{LEq}(F,G)^\otimes\rightarrow \cat{Fin}_*$ over $\langle i\rangle\in\cat{Fin}_*$ is of the desired form: 
		\begin{equation*}
			\operatorname{LEq}(F^\otimes,G^\otimes)\times_{\operatorname{LEq}(\id_{\cat{Fin}_*},\id_{\cat{Fin}_*})}\operatorname{LEq}\bigl(\id_{\{\langle i\rangle\}},\id_{\{\langle i\rangle\}}\bigr)\simeq \operatorname{LEq}(F^i,G^i)\simeq \operatorname{LEq}(F,G)^i\,.
		\end{equation*}
		Let us next check that $\operatorname{LEq}(F,G)^\otimes\rightarrow \cat{Fin}_*$ is a cocartesian fibration. For simplicity, we'll only describe locally cocartesian lifts of the unique active morphism $f_2\colon \langle 2\rangle\rightarrow \langle 1\rangle$; it will be obvious how to perform the construction in general, as will be the fact that the locally cocartesian lifts compose, so that we obtain a cocartesian fibration by the dual of \cite[Proposition~\chref{2.4.2.8}]{HTT}. So suppose we're given $((x_1,\varphi_1),(x_2,\varphi_2))\in \operatorname{LEq}(F,G)^2$, where $\varphi_1\colon F(x_1)\rightarrow G(x_1)$ and $\varphi_2\colon F(x_2)\rightarrow G(x_2)$. Let $\varphi$ denote the composite
		\begin{equation*}
			\varphi\colon F(x_1\otimes_\Cc x_2)\simeq F(x_1)\otimes_\Dd F(x_2)\xrightarrow{\varphi_1\otimes \varphi_2}G(x_1)\otimes_\Dd G(x_2)\longrightarrow G(x_1\otimes_\Cc x_2)\,,
		\end{equation*}
		where we use strict and lax symmetric monoidality of $F$ and $G$, respectively. Now let $\mu\colon (x_1,x_2)\rightarrow x_1\otimes_\Cc x_2$ be a locally cocartesian lift of $f_2$ along $\Cc^\otimes\rightarrow \cat{Fin}_*$. Moreover, let $\mu_F\simeq F^\otimes(\mu)\colon (F(x_1),F(x_2))\rightarrow F(x_1)\otimes_\Dd F(x_2)$ and $\mu_G\colon (G(x_1),G(x_2))\rightarrow G(x_1)\otimes_\Dd G(x_2)$ be locally cocartesian lifts of $f_2$ along $\Dd^\otimes\rightarrow \cat{Fin}_*$. We have $\varphi\circ \mu_F\simeq \mu_G\circ(\varphi_1,\varphi_2)$ by construction of $\varphi$, and so we obtain a morphism $((x_1,\varphi_1),(x_2,\varphi_2))\rightarrow (x_1\otimes_\Cc x_2,\varphi)$ in $\operatorname{LEq}(F,G)^\otimes$. Using the formula for mapping anima in lax equalisers from \cite[Proposition~{\chref{2.1.5}[II.1.5]}(ii)]{NikolausScholze} and the general criterion from the dual of \cite[Proposition~\chref{2.4.4.3}]{HTT}, it's straightforward to verify that this morphism is indeed a locally cocartesian lift of $f_2$, as desired.
		
		Therefore, $\operatorname{LEq}(F,G)^\otimes\rightarrow \cat{Fin}_*$ is indeed a cocartesian fibration. From the description of cocartesian lifts above, it's clear that $\operatorname{LEq}(F,G)^\otimes\rightarrow \Cc^\otimes$ preserves cocartesian lifts, hence $\operatorname{LEq}(F,G)\rightarrow \Cc$ is indeed symmetric monoidal. This finishes the proof sketch of \cref{enum:LEqCocartesian}.
		
		For \cref{enum:LEqPresentable}, we must check that $\operatorname{LEq}(F,G)$ is presentable and that the tensor product preserves colimits in either variable. Both assertions follow from \cite[Proposition~{\chref{2.1.5}[II.1.5]}(iv)--(v)]{NikolausScholze}.
	\end{proof}
	
	Let us now commence with the proof of \cref{prop:NaiveCyclonic}. The main ingredient is a formula that allows to compute genuine fixed points for finite cyclic groups in terms of homotopy fixed points, geometric fixed points, and the Tate construction.
	
	\begin{lem}\label{lem:GenuineFromGeometric}
		Let $X$ be a cyclonic spectrum and let $m\in\IN$. If the geometric fixed points $X^{\Phi C_d}$ are bounded below for all divisors $d\mid m$, then the following canonical $(S^1/C_m)$-equivariant map is an equivalence:
		\begin{equation*}
			X^{C_m}\overset{\simeq}{\longrightarrow}\eq\Biggl(\prod_{d\mid m}(X^{\Phi C_d})^{\h C_{m/d}}\overset{\operatorname{can}}{\underset{\phi}{\doublemorphism}}\prod_{p\vphantom{|}}\prod_{pd\mid m}\bigl((X^{\Phi C_d})^{\t C_p}\bigr)^{\h C_{m/pd}}\Biggr)\,.
		\end{equation*}
		Here the second product is taken over all primes~$p$. The two maps $\operatorname{can}$ and $\phi$ in the equaliser are given as follows:
		\begin{align*}
			(X^{\Phi C_d})^{\h C_{m/d}}\simeq \bigl((X^{\Phi C_d})^{\h C_p}\bigr)^{\h C_{m/pd}}&\longrightarrow \bigl((X^{\Phi C_d})^{\t C_p}\bigr)^{\h C_{m/pd}}\,,\\
			(X^{\Phi C_{pd}})^{\h C_{m/pd}}\simeq \bigl((X^{\Phi C_d})^{\Phi C_p}\bigr)^{\h C_{m/pd}}&\longrightarrow\bigl((X^{\Phi C_d})^{\t C_p}\bigr)^{\h C_{m/pd}}\,,
		\end{align*}
		using the natural transformations $(-)^{\h C_p}\Rightarrow (-)^{\t C_p}$ and $(-)^{\Phi C_p}\Rightarrow (-)^{\t C_p}$, respectively.
	\end{lem}
	\begin{proof}
		We use induction on~$m$. If $m=p^\alpha$ is a prime power, the assertion is \cite[Corollary~{\chref{2.4.7}[II.4.7]}]{NikolausScholze}. Now let $m$ be arbitrary. We may assume that all but one prime factors of $m$ act invertibly on $X$, because an arbitrary $X$ can be written as a finite \v Cech limit of such objects (also the assumption that all $X^{\Phi C_d}$ are bounded below is preserved under any localisation). Write $m=p^\alpha m_p$, where $p$ is the not necessarily invertible prime and $m_p$ is coprime to $p$. Using the inductive hypothesis and the fact that the Tate construction $(-)^{\t C_\ell}$ vanishes on $\IS[1/\ell]$-modules, we find
		\begin{equation*}
			X^{C_{m_p}}\simeq \prod_{d_p\mid m_p}\bigl(X^{\Phi C_{d_p}}\bigr)^{\h C_{m_p/d_p}}\,.
		\end{equation*}
		Also observe that all homotopy fixed points $(-)^{\h C_{m_p/d_p}}$ in this formula can be computed as finite limits, as $\B C_{m_p/d_p}$ has a finite cell structure once $m_p$ is invertible. An argument as in \cref{lem:GenuineGeometricCommute}\cref{enum:GenuineGeometricCommuteCoprime} then allows us to deduce that the formula above is also true as genuine $C_{p^\alpha}$-equivariant spectra and that the homotopy fixed points  $(-)^{\h C_{m_p/d_p}}$ commute with the geometric fixed points $(-)^{\Phi C_{p^i}}$. With these observations, the formula for $X^{C_m}\simeq (X^{C_{m_p}})^{C_{p^\alpha}}$ becomes precisely the desired equaliser.
	\end{proof}
	
	With a similar argument, one can show the following technical lemma.
	
	\begin{lem}\label{lem:NaiveCyclonicBoundedBelow}
		Let $Y=((Y_m)_m,(\phi_{p,m})_{p,m})$ be a naive cyclotomic spectrum. Then $Y$ is bounded below if and only if for all $m\in\IN$ the following equaliser is bounded below:
		\begin{equation*}
			\eq\Biggl(\prod_{d\mid m}Y_d^{\h C_{m/d}}\overset{\operatorname{can}}{\underset{\phi}{\doublemorphism}}\prod_{p\vphantom{|}}\prod_{pd\mid m}\bigl(Y_d^{\t C_p}\bigr)^{\h C_{m/pd}}\Biggr)\,.
		\end{equation*}
	\end{lem}
	\begin{proof}
		We only prove the \enquote{only if} part, the \enquote{if} will follow from \cref{prop:NaiveCyclonic} (and won't be used in the proof). So let $Y$ be bounded below. We may once again assume that all but one prime factors of~$m$ act invertibly on $Y$, since the property of being bounded below is preserved under finite \v Cech limits. So write $m=p^\alpha m_p$, where~$p$ is the not necessarily invertible prime and $m_p$ is coprime to~$p$. Since the Tate constructions $(-)^{\t C_\ell}$ vanish for all primes $p\neq \ell$, the equaliser simplifies to
		\begin{equation*}
			\eq\Biggl(\prod_{d\mid m}Y_d^{\h C_{m/d}}\overset{\operatorname{can}}{\underset{\phi_p}{\doublemorphism}}\prod_{pd\mid m}\bigl(Y_d^{\t C_p}\bigr)^{\h C_{m/pd}}\Biggr)
		\end{equation*}
		Let $pd\mid m$ and write $d=p^id_p$, where $i\leqslant \alpha-1$ and $d_p$ is coprime to~$p$. Using the Tate fixed point lemma \cite[Lemma~{\chref{2.4.1}[II.4.1]}]{NikolausScholze}, we find
		\begin{equation*}
			\fib\left(Y_{d}^{\h C_{m/d}}\rightarrow \bigl(Y_{d}^{\t C_p}\bigr)^{\h C_{m/pd}}\right)\simeq \bigl((Y_{d})_{\h C_{p^{\alpha-\smash{i}}}}\bigr)^{\h C_{m/p^\alpha d_p}}\,.
		\end{equation*}
		Since $(-)_{\h C_{p^{\alpha-i}}}$ preserves bounded below objects and $(-)^{\h C_{m/p^\alpha d_p}}$ can be written as a finite limit in our situation, we deduce that the fibre is bounded below. An easy induction shows that the equaliser in question must be bounded below as well.
	\end{proof}
	\begin{proof}[Proof of \cref{prop:NaiveCyclonic}]
		Let us first show that $(-)^{\Phi C}\colon \Cyclonic_+\rightarrow \Cyclonic_+^\mathrm{naiv}$ is fully faithful. For any $m\in\IN$, we have
		\begin{equation*}
			\IS_{S^1}[S^1/C_m]^{\Phi C_d}\simeq \begin{cases*}
				\IS[S^1/C_{m/d}] & if $d\mid m$\\
				0 & else
			\end{cases*}\,.
		\end{equation*}
		By unravelling the general formula for mapping anima/spectra in lax equalisers \cite[Proposition~{\chref{2.1.5}[II.1.5]}(ii)]{NikolausScholze}, we find that $\Hom_{\Cyclonic^\mathrm{naiv}}(\IS_{S^1}[S^1/C_m]^{\Phi C},X^{\Phi C})$ is given by the equaliser from \cref{lem:GenuineFromGeometric} for all cyclonic spectra $X$. If $X$ is bounded below, it follows that
		\begin{equation*}
			\Hom_{\Cyclonic^\mathrm{naiv}}\bigl(\IS_{S^1}[S^1/C_m]^{\Phi C},X^{\Phi C}\bigr)\simeq X^{C_m}\simeq \Hom_{\Cyclonic}\bigl(\IS_{S^1}[S^1/C_m],X\bigr)\,,
		\end{equation*}
		as desired. Since $\Cyclonic$ is generated under colimits by shifts of $\IS_{S^1}[S^1/C_m]$ for all $m\in\IN$, we deduce that $(-)^{\Phi C}\colon \Cyclonic_+\rightarrow \Cyclonic_+^\mathrm{naiv}$ is indeed fully faithful.
		
		Using \cite[Proposition~{\chref{2.1.5}[II.1.5]}(iv)--(v)]{NikolausScholze}, we see that $(-)^{\Phi C}\colon \Cyclonic\rightarrow \Cyclonic^\mathrm{naiv}$ is a colimit-preserving functor between presentable $\infty$-categories and so it admits a right adjoint $R\colon \Cyclonic^\mathrm{naiv}\rightarrow \Cyclonic$. We note that $R$ restricts to a functor $R\colon \Cyclonic_+^\mathrm{naiv}\rightarrow \Cyclonic_+$. Indeed, an analogous computation as above shows that
		\begin{equation*}
			R(Y)^{C_m}\simeq \Hom_{\Cyclonic^\mathrm{naiv}}\bigl(\IS_{S^1}[S^1/C_m]^{\Phi C},Y\bigr)\simeq \eq\Biggl(\prod_{d\mid m}Y_d^{\h C_{m/d}}\overset{\operatorname{can}}{\underset{\phi}{\doublemorphism}}\prod_{p\vphantom{|}}\prod_{pd\mid m}\bigl(Y_d^{\t C_p}\bigr)^{\h C_{m/pd}}\Biggr)
		\end{equation*}
		for all $Y\in \Cyclonic^\mathrm{naiv}$. Thus, if $Y$ is bounded below, \cref{lem:NaiveCyclonicBoundedBelow} shows that $R(Y)$ will be bounded below as well.
		
		The same calculation shows that $R$ is conservative. Indeed, if $Y\rightarrow Y'$ is a morphism of naive cyclonic spectra such that $R(Y)\rightarrow R(Y')$ is an equivalence, then the induced morphisms on the equalisers from \cref{lem:NaiveCyclonicBoundedBelow} are equivalences for all $m\in\IN$. Arguing inductively, this implies that $Y_m\rightarrow Y'_m$ must be an equivalence for all $m\in\IN$ and so $Y\rightarrow Y'$ is indeed an equivalence as well.
		
		In general, if the left adjoint in any adjunction is fully faithful and the right adjoint is conservative, the adjunction is a pair of inverse equivalences. This finishes the proof.
	\end{proof}
	
	\begin{rem}
		Ayala--Mazel-Gee--Rozenblyum derive another \enquote{naive} description of cyclonic spectra in \cite[Corollary~\chref{0.4}]{NaiveGenuine}. In contrast to \cref{prop:NaiveCyclonic}, which is only valid in the bounded below case, their result covers all cyclonic spectra. This comes at a cost of additional coherence data. The moral reason why, in the bounded below case, we can get away with only the maps $X^{ \Phi C_{pm}}\rightarrow (X^{\Phi C_m})^{\t C_p}$, with no coherence data to be specified, is the following: For $X$ bounded below, the composition maps for the \emph{proper} Tate construction are equivalences
		\begin{equation*}
			X^{\tau C_{mn}}\overset{\simeq}{\longrightarrow} (X^{\tau C_m})^{\tau C_n}\,,
		\end{equation*}
		and unless $m$ and $n$ are powers of the same prime, both sides vanish. This determines all coherence data uniquely. We expect that by formalising this observation, one can deduce \cref{prop:NaiveCyclonic} from \cite[Corollary~\chref{0.4}]{NaiveGenuine}, but we have not attempted to do so.
	\end{rem}
	\begin{numpar}[Cyclonic vs.\ cyclotomic spectra.]\label{par:CyclonicVsCyclotomic}
		Let $\cat{CyctSp}$ denote the $\infty$-category of cyclotomic spectra and let $\cat{CyctSp}^\mathrm{naiv}$ denote its naive variant introduced by Nikolaus--Scholze \cite[Definition~{\chref{2.1.6}[II.1.6]}(i)]{NikolausScholze}. We have a symmetric monoidal functor
		\begin{equation*}
			\cat{CyctSp}^\mathrm{naiv}\longrightarrow \Cyclonic^\mathrm{naiv}
		\end{equation*}
		sending a cyclotomic spectrum $X$ to the constant family $((X)_{m},(\phi_{p,m})_{p,m})$ in which each $\phi_{p,m}$ is given by the cyclotomic Frobenius $X\rightarrow X^{\t C_p}$. This functor is not fully faithful (this will become useful in \cref{par:CyclonicBase} below).
		
		One can also construct a functor $\cat{CyctSp}\rightarrow \Cyclonic$ on the non-naive $\infty$-categories (see \cite[\S{\chref[subsection]{2.5}}]{NaiveGenuine} for example) which agrees with the functor above on bounded below objects.
	\end{numpar}
	
	\subsection{Genuine equivariant \texorpdfstring{$\ku$}{ku}}\label{subsec:Genuineku}
	
	In this subsection we'll equip $\ku$ with the structure of a cyclonic spectrum and compute its genuine and geometric fixed points $\ku^{C_m}$ and $\ku^{\Phi C_m}$ for all $m$.
	
	\begin{numpar}[Cyclonic $\ku$.]\label{par:Cyclonicku}
		Recall that Schwede \cite[Construction~\chref{6.3.9}]{SchwedeGlobal} constructs a model $\ku_{\mathrm{gl}}$ of $\ku$ as an \emph{ultracommutative global%
			\footnote{\enquote{Global} in the sense of global homotopy theory, not in the sense of \cref{subsec:qdeRhamkuGlobal}. Very roughly, it means to have compatible trivial actions by all compact Lie groups. \enquote{Ultracommutative} refers to the fact that Schwede's model admits a strictly commutative multiplication on the point-set level.}
			ring spectrum}. Throwing away most of the structure, this yields an $\IE_\infty$-algebra $\ku_{S^1}\in \CAlg(\Sp_{S^1})$ with underlying non-equivariant $\IE_\infty$-algebra $\ku$. We still have a Bott map $\beta\colon \Sigma^2\IS_{S^1}\rightarrow \ku_{S^1}$ (in fact, $\beta$ already exists for $\ku_\mathrm{gl}$) and we define $\KU_{S^1}\coloneqq \ku_{S^1}[\beta^{-1}]$. In the following we'll often abusingly drop the index and just write $\ku$ or $\KU$ for the genuine $S^1$-equivariant versions. We also note that by restriction, $\ku$ and $\KU$ define $\IE_\infty$-algebras in cyclonic spectra.
	\end{numpar}

	\begin{numpar}[Genuine fixed points of $\ku$.]\label{par:kuGenuineFixedPoints}
		Let $q$ denote the standard representation of $S^1$ on $\IC$ via rotations, so that the complex representation rings of $S^1$ and $C_m$ are given by $\operatorname{RU}(S^1)\cong \IZ[q^{\pm 1}]$ and $\operatorname{RU}(C_m)\cong \IZ[q]/(q^m-1)$. Via the canonical map $\operatorname{RU}(S^1)\rightarrow \pi_0(\ku^{S^1})$, we can regard $q$ as a class in $\pi_0(\ku^{S^1})$, compatible with \cref{rem:kutCpNotation}. It's a well-known fact that $q$ is a \emph{strict} element, that is, it is detected by an $\IE_\infty$-algebra map $\IS[q]\rightarrow \ku^{S^1}$. See \cref{cor:qStrictElement} for a proof.
		
		For the finite groups $C_m$ the analogous maps $\operatorname{RU}(C_m)\rightarrow \pi_0(\ku^{C_m})$ are isomorphisms \cite[Theorem~\chref{6.3.33}]{SchwedeGlobal} and so, by equivariant Bott periodicity,
		\begin{equation*}
			\pi_*(\ku^{C_m})\cong \IZ[\beta,q]/(q^m-1)\quad\text{and}\quad \pi_*(\KU^{C_m})\cong \IZ[\beta^{\pm 1},q]/(q^m-1)\,.
		\end{equation*}
		In particular, $\ku^{C_m}\simeq \tau_{\geqslant 0}(\KU^{C_m})$. Using the homotopy fixed point spectral sequence, we can also compute the homotopy fixed points of the residual $(S^1/C_m)$-action:
		\begin{equation*}
			\pi_*\bigl((\ku^{C_m})^{\h (S^1/C_m)}\bigr)\cong \IZ[\beta,q]\llbracket t_m\rrbracket/\bigl(\beta t_m-(q^m-1)\bigr)\,,
		\end{equation*}
		where $\abs{t_m}=-2$. The canonical map $(\ku^{C_m})^{\h (S^1/C_m)}\rightarrow \ku^{\h S^1}$ sends $t_m\mapsto [m]_q t$. In particular, on $\pi_0$ this map recovers the $(q-1)$-completion $\IZ[q]_{(q^m-1)}^\complete\rightarrow \IZ\qpower$, and $t_m=[m]_{\ku}(t)$ agrees with the $m$-series of the formal group law of $\ku$.
	\end{numpar}
	
	\begin{numpar}[Inflation maps for $\ku$.]\label{par:kuInflations}
		Consider the inflation maps from \cref{par:Inflations} in the special case where $\varphi$ is the $n$\textsuperscript{th} power map $(-)^n\colon S^1\rightarrow S^1$ for some $n\geqslant 1$. We have $\varphi^*\ku_{S^1}\simeq \ku_{S^1}$, since the genuine $S^1$-equivariant structure comes from a global spectrum $\ku_\mathrm{gl}$, where all actions are trivial (compare \cite[\S{\chref[section]{4.1}}]{SchwedeGlobal}). Since $(-)^n$ maps the subgroups $C_{mn}$ to $C_m$, we get inflations
		\begin{equation*}
			\inf_n\colon \ku^{C_m}\longrightarrow \ku^{C_{mn}}\quad\text{and}\quad\inf_n\colon \ku^{\Phi C_m}\longrightarrow \ku^{\Phi C_{mn}}\,.
		\end{equation*} 
		These are maps of $\IE_\infty$-algebras in $\Sp_{S^1}$ for the residual genuine $S^1\simeq S^1/C_m$-equivariant structure on the left-hand sides and the residual $S^1\simeq S^1/C_{mn}$-equivariant structure on the right-hand sides. A straightforward check shows $\inf_n(q)=q^n$ and $\inf_n(\beta)=\beta$ (compare \cref{par:EquivariantBott}).
	\end{numpar}
	\begin{cor}\label{cor:kuCm}
		For all $m$ and $n$, the inflation map induces an $S^1$-equivariant equivalence of $\IE_\infty$-algebras
		\begin{equation*}
			\inf_n\colon \ku^{C_m}\otimes_{\IS[q],\psi^n}\IS[q]\longrightarrow \ku^{C_{mn}}\,,
		\end{equation*}
		where $\psi^n\colon \IS[q]\rightarrow\IS[q]$ is given by $\psi^n(q)\coloneqq q^n$.
	\end{cor}
	\begin{proof}
		This can be checked on homotopy groups, where it follows from \cref{par:kuGenuineFixedPoints} and~\cref{par:kuInflations}.
	\end{proof}
	\begin{rem}
		The notation $\psi^n\colon \IS[q]\rightarrow \IS[q]$ is chosen to be compatible with the Adams operations on the $\Lambda$-ring $\IZ[q]$. One can also construct equivariant Adams operations on $\ku$ (see \cref{par:EquivariantAdamsOperations}), but these \emph{do not} coincide with $\inf_n$.
	\end{rem}
	
	\begin{numpar}[Geometric fixed points of $\ku$.]
		To prove our Habiro descent result, it will be crucial to know the geometric fixed points $\ku^{\Phi C_m}$ as well, at least after inverting~$m$ and after $p$-completion for any prime $p\mid m$. This will be our goal for the rest of this subsection. Our strategy will be to compute the geometric fixed points inductively using \cref{lem:GenuineFromGeometric}. To apply said lemma, observe that we already know that each $\ku^{\Phi C_m}$ is bounded below thanks to \cref{lem:GenuineGeometricBoundedBelow}.
		
		For $\KU$, the geometric fixed points can essentially already be found in the literature (even though the author could only find the precise result in the case where $m$ is a prime power): We have an equivalence of $S^1$-equivariant $\IE_\infty$-ring spectra
		\begin{equation*}
			\KU^{C_m}\Bigl[\bigl\{(q^d-1)^{-1}\bigr\}_{d\mid m,\, d\neq m}\Bigr]\overset{\simeq}{\longrightarrow}\KU^{\Phi C_m}\,.
		\end{equation*}
		One way to prove this is via the corresponding statement for equivariant $\MU$ \cite[Proposition~\href{https://arxiv.org/pdf/math/9910024\#[{"num"\%3A88\%2C"gen"\%3A0}\%2C{"name"\%3A"FitH"}\%2C317]}{4.6}]{SinhaEquivariantMU} and the equivariant Conner--Floyd theorem \cite{EquivariantConnerFloyd}. The result can also be deduced from \cref{prop:kuPhiCm} below.
	\end{numpar}

	\begin{lem}\label{lem:kuPhiCmLocalisation}
		The canonical map $\ku[1/m]^{C_m}\rightarrow \ku[1/m]^{\Phi C_m}$ induces an equivalence of $S^1$-equivariant $\IE_\infty$-ring spectra
		\begin{equation*}
			\bigl(\ku\bigl[\localise{m}\bigr]^{C_m}\bigr)_{\Phi_m(q)}^\complete\overset{\simeq}{\longrightarrow}\ku\bigl[\localise{m}\bigr]^{\Phi C_m}\,.
		\end{equation*}
		In particular, $\pi_*(\ku[1/m]^{\Phi C_m})\cong \IZ[1/m,\beta,q]/\Phi_m(q)$.
	\end{lem}
	\begin{proof}
		Since we already know that $\ku^{\Phi C_d}$ is bounded below for all $d\mid m$, we can apply the formula from \cref{lem:GenuineFromGeometric} to $\ku[1/m]$. Because we've inverted~$m$, all Tate constructions will vanish, and the formula becomes an equivalence
		\begin{equation*}
			\ku\bigl[\localise{m}\bigr]^{C_m}\simeq \prod_{d\mid m}\ku\bigl[\localise{m}\bigr]^{\Phi C_d}\,.
		\end{equation*}
		The claim then follows via induction on $m$ and \cref{cor:kuCm}.
	\end{proof}

	\begin{lem}\label{lem:kuPhiCmpCompletion}
		Let $m=p^\alpha m_p$, where $p$ is a prime and $m_p$ is coprime to $p$.
		The inflation map induces an $S^1$-equivariant equivalence of $\IE_\infty$-ring spectra
		\begin{equation*}
			\inf_{m_p}\colon \Bigl(\bigl(\ku^{\Phi C_{p^\alpha}}\bigr)_p^\complete\otimes_{\IS[q],\psi^{m_p}}\IS[q]\Bigr)_{\Phi_m(q)}^\complete\overset{\simeq}{\longrightarrow} \bigl(\ku^{\Phi C_m}\bigr)_p^\complete\,.
		\end{equation*}
	\end{lem}
	\begin{proof}
		Note that $m_p$ is invertible on $(\ku^{\Phi C_{p^\alpha}})_p^\complete$. The same argument as in the proof of \cref{lem:kuPhiCmLocalisation} shows that the canonical map
		\begin{equation*}
			\bigl((\ku^{\Phi C_{p^\alpha}})_p^\complete\bigr)^{C_{m_p}}\simeq \bigl((\ku^{\Phi C_{p^\alpha}})^{C_{m_p}}\bigr)_p^\complete\longrightarrow \bigl((\ku^{\Phi C_{p^\alpha}})^{\Phi C_{m_p}}\bigr)_p^\complete\simeq \bigl(\ku^{\Phi C_m}\bigr)_p^\complete
		\end{equation*}
		exhibits the target as the $(p,\Phi_{m}(q))$-completion of the source. It remains to show that inflation induces an equivalence $(\ku^{\Phi C_{p^\alpha}})_p^\complete\otimes_{\IS[q],\psi^{m_p}}\IS[q]\simeq ((\ku^{\Phi C_{p^\alpha}})^{C_{m_p}})_p^\complete$. As both sides are $p$-complete, this can be checked modulo~$p$. Moreover, note that with geometric fixed points replaced by genuine fixed points, this would follow from \cref{cor:kuCm}. In fact, applying \cref{cor:kuCm} to $\ku^{C_{p^i}}$ for all $i\leqslant \alpha$, we deduce that
		\begin{equation*}
			\inf_{m_p}\colon \ku/p\otimes_{\IS[q],\psi^{m_p}}\IS[q]\overset{\simeq}{\longrightarrow}(\ku/p)^{C_{m_p}}
		\end{equation*}
		is an equivalence of genuine $C_{p^\alpha}$-equivariant spectra, as it induces equivalences on genuine fixed points for all subgroups (see \cref{par:GenuineFixedPoints}). Then it must induce equivalences on geometric fixed points as well, which proves what we want.
	\end{proof}
	
	
	\begin{lem}\label{lem:kuPhiCp}
		For all primes~$p$ and all $\alpha\geqslant 1$, the following assertions are true.
		\begin{alphanumerate}
			\item The canonical map $\ku^{\Phi C_{p^\alpha}}\rightarrow (\ku^{\Phi C_{p^{\alpha-\smash{1}}}})^{\t C_p}$ induces an $S^1$-equivariant equivalence of $\IE_\infty$-ring spectra\label{enum:kuPhipalpha}
			\begin{equation*}
				\bigl(\ku^{\Phi C_{p^\alpha}}\bigr)_p^\complete\overset{\simeq}{\longrightarrow} \tau_{\geqslant 0}\bigl((\ku^{\Phi C_{p^{\alpha-1}}})^{\t C_p}\bigr)\,.
			\end{equation*}
			\item On homotopy groups, we have $\pi_*((\ku^{\Phi C_{p^\alpha}})_p^\complete)\cong \IZ_p[u_{p^\alpha},q]/\Phi_{p^\alpha}(q)$ where $\abs{u_{p^\alpha}}=2$, and\label{enum:kuPhipalphaHomotopy}
			\begin{equation*}
				\pi_*\Bigl(\bigl((\ku^{\Phi C_{p^\alpha}})_p^\complete\bigr)^{\h (S^1/C_{p^\alpha})}\Bigr)\cong \IZ_p[u_{p^\alpha},q]\llbracket t_{p^\alpha}\rrbracket/\bigl(u_{p^\alpha}t_{p^\alpha}-\Phi_{p^\alpha}(q)\bigr)\,.
			\end{equation*}
			With notation as in \cref{par:kuGenuineFixedPoints}, the canonical map $(\ku^{C_{p^\alpha}})^{\h (S^1/C_{p^\alpha})}\rightarrow ((\ku^{\Phi C_{p^\alpha}})_p^\complete)^{\h (S^1/C_{p^\alpha})}$ sends $q\mapsto q$, $t_{p^\alpha}\mapsto t_{p^\alpha}$, and $\beta\mapsto (q^{p^{\alpha-1}}-1)u_{p^\alpha}$.
			\item The inflation map induces an equivalence of $S^1$-equivariant $\IE_\infty$-ring spectra\label{enum:kuPhipalphaInflation}
			\begin{equation*}
				\inf_{p^{\alpha-1}}\colon \left(\ku^{\Phi C_p}\otimes_{\IS[q],\psi^{p^{\alpha-\smash{1}}}}\IS[q]\right)_p^\complete\overset{\simeq}{\longrightarrow} \bigl(\ku^{\Phi C_{p^\alpha}}\bigr)_p^\complete\,.
			\end{equation*}
		\end{alphanumerate}
	\end{lem}
	\begin{proof}
		We show all three assertions at once using induction on $\alpha$. In general, using \cref{lem:GenuineFromGeometric}, or more directly the iterated pullback diagram from \cite[Corollary~{\chref{2.4.7}[II.4.7]}]{NikolausScholze}, we obtain a pullback square
		\begin{equation*}
			\begin{tikzcd}
				\ku^{C_{p^\alpha}}\dar\rar\drar[pullback] & \ku^{\Phi C_{p^\alpha}}\dar\\
				\bigl(\ku^{C_{p^{\alpha-1}}}\bigr)^{\h C_p}\rar & \bigl(\ku^{\Phi C_{p^{\alpha-1}}}\bigr)^{\t C_p}
			\end{tikzcd}
		\end{equation*}
		For $\alpha=1$, we see that $\ku^{C_{p}}\rightarrow \ku^{\h C_p}$ induces an equivalence $(\ku^{C_{p}})_p^\complete\simeq \tau_{\geqslant 0}(\ku^{\h C_p})_p^\complete$, and $(\ku^{\h C_p})_p^\complete\rightarrow \ku^{\t C_p}$ is an equivalence in homotopical degrees $\leqslant -1$. From the pullback square we deduce that $(\ku^{\Phi C_p})_p^\complete\simeq \tau_{\geqslant 0}(\ku^{\t C_p})$, proving \cref{enum:kuPhipalpha}. Assertion~\cref{enum:kuPhipalphaHomotopy} for $\alpha=1$ is then a standard calculation; see \cite[Proposition~\chref{3.3.1}]{DevalapurkarRaksitTHH} for example. Assertion~\cref{enum:kuPhipalphaInflation} is tautological for $\alpha=1$. For the inductive step, let $\alpha\geqslant 2$. We claim that
		\begin{align*}
			\pi_*\bigl(\ku^{C_{p^\alpha}}\bigr)&\cong \pi_*\bigl(\ku^{C_p}\bigr)\otimes_{\IZ[q],\psi^{p^{\alpha-\smash{1}}}}\IZ[q]\,,\\
			\pi_*\bigl((\ku^{C_{p^{\alpha-1}}})^{\h C_p}\bigr)&\cong \pi_*\bigl(\ku^{\h C_p}\bigr)\otimes_{\IZ[q],\psi^{p^{\alpha-\smash{1}}}}\IZ[q]\,,\\
			\pi_*\bigl((\ku^{\Phi C_{p^{\alpha-1}}})^{\t C_p}\bigr)&\cong \pi_*\bigl(\ku^{\t C_p}\bigr)\otimes_{\IZ[q],\psi^{p^{\alpha-\smash{1}}}}\IZ[q]  
		\end{align*}
		Indeed, the first two isomorphism follow from \cref{cor:kuCm} and the third one from \cref{enum:kuPhipalphaHomotopy} for $\ku^{\Phi C_{p^{\alpha-1}}}$, which we already know by induction. Then \cref{enum:kuPhipalphaHomotopy} and \cref{enum:kuPhipalphaInflation} follow immediately from the pullback square above. Moreover, we see that the vertical map $\ku^{C_{p^\alpha}}\rightarrow (\ku^{C_{p^{\smash{\alpha-1}\vphantom{\alpha}}}})^{\h C_p}$ induces an equivalence $(\ku^{C_{p^\alpha}})_p^\complete\simeq \tau_{\geqslant 0}((\ku^{C_{p^{\smash{\alpha-1}\vphantom{\alpha}}}})^{\h C_p})_p^\complete$, and that after $p$-completion the horizontal map $((\ku^{C_{p^{\smash{\alpha-1}\vphantom{\alpha}}}})^{\h C_p})_p^\complete\rightarrow (\ku^{\Phi C_{p^{\smash{\alpha-1}\vphantom{\alpha}}}})^{\t C_p}$ is an equivalence in homotopical degrees $\leqslant -1$. As in the case $\alpha=1$, this implies \cref{enum:kuPhipalpha}.
	\end{proof}
	\begin{rem}\label{rem:kutCpCyclotomic}
		Let $p>2$, so that $\THH(\IZ_p[\zeta_p]/\IS_p\qpower)_p^\complete\simeq \tau_{\geqslant 0}(\ku^{\t C_p})$ holds as $S^1$-equivariant $\IE_\infty$-ring spectra by \cref{thm:kutCp}. As a consequence of \cref{lem:kuPhiCp}\cref{enum:kuPhipalpha}, we get an equivalence
		\begin{equation*}
			\THH\bigl(\IZ_p[\zeta_p]/\IS_p\qpower\bigr)_p^\complete\simeq \bigl(\ku^{\Phi C_p}\bigr)_p^\complete
		\end{equation*}
		of $S^1$-equivariant $\IE_\infty$-ring spectra.
		
		But we can say even more. Devalapurkar shows in \cite[Theorem~\chref{6.4.1}]{DevalapurkarSpherochromatism} that the equivalence $\THH(\IZ_p[\zeta_p]/\IS_p\qpower)_p^\complete\simeq \tau_{\geqslant 0}(\ku^{\t C_p})$ holds as \emph{cyclotomic} $\IE_\infty$-ring spectra, where $\tau_{\geqslant 0}(\ku^{\t C_p})$ is equipped with the cyclotomic structure induced from the trivial cyclotomic structure on $\ku$ (see \cite[Construction~\chref{1.1.3}]{DevalapurkarRaksitTHH}). Since the inflation maps for $\ku$ are similarly induced via the trivial $S^1$-action on the global ultracommutative ring spectrum $\ku_\mathrm{gl}$, we see that the cyclotomic Frobenius
		\begin{equation*}
			\phi_p\colon \tau_{\geqslant 0}\bigl(\ku^{\t C_p}\bigr)\longrightarrow \bigl(\tau_{\geqslant 0}(\ku^{\t C_p})\bigr)^{\t C_p}
		\end{equation*}
		agrees, up to passing to the connective cover in the target, with $\inf_p\colon (\ku^{\Phi C_p})_p^\complete\rightarrow (\ku^{\Phi C_{p^2}})_p^\complete$, as maps of $S^1$-equivariant $\IE_\infty$-ring spectra. Therefore we obtain a commutative diagram
		\begin{equation*}
			\begin{tikzcd}[column sep=large]
				\bigl(\ku^{\Phi C_p}\bigr)_p^\complete\otimes_{\IS[q],\psi^p}\IS[q]\rar["\inf_p"]\dar["\simeq"'] & \bigl(\ku^{\Phi C_{p^2}}\bigr)_p^\complete\dar\\
				\THH\bigl(\IZ_p[\zeta_p]/\IS_p\qpower\bigr)_p^\complete\otimes_{\IS[q],\psi^p}\IS[q]\rar["\phi_{p/\IS[q]}"] & \THH\bigl(\IZ_p[\zeta_p]/\IS_p\qpower\bigr)^{\t C_p}
			\end{tikzcd}
		\end{equation*}
		of $S^1$-equivariant $\IE_\infty$-ring spectra.
	\end{rem}
	
	For our purposes, the description of $\ku^{\Phi C_m}$ that we get from \crefrange{lem:kuPhiCmLocalisation}{lem:kuPhiCp} would be enough, but for the sake of completeness, let us deduce a complete computation of $\pi_*(\ku^{\Phi C_m})$.
	
	\begin{prop}\label{prop:kuPhiCm}
		Let $m\in \IN$. For all divisors $d\mid m$ let $[d]_{\ku}(t)=\beta^{-1}(q^d-1)$ denote the $d$-series of the formal group law of $\ku$. Then
		\begin{equation*}
			\pi_*\bigl(\ku^{\Phi C_m}\bigr)\cong \IZ[\beta,t]/[m]_{\ku}(t)\Bigl[\bigl\{[d]_{\ku}(t)^{-1}\bigr\}_{d\mid m,d\neq m}\Bigr]^{\geqslant 0}\,,
		\end{equation*}
		where $(-)^{\geqslant 0}$ on the right-hand side denotes the restriction to non-negative graded degrees.
	\end{prop}
	\begin{proof}[Proof sketch]
		We use the arithmetic fracture square (see \cref{par:Notation}) for $\ku^{\Phi C_m}$:
		\begin{equation*}
			\begin{tikzcd}
				\ku^{\Phi C_m}\rar\dar\drar[pullback] & \prod_{p\mid m}\bigl(\ku^{\Phi C_m}\bigr)_p^\complete\dar\\
				\ku\bigl[\localise{m}\bigr]^{\Phi C_m}\rar & \prod_{p\mid m}\bigl(\ku^{\Phi C_m}\bigr)_p^\complete\bigl[\localise{p}\bigr]
			\end{tikzcd}
		\end{equation*}
		Using \crefrange{lem:kuPhiCmLocalisation}{lem:kuPhiCp}, one readily checks that the right vertical and bottom horizontal maps are jointly surjective on $\pi_*$. Therefore, we also get a pullback on $\pi_*$. It is then straightforward to construct a map $\IZ[\beta,t]/[m]_{\ku}(t)\bigl[\{[d]_{\ku}(t)^{-1}\}_{d\mid m,d\neq m}\bigr]_{\geqslant 0}\rightarrow \pi_*(\ku^{\Phi C_m})$. Whether this map is an equivalence can be checked after localising~$m$ and after $p$-completion for all $p\mid m$, which is again straightforward via \crefrange{lem:kuPhiCmLocalisation}{lem:kuPhiCp}.
	\end{proof}
	
	\subsection{Cyclonic even filtrations and Habiro descent of \texorpdfstring{$q$}{q}-Hodge complexes}\label{subsec:GenuineHabiroDescent}
	
	Let $A$ and $R$ be rings that satisfy the assumptions from \cref{par:NewAssumptions} and assume that $2\in R^\times$ (so that the addendum~\cref{enum:AssumptionOnR2} is automatically satisfied as well). In this subsection, we'll finally explain how to obtain $\qHhodge_{R/A}$ from a cyclonic structure on $\THH(\KU_R/\KU_A)$.

	To this end, let us first discuss how to equip $\THH(\ku_R/\ku_A)\simeq \THH(\IS_R/\IS_A)\otimes\ku$ with a suitable cyclonic structure. At first, one would expect that the cyclonic structure on $\THH(\IS_R/\IS_A)$ coming from its cyclotomic structure via \cref{par:CyclonicVsCyclotomic} would do the job. But it doesn't! For example, the constructions in \cite[\S{\chref[section]{3}}]{qWittHabiro} are all $A[q]$-linear. By contrast, the canonical map $\THH(\IS_R/\IS_A)^{\Phi C_p}\rightarrow \THH(\IS_R/\IS_A)^{\t C_p}$, which by definition agrees with the cyclotomic Frobenius, is not $\IS_A$-linear; instead, it is semilinear over the Tate-valued Frobenius $\phi_{\t C_p}\colon \IS_A\rightarrow \IS_A^{\t C_p}$. It is thus unclear how one would construct an $A[q]$-linear structure on the associated graded of some even filtration on $\THH(\ku_R/\ku_A)^{C_p}$.
	
	\begin{numpar}[Cyclonic structure on $\THH(\ku_R/\ku_A)$.]\label{par:CyclonicBase}
		To fix this, we need to modify the cyclonic structure on $\THH(\IS_R/\IS_A)$. This requires yet another assumption on $A$.
		\begin{alphanumerate}\itshape 
			\item[A_2] Let $\IS_A^{\mathrm{cyct}}$ and $\IS_A^\mathrm{triv}$ denote the cyclonic structures on $\IS_A$ given by the cyclotomic structure from \cref{par:NewAssumptions}\cref{enum:AssumptionsOnAglobal} and the trivial cyclotomic structure, respectively. Then we must assume that there exists a map\label{enum:AssumptionOnA2}
			\begin{equation*}
				\IS_A^\mathrm{cyct}\longrightarrow \IS_A^\mathrm{triv}
			\end{equation*}
			of $\IE_\infty$-algebras in $\Cyclonic$%
			\footnote{Beware that there may be more maps as cyclonic $\IE_\infty$-algebras than as cyclotomic $\IE_\infty$-algebras.}
			whose underlying map of $S^1$-equivariant $\IE_\infty$-algebras is the identity on $\IS_A$, equipped with the trivial action.
		\end{alphanumerate}
		Now let $\THH(\IS_R/\IS_A)^\mathrm{cyct}$ denotes the cyclonic structure on $\THH(\IS_R/\IS_A)$ coming from the usual cyclotomic structure. Assuming~\cref{enum:AssumptionOnA2}, we can instead consider the following cyclonic structure:
		\begin{equation*}
			\THH(\IS_R/\IS_A)^\mathrm{cyct}\otimes_{\IS_A^\mathrm{cyct}}\IS_A^\mathrm{triv}\,.
		\end{equation*}
		We'll then regard $\THH(\ku_R/\ku_A)\simeq \THH(\IS_R/\IS_A)\otimes\ku$ as a cyclonic spectrum in the apparent way, using the above cyclonic structure on $\THH(\IS_R/\IS_A)$ as well as the cyclonic structure on $\ku$ from \cref{par:Cyclonicku}. As we'll see, this has the desired properties.
		
		Let us unravel Assumption~\cref{enum:AssumptionOnA2}. Since both $\IS_A^{\mathrm{cyct}}$ and $\IS_A^\mathrm{triv}$ are cyclotomic spectra, we have $(\IS_A^\mathrm{cyct})^{\Phi C_m}\simeq \IS_A$ and $(\IS_A^\mathrm{triv})^{\Phi C_m}\simeq \IS_A$ for all $m$, identifying the residual $(S^1/C_m)$-action with the trivial $S^1$-action on $\IS_A$. In particular, after taking $(-)^{\Phi C_m}$, a map $\IS_A^\mathrm{cyct}\rightarrow \IS_A^\mathrm{triv}$ induces $S^1$-equivariant $\IE_\infty$-maps $\psi^m\colon \IS_A\rightarrow \IS_A$ that fit into commutative diagrams
		\begin{equation*}
			\begin{tikzcd}[column sep=large]
				\IS_A\rar["\psi^{pm}"]\dar["\phi_p"'] & \IS_A\dar\\
				\IS_A^{\t C_p}\rar["(\psi^m)^{\t C_p}"] & \IS_A^{\t C_p}
			\end{tikzcd}
		\end{equation*}
		for all $m\in\IN$ and all primes~$p$. It follows inductively that $\psi^m\colon \IS_A\rightarrow \IS_A$ must be a lift of the $\Lambda$-ring Adams operation $\psi^m\colon A\rightarrow A$.
	\end{numpar}
	
	\begin{lem}\label{lem:A2FrobeniusLifts}
		The data of $S^1$-equivariant $\IE_\infty$-maps $\psi^m\colon \IS_A\rightarrow \IS_A$ together with commutative diagrams as above uniquely determines a map $\IS_A^\mathrm{cyct}\rightarrow \IS_A^\mathrm{triv}$ of cyclonic $\IE_\infty$-algebras.
	\end{lem}
	\begin{proof}
		By \cref{prop:NaiveCyclonic} we may equivalently construct $\IS_A^\mathrm{cyct}\rightarrow\IS_A^\mathrm{triv}$ as a map of $\IE_\infty$-algebras in naive cyclonic spectra. It's clear from the construction in \cref{lem:LEqSymmetricMonoidal} that
		\begin{equation*}
			\CAlg\left(\Cyclonic^\mathrm{naiv}\right)\simeq \operatorname{LEq}\Biggl(\begin{tikzcd}[cramped,column sep=huge]
				\prod_{m\in\IN}\CAlg\bigl(\Sp^{\B(S^1/C_m)}\bigr)\rar[shift left=0.2em,"\operatorname{can}"]\rar[shift right=0.2em,"((-)^{\t C_p})_{p,m}"'] & \prod_{p\vphantom{\IN}}\prod_{m\in\IN}\CAlg\bigl(\Sp^{\B(S^1/C_{pm})}\bigr)
			\end{tikzcd}\biggr)\,,
		\end{equation*}
		and so the given data indeed uniquely determines such a map.
	\end{proof}
	
	We'll verify in \cref{subsec:CyclotomicBases} that in all examples we can construct, Assumption~\cref{enum:AssumptionOnA2} is satisfied as well. This concludes our discussion of the cyclonic structure on $\THH(\ku_R/\ku_A)$. For convenience, let us also introduce the following notation.
	\begin{defi}\label{def:ThisShouldBeTR}
		For all $m\in\IN$, the \emph{$m$\textsuperscript{th} topological cyclonic homology of $\ku_R$ over $\ku_A$} is the spectrum
		\begin{equation*}
			\TCn{m}(\ku_R/\ku_A)\coloneqq \left(\THH(\ku_R/\ku_A)^{C_m}\right)^{\h (S^1/C_m)}\,.
		\end{equation*}
	\end{defi}

	
	\begin{numpar}[Cyclonic even filtrations in general.]\label{par:CyclonicEvenFiltration}
		Let $T$ be a cyclonic $\IE_1$-algebra and let $M$ be a cyclonic left $T$-module. Suppose that $T$ and $M$ are bounded below and that for all $m\in\IN$ the geometric fixed points $T^{\Phi C_m}$ are complex orientable (but we don't require any genuine equivariant or cyclonic complex orientation). In this situation, we expect that the correct filtration to put on $M^{\Phi C_m}$ is simply the non-equivariant even perfect filtration $\fil_{\ev}^\star M^{\Phi C_m}\coloneqq \fil_{\Pev/T^{\Phi C_m}}^\star M^{\Phi C_m}$ of $M^{\Phi C_m}$ as a left module over $T^{\Phi C_m}$. Moreover, the genuine $C_m$-fixed points should be equipped with the filtration
		\begin{equation*}
			\fil_{\ev/T,C_m}^\star M^{C_m}\coloneqq \eq\Biggl(\prod_{d\mid m}\bigl(\fil_{\ev}^\star M^{\Phi C_d}\bigr)^{\h C_{m/d,\ev}}\overset{\operatorname{can}}{\underset{\phi}{\doublemorphism}}\prod_{p\vphantom{|}}\prod_{pd\mid m}\bigl((\fil_{\ev}^\star M^{\Phi C_d})^{\t C_{p,\ev}}\bigr)^{\h C_{m/pd,\ev}}\Biggr)\,.
		\end{equation*}
		Here $(-)^{\h C_{m/d,\ev}}$, $(-)^{t C_{p,\ev}}$, and $(-)^{\h C_{m/pd,\ev}}$ refer to the filtered fixed points and Tate construction defined \cite[\S{\chref[subsection]{2.3}}]{CyclotomicSynthetic}.%
		\footnote{This needs the residual $S^1$-actions, so as stated the formula above only applies in the cyclonic setting but not in the genuine $C_m$-equivariant setting.}
		The map $\operatorname{can}$ in the equaliser is induced by the natural transformation $(-)^{\h C_{p,\ev}}\Rightarrow (-)^{\t C_{p,\ev}}$ and the map $\phi$ is induced by the canonical maps 
		\begin{equation*}
			\fil_{\Pev/T^{\Phi C_{pd}}}^\star \bigl((M^{\Phi C_d})^{\t C_p}\bigr)\longrightarrow\fil_{\Pev/(T^{\Phi C_{d}})^{\t C_p}}^\star \bigl((M^{\Phi C_d})^{\t C_p}\bigr)\longrightarrow \bigl(\fil_{\Pev/T^{\Phi C_d}}^\star M^{\Phi C_d}\bigr)^{\t C_{p,\ev}}
		\end{equation*}
		using the construction from~\cref{par:EvenFiltrationTate}. To apply this construction, we need the additional assumption that $(M^{\Phi C_m})^{\h C_p}$ is homologically even over $(T^{\Phi C_m})^{\h C_p}$; this is certainly satisfied in the case $M=T$ that is relevant for us.
		
		A genuine equivariant version of the even filtration is currently in the works; for example, the author has been informed of (independent) work in progress by Jeremy Hahn and Lucas Piessevaux. We have little doubt that in the foreseeable future, an intrinsically defined genuine equivariant even filtration will be available and we expect that for $M$ as above (maybe subject to some extra assumptions), the true even filtration will agree with our formula.
	\end{numpar}
	\begin{numpar}[Cyclonic even filtrations on $\THH(\ku_R/\ku_A)$.]\label{par:CyclonicEvenFiltrationTHH}
		Put $R^{(m)}\coloneqq R\lotimes_{A,\psi^m}A$. Note that $R^{(m)}$ is static, since the Adams operation $\psi^m$ is flat in any perfectly covered $\Lambda$-ring. Moreover, $R^{(m)}$ satisfies the assumptions from \cref{par:NewAssumptions}\cref{enum:AssumptionsOnRglobal}; in particular, it admits a spherical lift given by $\IS_{R^{(m)}}\coloneqq \IS_R\otimes_{\IS_A,\psi^m}\IS_A$, where $\psi^m\colon \IS_A\rightarrow \IS_A$ is the lift of the $\Lambda$-ring Adams operation from \cref{par:CyclonicBase}. We may thus define $\fil_{\ev}^\star\THH(\ku_{R^{(m)}}/\ku_A)$ via \cref{par:GlobalEvenFiltration}. Via base change along the inflation $\inf_m\colon \ku\rightarrow \ku^{\Phi C_m}$, we may then equip the geometric fixed points $\THH(\ku_R/\ku_A)^{\Phi C_m}$ with the filtration
		\begin{equation*}
			\fil_{\ev}^\star \THH(\ku_R/\ku_A)^{\Phi C_m}\coloneqq \fil_{\ev}^\star \THH\left(\ku_{R^{(m)}}/\ku_A\right)\otimes_{\ku_{\ev}}\ku^{\Phi C_m}_{\ev}\,,
		\end{equation*}
		where $\ku^{\Phi C_m}_{\ev}\coloneqq \tau_{\geqslant 2\star}(\ku^{\Phi C_m})$ denotes the double-speed Whitehead filtration. We'll check in \cref{lem:GeometricEvenFiltration} below that this agrees with the usual perfect even filtration on $\THH(\ku_R/\ku_A)^{\Phi C_m}$, as long as the latter is defined. Next, we construct the filtration on genuine fixed points
		\begin{equation*}
			\fil_{\ev,C_m}^\star\THH(\ku_R/\ku_A)^{C_m}
		\end{equation*}
		via the formula in \cref{par:CyclonicEvenFiltration}. Finally, using the notation introduced in \cref{def:ThisShouldBeTR}, we define
		\begin{equation*}
			\fil_{\ev,S^1}^\star\TCn{m}(\ku_R/\ku_A)\coloneqq \left(\fil_{\ev,C_m}^\star\THH(\ku_R/\ku_A)^{C_m}\right)^{\h (\IT/C_m)_{\ev}}\,,
		\end{equation*}
		where $(-)^{\h (\IT/C_m)_{\ev}}$ denotes fixed points in the sense of \cite[\S{\chref[subsection]{2.3}}]{CyclotomicSynthetic} with respect to the even filtration on $\IS[S^1/C_m]$.
	\end{numpar}
	Here are two sanity checks:
	
	\begin{lem}\label{lem:GeometricEvenFiltration}
		Suppose we chose condition~\cref{par:AssumptionsOnR}\cref{enum:E2Lift} for all primes~$p$, so that $\ku_R$ is an $\IE_2$-algebra in $\ku_A$-modules. Then $\fil_{\ev}^\star\THH(\ku_R/\ku_A)^{\Phi C_m}$ agrees with Pstr\k{a}gowski's perfect even filtration on the $\IE_1$-ring $\THH(\ku_R/\ku_A)^{\Phi C_m}$.
	\end{lem}
	\begin{proof}[Proof sketch]
		We know from \cref{lem:GlobalEvenFiltration} that $\fil_{\ev}\THH(\ku_{R^{(m)}}/\ku_A)$ agrees with Pstr\k{a}gowski's perfect even filtration. It will thus be enough to show that the canonical base change map
		\begin{equation*}
			\fil_{\Pev}^\star \THH\left(\ku_{R^{(m)}}/\ku_A\right)\otimes_{\ku_{\ev}}\ku^{\Phi C_m}_{\ev}\longrightarrow \fil_{\Pev}^\star \THH(\ku_R/\ku_A)^{\Phi C_m}
		\end{equation*}
		is an equivalence. It's enough to check this on associated gradeds as both sides are exhaustive filtrations on $\THH(\ku_{R^{(m)}}/\ku_A)\otimes_\ku\ku^{\Phi C_m}\simeq \THH(\ku_R/\ku_A)^{\Phi C_m}$. Now on associated gradeds (and in fact, one the nose) both sides can be computed by a cosimplicial resolution as in \cref{prop:EvenResolution}, because $\THH(\IS_P)\rightarrow \IS_P$ is faithfully even flat. We can then use a similar argument as in \cref{cor:EvenFiltrationBaseChange} to show the desired base change equivalence. Here we use that $\ku^{\Phi C_m}$ is even with $p$-torsion free homotopy groups for all primes~$p$ by \cref{prop:kuPhiCm}.
	\end{proof}
	\begin{lem}\label{lem:GenuineEvenFiltrationTHHkuCompleteExhaustive}
		For all $m\in\IN$,
		\begin{equation*}
			\fil_{\ev,C_m}^\star\THH(\ku_R/\ku_A)^{C_m}\quad\text{and}\quad\fil_{\ev,S^1}^\star\TCn{m}(\ku_R/\ku_A)
		\end{equation*}
		are complete exhaustive filtrations on $\THH(\ku_R/\ku_A)^{C_m}$ and $\TCn{m}(\ku_R/\ku_A)$, respectively.
	\end{lem}
	\begin{proof}[Proof sketch]
		For completeness, apply \cite[Lemma~\chref{2.75}(iv)]{CyclotomicSynthetic} to each of the constituents of the equaliser from \cref{par:CyclonicEvenFiltration}. The only non-obvious thing to check is that $(\fil_{\ev}^\star\THH(\ku_R/\ku_A)^{\Phi C_d})^{\t C_{p,ev}}$ is complete, which follows from an argument as in~\cref{par:kuComparisonII}. For exhaustiveness, apply \cite[Lemma~\chref{2.75}(iv)]{CyclotomicSynthetic} to each of the constituents in the equaliser from \cref{par:CyclonicEvenFiltration}. To see that this lemma applies, one can use \cref{cor:Bifiltration}.
	\end{proof}
	
	We can now formulate the main result of \cref{sec:Genuine}.
	
	\begin{numpar}[The twisted $q$-Hodge filtration.]\label{par:ReturnOfTwistedqHodge}
		The $q$-Hodge filtration $\fil_{\qHodge}^\star\qdeRham_{R/A}$ from \cref{thm:qdeRhamkuGlobal} can be plugged into the construction from \cite[\chref{3.38}]{qWitt} to obtain the \emph{twisted $q$-Hodge filtration}
		\begin{equation*}
			\fil_{\qHhodge_m}^\star\qdeRham_{R/A}^{(m)}\,.
		\end{equation*}
		We will relate this to $\gr_{\ev,S^1}^*\TCn{m}(\ku_R/\ku_A)$. To this end, we must first explain how the latter acquires a filtered structure.
		
		Observe that $\fil_{\ev,S^1}^\star\TCn{m}(\ku/\ku)\simeq\tau_{\geqslant 2\star}((\ku^{C_m})^{\h (S^1/C_m)})$. This computation is not completely trivial, but it can be done in the same way as \cref{thm:qdeRhamkuGenuine} below.%
		\footnote{In fact, it is almost a special case of that theorem, except that $2\notin\IZ^\times$. Even so, to formulate the theorem properly, we need this special case first.}
		As a consequence, we see that
		\begin{equation*}
			\Sigma^{-2*}\gr_{\ev,S^1}^*\TCn{m}(\ku_R/\ku_A)
		\end{equation*}
		is a module over the graded ring $\IZ[\beta,q]\llbracket t_m\rrbracket/(\beta t_m-(q^m-1))\cong \pi_{2*}((\ku^{C_m})^{\h (S^1/C_m)})$ (see \cref{par:kuGenuineFixedPoints}). Regarding $t_m$ as the filtration parameter, this graded ring can be identified with the $(q^m-1)$-adic filtration $(q^m-1)^\star\IZ[q]_{(q^m-1)}^\complete$.
	\end{numpar}

	\begin{thm}\label{thm:qdeRhamkuGenuine}
		Let $m\in\IN$. Suppose $A$ and $R$ satisfy the assumptions from~\cref{par:NewAssumptions} along with the addenda $2\in R^\times$ and~\cref{par:CyclonicBase}\cref{enum:AssumptionOnA2}. Then there exists a canonical equivalence of filtered $\IZ[\beta,q]\llbracket t_m\rrbracket/(\beta t_m-(q^m-1))$-modules
		\begin{equation*}
			\fil_{\qHhodge_m}^\star \qhatdeRham_{R/A}^{(m)}\overset{\simeq}{\longrightarrow}\Sigma^{-2*}\gr_{\ev,S^1}^*\TC^{-(m)}(\ku_R/\ku_A)\,,
		\end{equation*}
		where the left-hand side denotes the completion of the twisted $q$-Hodge filtration $\fil_{\qHhodge_m}^\star\qdeRham_{R/A}^{(m)}$ from \cref{par:ReturnOfTwistedqHodge} and the right-hand side is defined in \cref{par:CyclonicEvenFiltrationTHH}.
	\end{thm}
	
	To show \cref{thm:qdeRhamkuGenuine}, we'll decompose $\Sigma^{-2*}\gr_{\ev,S^1}^*\TC^{-(m)}(\ku_R/\ku_A)$ into a fracture square and match it up with the construction from~\cite[\chref{3.38}]{qWitt}.
	
	\begin{numpar}[Fracture squares for even filtrations.]\label{par:EvenFiltrationFractureSquares}
		Let $N$ be a positive integer. We construct an even filtration
		\begin{equation*}
			\fil_{\ev}^\star\THH\bigl(\ku_R\bigl[\localise{N}\bigr]/\ku_A\bigl[\localise{N}\bigr]\bigr)
		\end{equation*}
		as in \cref{par:GlobalEvenFiltration}, except that we replace every occurence of $\ku$ by a $\ku[1/N]$. Moreover, for any prime~$p$ we let
		\begin{equation*}
			\fil_{\ev}^\star\THH_\solid\bigl(\ku_{\smash{\widehat{R}}_p}/\ku_{\smash{\widehat{A}}_p}\bigr)\,,\quad\text{and}\quad \fil_{\ev}^\star\THH_\solid\bigl(\ku_{\smash{\widehat{R}}_p}\bigl[\localise{p}\bigr]/\ku_{\smash{\widehat{A}}_p}\bigl[\localise{p}\bigr]\bigr)
		\end{equation*}
		be the even filtrations given by applying \cref{par:EvenFiltration} for $k=\ku$ and $k=\ku[1/p]$, respectively. By construction, we then have a pullback square
		\begin{equation*}
			\begin{tikzcd}
				\fil_{\ev}^\star\THH(\ku_R/\ku_A)\rar\dar\drar[pullback] & \prod_{p\mid N} \fil_{\ev}^\star\THH_\solid\bigl(\ku_{\smash{\widehat{R}}_p}/\ku_{\smash{\widehat{A}}_p}\bigr)\dar\\
				\fil_{\ev}^\star\THH\bigl(\ku_R\bigl[\localise{N}\bigr]/\ku_A\bigl[\localise{N}\bigr]\bigr)\rar & \prod_{p\mid N}\fil_{\ev}^\star\THH_\solid\bigl(\ku_{\smash{\widehat{R}}_p}\bigl[\localise{p}\bigr]/\ku_{\smash{\widehat{A}}_p}\bigl[\localise{p}\bigr]\bigr)
			\end{tikzcd}
		\end{equation*}
		A similar fracture square exists for the geometric $C_m$-fixed points. To this end, replace $R$ by $R^{(m)}$ in the above construction and apply the base change $-\otimes_{\ku_{\ev}}\ku_{\ev}^{\Phi C_m}$ to obtain
		\begin{gather*}
			\fil_{\ev}^\star\THH\bigl(\ku_R\bigl[\localise{N}\bigr]/\ku_A\bigl[\localise{N}\bigr]\bigr)^{\Phi C_m}\,,\\
			\fil_{\ev}^\star\THH_\solid\bigl(\ku_{\smash{\widehat{R}}_p}/\ku_{\smash{\widehat{A}}_p}\bigr)^{\Phi C_m}\quad\text{and}\quad \fil_{\ev}^\star\THH_\solid\bigl(\ku_{\smash{\widehat{R}}_p}\bigl[\localise{p}\bigr]/\ku_{\smash{\widehat{A}}_p}\bigl[\localise{p}\bigr]\bigr)^{\Phi C_m}\,.
		\end{gather*}
		These fit into a pullback square
		\begin{equation*}
			\begin{tikzcd}
				\fil_{\ev}^\star\THH(\ku_R/\ku_A)^{\Phi C_m}\rar\dar\drar[pullback] & \prod_{p\mid N} \fil_{\ev}^\star\THH_\solid\bigl(\ku_{\smash{\widehat{R}}_p}/\ku_{\smash{\widehat{A}}_p}\bigr)^{\Phi C_m}\dar\\
				\fil_{\ev}^\star\THH\bigl(\ku_R\bigl[\localise{N}\bigr]/\ku_A\bigl[\localise{N}\bigr]\bigr)^{\Phi C_m}\rar & \prod_{p\mid N}\fil_{\ev}^\star\THH_\solid\bigl(\ku_{\smash{\widehat{R}}_p}\bigl[\localise{p}\bigr]/\ku_{\smash{\widehat{A}}_p}\bigl[\localise{p}\bigr]\bigr)^{\Phi C_m}
			\end{tikzcd}
		\end{equation*}
		We also note that if we define the \emph{$\infty$-category of cyclonic solid condensed spectra} as the Lurie tensor product $\Cyclonic\otimes\Sp_\solid$, then $\THH_\solid(\ku_{\smash{\widehat{R}}_p}/\ku_{\smash{\widehat{A}}_p})$ and $\THH_\solid(\ku_{\smash{\widehat{R}}_p}[1/p]/\ku_{\smash{\widehat{A}}_p}[1/p])$ can be equipped with cyclonic solid condensed structures as in \cref{par:CyclonicBase} and so the expressions $\THH_\solid(\ku_{\smash{\widehat{R}}_p}/\ku_{\smash{\widehat{A}}_p})^{\Phi C_m}$ and $\THH_\solid(\ku_{\smash{\widehat{R}}_p}[1/p]/\ku_{\smash{\widehat{A}}_p}[1/p])^{\Phi C_m}$ make sense. Finally, the constructions from \cref{par:CyclonicEvenFiltrationTHH} can also be applied in this setting, and so we obtain
		\begin{gather*}
			\fil_{\ev,S^1}^\star\TCn{m}\bigl(\ku_R\bigl[\localise{N}\bigr]/\ku_A\bigl[\localise{N}\bigr]\bigr)\,,\\
			\fil_{\ev,S^1}^\star\TCn{m}_\solid\bigl(\ku_{\smash{\widehat{R}}_p}/\ku_{\smash{\widehat{A}}_p}\bigr)\quad\text{and}\quad \fil_{\ev,S^1}^\star\TCn{m}_\solid\bigl(\ku_{\smash{\widehat{R}}_p}\bigl[\localise{p}\bigr]/\ku_{\smash{\widehat{A}}_p}\bigl[\localise{p}\bigr]\bigr)\,,
		\end{gather*}
		which fit into a pullback square
		\begin{equation*}
			\begin{tikzcd}
				\fil_{\ev,S^1}^\star\TCn{m}(\ku_R/\ku_A)\rar\dar\drar[pullback] & \prod_{p\mid N} \fil_{\ev,S^1}^\star\TCn{m}_\solid\bigl(\ku_{\smash{\widehat{R}}_p}/\ku_{\smash{\widehat{A}}_p}\bigr)\dar\\
				\fil_{\ev,S^1}^\star\TCn{m}\bigl(\ku_R\bigl[\localise{N}\bigr]/\ku_A\bigl[\localise{N}\bigr]\bigr)\rar & \prod_{p\mid N}\fil_{\ev,S^1}^\star\TCn{m}_\solid\bigl(\ku_{\smash{\widehat{R}}_p}\bigl[\localise{p}\bigr]/\ku_{\smash{\widehat{A}}_p}\bigl[\localise{p}\bigr]\bigr)
			\end{tikzcd}
		\end{equation*}
	\end{numpar}

	
	We will now analyse this pullback. Let us begin with the part where $N$ is invertible.
	
	\begin{lem}\label{lem:qdeRhamkuGeometricBaseChange}
		Suppose $N$ is divisible by $m$. Then the inflation map $\inf_m\colon \ku\rightarrow \ku^{\Phi C_m}$ induces a filtered $S^1$-equivariant \embrace{or more precisely, $\IT_{\ev}$-module} equivalence
		\begin{equation*}
			\Bigl(\fil_{\ev}^\star \THH\bigl(\ku_R^{(m)}\bigl[\localise{N}\bigr]/\ku_A\bigl[\localise{N}\bigr]\bigr)\otimes_{\IS[q],\psi^m}\IS[q]\Bigr)_{\Phi_m(q)}^\complete\overset{\simeq}{\longrightarrow}\fil_{\ev}^\star\THH\bigl(\ku_R\bigl[\localise{N}\bigr]/\ku_A\bigl[\localise{N}\bigr]\bigr)^{\Phi C_m}\,.
		\end{equation*}
	\end{lem}
	\begin{proof}
		Observe that the $\Phi_m(q)$-adic completion is just the projection to the $m$\textsuperscript{th} factor in the decomposition
		\begin{equation*}
			\IS\bigl[\localise{N},q\bigr]/(q^m-1)\simeq \prod_{d\mid m}\IS\bigl[\localise{N},q\bigr]/\Phi_d(q)\,.
		\end{equation*}
		The claim then follows from \cref{lem:kuPhiCmLocalisation} and the definition of $\fil_{\ev}^\star \THH(\ku_R[1/N]/\ku_A[1/N])$ and $\fil_{\ev}^\star\THH_\solid(\ku_{\smash{\widehat{R}}_p}[1/p]/\ku_{\smash{\widehat{A}}_p}[1/p]\bigr)$ as base changes along $-\otimes_{\ku_{\ev}}\ku_{\ev}^{\Phi C_m}$.
	\end{proof}

	Let us now analyse the $p$-adic part.
	
	\begin{lem}\label{lem:qdeRhamkuGeometricpAdicLocalisation}
		For all primes~$p$, the inflation map $\inf_m\colon \ku\rightarrow\ku^{\Phi C_m}$ induces a filtered $S^1$-equivariant \embrace{or more precisely, $\IT_{\ev}$-module} equivalence
		\begin{equation*}
			\Bigl(\fil_{\ev}^\star \THH_\solid\bigl(\ku_{\smash{\widehat{R}_p^{(m)}}}\bigl[\localise{p}\bigr]/\ku_{\smash{\widehat{A}}_p}\bigl[\localise{p}\bigr]\bigr)\otimes_{\IS[q],\psi^m}\IS[q]\Bigr)_{\Phi_m(q)}^\complete\overset{\simeq}{\longrightarrow}\fil_{\ev}^\star\THH_\solid\bigl(\ku_{\smash{\widehat{R}}_p}\bigl[\localise{p}\bigr]/\ku_{\smash{\widehat{A}}_p}\bigl[\localise{p}\bigr]\bigr)^{\Phi C_m}\,.
		\end{equation*}
	\end{lem}
	\begin{proof}
		Analogous to \cref{lem:qdeRhamkuGeometricBaseChange}.
	\end{proof}
	
	\begin{lem}\label{lem:qdeRhamkuGeometricpAdicBaseChange}
		Write $m=p^\alpha m_p$, where $p$ is a prime and $m_p$ is coprime to $p$. Then the inflation map $\inf_{m_p}\colon \ku^{\Phi C_{p^\alpha}}\rightarrow \ku^{\Phi C_m}$ induces a filtered $S^1$-equivariant \embrace{or more precisely, $\IT_{\ev}$-module} equivalence
		\begin{equation*}
			\Bigl(\fil_{\ev}^\star \THH_\solid\bigl(\ku_{\smash{\widehat{R}_p^{(m_p)}}}/\ku_{\smash{\widehat{A}}_p}\bigr)^{\Phi C_{p^\alpha}}\otimes_{\IS[q],\psi^{m_p}}\IS[q]\Bigr)_{\Phi_m(q)}^\complete\overset{\simeq}{\longrightarrow}\fil_{\ev}^\star \THH_\solid\bigl(\ku_{\smash{\widehat{R}}_p}/\ku_{\smash{\widehat{A}}_p}\bigr)^{\Phi C_m}\,.
		\end{equation*}
	\end{lem}
	\begin{proof}
		As in the proof of \cref{lem:qdeRhamkuGeometricBaseChange}, observe that the $\Phi_m(q)$-adic completion, which agrees with $\Phi_{m_p}(q)$-adic completion as everything is already $p$-complete, is just a projection to the $m_p^{\text{th}}$ factor in the product decomposition
		\begin{equation*}
			\bigl(\IS[q]/(q^m-1)\bigr)_p^\complete\simeq \prod_{d_p\mid m_p}\bigl(\IS_p[q]/(q^m-1)\bigr)_{(p,\Phi_{d_p}(q))}^\complete\,.
		\end{equation*}
		The claim then follows from the constructions and \cref{lem:kuPhiCmpCompletion}.
	\end{proof}
	
	\begin{lem}\label{lem:qdeRhamkuGeometricpAdic}
		In the case $m=p^\alpha$, where $p>2$ is a prime and $\alpha\geqslant 1$, we have a canonical equivalence of filtered $\IZ_p[u_{p^\alpha},q]\llbracket t_{p^\alpha}\rrbracket/(u_{p^\alpha} t_{p^\alpha}-\Phi_{p^\alpha}(q))$-modules
		\begin{equation*}
			\fil_{\Nn}^\star\bigl(\qdeRham_{R/A}^{(p^\alpha)}\bigr)_{(p,\Nn)}^\complete \overset{\simeq}{\longrightarrow} \Sigma^{-2*}\gr^\star\left( \bigl(\fil_{\ev}^*\THH_\solid\bigl(\ku_{\smash{\widehat{R}}_p}/\ku_{\smash{\widehat{A}}_p}\bigr)^{\Phi C_{p^\alpha}}\bigr)^{\h(\IT/C_{p^\alpha})_{\ev}}\right)\,.
		\end{equation*}
	\end{lem}
	\begin{proof}
		We'll explain the case $\alpha=1$; the general case will follow from an analogous argument using \cref{lem:kuPhiCp}\cref{enum:kuPhipalphaInflation}. Let $\widehat{R}_p^{(p)}$, $\IS_{\smash{\widehat{R}_p^{(p)}}}$, and $\ku_{\smash{\widehat{R}_p^{(p)}}}$ denote the $p$-completions of $R^{(p)}$, $\IS_{R^{(p)}}$, and $\ku_{R^{(p)}}$, respectively. By \cref{rem:kutCpCyclotomic}, $(\ku^{\Phi C_p})_p^\complete \simeq \THH_\solid(\IZ_p[\zeta_p]/\IS_p\qpower)$, and so we get $S^1$-equivariant equivalences
		\begin{equation*}
			\THH_\solid\bigl(\ku_{\smash{\widehat{R}}_p}/\ku_{\smash{\widehat{A}}_p}\bigr)^{\Phi C_p}\simeq \THH_\solid \bigl(\IS_{\smash{\widehat{R}_p^{(p)}}}/\IS_{\smash{\widehat{A}}_p}\bigr)\soltimes\ku^{\Phi C_p}\simeq \THH_\solid\bigl(\widehat{R}_p^{(p)}[\zeta_p]/\IS_{\smash{\widehat{A}}_p}\qpower\bigr)\,.
		\end{equation*}
		This also induces an equivalence of $S^1$-equivariant even filtrations
		\begin{equation*}
			\left(\fil_{\ev}^\star \THH_\solid\bigl(\ku_{\smash{\widehat{R}}_p}/\ku_{\smash{\widehat{A}}_p}\bigr)^{\Phi C_p}\right)^{\h (\IT/C_p)_{\ev}}\simeq \fil_{\HRWev,\h S^1}^\star \TC^-\bigl(\widehat{R}_p^{(p)}[\zeta_p]/\IS_{\smash{\widehat{A}}_p}\qpower\bigr)_p^\complete\,.
		\end{equation*}
		Indeed, depending on whether we are in case~\cref{par:AssumptionsOnR}\cref{enum:E1Lift} or~\cref{enum:E2Lift}, the given resolution $\widehat{R}_p\rightarrow \widehat{R}_{p,\infty}^\bullet$ or the resolution from \cref{prop:EvenResolution} will also compute the Hahn--Raksit--Wilson even filtration. By \cref{prop:qdeRhamTC-} and \cref{par:qdeRhamViaTC-}, the associated graded
		\begin{equation*}
			\Sigma^{-2*}\gr_{\HRWev,\h S^1}^*\TC^-\bigl(\widehat{R}_p^{(p)}[\zeta_p]/\IS_{\smash{\widehat{A}}_p}\qpower\bigr)_p^\complete\simeq \fil_\Nn^\star \bigl(\qdeRham_{R/A}^{(p)}\bigr)_{(p,\Nn)}^\complete
		\end{equation*}
		is the completion of the Nygaard filtration on $\bigl(\qdeRham_{R/A}^{(p)}\bigr)_p^\complete$, as desired.
	\end{proof}
	
	
	\begin{lem}\label{lem:qdeRhamkuGenuinepAdic}
		In the case $m=p^\alpha$, where $p>2$ is a prime and $\alpha\geqslant 1$, we have a canonical equivalence of filtered $\IZ_p[\beta,q]\llbracket t_{p^\alpha}\rrbracket/(\beta t_{p^\alpha}-(q^{p^\alpha}-1))$-modules
		\begin{equation*}
			\fil_{\qHhodge_{p^\alpha}}^\star\bigl(\qhatdeRham_{R/A}^{(p^\alpha)}\bigr)_p^\complete \overset{\simeq}{\longrightarrow} \Sigma^{-2*}\gr_{\ev,S^1}^* \TCn{p^\alpha}_\solid\bigl(\ku_{\smash{\widehat{R}}_p}/\ku_{\smash{\widehat{A}}_p}\bigr)\,.
		\end{equation*}
	\end{lem}
	\begin{proof}
		We use induction on $\alpha$. Unravelling the equaliser from \cref{par:CyclonicEvenFiltration} in the case $m=p^\alpha$ provides us with a pullback diagram
		\begin{equation*}
			\begin{tikzcd}
				\fil_{\ev,S^1}^\star\TCn{p^\alpha}_\solid\bigl(\ku_{\smash{\widehat{R}}_p}/\ku_{\smash{\widehat{A}}_p}\bigr)\rar\dar\drar[pullback] & \left(\fil_{\ev}^\star \THH_\solid\bigl(\ku_{\smash{\widehat{R}}_p}/\ku_{\smash{\widehat{A}}_p}\bigr)^{\Phi C_{p^\alpha}}\right)^{\h (\IT/C_{p^\alpha})_{\ev}}\dar\\
				\fil_{\ev,S^1}^\star\TCn{p^{\alpha-1}}_\solid\bigl(\ku_{\smash{\widehat{R}}_p}/\ku_{\smash{\widehat{A}}_p}\bigr)\rar & \left(\bigl(\fil_{\ev}^\star\THH_\solid\bigl(\ku_{\smash{\widehat{R}}_p}/\ku_{\smash{\widehat{A}}_p}\bigr)^{\Phi C_{p^{\alpha-1}}}\bigr)^{\t C_{p,\ev}}\right)^{\h(\IT/C_{p^\alpha})_{\ev}}
			\end{tikzcd}
		\end{equation*}
		Let us first consider the case $\alpha=1$. In this case the bottom left corner of the diagram above is just $\fil_{\ev,\h S^1}^\star\TC_\solid^-(\ku_{\smash{\widehat{R}}_p}/\ku_{\smash{\widehat{A}}_p})$, whose associated graded is $\fil_{\qHodge}^*(\qhatdeRham_{R/A})_p^\complete$ by \cref{thm:qdeRhamkupComplete}. The argument in \cref{par:kuComparisonII} shows that the bottom right corner can be identified with $\fil_{\ev,\t S^1}^\star\TP_\solid(\ku_{\smash{\widehat{R}}_p}/\ku_{\smash{\widehat{A}}_p})$, whose associated graded is $(\qhatdeRham_{R/A})_p^\complete$ in every degree. The associated graded of the top right corner has been computed in \cref{lem:qdeRhamkuGeometricpAdic}. We conclude that the associated graded of the pullback diagram above will be of the form
		\begin{equation*}
			\begin{tikzcd}
				\Sigma^{-2*}\gr_{\ev,S^1}^*\TCn{p}_\solid\bigl(\ku_{\smash{\widehat{R}}_p}/\ku_{\smash{\widehat{A}}_p}\bigr)\rar\dar\drar[pullback] & \fil_\Nn^\star \bigl(\qdeRham_{R/A}^{(p)}\bigr)_{(p,\Nn)}^\complete\dar["\phi_{p/A[q]}"]\\
				\fil_{\qHodge}^\star\bigl(\qhatdeRham_{R/A}\bigr)_p^\complete\rar & \bigl(\qhatdeRham_{R/A}\bigr)_p^\complete
			\end{tikzcd}
		\end{equation*}
		By \cref{par:qdeRhamFrobenii} and the construction of the comparison map in \cref{par:kuComparisonI}--\cref{par:kuComparisonII}, we see that the right vertical map is indeed the relative Frobenius $\phi_{p/A[q]}$ on $q$-de Rham cohomology.
		
		The filtered structure on $\fil_\Nn^\star (\qdeRham_{R/A}^{(p)})_{(p,\Nn)}^\complete$ comes from the structure as a graded module over $\IZ_p[u_p,q]\llbracket t\rrbracket/(u_pt_p-\Phi_p(q))$, whereas the filtered structure on $\fil_{\qHodge}^*(\qdeRham_{R/A})_p^\complete$ and the constant filtration on $(\qdeRham_{R/A})_p^\complete$ are presented as graded $\IZ[\beta]\llbracket t\rrbracket$-modules. Changing the filtration parameter from $t$ to $t_p=\Phi_p(q)t$ has the effect of \enquote{rescaling} filtrations by $\Phi_p(q)$ as in \cite[\chref{3.32}]{qWitt}. The resulting diagram almost looks like the completion of the defining pullback of $\fil_{\qHhodge_p}^\star(\qdeRham_{R/A}^{(p)})_p^\complete$, except for the following subtlety: The rescaled filtrations
		\begin{equation*}
			\Phi_p(q)^\star \fil_{\qHodge}^\star(\qdeRham_{R/A})_p^\complete\quad\text{and}\quad\Phi_p(q)^\star(\qdeRham_{R/A})_p^\complete
		\end{equation*}
		are already complete, so  $\Phi_p(q)^\star\fil_{\qHodge}^\star(\qhatdeRham_{R/A})_p^\complete$ and $\Phi_p(q)^\star(\qhatdeRham_{R/A})_p^\complete$ are \emph{not} the completions of these filtrations. To see that the pullback above still yields the completion of $\fil_{\qHhodge_p}^\star(\qdeRham_{R/A}^{(p)})_p^\complete$, just observe that the pullback
		\begin{equation*}
			\begin{tikzcd}
				\fil_{\qHodge}^\star\left(\qdeRham_{R/A}\right)_p^\complete\dar\rar\drar[pullback] & \left(\qdeRham_{R/A}\right)_p^\complete\dar\\
				\fil_{\qHodge}^\star\bigl(\qhatdeRham_{R/A}\bigr)_p^\complete\rar & \bigl(\qhatdeRham_{R/A}\bigr)_p^\complete
			\end{tikzcd}
		\end{equation*}
		stays a pullback after rescaling everything by $\Phi_p(q)$. This is clear since rescaling preserves all limits. This concludes the proof in the case $\alpha=1$.
		
		Now let $\alpha\geqslant 2$. Using a similar argument as in~\cref{par:kuComparisonII}, we see that the associated graded of $((\fil_{\ev}^\star\THH_\solid(\ku_{\smash{\widehat{R}}_p}/\ku_{\smash{\widehat{A}}_p})^{\Phi C_{p^{\alpha-1}}})^{\t C_{p,\ev}})^{\h(\IT/C_{p^\alpha})_{\ev}}$ is given by $\bigl(\qdeRham_{R/A}^{(p^{\alpha-1})}\bigr)_{(p,\Nn)}^\complete$ in every degree. Thus, the associated graded of the pullback diagram from the beginning of the proof will take the form
		\begin{equation*}
			\begin{tikzcd}
				\Sigma^{-2*}\gr_{\ev,S^1}^*\TCn{p^\alpha}_\solid\bigl(\ku_{\smash{\widehat{R}}_p}/\ku_{\smash{\widehat{A}}_p}\bigr)\rar\dar\drar[pullback] & \fil_\Nn^\star \bigl(\qdeRham_{R/A}^{(p^\alpha)}\bigr)_{(p,\Nn)}^\complete\dar["\phi_{p/A[q]}"]\\
				\fil_{\qHhodge_{p^{\alpha-1}}}^\star\bigl(\qhatdeRham_{R/A}^{(p^{\alpha-1})}\bigr)_p^\complete\rar & \bigl(\qdeRham_{R/A}^{(p^{\alpha-1})}\bigr)_{(p,\Nn)}^\complete
			\end{tikzcd}
		\end{equation*}
		Again, changing the filtration parameter from $t_{p^{\alpha-1}}$ to $t_{p^\alpha}$ introduces a \enquote{rescaling} by $\Phi_{p^\alpha}(q)$ in the bottom row. The resulting diagram looks almost like the completion of the defining pullback of $\fil_{\qHhodge_{p^\alpha}}^\star \bigl(\qdeRham_{R/A}^{(p^\alpha)}\bigr)_p^\complete$, except that again the rescaled filtrations are already complete. To fix this and to finish the proof, it will be enough to check that the diagram
		\begin{equation*}
			\begin{tikzcd}
				\fil_{\qHhodge_{p^{\alpha-1}}}^\star \bigl(\qdeRham_{R/A}^{(p^{\alpha-1})}\bigr)_p^\complete\dar\rar\drar[pullback] & \fil_\Nn^\star\bigl(\qdeRham_{R/A}^{(p^{\alpha-1})}\bigr)_p^\complete\rar\dar\drar[pullback] & \bigl(\qdeRham_{R/A}^{(p^{\alpha-1})}\bigr)_{p}^\complete\dar\\
				\fil_{\qHhodge_{p^{\alpha-1}}}^\star \bigl(\qhatdeRham_{R/A}^{(p^{\alpha-1})}\bigr)_p^\complete\rar  & \fil_\Nn^\star\bigl(\qdeRham_{R/A}^{(p^{\alpha-1})}\bigr)_{(p,\Nn)}^\complete \rar & \bigl(\qdeRham_{R/A}^{(p^{\alpha-1})}\bigr)_{(p,\Nn)}^\complete
			\end{tikzcd}
		\end{equation*}
		consists of two pullback squares (so that we still get a pullback after rescaling the outer rectangle by $\Phi_{p^\alpha}(q)$). Now the right square is a pullback since every filtration is the pullback of its completion. To see that the left square is a pullback, we observe that in the definition of $\fil_{\qHhodge_{p^{\smash{\alpha-1}}}}^\star \bigl(\qdeRham_{R/A}^{(p^{\alpha-1})}\bigr)_p^\complete$ the only occuring non-complete filtration is $\fil_\Nn^\star\bigl(\qdeRham_{R/A}^{(p^{\alpha-1})}\bigr)_p^\complete$, as the other two filtrations are rescaled by $\Phi_{p^{\alpha-1}}(q)$ and thus automatically complete.
	\end{proof}
	\begin{proof}[Proof sketch of \cref{thm:qdeRhamkuGenuine}]
		We analyse the factors of the last fracture square from \cref{par:EvenFiltrationFractureSquares} in the case where~$N$ is divisible by~$m$ and check that they match up with those from \cite[\chref{3.38}]{qWitt}.
		\begin{alphanumerate}
			\item Once we invert~$N$, all filtered Tate constructions $(-)^{\t C_{p,\ev}}$ for $p\mid m$ will vanish, using that the non-filtered Tate construction $(-)^{\t C_p}$ vanishes on $\IS[1/p]$-modules plus an argument as in \cref{par:kuComparisonII}. So the equaliser from \cref{par:CyclonicEvenFiltration} will just be a product. Together with \cref{lem:qdeRhamkuGeometricBaseChange}, we conclude that $\fil_{\ev,S^1}^\star \TCn{m}(\ku_R[1/N]/\ku_A[1/N])$ is the product\label{enum:EvenFiltrationFractureSquareA}
			\begin{equation*}
				\prod_{d\mid m}
				\Bigl(\fil_{\ev}^\star \TC^-\bigl(\ku_R\bigl[\localise{N}\bigr]/\ku_A\bigl[\localise{N}\bigr]\bigr)\otimes_{\IS[q],\psi^d}\IS[q]\Bigr)_{\Phi_d(q)}^\complete
			\end{equation*}
			and therefore $\Sigma^{-2*}\gr_{\ev,S^1}^* \TCn{m}(\ku_R[1/N]/\ku_A[1/N])$ is the completion of the filtered $\IZ[\beta,q]\llbracket t_m\rrbracket/(\beta t_m-(q^m-1))$-module
			\begin{equation*}
				\prod_{d\mid m}\Bigl(\fil_{\qHodge}^\star \qdeRham_{R/A}\lotimes_{A[q],\psi^d}A\bigl[\localise{N},q\bigr]\Bigr)_{\Phi_d(q)}^\complete\,.
			\end{equation*}
			\item A similar analysis as in \cref{enum:EvenFiltrationFractureSquareA} shows that $\Sigma^{-2*}\gr_{\ev,S^1}^* \TCn{m}(\ku_{\smash{\widehat{R}}_p}[1/p]/\ku_{\smash{\widehat{A}}_p}[1/p])$ is the completion of the filtered $\IZ[\beta,q]\llbracket t_m\rrbracket/(\beta t_m-(q^m-1))$-module\label{enum:EvenFiltrationFractureSquareB}
			\begin{equation*}
				\prod_{d\mid m}\Bigl(\fil_{\qHodge}^\star \qdeRham_{R/A}\lotimes_{A[q],\psi^d}A[q]\Bigr)_p^\complete\bigl[\localise{p}\bigr]_{\Phi_d(q)}^\complete\,.
			\end{equation*}
			\item After $p$-completion for any $p\mid N$, we observe as in \cref{enum:EvenFiltrationFractureSquareA} that all filtered Tate constructions $(-)^{\t C_{\ell,ev}}$ vanish for $\ell\neq p$. Simplifying the equaliser accordingly and using \cref{lem:qdeRhamkuGeometricpAdicBaseChange}, we find that $\fil_{\ev,S^1}^\star\TCn{m}(\ku_{\smash{\widehat{R}}_p}/\ku_{\smash{\widehat{A}}_p})$ is given by the product\label{enum:EvenFiltrationFractureSquareC}
			\begin{equation*}
				\prod_{d_p\mid m_p}\left(\fil_{\ev,S^1}^\star\TCn{p^\alpha}\bigl(\ku_{\smash{\widehat{R}_p^{(d_p)}}}/\ku_{\smash{\widehat{A}}_p}\bigr)\otimes_{\IS[q],\psi^{d_p}}\IS[q]\right)_{\Phi_{d_p}(q)}^\complete\,,
			\end{equation*}
			where we put $m=p^\alpha m_p$ with $m_p$ coprime to~$p$. Using \cref{lem:qdeRhamkuGeometricpAdic}, we deduce that the sheared associated graded $\Sigma^{-2*}\gr_{\ev,S^1}^*\TCn{m}(\ku_{\smash{\widehat{R}}_p}/\ku_{\smash{\widehat{A}}_p})$ is the completion of the filtered $\IZ[\beta,q]\llbracket t_m\rrbracket/(\beta t_m-(q^m-1))$-module
			\begin{equation*}
				\prod_{d_p\mid m_p}\left(\fil_{\qHhodge_{p^\alpha}}^\star \bigl(\qdeRham_{R/A}^{(p^\alpha)}\bigr)_p^\complete\lotimes_{A[q],\psi^{d_p}}A[q]\right)_{(p,\Phi_{d_p}(q))}^\complete
			\end{equation*}
		\end{alphanumerate}
		Evidently, \cref{enum:EvenFiltrationFractureSquareA}--\cref{enum:EvenFiltrationFractureSquareC} above match up with  \cite[\chref{3.38}({\chref[Item]{30}[$a$]})--({\chref[Item]{32}[$c$]})]{qWitt}. It's straightforward to check (using \cref{lem:psiRrational}) that also the maps between them match up. This proves what we want.
	\end{proof}
	As a consequence we obtain a \enquote{$\TR$-style} description of derived $q$-de Rham--Witt complexes. The question whether such a description exists was first raised by Johannes Anschütz in the author's Master's thesis defense.
	
	\begin{cor}
		The associated graded of the even filtration $\fil_{\ev,C_m}^\star\THH(\ku_R/\ku_A)^{C_m}$ is given by
		\begin{equation*}
			\Sigma^{-2*}\gr_{\ev,C_m}^*\THH(\ku_R/\ku_A)^{C_m}\simeq \qIW_m\deRham_{R/A}^*\,.
		\end{equation*}
	\end{cor}
	\begin{proof}[Proof sketch]
		This follows from \cref{thm:qdeRhamkuGenuine} and \cite[Proposition~\chref{3.49}]{qWittHabiro}.
	\end{proof}

	Finally, let us explain how to recover the Habiro-Hodge complex $\qHhodge_{R/A}$.
	
	\begin{numpar}[Cyclonic even filtrations on $\THH(\KU_R/\KU_A)$.]\label{par:CyclonicEvenFiltrationTHHKU}
		Put $\KU_A\coloneqq \KU\otimes\IS_A$ and $\KU_R\coloneqq \KU\otimes\IS_R$. We equip $\KU$ with its cyclonic structure from \cref{par:Cyclonicku} and
		\begin{equation*}
			\THH(\KU_R/\KU_A)\simeq \THH(\ku_R/\ku_A)\otimes_{\ku}\KU
		\end{equation*}
		with the base change of the cyclonic structure from \cref{par:CyclonicBase}. We also let
		\begin{equation*}
			\fil_{\ev,C_m}^\star \THH(\KU_R/\KU_A)^{C_m}\coloneqq \fil_{\ev,C_m}^\star \THH(\ku_R/\ku_A)^{C_m}\otimes_{\ku_{\ev}^{C_m}}\KU_{\ev}^{C_m}\,,
		\end{equation*}
		where $\ku_{\ev}^{C_m}\coloneqq \tau_{\geqslant 2\star}(\ku^{C_m})$ and $\KU_{\ev}^{C_m}\coloneqq \tau_{\geqslant 2\star}(\KU^{C_m})$. Observe that $-\otimes_{\ku_{\ev}^{C_m}}\KU_{\ev}^{C_m}$ can be regarded as a localisation at the element $\beta$ sitting in homotopical degree~$2$ and filtration degree~$1$. Finally, we construct
		\begin{equation*}
			\fil_{\ev,S^1}^\star\TCn{m}(\KU_R/\KU_A)\coloneqq \left(\fil_{\ev,C_m}^\star\THH(\KU_R/\KU_A)^{C_m}\right)^{\h (\IT/C_{m})_{\ev}}\,.
		\end{equation*}
	\end{numpar}
	\begin{rem}
		If we believe that our construction of $\fil_{\ev,C_m}^\star \THH(\ku_R/\ku_A)^{C_m}$ is the \enquote{correct} filtration to put on $\THH(\ku_R/\ku_A)^{C_m}$ (see the discussion in~\cref{par:CyclonicEvenFiltration}), then the construction from~\cref{par:CyclonicEvenFiltrationTHHKU} provides the correct even filtration for $\THH(\KU_R/\KU_A)^{C_m}$, since taking even filtrations should commute with filtered colimits. 
	\end{rem}
	\begin{lem}
		For all $m\in\IN$, the filtered objects
		\begin{equation*}
			\fil_{\ev,C_m}^\star\THH(\KU_R/\KU_A)^{C_m}\quad\text{and}\quad\fil_{\ev,S^1}^\star\TCn{m}(\KU_R/\KU_A)
		\end{equation*}
		are complete and exhaustive filtrations on $\THH(\KU_R/\KU_A)^{C_m}$ and $\TCn{m}(\KU_R/\KU_A)$, respectively.
	\end{lem}
	\begin{proof}[Proof sketch]
		Observe that inverting the element $\beta$ in homotopical degree~$2$ and filtration degree~$1$ preserves the assumptions of \cite[Lemma~\chref{2.75}(iv)]{CyclotomicSynthetic}. We can thus use the same argument as in \cref{lem:GenuineEvenFiltrationTHHkuCompleteExhaustive}.
	\end{proof}
	\begin{rem}\label{rem:CyclonicFixedPointsTransitionMaps}
		In the general setup of \cref{par:CyclonicEvenFiltration}, we have canonical maps
		\begin{equation*}
			\fil_{\ev/T,C_m}^\star M^{C_m}\longrightarrow \bigl(\fil_{\ev/T,C_n}^\star M^{C_n}\bigr)^{\h C_{m/n,\ev}}
		\end{equation*}
		whenever $n\mid m$. Indeed, upon applying $(-)^{\h C_{m/n,\ev}}$, the equaliser diagram for $\fil_{\ev/T,C_n}^\star M^{C_n}$ becomes a subdiagram of that for $\fil_{\ev/T,C_m}^\star M^{C_m}$. As a consequence, we get canonical maps
		\begin{equation*}
			\fil_{\ev,S^1}^\star\TCn{m}(\KU_R/\KU_A)\longrightarrow \fil_{\ev,S^1}^\star\TCn{n}(\KU_R/\KU_A)\,.
		\end{equation*}
		and similarly for $\ku$. It's possible to construct these maps coherently, that is, assemble them into functor $\IN\rightarrow \cat{SynSp}$. Since we're only interested in the limit, the individual maps will suffice, as we can always restrict to the sequential subposet $\{n!\}_{n\geqslant 1}\subseteq \IN$.
	\end{rem}
	\begin{thm}\label{thm:qdeRhamKUGenuine}
		Let $m\in\IN$. Suppose $A$ and $R$ satisfy the assumptions from~\cref{par:NewAssumptions} along with the addenda $2\in R^\times$ and~\cref{par:CyclonicBase}\cref{enum:AssumptionOnA2}. Then there exists a canonical $\IZ[\beta^{\pm 1}]$-linear equivalence
		\begin{equation*}
			\qHhodge_{R/A}[\beta^{\pm 1}]\overset{\simeq}{\longrightarrow}\Sigma^{-2*}\gr^*\Bigl(\limit_{m\in\IN}\fil_{\ev,S^1}^\star\TCn{m}(\KU_R/\KU_A)\Bigr)\,.
		\end{equation*}
	\end{thm}
	\begin{proof}
		Let us first verify that
		\begin{equation*}
			\Bigl(\bigl(\fil_{\ev,S^1}^\star\TCn{m}(\ku_R/\ku_A)\bigr)[\beta^{-1}]\Bigr)_{t_m}^\complete\overset{\simeq}{\longrightarrow} \fil_{\ev,S^1}^\star\TCn{m}(\KU_R/\KU_A)\,,
		\end{equation*}
		where $\beta$ sits in homotopical degree~$2$ and filtration degree~$1$, whereas $t_m$ sits in homotopical degree~$-2$ and filtration degree~$-1$ of $\tau_{\geqslant 2\star}\bigl((\ku^{C_m})^{\h (S^1/C_m)}\bigr)$. Indeed, we can identify the $t_m$-adic filtration on $(-)^{\h(\IT/C_m)_{\ev}}$ with the filtration coming from the \emph{CW filtration} on $\ku[S^1/C_m]_{\ev}$ in the sense of \cite[Construction~\chref{2.52}]{CyclotomicSynthetic}. This shows that both sides above are $t_m$-complete, so the map exists, and after reduction modulo~$t_m$ we recover the defining equivalence $\fil_{\ev,C_m}^\star\THH(\ku_R/\ku_A)^{C_m}[\beta^{-1}]\simeq \fil_{\ev,C_m}^\star\THH(\KU_R/\KU_A)^{C_m}$, so also the map above is an equivalence.
		
		As a consequence of this observation and \cref{thm:qdeRhamkuGenuine}, we obtain that the filtration $\fil_{\ev,S^1}^\star\TCn{m}(\KU_R/\KU_A)$ is periodic and each graded piece is equivalent to
		\begin{equation*}
			\gr_{\ev,S^1}^0\TCn{m}(\KU_R/\KU_A)\simeq \qhatdeRham_{R/A}^{(m)}\left[\frac{\fil_{\qHhodge_m}^i}{(q^m-1)^i}\ \middle|\ i\geqslant 1\right]_{(q^m-1)}^\complete\,,
		\end{equation*}
		where we use the notation from \cite[Construction~\chref{3.42}]{qWittHabiro}. Also observe that since we complete at $(q^m-1)$ anyway, it doesn't matter whether we use $\qhatdeRham_{R/A}^{(m)}$ or $\qdeRham_{R/A}^{(m)}$ in this formula, so the right-hand side agrees with $\qHhodge_{R/A,m}$.
		
		By tracing through the constructions it's straightforward to check that for any $n\mid m$ the map $\Sigma^{-2*}\gr_{\ev,S^1}^*\TCn{m}(\ku_R/\ku_A)\rightarrow \Sigma^{-2*}\gr_{\ev,S^1}^*\TCn{n}(\ku_R/\ku_A)$ from \cref{rem:CyclonicFixedPointsTransitionMaps} is the completion of the transition map
		\begin{equation*}
			\fil_{\qHhodge_m}^\star\qdeRham_{R/A}^{(m)}\longrightarrow \fil_{\qHhodge_n}^\star\qdeRham_{R/A}^{(n)}
		\end{equation*}
		from \cite[Construction~\chref{3.41}]{qWittHabiro}. Therefore,
		\begin{equation*}
			\Sigma^{-2*}\gr^*\Bigl(\limit_{m\in\IN}\fil_{\ev,S^1}^\star\TCn{m}(\KU_R/\KU_A)\Bigr)\simeq \limit_{m\in\IN}\qHhodge_{R/A,m}[\beta^{\pm 1}]\simeq \qHhodge_{R/A}[\beta^{\pm 1}]\,,
		\end{equation*}
		as desired.
	\end{proof}

	\newpage 
	
	\section{Examples}\label{sec:Examples}
	
	\subsection{Examples of spherical lifts}\label{subsec:CyclotomicBases}
	
	The assumptions of our main results---\cref{thm:qdeRhamkuGlobal,thm:qdeRhamkuGenuine,thm:qdeRhamKUGenuine}---seem quite restrictive at first. In this subsection we'll show that there are nevertheless many nontrivial examples to which the theorems apply. We'll start with examples of $\Lambda$-rings $A$ that satisfy the assumptions from~\cref{par:NewAssumptions}\cref{enum:AssumptionsOnAglobal}.
	
	\begin{exm}\label{exm:PolynomialLambdaRings}
		If $A=\IZ[x_i\ |\ i\in I]$ is a polynomial ring equipped with the \emph{toric $\Lambda$-structure} in which $\psi^m(x_i)=x_i^m$ for all $m$, then the assumptions from~\cref{par:AssumptionsOnA} are satisfied. Indeed, we can choose $\IS_A\simeq \IS[x_i\ |\ i\in I]$ to be flat spherical polynomial ring. As explained in \cite[Proposition~\chref{11.3}]{BMS2}, this is a cyclotomic basis and for every prime~$p$ the Tate-valued Frobenius satisfies $\phi_{\t C_p}(x_i)=x_i^p=\psi^p(x_i)$.
	\end{exm}
	\begin{exm}\label{exm:PerfectLambdaRings}
		If $A$ is a perfect $\Lambda$-ring, then the assumptions from~\cref{par:AssumptionsOnA} are also satisfied: For every prime~$p$, the spherical Witt vector ring $\IS_{\W(A/p)}$ from \cite[Example~\chref{5.2.7}]{LurieEllipticII} yields a $p$-complete lift of $A$. These can be glued with $A\otimes\IQ$ in a canonical way to yield $\IS_A$. To construct the structure of a cyclotomic base and check \cref{par:AssumptionsOnA}\cref{enum:CyclotomicLift} for all primes~$p$, we must equip the Tate-valued Frobenius
		\begin{equation*}
			\phi_{\t C_p}\colon \IS_{A}\longrightarrow \IS_{A}^{\t C_p} 
		\end{equation*}
		with an $S^1$-equivariant structure, where $\IS_{A}$ receives the trivial action and $\IS_{A}^{\t C_p}$ the residual $S^1/C_p\simeq S^1$-action. Equivalently, we must factor $\phi_{\t C_p}$ through an $\IE_\infty$-map
		\begin{equation*}
			\IS_{A}\longrightarrow \bigl(\IS_{A}^{\t C_p}\bigr)^{\h (S^1/C_p)}\simeq \bigl(\IS_{A}^{\t S^1}\bigr)_p^\complete\,.
		\end{equation*}
		By the universal property of spherical Witt vectors, for all $m\in\IN$ and all primes~$p$ the Adams operation $\psi^m\colon A\rightarrow A$ lifts to an $\IE_\infty$-map $\psi^m\colon \IS_{\W(A/p)}\rightarrow \IS_{\W(A/p)}$. These can be glued with the rationalisation to obtain an $\IE_\infty$-map $\psi^m\colon \IS_A\rightarrow \IS_A$. From the trivial $S^1$-action we also obtain a map $\IS_{A}\rightarrow \IS_{A}^{\h S^1}$ that splits the usual limit projection. The desired factorisation of $\phi_{\t C_p}$ is then given by
		\begin{equation*}
			\IS_{A}\overset{\psi^p}{\longrightarrow}\IS_{A}\longrightarrow \IS_{A}^{\h S^1}\longrightarrow \bigl(\IS_{A}^{\t S^1}\bigr)_p^\complete\longrightarrow \IS_{A}^{\t C_p} \,.
		\end{equation*}
		To see that the composition is really $\phi_{\t C_p}$, we use the universal property of spherical Witt vectors again: It's enough to check that the map on $\pi_0(-)/p$ is the Frobenius on $A/p$, which is clear from the construction.
	\end{exm}
	\begin{exm}
		We can also combine \cref{exm:PolynomialLambdaRings,exm:PerfectLambdaRings} and consider $A$ to be a polynomial ring over a perfect $\Lambda$-ring, or even a localisation of such a ring, as long as it still carries a $\Lambda$-structure.
	\end{exm}
	
	The examples where $A$ is a polynomial ring (over a perfect $\Lambda$-ring) are the most relevant for us, since they are expected to show up in the connection with the work of Garoufalidis--Scholze--Wheeler--Zagier (\cite{HabiroRingOfNumberField}, but the relative case was only discussed in \cite{HabiroRingLecture}). Nevertheless, there are examples that are not of this form, such as the following.
	\begin{exm}\label{exm:Chebyshev}
		Recall that the polynomial ring $\IZ[y]$ admits one more $\Lambda$-structure besides the toric one (\cite{ClauwensLambda}; see also \cite{ManinRealF1}). This other $\Lambda$-structure is called the \emph{Chebyshev $\Lambda$-structure}, since $\psi^m(y)$ is given by the Chebyshev polynomial $T_m(y)$. If $\IZ[x^{\pm 1}]$ is equipped with the toric $\Lambda$-structure, then the Chebyshev $\Lambda$-structure on $\IZ[y]$ can be identified with the fixed points of the $C_2$-action on $\IZ[x^{\pm 1}]$ that sends $x\mapsto x^{-1}$. Under this identification we have $y=x+x^{-1}$.
		
		We'll show that $A=\IZ[\frac12,y]$ still satisfies \cref{par:NewAssumptions}\cref{enum:AssumptionsOnAglobal}. Indeed, as soon as $2$ is invertible, the homotopy fixed points $\IS[\frac12,y]\coloneqq \IS[\frac12,x^{\pm 1}]^{\h C_2}$ define the desired $\IE_\infty$-lift. To verify that \cref{par:AssumptionsOnA}\cref{enum:CyclotomicLift} is satisfied for all primes~$p$, there's nothing to do for $p=2$, as then $\IS[\frac12,y]^{\t C_2}\simeq 0$. For $p\neq 2$, $(-)^{\t C_p}$ and $(-)^{\h C_2}$ commute (see \cite[Lemma~\chref{9.3}]{KrauseNikolausTHHLectures} for example) and so \cref{par:AssumptionsOnA}\cref{enum:CyclotomicLift} follows from the corresponding assertions for $\IS[\frac12,x^{\pm 1}]$ by applying $(-)^{\h C_2}$. The same argument shows that the addendum from~\cref{par:CyclonicBase}\cref{enum:AssumptionOnA2} is satisfied as well.
	\end{exm}
	
	\begin{rem}
		Recall that a cyclotomic spectrum $X$ has \emph{Frobenius lifts} in the sense of \cite[Definition~\chref{8.2}]{KrauseNikolausTHHLectures} if for each prime~$p$ the cyclotomic Frobenius $\phi_p\colon X\rightarrow X^{\t C_p}$ factors $S^1$-equivariantly through a map $\psi_p\colon X\rightarrow X^{\h C_p}$ such that the $\psi_p$ commute for different primes.
		
		In each of \crefrange{exm:PolynomialLambdaRings}{exm:Chebyshev} it's clear that $\IS_A$ admits Frobenius lifts as a cyclotomic $\IE_\infty$-algebra. Using \cref{lem:A2FrobeniusLifts}, this implies that Assumption~\cref{par:CyclonicBase}\cref{enum:AssumptionOnA2} is satisfied. Indeed, since the $S^1$-action is trivial, we may equivalently regard $\psi_p\colon \IS_A\rightarrow \IS_A^{\h C_p}$ as an $S^1$-equivariant $\IE_\infty$-algebra map $\psi^p\colon \IS_A\rightarrow \IS_A$. The commutativity datum simply provides homotopies $\psi^p\circ\psi^\ell\simeq \psi^\ell\circ \psi^p$ for all $p\neq \ell$. Inductively defining $\psi^1\coloneqq \id$, $\psi^{pm}\coloneqq \psi^m\circ \psi^p$, we obtain the necessary commutative diagrams
		\begin{equation*}
			\begin{tikzcd}[column sep=large]
				\IS_A\rar["\psi^{pm}"]\dar["\phi_p"'] & \IS_A\dar\\
				\IS_A^{\t C_p}\rar["(\psi^m)^{\t C_p}"] & \IS_A^{\t C_p}
			\end{tikzcd}
		\end{equation*}
		and thus the desired map $\IS_A^\mathrm{cyct}\rightarrow \IS_A^\mathrm{triv}$.
	\end{rem}
	
	\begin{numpar}[Non-example.]
		In the case where $A=\IZ\{x\}_\Lambda$ is a free $\Lambda$-ring, it's not known whether a spherical lift $\IS_A$ as in \cref{par:AssumptionsOnA} exist.%
		\footnote{In fact, it is a conjecture of Thomas Nikolaus that such a spherical lift \emph{doesn't} exist.}
	\end{numpar}
	
	Let us now give several examples of $A$-algebras $R$ that satisfy the assumptions of \cref{par:NewAssumptions}\cref{enum:AssumptionsOnRglobal}.
	
	\begin{exm}\label{exm:LiftsCoordinateCase}
		Suppose that $S$ is a smooth $A$-algebra equipped with an étale map $\square\colon A[x_1,\dotsc,x_n]\rightarrow S$. By \cite[Theorem~\chref{7.5.4.3}]{HA}, $\square$ lifts uniquely to an étale map $\IS_A[x_1,\dotsc,x_n]\rightarrow \IS_{S,\square}$ of $\IE_\infty$-ring spectra. Then $R=S$ satisfies the assumptions of \cref{par:NewAssumptions}\cref{enum:AssumptionsOnRglobal}, choosing \cref{par:AssumptionsOnR}\cref{enum:E2Lift} for every prime~$p$. We'll continue to study this example in \cref{subsec:Raksit} below.
	\end{exm}
	\begin{exm}\label{exm:LiftsCoordinateCaseII}
		In the setting from \cref{exm:LiftsCoordinateCase}, suppose that $(y_1,\dotsc,y_r)$ is a regular sequence in $S$. By Burklund's theorem about multiplicative structures on quotients \cite[Theorem~\chref{1.5}]{BurklundMooreSpectra}, the spectrum 
		\begin{equation*}
			\IS_R\coloneqq \IS_{S,\square}/\bigl(y_1^{\alpha_1},\dotsc,y_r^{\alpha_r}\bigr)\simeq \IS_{S,\square}/y_1^{\alpha_1}\otimes_{\IS_{S,\square}}\dotsb\otimes_{\IS_{S,\square}}\IS_{S,\square}/y_r^{\alpha_r}
		\end{equation*}
		admits an $\IE_2$-structure in $\IS_A$-modules (even in $\IS_{S,\square}$-modules) if all $\alpha_i$ are even and $\geqslant 6$. If $2$ is invertible in $S$, it's already enough to have all $\alpha_i\geqslant 3$, with no evenness assumption. In either case, we see that $R=S/(y_1^{\alpha_1},\dotsc,y_r^{\alpha_r})$ satisfies the assumptions of \cref{par:NewAssumptions}\cref{enum:AssumptionsOnRglobal}, choosing \cref{par:AssumptionsOnR}\cref{enum:E2Lift} for every prime~$p$.
		
		If we only assume that all $\alpha_i$ are even and $\geqslant 4$, or $2$ is invertible in $S$ and all $\alpha_i\geqslant 2$, then $\IS_R$ still admits an $\IE_1$-structure in $\IS_{S,\square}$-modules. Provided that $R$ is $p$-torsion free, condition~\cref{par:AssumptionsOnR}\cref{enum:E1Lift} is satisfied for every prime~$p$. Indeed, if we put
		\begin{equation*}
			\widehat{R}_{p,\infty}\coloneqq \left(\widehat{A}_p\bigl\langle x_1^{1/p^\infty},\dotsc,x_n^{1/p^\infty}\bigr\rangle\otimes_{\widehat{A}_p\langle x_1,\dotsc,x_n\rangle}\widehat{R}_p\right)_p^\complete\,.
		\end{equation*}
		then the $p$-completed \v Cech nerve of $\widehat{R}_p\rightarrow \widehat{R}_{p,\infty}$ admits a spherical $\IE_1$-lift, given by the $p$-completed base change along $\IS_A[x_1,\dotsc,x_n]\rightarrow \IS_R$ of the \v Cech nerve of the $\IE_\infty$-algebra map $\IS_A[x_1,\dotsc,x_n]\rightarrow \IS_A[x_1^{1/p^\infty},\dotsc,x_n^{1/p^\infty}]$.
	\end{exm}
	
	\begin{exm}
		The easiest way for \cref{par:AssumptionsOnR}\cref{enum:E1Lift} to be satisfied is the case where $R/p$ is already relatively semiperfect over $A$, so that we can take the trivial descent diagram for the identity on $\widehat{R}_p$. Then the only condition is for $\widehat{R}_p$ to admit an $\IE_1$-lift $\IS_{\smash{\widehat{R}}_p}$ in $\IS_A$-modules.
		
		Thanks to Burklund's result again, it's easy to write down rings for which this is satisfied for all primes~$p$. Here's one possible construction: Let $B$ be a relatively perfect $\Lambda$-$A$-algebra such that $A\rightarrow B$ is quasi-lci.%
		\footnote{For every prime~$p$, the relatively perfect map of $\delta$-rings $\widehat{A}_p\rightarrow \widehat{B}_p$ will automatically be $p$-quasi-lci, so $A\rightarrow B$ being quasi-lci is a rational condition.}
		For example, we could take $B=A[x^{1/n}\ |\ n\geqslant 1]$ with the toric $\Lambda$-structure or $B=A\otimes_\IZ\IZ\{x\}_{\Lambda,\mathrm{perf}}$, the free $\Lambda$-$A$-algebra on a perfect generator. Let $B'$ be an étale $B$-algebra and let $(y_1,\dotsc,y_r)$ be a regular sequence in $B'$. Then $R\cong B'/(y_1^{\alpha_1},\dotsc,y_r^{\alpha_r})$ satisfies \cref{par:AssumptionsOnR}\cref{enum:E1Lift} if all $\alpha_i$ are even and $\geqslant 4$. If $2$ is invertible in~$R$, it's already enough to have all $\alpha_i\geqslant 2$ with no evenness assumption.
		
		Indeed, since each $p$-completions $\widehat{B}'_p$ is all $p$-completely formally étale over $A$, it lifts uniquely to a $p$-complete connective $\IE_\infty$-$\IS_A$-algebra $\IS_{\smash{\widehat{B}}'_p}$. Our assumptions on the $\alpha_i$ ensure that \cite[Theorem~\chref{1.5}]{BurklundMooreSpectra} applies, so that
		\begin{equation*}
			\IS_{\smash{\widehat{R}}_p}\coloneqq\IS_{\smash{\widehat{B}}'_p}/\bigl(y_1^{\alpha_1},\dotsc,y_r^{\alpha_r}\bigr)
		\end{equation*}
		admits an $\IE_1$-structure in $\IS_A$-modules (even in $\IS_{\smash{\widehat{R}}_p}$-modules), as desired.
	\end{exm}
	
	\subsection{The case of a framed smooth algebra}\label{subsec:Raksit}
	
	In the situation of \cref{exm:LiftsCoordinateCase}, the $q$-deformation of the Hodge filtration that we see has a very nice explicit description. This result is due to Arpon Raksit; in fact, his result is what motivated our investigation. To formulate the result, recall that in the situation at hand, the (underived) $q$-de Rham complex $\qOmega_{S/A}$ can be represented by an explicit complex
	\begin{equation*}
		\qOmega_{S/A, \square}^*=\Big(S\qpower\xrightarrow{\q\nabla}\Omega_{S/A}^1\qpower\xrightarrow{\q\nabla}\dotsb\xrightarrow{\q\nabla}\Omega_{S/A}^n\qpower\Big)\,.
	\end{equation*}
	\begin{thm}[Raksit, unpublished]\label{thm:Raksit}
		Let $(S,\square)$ be a framed smooth $A$-algebra as in \cref{exm:LiftsCoordinateCase} and put $\ku_{S,\square}\coloneqq \ku\otimes\IS_{S,\square}$. For all integers $i$ we let $\fil_{\qHodge, \square}^i\qOmega_{S/A, \square}^*$ denote the subcomplex
		\begin{equation*}
			\Bigl((q-1)^iS\qpower\rightarrow (q-1)^{i-1}\Omega_{S/A}^1\qpower\rightarrow\dotsb\rightarrow \Omega_{S/A}^i\qpower\rightarrow\dotsb\rightarrow \Omega_{S/A}^n\qpower\Bigr)\,.
		\end{equation*}
		of the coordinate-dependent $q$-de Rham complex $\qOmega_{S/A, \square}^*$ \embrace{which we regard as sitting in homotopical degrees $[-n,0]$}. Then
		\begin{equation*}
			\Sigma^{-2*}\gr_{\ev}^i\TC^-(\ku_{S,\square}/\ku_A)\simeq \fil_{\qHodge, \square}^\star\qOmega_{S/A, \square}^*\,.
		\end{equation*}
	\end{thm}
	While Raksit's original proof uses geometric arguments, we'll give a more algebraic proof of \cref{thm:Raksit}. We first need a general fact about $q$-divided powers.
	\begin{lem}\label{lem:qPDrank1}
		Fix a prime~$p$. Consider $\IZ_p[x,y,q]$, equipped with the toric $\delta$-structure, and let
		\begin{equation*}
			\q D\coloneqq \IZ_p[x,y,q]\left\{\frac{\phi(x-y)}{[p]_q}\right\}_{(p,q-1)}^\complete\,.
		\end{equation*}
		Then $\q D$ is the $(p,q-1)$-completion of the subalgebra of $\IQ_p[x,y]\qpower$ generated by $\IZ_p[x,y,q]$ as well as elements $(q-1)^d\widetilde{\gamma}_q^d(x-y)$ for all $d\geqslant 1$, where we put
		\begin{equation*}
			\widetilde{\gamma}_q^d(x-y)\coloneqq \frac{(x-y)(x-qy)\dotsm (x-q^{d-1}y)}{(q;q)_d}\,.
		\end{equation*}
	\end{lem}
	\begin{proof}
		It will be enough to show that $\q D$ contains $(q-1)^d\widetilde{\gamma}_q^d(x-y)$ for all $d\geqslant 1$, as then the fact that these are generators as well as the claimed description of $\q D$ can be checked modulo~$(q-1)$.
		
		First observe that $(p,q-1)$ is a regular sequence in $\q D$. Indeed, $\q D/(q-1)$, where the quotient is taken in the derived sense as usual, is the PD-envelope of $(x-y)\subseteq \IZ_p[x,y]$, which is a $p$-torsion free ring. It follows that $(p,(q;q)_d)$ is a regular sequence for all $d\geqslant 1$. Indeed, up to factors that are invertible in $\q D$, the Pochhammer symbol is a product of factors of the form $(1-q^{p^\alpha})$, and $(1-q^{p^\alpha})\equiv (1-q)^{p^\alpha}\mod p$. In particular, each $(q;q)_d$ is a non-zerodivisor in $\q D$.
		
		If we equip $\IZ_p[x,y,q]$ with the toric $\Lambda$-structure, then the Adams operations $\psi^\ell$ for $\ell\neq p$ are $\delta$-ring maps. Using the universal property it is then straightforward to check that the $\psi^\ell$ extend to $\q D$, hence $\q D$ carries a $\Lambda$-$\IZ_p[x,y,q]$-structure extending the given $\delta$-structure. This $\Lambda$-structure extends then uniquely to the localisation $\q D[(q;q)_d^{-1}\ |\ d\geqslant 1]$. In the localisation, we have
		\begin{equation*}
			\lambda^d\left(\frac{x-y}{q-1}\right)=\widetilde{\gamma}_q^d(x-y)\,;
		\end{equation*}
		see \cite[Lemma~\chref{1.3}]{Pridham}. So we must show $(q-1)^d\lambda^d\bigl(\frac{x-y}{q-1}\bigr)\in \q D$. To this end, first observe that $(q-1)\psi^d\bigl(\frac{x-y}{q-1}\bigr)\in \q D$ for all $d\geqslant 1$. Indeed, it's enough to check this if $d=p^\alpha$ is a power of $p$. So we must check that $x^{p^\alpha}-y^{p^{\alpha}}$ is divisible by $[p^\alpha]_q$ in $\q D$. Since $\q D$ is $(p,q-1)$-completely flat over $\IZ_p\qpower$ by \cite[Lemma~\chref{16.10}]{Prismatic} and thus flat on the nose over $\IZ[q]$, it will be enough to check that $x^{p^\alpha}-y^{p^\alpha}$ is divisible by each cyclotomic polynomial in the factorisation $[p^\alpha]_q=\Phi_p(q)\Phi_{p^2}(q)\dotsm \Phi_{p^\alpha}(q)$. Since $x^{p^i}-y^{p^i}$ divides $x^{p^\alpha}-y^{p^\alpha}$ for $i\leqslant \alpha$, it suffices to show that $x^{p^\alpha}-y^{p^{\alpha}}$ is divisible by $\Phi_{p^\alpha}(q)$, which follows by applying $\phi^{\alpha-1}$ to $\phi(x-y)/[p]_q$.
		
		Now let us put $\lambda_t(-)\coloneqq \sum_{d\geqslant 0}\lambda^d(-)t^n$ and $\psi_t(-)\coloneqq \sum_{d\geqslant 1}\psi^d(-)t^d$, where $t$ is a formal variable. Our observation above shows that $\psi_{(q-1)t}\bigl(\frac{x-y}{q-1}\bigr)$ has coefficients in $\q D$. From the general $\Lambda$-ring formula $\psi_t=-t\frac{\d}{\d t}\log\lambda_{-t}$ we deduce that $\lambda_{(q-1)t}\bigl(\frac{x-y}{q-1}\bigr)$ has coefficients in $\q D[p^{-1}]$. Since $(p,(q;q)_d)$ is a regular sequence in $\q D$, the we get
		\begin{equation*}
			\q D\bigl[p^{-1}\bigr]\cap \q D\bigl[(q;q)_d^{-1}\bigr]=\q D\,,
		\end{equation*}
		where the intersection is taken in $\q D[p^{-1},(q;q)_d^{-1}]$ (and on the level of sets---nothing derived is happening). This shows $(q-1)^d\lambda^d\bigl(\frac{x-y}{q-1}\bigr)\in \q D$, as desired.
	\end{proof}
	
	\begin{numpar}[A cosimplicial resolution.]
		To show \cref{thm:Raksit}, we'll compute the even filtration via an explicit resolution. To this end, let us fix the following notation:
		\begin{alphanumerate}
			\item Let $P\coloneqq A[x_1,\dotsc,x_n]$ and $\IS_P\coloneqq \IS_A[x_1,\dotsc,x_n]$. Let $A\rightarrow P^\bullet$ and $\IS_A\rightarrow \IS_{P^\bullet}$ denote the \v Cech nerves of $A\rightarrow P$ and $\IS_A\rightarrow \IS_P$ and put $\ku_{P^\bullet}\coloneqq \ku\otimes\IS_{P^\bullet}$.
			\item Let $x_i^{(r)}\in P^\bullet\cong P^{\otimes_A(\bullet+1)}$ denote the element $1\otimes\dotsb\otimes 1\otimes x_i\otimes 1\otimes \dotsb\otimes 1$ coming from the $r$\textsuperscript{th} tensor factor for any $1\leqslant r\leqslant \bullet+1$.
			\item Let $\q D^\bullet$ denote the $(q-1)$-completion of the sub-algebra of $(S^{\otimes_A(\bullet+1)}\otimes_\IZ\IQ)\qpower$ generated by $S^{\otimes_A(\bullet+1)}\qpower$ as well as the elements $(q-1)^d\widetilde{\gamma}_q^d\bigl(x_i^{(r)}-x_i^{(s)}\bigr)$ for all integers $d\geqslant 1$, all tensor factors $1\leqslant r,s\leqslant \bullet+1$, and all indices $1\leqslant i\leqslant n$.
			\item Let $\fil_{\qHodge}^\star\q D^\bullet$ be the descending filtration of ideals generated by $(q-1)$ in filtration degree~$1$ and the elements $(q-1)^d\widetilde{\gamma}_q^d\bigl(x_i^{(r)}-x_i^{(s)}\bigr)$ in filtration degree~$d$, and let $\fil_{\qHodge}^\star\q \widehat{D}^\bullet$ denote the completion of this filtration.
		\end{alphanumerate}
	\end{numpar}

	\begin{lem}\label{lem:ExplicitqHodgeFiltration}
		With notation as above, there exists a canonical isomorphism of graded $\IZ[\beta]\llbracket t\rrbracket\cong (q-1)^\star\IZ\qpower$-modules
		\begin{equation*}
			\pi_{2*}\TC^-(\ku_{S,\square}/\ku_{P^\bullet})\cong \fil_{\qHodge}^\star\q\widehat{D}^\bullet\,.
		\end{equation*}
	\end{lem}
	\begin{proof}
		We know from \cref{thm:qdeRhamkuGlobal} that $\pi_{2\star}\TC^-(\ku_{S,\square}/\ku_{P^\bullet})$ is the completion of a filtration $\fil_{\qHodge}^\star\qdeRham_{S/P^\bullet}$. Consider the arithmetic fracture square for the completed filtration:
		\begin{equation*}
			\begin{tikzcd}
				\fil_{\qHodge}^\star\qhatdeRham_{S/P^\bullet}\rar\dar\drar[pullback] & \prod_p\fil_{\qHodge}^\star \bigl(\qhatdeRham_{S/P^\bullet}\bigr)_p^\complete\dar\\
				\fil_{(\Hodge,q-1)}^\star\bigl(\deRham_{S/P^\bullet}\otimes_\IZ\IQ\bigr)_{\Hodge}^\complete\qpower\rar & \fil_{(\Hodge,q-1)}^\star\biggl(\prod_p	\left(\deRham_{S/P^\bullet}\right)_p^\complete\otimes_\IZ\IQ\biggr)_{\Hodge}^\complete\qpower
			\end{tikzcd}
		\end{equation*}
		Observe that all corners of this pullback square are static in every filtration degree. Indeed, this can easily be checked modulo~$(q-1)$. More precisely, if we identify the $(q-1)$-adic filtration $(q-1)^\star \IZ\qpower$ with the graded ring $\IZ[\beta]\llbracket t\rrbracket$ as in \cref{par:GlobalqHodgeFiltration}, then everything is $\beta$-complete; modulo~$\beta$, we're then reduced to checking that $\fil_{\Hodge}^\star\hatdeRham_{S/P^\bullet}$ as well as its $p$-completions and its Hodge-completed rationalisation are static, which is standard.
		
		We conclude that this diagram is also a pullback of filtered abelian groups, which will make it easy to construct a map $\fil_{\qHodge}^\star\q\widehat{D}^\bullet\rightarrow \fil_{\qHodge}^\star\qhatdeRham_{S/P^\bullet}$. To this end, let us now analyse the factors of the pullback. Let us start with the $p$-completed $q$-de Rham complex $(\qdeRham_{S/P^\bullet})_p^\complete$. Since $\delta$-structures extend uniquely along $p$-completely étale maps, the toric $\delta$-$A$-algebra structure on $\widehat{P}_p^\bullet$ extends uniquely to a $\delta$-$A$-algebra structure on $(S^{\otimes_A(\bullet+1)})_p^\complete$. Then $(\qdeRham_{S/P^\bullet})_p^\complete$ is the $q$-PD-envelope in the sense of \cite[Lemma~\chref{16.10}]{Prismatic} of the $(p,q-1)$-completely regular ideal
		\begin{equation*}
			\widehat{J}_p^\bullet\coloneqq \ker\Bigl(\bigl(S^{\otimes_A(\bullet+1)}\bigr)_p^\complete\twoheadrightarrow \widehat{S}_p\Bigr)
		\end{equation*}
		Using \cref{lem:qPDrank1} we see that $(\qdeRham_{S/P^\bullet})_p^\complete$ contains all the elements $(q-1)^d\widetilde{\gamma}_q^d\bigl(x_i^{(r)}-x_i^{(s)}\bigr)$. By \cref{thm:qdeRhamkuRegularQuotient}, for any fixed $d$, these elements are contained in $\fil_{\qHodge}^d (\qdeRham_{S/P^\bullet})_p^\complete$.
		
		The rational factor is similar: Since $P\rightarrow S$ is étale, the Hodge-completed de Rham complex satisfies $\hatdeRham_{S/P^\bullet}\simeq \hatdeRham_{S/S^{\smash{\otimes_A(\bullet+1)}}}$, and so $(\deRham_{S/P^\bullet}\otimes_\IZ\IQ)_{\Hodge}^\complete\qpower$ is the $(J_\IQ^\bullet,q-1)$-adic completion of $(S^{\otimes_A(\bullet+1)}\otimes_\IZ\IQ)\qpower$, where
		\begin{equation*}
			J_\IQ^\bullet\coloneqq \ker\Bigl(\bigl(S^{\otimes_A(\bullet+1)}\otimes_\IZ\IQ\bigr)\twoheadrightarrow (S\otimes_\IZ\IQ)\Bigr)\,.
		\end{equation*}
		Since $x_i^{(r)}-x_i^{(s)}$ is an element of $J_{\IQ}^\bullet$, it's also clear that $\widetilde{\gamma}_q^d\bigl(x_i^{(r)}-x_i^{(s)}\bigr)$ is contained in $\fil_{(\Hodge,q-1)}^\star(\deRham_{S/P^\bullet}\otimes_\IZ\IQ)_{\Hodge}^\complete\qpower$. Using the pullback above we get a filtered map
		\begin{equation*}
			\fil_{\qHodge}^\star\q\widehat{D}^\bullet\longrightarrow \fil_{\qHodge}^\star\qhatdeRham_{S/P^\bullet}\,.
		\end{equation*}
		Reducing modulo~$(q-1)$, or more precisely modulo~$\beta$, we see that this map is an isomorphism, which finishes the proof.
	\end{proof}
	
	
	%
	%
	
	\begin{proof}[Proof of \cref{thm:Raksit}]
		The even filtration in question can be computed via the cosimplicial resolution
		\begin{equation*}
			\fil_{\ev}^\star\TC^-(\ku_{S,\square}/\ku_A)\simeq \limit_{\IDelta}\tau_{\geqslant 2\star}\TC^-(\ku_{S,\square}/\ku_{P^\bullet})\,.
		\end{equation*}
		Using \cref{lem:ExplicitqHodgeFiltration}, it remains to show that the totalisation of the cosimplicial filtered ring $\fil_{\qHodge}^\star\q\widehat{D}^\bullet$ is quasi-isomorphic to the filtered complex $\fil_{\qHodge,\square}^\star\qOmega_{S/A, \square}^*$. We'll show this using a similar argument as in the proof of \cite[Theorem~\chref{16.22}]{Prismatic}.
		
		To this end, first observe that the $q$-divided powers from \cref{lem:qPDrank1} interact with the $q$-derivatives as follows:
		\begin{equation*}
			\q\partial_x\bigl(\widetilde{\gamma}_q^d(x-y)\bigr)=\widetilde{\gamma}_{q}^{d-1}(x-y)\quad\text{and}\quad \q\partial_y\bigl(\widetilde{\gamma}_q^d(x-y)\bigr)=-\widetilde{\gamma}_q^{d-1}(x-qy)\,.
		\end{equation*}
		It follows that the $q$-derivatives extend to $\q\widehat{D}^\bullet$. We can then consider the filtered cosimplicial filtered complex $\fil^\star\q M^{\bullet,*}$ given by
		\begin{equation*}
			\Bigl(\fil_{\qHodge}^\star\q\widehat{D}^\bullet\xrightarrow{\q\nabla}\fil_{\qHodge}^{\star-1}\q\widehat{D}^\bullet\otimes_{P^\bullet}\Omega_{P^\bullet/A}^1\xrightarrow{\q\nabla}\fil_{\qHodge}^{\star-2}\q\widehat{D}^\bullet\otimes_{P^\bullet}\Omega_{P^\bullet/A}^2\xrightarrow{\q\nabla}\dotsb\Bigr)\,.
		\end{equation*}
		Then each column $\fil^\star\q M^{i,*}$ is quasi-isomorphic to $\fil^\star\q M^{0,*}$; indeed, this can be checked modulo~$(q-1)$, and then it follows from the Poincaré lemma for the completed Hodge-filtered de Rham complex. On the other hand the rows $\fil^\star\q M^{\bullet,j}$ for $j>0$ are acyclic; this can be seen e.g.\ by \cite[Tag~\chref{07L7}]{Stacks} applied to the cosimplicial filtered ring $\fil_{\qHodge}^\star \q\widehat{D}^\bullet$. It follows formally that the $0$\textsuperscript{th} column $\fil^\star \q M^{0,*}$ is quasi-isomorphic to the totalisation of the $0$\textsuperscript{th} row $\fil^\star\q M^{\bullet,0}$, which is exactly what we wanted to show.
	\end{proof}
	\begin{rem}\label{rem:RaksitqOmegaDAlg}
		As a consequence of \cref{thm:Raksit}, the filtered complex $\fil_{\qHodge, \square}^i\qOmega_{S/A, \square}^*$ can be promoted to a filtered $\IE_\infty$-algebra over the filtered ring $(q-1)^\star A\qpower$. In fact, we even get the structure of a filtered derived commutative algebra in the sense of \cite[Definition~\chref{4.3.4}]{RaksitFilteredCircle}.
	\end{rem}
	
	\subsection{The Habiro ring of a number field, homotopically}\label{subsec:HabiroRingOfNumberField}
	
	As a final example, let us give a homotopical description of the Habiro ring of a number field from \cite[Definition~\chref{1.1}]{HabiroRingOfNumberField}.
	
	\begin{cor}\label{cor:HabiroRingOfNumberFieldKU}
		Let $F$ be a number field and let $\Delta$ be divisible by $6$ and by the discriminant of~$F$. Let $\IS_{\Oo_F[1/\Delta]}$ denote the unique lift of $\Oo_F[1/\Delta]$ to an étale extension of $\IS$. Then
		\begin{equation*}
			\Hh_{\Oo_F[1/\Delta]}\cong \pi_0\Bigl(\limit_{m\in\IN}\bigl(\THH(\KU\otimes\IS_{\Oo_F[1/\Delta]}/\KU)^{C_m}\bigr)^{\h (S^1/C_m)}\Bigr)\,.
		\end{equation*}
	\end{cor}
	\begin{proof}
		By \cite[Corollary~\chref{3.12}]{qWittHabiro}, $\qHhodge_{\Oo_F[1/\Delta]/\IZ}\simeq \Hh_{\Oo_F[1/\Delta]}$. In particular, the Habiro--Hodge complex must be static. By \cref{thm:qdeRhamKUGenuine}, $\limit_{m\in\IN}\fil_{\ev,S^1}^\star\TCn{m}(\KU\otimes\IS_{\Oo_F[1/\Delta]}/\KU)$ must be the double-speed Whitehead filtration $\tau_{\geqslant 2\star}$ and the result follows.
	\end{proof}

	\newpage
	\appendix
	\renewcommand{\thetheorem}{\thesection.\arabic{theorem}}
	\renewcommand{\SectionPrefix}{Appendix~}

	\section{The \texorpdfstring{$q$}{q}-de Rham complex via \texorpdfstring{$\TC^-$}{TC-}}\label{appendix:BMS2}
	
	In \cite[\S{\chref[section]{11}}]{BMS2} and \cite[\S{\chref[section]{15.2}}]{Prismatic}, it is explained how prismatic cohomology relative to a Breuil--Kisin prism $(\W(k)\llbracket z\rrbracket,E(z))$ can be understood in terms of $\TC^-(-/\IS[z])_p^\complete$. In this subsection, we'll show how the $p$-complete $q$-de Rham complex can be understood in a completely analogous way.
	
	For this to work, we assume that $A$ satisfies the conditions from \cref{par:AssumptionsOnA}, that is, $A$ is a $p$-complete and $p$-completely covered $\delta$-ring with a flat spherical lift $\IS_A$ which admits the structure of a $p$-cyclotomic base.
	
	

	\begin{lem}\label{lem:S_Ainfty}
		The $p$-completed colimit-perfection $A_\infty$ of $A$ admits a unique lift to a $p$-complete connective $\IE_\infty$-ring spectrum $\IS_{A_\infty}$ and $A\rightarrow A_\infty$ can be lifted to an $\IE_\infty$-map $\IS_A\rightarrow \IS_{A_\infty}$.
	\end{lem}
	\begin{proof}
		Since $A_\infty$ is a perfect $\delta$-ring, the lift $\IS_{A_\infty}$ exists uniquely; it is given by the spherical Witt vectors $\IS_{\W(A_\infty^\flat)}$ from Example~\cite[\chref{5.2.7}]{LurieEllipticII}.
		
		To construct the map $\IS_A\rightarrow \IS_{A_\infty}$, first observe that the canonical map $\IS_A\rightarrow\IS_A^{\t C_p}$ is an equivalence. Indeed, we can choose a two-term resolution $0\rightarrow \bigoplus_I\IZ_p\rightarrow\bigoplus_J\IZ_p\rightarrow A\rightarrow 0$ and lift it to a cofibre sequence $\bigoplus_I\IS_p\rightarrow \bigoplus_J\IS_p\rightarrow\IS_A$ of spectra. By the Segal conjecture, $(\bigoplus_I\IS_p)^{\t C_p}\simeq (\bigoplus_I\IS_p)_p^\complete$ and likewise for $J$, so the same will be true for $\IS_A$. We can then form the sequential colimit
		\begin{equation*}
			\colimit\left(\IS_A\xrightarrow{\phi_{\t C_p}}\IS_A^{\t C_p}\simeq \IS_A\xrightarrow{\phi_{\t C_p}}\dotsb \right)_p^\complete\,.
		\end{equation*}
		By our assumptions on $A$, the Tate-valued Frobenius $\phi_{\t C_p}$ agrees with $\phi$ on $\pi_0$, and so this colimit is a $p$-complete connective $\IE_\infty$-lift of $A_\infty$. By uniqueness, it must agree with $\IS_{A_\infty}$, and so we get our desired map $\IS_A\rightarrow\IS_{A_\infty}$.
	\end{proof}
	
	\begin{lem}\label{lem:TC-Zpzeta}
		There are generators $u$ and $v$ in $\pi_2$ and $\pi_{-2}$ of $\TC^-(\IZ_p[\zeta_p]/\IS_p\qpower)_p^\complete$ such that
		\begin{equation*}
			\pi_*\TC^-\bigl(\IZ_p[\zeta_p]/\IS_p\qpower\bigr)_p^\complete\simeq \IZ_p\qpower[u,v]/\bigl(uv-[p]_q\bigr)\,.
		\end{equation*}
	\end{lem}
	\begin{proof}
		This can be shown in the same way as \cite[Proposition~\chref{11.10}]{BMS2}, using base change along $\IS\qpower\rightarrow \IS\llbracket q^{1/p^\infty}-1\rrbracket$.
	\end{proof}
	\begin{prop}\label{prop:qdeRhamTC-}
		Let $S$ be a $p$-complete $p$-quasi-lci $A[\zeta_p]$-algebra of bounded $p^\infty$-torsion. Then there is an equivalence of graded $\IE_\infty$-$\IZ_p\qpower[u,v]/(uv-[p]_q)$-algebras
		\begin{equation*}
			\Sigma^{-2*}\gr_{\HRWev,\h S^1}^*\TC^-\bigl(S/\IS_A\qpower\bigr)_p^\complete\simeq  \fil_\Nn^*\widehat{\Prism}_{S/A\qpower}^{(p)}\,,
		\end{equation*}
		where $\gr_{\HRWev,\h S^1}^*$ denotes the associated graded of the $p$-complete $S^1$-equivariant Hahn--Raksit--Wilson even filtration and $(-)^{(p)}$ \embrace{instead of $(-)^{(1)}$} denotes the Frobenius twist of prismatic cohomology. Moreover, after inverting $u$, we get an equivalence of graded $\IE_\infty$-$\IZ_p[u^{\pm 1}]\qpower$-algebras
		\begin{equation*}
			\Sigma^{-2*}\gr_{\HRWev,\h S^1}^*\left(\TC^-\bigl(S/\IS_A\qpower\bigr)\bigl[\localise{u}\bigr]_{(p,q-1)}^\complete\right)\simeq \Prism_{S/A\qpower}[u^{\pm 1}]\,,
		\end{equation*}
		where now $\gr_{\HRWev}^*$ refers to the $p$-complete $S^1$-equivariant Hahn--Raksit--Wilson even filtration on $\THH(S/\IS_A\qpower)[1/u]_p^\complete$.
	\end{prop}
	\begin{proof}[Proof sketch]
		First observe that $S_\infty\coloneqq (S\lotimes_{A\qpower}A_\infty\llbracket q^{1/p^\infty}-1\rrbracket)_p^\complete$ will be static and of bounded $p^\infty$-torsion, as $\phi\colon A\rightarrow A$ is $p$-completely flat. Moreover, $S_\infty$ will be $p$-quasi-lci over $A_\infty\llbracket q^{1/p^\infty}-1\rrbracket$, hence over $\IZ_p$, as the cotangent complex $\L_{A_\infty\llbracket q^{1/p^\infty}-1\rrbracket/\IZ_p}$ vanishes after $p$-completion. Thus $S_\infty$ is $p$-quasi-syntomic.
		
		If $S$ is \emph{large} in the sense that there exists a surjection $A\langle x_i^{1/p^\infty}\ \big|\ i\in I\rangle\twoheadrightarrow S$, then $\TC^-(S/\IS_A\qpower)_p^\complete$ will be even. Indeed, evenness can be checked after base change along $\IS_A\qpower\rightarrow \IS_{A_\infty}\llbracket q^{1/p^\infty}-1\rrbracket$. By an analogous argument as in \cite[Proposition~\chref{11.7}]{BMS2},
		\begin{equation*}
			\THH\bigl(\IS_{A_\infty}\llbracket q^{1/p^\infty}-1\rrbracket\bigr)\rightarrow\IS_{A_\infty}\llbracket q^{1/p^\infty}-1\rrbracket
		\end{equation*}
		is an equivalence after $p$-completion. This reduces the assertion to $\TC^-(S_\infty)_p^\complete$ being even, which is shown in \cite[Theorem~\chref{7.2}]{BMS2}.
		
		Via quasi-syntomic descent from the large case, we can now construct a filtration on $\TC^-(S/\IS_A\qpower)_p^\complete$. Arguing as in \cite[\S{\chref[section]{11.2}}]{BMS2} and \cite[\S{\chref[section]{15.2}}]{Prismatic}, we find that the associated graded of this filtration yields the completion of the Nygaard filtration on the Frobenius-twisted prismatic cohomology relative to the $q$-de Rham prism $(A\qpower,[p]_q)$. To see that the filtration agrees with the $p$-complete $S^1$-equivariant Hahn--Raksit--Wilson even filtration, we argue as in the proof of \cite[Theorem~\chref{5.0.3}]{EvenFiltration}. Choose a surjection from a polynomial ring $\IZ[x_i\ |\ i\in I] \twoheadrightarrow S$. Both filtrations satisfy descent along the $p$-completely eff map $\THH(\IS[x_i\ |\ i\in I])\rightarrow \THH(\IS[x_i^{1/p^\infty}\ |\ i\in I])$. By descent, it will then be enough to check that the filtrations agree when $S$ is large, which is clear by evenness.
		
		After inverting~$u$, the argument is analogous: As in \cite[\S{\chref[section]{11.3}}]{BMS2}, we use quasi-syntomic descent again to construct a filtration
		\begin{equation*}
			\fil_{\BMSev}^\star\left(\TC^-\bigl(S/\IS_A\qpower\bigr)\bigl[\localise{u}\bigr]_{(p,q-1)}^\complete\right)
		\end{equation*}
		and check via descent along $\THH(\IS[x_i\ |\ i\in I])\rightarrow \THH(\IS[x_i^{1/p^\infty}\ |\ i\in I])$ that this filtration is really the Hahn--Raksit--Wilson even filtration. To see $\gr_{\HRWev,\h S^1}^*\simeq \Prism_{S/A\qpower}[u^{\pm 1}]$, observe that inverting the degree~$2$ class~$u$ amounts to adjoining $[p]_q^{-i} \fil_\Nn^i$ for all $i\geqslant 0$ in the sense of \cite[Construction~\chref{3.42}]{qWittHabiro}. We must then show that the relative Frobenius induces an equivalence
		\begin{equation*}
			\phi_{/A\qpower}\colon \widehat{\Prism}_{S/A\qpower}^{(p)}\left[\frac{\fil_\Nn^i}{[p]_q^i}\ \middle|\ i\geqslant 0\right]_{(p,q-1)}^\complete\overset{\simeq}{\longrightarrow} \Prism_{S/A\qpower}\,.
		\end{equation*}
		This is a general fact about the Nygaard filtration on prismatic cohomology; it follows, for example, from \cite[Theorem~\chref{15.2}(2)]{Prismatic} via quasi-syntomic descent. See also \cite[Lemma~\chref{3.44}]{qWittHabiro}.
	\end{proof}
	\begin{numpar}[Frobenii.]\label{par:qdeRhamFrobenii}
		The same argument as in \cite[Proposition~\chref{11.10}]{BMS2} shows that the $p$-cyclotomic Frobenius
		\begin{equation*}
			\phi_p^{\h S^1}\colon\TC^-\bigl(\IZ_p[\zeta_p]/\IS\qpower\bigr)_p^\complete\longrightarrow\TP\bigl(\IZ_p[\zeta_p]/\IS\qpower\bigr)_p^\complete
		\end{equation*}
		inverts the generator~$u$ in degree~$2$. Moreover, the $p$-cyclotomic Frobenius on $\THH(-/\IS_{A,p}\qpower)$ is semilinear with respect to the Tate-valued Frobenius $\phi_{\t C_p}\colon \IS_{A,p}[q]\rightarrow \IS_{A,p}[q]$, which on $\pi_0$ is given by $\phi\colon A\rightarrow A$ and $q\mapsto q^p$. It follows that the $p$-cyclotomic Frobenius induces a map
		\begin{equation*}
			\left(\TC^-\bigl(S/\IS_A\qpower\bigr)\bigl[\localise{u}\bigr]\otimes_{\IS_A[q],\phi_{\t C_p}}\IS_A[q]\right)_{(p,q-1)}^\complete\longrightarrow\TP\bigl(S/\IS_A\qpower\bigr)_p^\complete\,.
		\end{equation*}
		On $\gr_{\HRWev}^0$, this map agrees with the Nygaard completion $\Prism_{S/A\qpower}^{(p)}\rightarrow\widehat{\Prism}_{S/A\qpower}^{(p)}$, as the proof of \cref{prop:qdeRhamTC-} shows. The relative Frobenius on prismatic cohomology,
		\begin{equation*}
			\phi_{/A\qpower}\colon \widehat{\Prism}_{S/A\qpower}^{(p)}\longrightarrow \Prism_{S/A\qpower}\,,
		\end{equation*}
		can then be identified as the composition of $\gr_{\HRWev}^0\TP\simeq \gr_{\HRWev}^0\TC^-$ with
		\begin{equation*}
			\gr_{\HRWev}^0\TC^-\bigl(S/\IS_A\qpower\bigr)_p^\complete\longrightarrow \gr_{\HRWev}^0\left(\TC^-\bigl(S/\IS_A\qpower\bigr)\bigl[\localise{u}\bigr]_{(p,q-1)}^\complete\right)\,.
		\end{equation*}
	\end{numpar}
	\begin{numpar}[Recovering $q$-de Rham cohomology.]\label{par:qdeRhamViaTC-}
		Let $R$ be a $p$-torsion free $p$-quasi-lci $A$-algebra and let $R^{(p)}\coloneqq (R\lotimes_{A,\phi}A)_p^\complete$. Then \cite[Theorem~\chref{16.18}]{Prismatic} shows
		\begin{equation*}
			\qdeRham_{R/A}\simeq \Prism_{R^{(p)}[\zeta_p]/A\qpower}\,.
		\end{equation*}
		Therefore \cref{prop:qdeRhamTC-} and \cref{par:qdeRhamFrobenii} contain $q$-de Rham cohomology (which is implicitly $p$-completed by our convention in \cref{par:Notation}) as a special case.
	\end{numpar}
	
	\begin{numpar}[The Adams action.]\label{par:AdamsAction}
		In \cite[\S{\chref[subsection]{3.8}}]{BhattLurieI}, Bhatt--Lurie describe an action of $\IZ_p^\times$ on the $q$-de Rham prism $(A\qpower,[p]_q)$, where $u\in\IZ_p^\times$ acts by sending $q\mapsto q^u$. Here $q^u$ denotes the convergent power series
		\begin{equation*}
			q^u\coloneqq \sum_{n\geqslant 0}\binom{u}{n}(q-1)^n\,.
		\end{equation*}
		By functoriality of prismatic cohomology, the action on the prism induces an action of $\IZ_p^\times$ on $\qdeRham_{R/A}$, which is precisely the action predicted in \cite[Conjecture~\chref{6.2}]{Toulouse}.
		
		Under the identification
		\begin{equation*}
			\qdeRham_{R/A}\simeq \gr_{\HRWev}^0\left(\TC^-\bigl(R^{(p)}[\zeta_p]/\IS_A\qpower\bigr)\bigl[\localise{u}\bigr]_{(p,q-1)}^\complete\right)\,,
		\end{equation*}
		this action comes from an action of $\IZ_p^\times$ on $\IS_A\qpower$. Indeed, following \cite[Notation~\chref{3.3.3}]{DevalapurkarRaksitTHH}, we can write
		\begin{equation*}
			\IS_A\qpower\simeq \limit_{\alpha\geqslant 0}\IS_A[q]/\bigl(q^{p^\alpha}-1\bigr)\simeq \limit_{\alpha\geqslant 0}\IS_A[\IZ/p^\alpha]
		\end{equation*}
		and then let $\IZ_p^\times$ act on $\IZ/p^\alpha$ via multiplication (this is another way of making precise what $q^u$ is supposed to mean). To see that this induces the same action on $(\qdeRham_{R/A})_p^\complete$ as above, we can use quasi-syntomic descent as in the proof of \cref{prop:qdeRhamTC-} to reduce to an even situation, where the claim is straightforward to verify.
		
		We call this action the \emph{Adams action}, since it turns out to agree with the action of $\IZ_p^\times$ on $\ku_p^\complete$ via Adams operations (see \cref{subsec:qdeRhamkupComplete}).
	\end{numpar}
	
	\newpage

	\section{Even \texorpdfstring{$\IE_2$}{E2}-cell structures on flat polynomial rings}\label{appendix:E2}
	
	In this appendix we show the following technical result.
	
	\begin{lem}\label{lem:E2CellStructure}
		Let $\IS[x_i\ |\ i\in I]$ be the flat graded polynomial ring on generators $x_i$ in graded degree~$1$ and homotopical degree~$0$. As a graded $\IE_2$-ring, $\IS[x_i\ |\ i\in I]$ admits a cell decomposition with all cells in even homotopical degree.
	\end{lem}
	\begin{rem}
		For polynomial rings in one variable this is shown in \cite[Proposition~\chref{3.11}]{AdjoiningRootsToE2}. We believe the argument given there can be adapted to several variables as well. The authors of that paper also remark that an alternative proof of the one-variable case is given in the second (but not in the final) arXiv version of \cite{BPRedshift}; we'll follow the proof given therein.
	\end{rem}
	\begin{proof}[Proof of \cref{lem:E2CellStructure}]
		To avoid issues with double duals of infinite direct sums, we work in the $\infty$-category of graded solid condensed spectra $\Gr(\Sp_\solid)$. Usual graded spectra embed fully faithfully as the full sub-$\infty$-category of graded discrete solid condensed spectra. We let 
		\begin{equation*}
			\ID^{(2)}\coloneqq \Hhom_{\Gr(\Sp_\solid)}\bigl(\operatorname{Bar}^{(2)}(-),\IS\bigr)\colon \Alg_{\IE_2}\bigl(\Gr(\Sp_\solid)\bigr)\longrightarrow \Alg_{\IE_2}\bigl(\Gr(\Sp_\solid)\bigr)
		\end{equation*}
		denote the $\IE_2$-Koszul duality functor.
		
		Let us first compute $D\coloneqq \ID^{(2)}(\IS[x_i\ |\ i\in I])$. A standard computation shows that the double Bar construction $\operatorname{Bar}^{(2)}(\IS[x_i])$ is given by $\bigoplus_{n\geqslant 0}\Sigma^{2n}\IS(n)$ as a graded spectrum. Thus, if $I_n\coloneqq \Sym^nI$ denotes the $n$\textsuperscript{th} symmetric power of $I$ as a set, then
		\begin{equation*}
			D\simeq \bigoplus_{n\geqslant 0}\Sigma^{-2n}\prod_{I_n}\IS(-n)\,.
		\end{equation*}
		If $D_{\geqslant -n}$ denotes the restriction of $D$ to graded degrees $\geqslant-n$, then $D$ is the limit of the tower of square-zero extensions $\dotsb\rightarrow D_{\geqslant -2}\rightarrow D_{\geqslant -1}\rightarrow D_{\geqslant 0}$. For all $n\geqslant 1$, the square-zero extension $D_{\geqslant -n}\rightarrow D_{\geqslant -(n-1)}$ is determined by a pullback diagram
		\begin{equation*}
			\begin{tikzcd}
				D_{\geqslant -n}\rar\dar\drar[pullback] & \IS\dar\\
				D_{\geqslant -(n-1)}\rar & \IS\oplus\Sigma^{-2n+1}\prod_{I_n}\IS(-n)
			\end{tikzcd}
		\end{equation*}
		After applying the Koszul duality functor, this becomes a pushout diagram
		\begin{equation*}
			\begin{tikzcd}
				\operatorname{Free}_{\IE_2}\biggl(\Sigma^{2n+1}\bigoplus_{I_n}\IS(n)\biggr)\rar\dar\drar[pushout] & \ID^{(2)}\bigl(D_{\geqslant -(n-1)}\bigr)\dar\\
				\IS\rar & \ID^{(2)}\bigl(D_{\geqslant -n}\bigr)
			\end{tikzcd}
		\end{equation*}
		Here we use $\Hom_{\Sp_\solid}(\prod_{I_n}\IS,\IS)\simeq \bigoplus_{I_n}\IS$; this is the advantage of working in solid condensed spectra. Taking the colimit, we see that $\ID^{(2)}(D)$ has an $\IE_2$-cell decomposition with cells in even homotopical degrees. Once again using that we're working in the solid condensed world, we find $\ID^{(2)}(D)\simeq \IS[x_i\ |\ i\in I]$ and so we're done.
	\end{proof}

	\newpage
	
	\section{On the equivariant Snaith theorem}\label{appendix:EquivariantSnaith}
	
	For abelian compact Lie groups, Spitzweck and \O{}stv\ae{}r \cite{EquivariantSnaith} show a genuine equivariant form of Snaith's theorem. However, the equivalence they construct is only one of homotopy ring spectra. In this short appendix, we explain how to make their equivalence $\IE_\infty$-algebras. We'll restrict to $S^1$ for simplicity, but the argument would work for any abelian compact Lie group.
	
	\begin{con}
		In \cite[(\chref{2.3.20})]{SchwedeGlobal} Schwede introduces an orthogonal space $\IP^\IC$ that sends an inner product space~$V$ to the infinite projective space $\IP(\Sym^*_\IC V_\IC)$. We can construct a morphism of orthogonal spaces
		\begin{equation*}
			c\colon \IP^\IC\longrightarrow \Omega^\bullet\ku_{\mathrm{gl}}
		\end{equation*}
		using a similar construction as in \cite[Construction~\chref{6.3.24}]{SchwedeGlobal}: Namely, for any inner product space $V$, the required map $c(V)\colon \IP(\Sym^*_\IC V_\IC)\rightarrow \operatorname{Map}_*(S^V,\ku_{\mathrm{gl}}(V))$ is adjoint to the tautological map $\IP(\Sym^*_\IC V_\IC)\wedge S^V\rightarrow \ku_{\mathrm{gl}}(V)$ that sends $(L,v)\mapsto [L;v]$ for any line $L\subseteq \Sym_\IC^*V_\IC$ and any point $v\in S^V$.
		
		Schwede equips $\IP^\IC$ with an ultracommutative monoid structure by sending a pair of lines $(L_1\subseteq \Sym_\IC^*V_\IC,L_2\subseteq \Sym_\IC^*W_\IC)$ to $L_1\otimes_\IC L_2\subseteq \Sym_\IC^*V_\IC\otimes_\IC\Sym_\IC^*W_\IC\cong \Sym_\IC^*(V\oplus W)_\IC$. It's clear from the construction that $c$ is multiplicative. Thus, by adjunction, it induces a map of ultracommutative global ring spectra
		\begin{equation*}
			\IS_{\mathrm{gl}}[\IP^\IC]\longrightarrow \ku_{\mathrm{gl}}\,.
		\end{equation*}
	\end{con}
	Before we continue, let us deduce that the element $q\in \pi_0(\ku^{S^1})$ is \emph{strict}.
	
	\begin{cor}\label{cor:qStrictElement}
		Let $q\in \pi_0(\ku^{S^1})$ be the image of the standard representation of $S^1$ under $\operatorname{RU}(S^1)\rightarrow \pi_0(\ku^{S^1})$. Then $q$ is detected by an $\IE_\infty$-algebra map
		\begin{equation*}
			\IS_{S^1}[q]\longrightarrow \ku_{S^1}
		\end{equation*}
		in $\Sp_{S^1}$. In particular, $q$ is a strict element in $(\ku^{C_m})^{\h (S^1/C_m)}$ for all~$m$.
	\end{cor}
	\begin{proof}
		By \cite[Proposition~\chref{4.1.8}]{SchwedeGlobal} (plus a simple argument to get rid of the telescope), the restriction of $\IS_\mathrm{gl}[\IP^\IC]$ to a genuine $S^1$-equivariant ring spectrum is given by $\IS_{S^1}[\IP^\IC]$, where $\Uu$ is any complete complex $S^1$-universe, that is, a direct sum of countably many copies of each irreducible complex $S^1$-representation. Choosing any copy of the standard representation $q$ inside $\Uu$, we get a $\IC$-algebra map $\IC\oplus q\oplus q^2\oplus\dotsb\rightarrow \Sym^*\Uu$, which induces an $S^1$-equivariant monoid map $\{1,q,q^2,\dotsc\}\simeq \IP(\IC)\sqcup \IP(q)\sqcup\IP(q^2)\sqcup\dotsb\rightarrow \IP(\Sym^*\Uu)$ and thus the desired map of $\IE_\infty$-algebras in $\Sp_{S^1}$
		\begin{equation*}
			\IS_{S^1}[q]\longrightarrow \IS_{S^1}[\IP^\IC]\longrightarrow \ku_{S^1}\,.\qedhere
		\end{equation*}
	\end{proof}

	\begin{numpar}[The Bott element.]\label{par:EquivariantBott}
		Let $\Uu$ be a complete complex $S^1$-universe as in the proof above. Let $\epsilon$ denote any copy of the trivial representation inside $\Uu$. The inclusion $\IC\oplus \epsilon\subseteq \Sym_\IC^*\Uu$, where $\IC$ denotes the unit component of the symmetric algebra, defines a map of genuine $S^1$-equivariant spectra $\xi\colon \IS_{S^1}[\IP(\IC\oplus\epsilon)]\rightarrow \IS_{S^1}[\IP(\Sym_\IC^*\Uu)]$. When we restrict to $\IS_{S^1}\simeq \IS_{S^1}[\IP(\IC)]$ in source and target, $\xi$ is canonically the identity, and so we can construct the \emph{Bott map} as the factorisation
		\begin{equation*}
			\begin{tikzcd}
				\IS_{S^1}\oplus \Sigma^2\IS_{S^1}\dar["\simeq"']\rar & \Sigma^2\IS_{S^1}\dar[dashed,"\beta"]\\
				\IS_{S^1}\bigl[\IP(\IC\oplus \epsilon)\bigr]\rar["1-\xi"] & \IS_{S^1}\bigl[\IP(\Sym_\IC^*\Uu)\bigr]
			\end{tikzcd}
		\end{equation*}
		It's clear from the construction that the $\IE_\infty$-map $\IS_{S^1}[\IP(\Sym_\IC^*\Uu)]\rightarrow \ku_{S^1}$, that was constructed in the proof of \cref{cor:qStrictElement}, sends $\beta\mapsto \beta$.
		
		We also note that if $\epsilon'$ is another copy of the trivial representation inside $\Uu$, then the map $\IS_{S^1}[\IP(\epsilon\oplus\epsilon')]\rightarrow \IS_{S^1}[\IP(\Sym_\IC^*\Uu)]$ is homotopic to $\beta$. Indeed, already the inclusions of $\IP(\IC\oplus\epsilon)$ and $\IP(\epsilon\oplus\epsilon')$ into $\IP(\IC\oplus \epsilon\oplus\epsilon')$ are $S^1$-equivariantly homotopic. It follows that $\beta$ already factors through the map $\IS_{S^1}[\IP(\Uu)]\rightarrow \IS_{S^1}[\IP(\Sym_\IC^*\Uu)]$ induced by $\Uu\cong \Sym_\IC^1\Uu\subseteq\Sym_\IC^*\Uu$. Finally, recall that Spitzweck and \O{}stv\ae{}r construct a homotopy ring spectrum structure on $\IS_{S^1}[\IP(\Uu)]$, so that we can consider the localisation $\IS_{S^1}[\IP(\Uu)][\beta^{-1}]$.
	\end{numpar}
	\begin{lem}\label{lem:EquivariantSnaith}
		The induced map of $\IE_\infty$-algebras in $\Sp_{S^1}$
		\begin{equation*}
			\IS_{S^1}\bigl[\IP(\Sym_\IC^*\Uu)\bigr][\beta^{-1}]\overset{\simeq}{\longrightarrow}\KU_{S^1}
		\end{equation*}
		is an equivalence. Moreover, its precomposition with $\IS_{S^1}[\IP(\Uu)][\beta^{-1}]\rightarrow \IS_{S^1}[\IP(\Sym_\IC^*\Uu)][\beta^{-1}]$ is the equivalence constructed in \cite{EquivariantSnaith}.
	\end{lem}
	\begin{proof}
		Since $\IS_{S^1}[\IP(\Uu)]\rightarrow \IS_{S^1}[\IP(\Sym_\IC^*\Uu)]$ is an equivalence as both $\Uu$ and $\Sym_\IC^*\Uu$ are complete complex $S^1$-universes, it will be enough to show the second statement.
		
		To this end, let $\Gr^\IC$ be the orthogonal space from \cite[Example~\chref{2.3.16}]{SchwedeGlobal} that sends an inner product space $V$ to $\coprod_{i\geqslant 0}\Gr_i^\IC(V_\IC)$, where $\Gr_i^\IC$ denotes the Grassmannian of $i$-dimensional complex subspaces. Let $\Gr_1^\IC\rightarrow \Gr^\IC$ be the component where $i=1$. Using \cite[Proposition~\chref{4.1.8}]{SchwedeGlobal} (plus a simple argument to get rid of the telescope), we see that $\IS_{S^1}[\IP(\Uu)]$ is the restriction of the global spectrum $\IS_\mathrm{gl}[\Gr_1^\IC]$ to a genuine $S^1$-equivariant spectrum. By unravelling the proof of \cref{cor:qStrictElement}, we immediately see that the diagram
		\begin{equation*}
			\begin{tikzcd}
				\IS_\mathrm{gl}[\Gr_1^\IC]\rar\dar & \IS_\mathrm{gl}[\IP^\IC]\dar\\
				\IS_\mathrm{gl}[\Gr^\IC]\rar & \ku_\mathrm{gl}
			\end{tikzcd}
		\end{equation*}
		commutes, where the bottom map is the adjoint of \cite[Construction~\chref{6.3.24}]{SchwedeGlobal}. By another straightforward unravelling, the composition $\IS_\mathrm{gl}[\Gr_1^\IC]\rightarrow\IS_\mathrm{gl}[\Gr^\IC]\rightarrow\ku_\mathrm{gl}$ restricts to the map $\IS_{S^1}[\IP(\Uu)]\rightarrow \ku_{S^1}$ constructed in \cite{EquivariantSnaith}.
	\end{proof}
	\begin{numpar}[Equivariant Adams operations]\label{par:EquivariantAdamsOperations}
		Let $\rho_n$ denote the $n$\textsuperscript{th} power map $(-)^n\colon S^1\rightarrow S^1$. Writing the monoid operation multiplicatively, we also consider the monoid endomorphism $(-)^n\colon \IP(\Sym_\IC^*\Uu)\rightarrow\IP(\Sym_\IC^*\Uu)$. This is equivariant over $\rho_n$ and therefore induces an endomorphism
		\begin{equation*}
			\psi^n\colon \rho_n^*\IS_{S^1}\bigl[\IP(\Sym_\IC^*\Uu)\bigr]\longrightarrow\IS_{S^1}\bigl[\IP(\Sym_\IC^*\Uu)\bigr]
		\end{equation*}
		of $\IE_\infty$-algebras in $S^1$-equivariant spectra. Clearly $\psi^n(q)=q^n$. Moreover, $\psi^n(\beta)=n\beta$ holds $S^1$-equivariantly. Indeed, to see this, let $\Uu_{\mathrm{triv}}\subseteq \Uu$ be the direct summand consisting of all copies of the trivial $S^1$-representation. Then the usual non-equivariant argument can be applied to $\IS_{S^1}[\IP(\Sym_\IC^*\Uu_{\mathrm{triv}})]$. Inverting $\beta$ and passing to connected covers, we obtain maps
		\begin{equation*}
			\psi^n\colon \KU_{S^1}\longrightarrow \KU_{S^1}\bigl[\localise{n}\bigr]\quad\text{and}\quad \psi^n\colon \ku_{S^1}\longrightarrow \ku_{S^1}\bigl[\localise{n}\bigr]
		\end{equation*}
		of $\IE_\infty$-algebras in $\Sp_{S^1}$. Here we also use $\rho_n^*\ku_{S^1}\simeq \ku_{S^1}$ and $\rho_n^*\ku_{S^1}\simeq \ku_{S^1}$, since we've modelled $\ku$ by an ultracommutative global ring spectrum $\ku_\mathrm{gl}$, where everything acts trivially.
	\end{numpar}

	\newpage 
	\renewcommand{\SectionPrefix}{}
	
	\renewcommand{\bibfont}{\small}
	\printbibliography
\end{document}